\documentclass[reqno,11pt]{amsart}
\UseRawInputEncoding
\pdfoutput=1
\usepackage{amsfonts}
\usepackage{graphicx}
\usepackage{amsmath,amssymb,amsfonts,amsthm,mathrsfs}
\usepackage[english]{babel}
\usepackage{graphicx}
\usepackage{subfigure}
\usepackage{float}
\usepackage{tikz}
\usepackage{marginnote}
\usepackage{color}
\usepackage{enumerate}

\newcommand{\dis}{\displaystyle}

\numberwithin{equation}{section}
\newtheorem{Theorem}{Theorem}[section]
\newtheorem{Lemma}[Theorem]{Lemma}
\newtheorem{Corollary}{Corollary}[section]
\theoremstyle{plian}
\newtheorem{Definition}{Definition}[section]
\newtheorem{remark}[Theorem]{Remark}
\newtheorem{Proposition}[Theorem]{Proposition}

\newcommand{\R}{\mathbb{R}}
\newcommand{\FP}{\mathbf{P}}
\newcommand{\FL}{\mathbf{L}}
\newcommand{\FI}{\mathbf{I}}

\newcommand{\mf}{\mathbf{f}}

\newcommand{\mv}{\mathfrak{v}}
\newcommand{\fu}{\mathfrak{u}}
\newcommand{\bp}{\partial_{t,\phi}}

\def\v{\varepsilon}
\def\f{\frac}
\def\d{\delta}

\newcounter{RomanNumber}

\allowdisplaybreaks[4]
\begin{document}

	\title [Hilbert Expansion of BE in 2D Disk] {Hilbert expansion of the Boltzmann equation on a 2-dimensional disk with specular boundary condition}
	
	\author[F.M. Huang]{Feimin Huang}
	\address[F.M. Huang]{State Key Laboratory of Mathematical Sciences, Academy of Mathematics and Systems Science, Chinese Academy of Sciences, Beijing 100190, China; School of Mathematical Sciences, University of Chinese Academy of Sciences, Beijing 100049, China}
	\email{fhuang@amt.ac.cn}
	
	\author[J. Ouyang]{Jing Ouyang}
	\address[J. Ouyang]{Academy of Mathematics and Systems Science, Chinese Academy of Sciences, Beijing 100190, China; School of Mathematical Sciences, University of Chinese Academy of Sciences, Beijing 100049, China}
	\email{ouyangjing@amss.ac.cn}
	
	\author[Y. Wang]{Yong Wang}
	\address[Y. Wang]{State Key Laboratory of Mathematical Sciences, Academy of Mathematics and Systems Science, Chinese Academy of Sciences, Beijing 100190, China; School of Mathematical Sciences, University of Chinese Academy of Sciences, Beijing 100049, China}
	\email{yongwang@amss.ac.cn}
	
\begin{abstract}
%The study of the Boltzmann equation with curved boundary is both important and challenging. This article mainly establishes the hydrodynamic limit of  the Boltzmann equation with specular reflection boundary in a two-dimensional disk to the compressible Euler equation. 
%In practical problems, the boundaries of objects are often not flat but with non-zero curvature. 
In the present paper, we concern the hydrodynamic limit of Boltzmann equation with specular reflection boundary condition in a two-dimensional  disk  to the compressible Euler equations.
%We systematically construct the internal layer, the viscous boundary layer, and the Knudsen boundary layer. 
Due to the non-zero curvature and non-zero tangential velocity of compressible Euler solution on the boundary, new difficulties arise in the construction of Knudsen boundary layer.
By employing the geometric correction, and  an innovative and refined $L^2-L^\infty$ method, we establish the existence and space-decay for a truncated Knudsen boundary layer.
Then, by the Hilbert expansion of multi-scales, we successfully justify the hydrodynamic limit of Boltzmann equation with specular reflection boundary condition to the compressible Euler equations in the two-dimensional  disk.
\end{abstract}
	
	
	\date{\today}
	\maketitle
	\setcounter{tocdepth}{1}
	\tableofcontents

\section{Introduction}\label{sec1}
	
\subsection{Background}
%The Boltzmann equation is a partial differential equation describing the statistical behavior of thermodynamic systems in non-thermal equilibrium states.
The Boltzmann equation was proposed by Boltzmann \cite{Boltzmann-1872} in 1872 to study the motion and interaction of particles, especially in non-equilibrium conditions. The seminal work by Maxwell \cite{Maxwell} and Boltzmann \cite{Boltzmann-1872} established the intrinsic connection between the Boltzmann equation and fluid dynamics, including both compressible and incompressible flow systems.
	
\smallskip
	
The hydrodynamic limit of the Boltzmann equation has been a central topic in the kinetic theory. In 1912, Hilbert introduced a systematic asymptotic expansion for the Knudsen number $\mathscr{K}_n\ll1$, which significantly advanced the study of the Boltzmann equation. Later, Chapman and Enskog independently proposed a different expansion methods in 1916 and 1917. These expansion techniques have successfully derived various fluid equations, including the compressible Euler and Navier-Stokes systems, the incompressible Euler and Navier-Stokes(-Fourier) systems, as well as the acoustic system, among others. 
%However, the rigorous mathematical validation of these approximate formulations remains a significant challenge.
It is still a challenging problem to rigorously justify these formal approximation, especially for the problems with physical boundaries.
	
	\smallskip

Based on the truncated Hilbert expansion, Caflish \cite{Caflish} rigorously justified the hydrodynamic limit from the Cauchy solutions of Boltzmann equation to the compressible Euler equations under the framework of smooth solutions, see also \cite{Grad,Lachowicz,Nishida,Ukai-Asano} and \cite{Guo Jang Jiang-1,Guo-2010-CPAM}  via the $L^2$-$L^\infty$ method. The hydrodynamic limit of Boltzmann to wave patterns have been proved \cite{HJWY,HWY,HWY2,HWWY,XZ,Yu} in one dimensional case, see also \cite{WWZ} for planar rarefaction wave in $\mathbb{R}^3$. Studies on the hydrodynamic limit to acoustic equations can be found in \cite{ Diper,Guo-2010-CPAM,GHW-2021-ARMA,Jang-Jiang} and related works.
For the derivation of incompressible Navier-Stokes system and incompressible Euler equations, we refer to \cite{Bardos, Bardos-2,Bardos-Ukai,CJK-2023-JDE,CEMP,Diper,EGMW,ELM-1994,E-Guo-M,E-Guo-K-M,Golse-Saint-Raymond,Guo2006,Jang-Kim,Jiang-Masmoudi,Liu-Yu,Masmoudi-Raymond,Saint-Raymond2009,Wu-2016-JDE} and references cited therein. 
	
\smallskip
	
Most studies have been conducted in $\R^3$ or $\mathbb{T}^3$, the investigation of the Boltzmann equation with boundary condition remains a crucial yet challenging task in the field of fluid kinetic theory. This is primarily due to the intricate dependence of boundary condition on the solution structure and the added complexity in establishing stable solutions.
	
\smallskip
	
In 2010, Guo \cite{Guo-2010-ARMA} utilized the $L^2-L^\infty$ method to demonstrate the exponential decay of solutions to the Boltzmann equation in a smooth domain with one of four common boundary conditions, see also \cite{Guo-2016-ARMA,Guo-2017-Inv} for the regularities in strictly convex domain and non-convex domain. In 2021, Guo-Huang-Wang \cite{GHW-2021-ARMA} established the validity of the Hilbert expansion for the Boltzmann equation with specular reflection boundary condition, which leads to derivations of compressible Euler equations and acoustic equations in half-space. 
Later, Jiang-Luo-Tang \cite{Jiang-2024-Trans} derived the hydrodynamic limit from the Boltzmann equation with diffuse reflection boundary condition to the compressible Euler equations in half-space. 
	
\smallskip
	
In 2016, by introducing a geometrically modified Knudsen boundary layer, Wu \cite{Wu-2016-JDE} justified the hydrodynamic limit of the Boltzmann equation with an incoming boundary condition to the incompressible Navier-Stokes equations in a two-dimensional disk, see very recent results \cite{Wu-2020-Arxiv,Wu-2023-arxiv}  on general smooth domains. We remark that the geometric correction was firstly introduced in the study of neutral transport equation \cite{GW-2015}. %However, his study was limited to low-order expansions and assumed that the macroscopic tangential velocity on the boundary was set to zero. 
	
%在郭黄王的工作中，尽管切向速度不为零，但是由于没有几何的作用，也不会出现本文中的困难项。(这个也需要指出)
	
\smallskip
	
We point out that there is no geometric influence in \cite{GHW-2021-ARMA,Jiang-2024-Trans} for the compressible Euler limit of Boltzmann equation since the boundary of half-space  is flat. In this paper, we mainly concern the hydrodynamic limit of Boltzmann equation with specular reflection boundary condition to the compressible Euler equations within a two-dimensional disk, and  the interaction between non-zero boundary curvature and non-zero tangential flow velocity of compressible Euler equations introduces new challenges in analyzing Knudsen boundary layer problems. 
	
	\smallskip
	
In this paper, we consider the  Boltzmann equation
	\begin{equation}\label{BE}
		F_t+v\cdot\nabla_xF=\frac{1}{\mathscr{K}_n}Q(F,F),
	\end{equation}
	where $F(t,x,v)$ is the density distribution function for the gas particles with position $x\in B_1:=\{x\in\R^2:|x|\leq 1\}$ and velocity $v\in\R^2$ at time $t\geq 0$. The Boltzmann collision term $Q(F,G)$ is defined by
	\begin{align}\label{1.2}
		Q(F,G)=\int_{\mathbb{S}^1}\int_{\mathbb{R}^2}B(|v-u|,\omega)[F(v')G(u')-F(v)G(u)]dud\omega,
	\end{align}
	where the relationship between the post-collision velocity $v',u'$ of two particles with the pre-collision velocity $v,u$ is given by
	\begin{align*}
		u'=u+[(v-u)\cdot \omega]\cdot \omega,\quad v'=v-[(v-u)\cdot \omega]\cdot \omega,
	\end{align*}
	for $\omega\in \mathbb{S}^1$, which can be determined by conservation laws of momentum and energy
	\begin{align*}
		v'+u'=v+u,\quad |v'|^2+|u'|^2=|v|^2+|u|^2.
	\end{align*}

	In the present paper, we consider hard sphere case, i.e.
	\begin{align}
		B(|v-u|,\omega)=|(v-u)\cdot\omega|.
	\end{align} 
	For later use, we divide the boundary $\gamma:=\partial B_1\times\mathbb{R}^2$  into incoming boundary $\gamma_-$, outgoing boundary $\gamma_+$ and grazing set $\gamma_0$:
	\begin{align}
		&\gamma_-=\{(x,v)\in\partial B_1\times\mathbb{R}^2:v\cdot \vec{n} \textless0\},\nonumber\\
		&\gamma_+=\{(x,v)\in\partial B_1\times\mathbb{R}^2:v\cdot \vec{n} \textgreater0\},\\
		&\gamma_0=\{(x,v)\in\partial B_1\times\mathbb{R}^2:v\cdot \vec{n} =0\},\nonumber
	\end{align}
	where $\vec{n}$ is the outgoing normal vector of $B_1$. In the present paper, we adopt the specular reflection boundary condition, i.e.
	\begin{align}\label{1.5}
		F(t,x,v)|_{\gamma_-}=F(t,x,R_xv),\quad \text{ where } R_xv=v-2(v\cdot \vec{n})\vec{n}.
	\end{align}

	\subsection{Formal Hilbert expansion} \label{sec1.2}
	Inspired by \cite{GHW-2021-ARMA}, we know that the thickness of viscous boundary layer is $\sqrt{\mathscr{K}_n}$.
	For later use, we denote $\v:=\sqrt{\mathscr{K}_n}$ and $F$ as $F^\v$.  Then \eqref{BE} is rewritten as
	\begin{align}\label{1.6}
		\partial_tF^\v+v\cdot\nabla_xF^\v=\frac{1}{\v^2}Q(F^\v,F^\v).
	\end{align}
	\subsubsection{Formal internal expansion}
	Define the internal expansion by $F^\v\sim \sum\limits_{k=0}^\infty\v^kF_k(t,x,v)$, then one has 
	\begin{align}\label{1.7}
		\sum\limits_{k=0}^\infty\v^k(\partial_tF_k+v\cdot\nabla_xF_k)&=\frac{1}{\v^2}Q(\sum\limits_{k=0}^\infty\v^kF_k,\sum\limits_{k=0}^\infty\v^kF_k)\nonumber\\
		&=\sum\limits_{l=0}^\infty\sum\limits_{i,j\geq0,\atop i+j=l}\v^{l-2}Q(F_i,F_j).
	\end{align}
	Comparing the coefficients to each order of $\v$, one  obtains the following equations:
	\begin{align}\label{ie}
		\begin{split}
		&\v^{-2}:\quad0=Q(F_0,F_0),\\
		&\v^{-1}:\quad0=Q(F_0,F_1)+Q(F_1,F_0),\\
		&\v^0:\quad \{\partial_t+v\cdot\nabla_x\}F_0=Q(F_0,F_2)+Q(F_2,F_0)+Q(F_1,F_1),\\
		&\vdots\\
		&\v^k:\quad\{\partial_t+v\cdot\nabla_x\}F_k=Q(F_0,F_{k+2})+Q(F_{k+2},F_0)+\sum\limits_{i,j\geq1\atop i+j=k+2}Q(F_i,F_j), \quad \text{for}\, k\geq 0.
		\end{split}
	\end{align}
	It follows from $\eqref{ie}_1$ and the celebrated H-theorem that $F_0$ is a local Maxwellian:
	\begin{equation*}
		\mu(t,x,v):=F_0(t,x,v)\equiv\frac{\rho(t,x)}{2\pi T(t,x)}\exp\left\{-\frac{|v-\mathfrak{u}|^2}{2T(t,x)}\right\},
	\end{equation*}
	where %$\rho(t,x),\mathfrak{u}(t,x),T(t,x)$ are defined by
	\begin{align*}
		\rho=\int_{\mathbb{R}^2}F_0dv,\quad \rho \mathfrak{u} = \int_{\mathbb{R}^2}v\cdot F_0dv,\quad\rho|\mathfrak{u}|^2+2\rho T=\int_{\mathbb{R}^2}|v|^2F_0dv,
	\end{align*}
	which represent macroscopic density, velocity and temperature. \\
	
	Multiplying $\eqref{ie}_3$ by $1,v,|v|^2$, which are four collision invariants for the Boltzmann collision operator $Q(\cdot,\cdot)$, one obtains that $(\rho,\mathfrak{u},T)$ satisfies the compressible Euler system
	\begin{align}\label{1.9}
		\begin{cases}
			\dis
			\partial_t\rho+\text{div}(\rho \mathfrak{u}) = 0,\\
			\dis \partial_t(\rho \mathfrak{u})+\text{div}(\rho \mathfrak{u}\otimes \mathfrak{u})+\nabla(\rho T)=0,\\
			\dis \partial_t\left(\frac{\rho|\mathfrak{u}|^2+2\rho T}{2}\right)+\text{div}\left(\mathfrak{u}\cdot\frac{(\rho|\mathfrak{u}|^2+2\rho T)}{2}\right)+\text{div}(\mathfrak{u}\cdot\rho T)=0.
		\end{cases}
	\end{align}
	Usually, we denote $p:=\rho T$ to be the pressure function. For \eqref{1.9}, we consider the slip boundary
	\begin{align}\label{1.9-0}
		\mathfrak{u}\cdot \vec{n}=0,
	\end{align}  which satisfies the specular reflection boundary condition, i.e., $F_0(t,x,v)|_{\gamma_-}=F_0(t,x,R_xv)$. The initial data is given by 
	\begin{align}\label{1.9-1}
		(\rho,\mathfrak{u},T)(0,x)=(1+\delta \phi_0,\delta\Phi_0,1+\delta\vartheta_0)(x),
	\end{align}
	with $\|(\phi_0,\Phi_0,\vartheta_0)\|_{H^{s_0}}\leq 1$, where $0<\delta\ll1$ is a parameter and $s_0\geq 2$ is some given number. Then $1+\delta\phi_0, 1+\delta\vartheta_0$ are positive. There is a family of local-in-time classical solutions $(\rho^\delta,\mathfrak{u}^\delta,T^\delta)\in C([0,\tau^\delta];H^{s_0}(B_1))\cap C^1([0,\tau^\delta];H^{s_0-1-\d}(B_1))$ of the compressible Euler equations \eqref{1.9}--\eqref{1.9-1} such that $\rho^\d>0,T^\d>0$, here the life-span  $\tau^\d>0$ depends on $\delta^{-1}$; see Lemma \ref{lem3.1}.
	
	For later use, we define the linear collision operator $\mathbf{L}$ by
	\begin{equation}\label{1.11-1}
		\mathbf{L}g=-\frac{1}{\sqrt{\mu}}\{Q(\mu,\sqrt{\mu}g)+Q(\sqrt{\mu}g,\mu)\},
	\end{equation}
and  denote the null space of $\mathbf{L}$ as $\mathcal{N}$ which is generated by
\begin{align}\label{1.13-0}
\begin{split}
&\chi_0(v)=\frac{1}{\sqrt{\rho}}\sqrt{\mu},\\
&\chi_i(v)=\frac{v_i-\mathfrak{u}_i}{\sqrt{\rho T}}\sqrt{\mu},\quad i=1,2,\\
&\chi_3(v)=\frac{1}{2\sqrt{\rho}}\left\{\frac{|v-\mathfrak{u}|^2}{T}-2\right\}\sqrt{\mu}.
\end{split}
\end{align}
	
	For $k\geq1$, we define the macroscopic part and the microscopic part of $f_k:=\frac{F_k}{\sqrt{\mu}}$ as $\FP f_k$ and $(\FI-\FP) f_k$ where $\FP$ is the projection operator on $\mathcal{N}$. 	We denote 
	\begin{align*}
	\rho_k=\int_{\mathbb{R}^2}F_k\, dv,\quad \rho u_k=\int_{\mathbb{R}^2}(v-\mathfrak{u})\cdot F_k\, dv,\quad \rho\theta_k=\int_{\mathbb{R}^2}\left(|v-\mathfrak{u}|^2-2T\right)F_k\, dv.
	\end{align*}
then  we obtain from \eqref{ie} that
	\begin{align}\label{1.10-0}
		\frac{F_k}{\sqrt{\mu}}&=\mathbf{P}\left\{\frac{F_k}{\sqrt{\mu}}\right\}+\{\mathbf{I}-\mathbf{P}\}\Big\{\frac{F_k}{\sqrt{\mu}}\Big\}\nonumber\\
		&=\left\{\frac{\rho_k}{\sqrt{\rho}}\chi_0+\sum\limits_{j=1}^2\sqrt{\frac\rho T}u_{k,j}\cdot\chi_j+\frac{\sqrt{\rho}}{2}\frac{\theta_k}{T}\chi_3\right\}+\{\mathbf{I}-\mathbf{P}\}\Big\{\frac{F_k}{\sqrt{\mu}}\Big\}\nonumber\\
		&=\left\{\frac{\rho_k}{\rho}+\frac{u_k}{T}\cdot(v-\mathfrak{u})+\frac{\theta_k}{4T}\cdot\left(\frac{|v-\mathfrak{u}|^2}{T}-2\right)\right\}\sqrt{\mu}+\{\mathbf{I}-\mathbf{P}\}\Big\{\frac{F_k}{\sqrt{\mu}}\Big\},
	\end{align}
	with
	\begin{align}\label{1.10}
		\{\mathbf{I}-\mathbf{P}\}\Big\{\frac{F_{k+1}}{\sqrt{\mu}}\Big\}=\mathbf{L}^{-1}\left\{\frac{1}{\sqrt{\mu}}\left[-\{\partial_t+v\cdot\nabla_x\}F_{k-1}+\sum\limits_{ i,j\geq1,\atop i+j=k+1}Q(F_i,F_j)\right]\right\}.
	\end{align}
It is clear from \eqref{1.10} that $\{\mathbf{I}-\mathbf{P}\}\left\{\frac{F_{k}}{\sqrt{\mu}}\right\} $ is determined by $F_i,0\leq i\leq k-1$. Also it follows from $\eqref{ie}_2$ that $f_1\in\mathcal{N}$. 

Motivated by \cite{GHW-2021-ARMA}, we shall construct $F_k$ by solving a hyperbolic boundary value problem, see Sections 2.1 and \ref{sec3.1}.

%%%%%%%%%%%%%%%%%%%%%%%%%%%%%%	
\subsubsection{Formal viscous boundary layer expansion}
We remark that, in general, the internal solutions $F_0, F_1, ..., F_k$ in Section 1.2.1 do not satisfy the specular boundary condition \eqref{1.5}. Motivated by \cite{GHW-2021-ARMA,Sone}, we introduce the viscous boundary layers near $\partial B_1$. Since the viscous boundary layers change dramatically along the normal direction, it is necessary to use the polar coordinate, i.e.
	\begin{align*}
		x_1=r\cos\phi,\quad x_2=r\sin\phi,\quad (r,\phi)\in[0,1]\times[0,2\pi).
	\end{align*}
	Then we make change of variables $(r,\phi,v)\rightarrow(y,\phi,\bar{v})$ that 	
	\begin{align*}
		y=\frac{1-r}{\v},\quad \bar{v}=-v.
	\end{align*}
We remark that the scaled normal coordinate $y$ is essential for the viscous boundary layer. Then it is clear that
	\begin{align*}
		v\cdot\nabla_x=\frac1\v(\bar{v}\cdot \vec{n})\partial_y-\frac{1}{1-\v y}(\bar{v}\cdot\vec{\tau})\partial_\phi,
	\end{align*}
where $\vec{n}=(\cos\phi,\sin\phi),\vec{\tau}=(-\sin\phi,\cos\phi)$.

We consider the viscous boundary layer in the following form 
\begin{align}
\bar{F}^\v(t,y,\phi,\bar{v}) \sim\sum\limits_{k=1}^\infty  \v^k\bar{F}_k(t,y,\phi,\bar{v}).
\end{align}
Then substituting $F^\v+\bar{F}^\v$ into \eqref{1.6}, and using \eqref{1.7}, one obtains
%	\begin{align*}
%		(\partial_t+v\cdot \nabla_x)\bar{F}^\v=\f{1}{\v^2}[Q(F^\v,\bar{F}^\v)+Q(\bar{F}^\v,F^\v)+Q(\bar{F}^\v,\bar{F}^\v)],
%	\end{align*}
%	which is rewritten as
\begin{align}\label{1-12}
&\partial_t\bar{F}^{\v} +\frac1\v(\bar{v}\cdot \vec{n})\partial_y \bar{F}^{\v} -\frac{1}{1-\v y}(\bar{v}\cdot\vec{\tau})\partial_\phi \bar{F}^{\v} \nonumber\\
&=\frac{1}{\v^2}\big[Q(F^{\v},\bar{F}^{\v})+Q(\bar{F}^{\v},F^{\v})+Q(\bar{F}^{\v},\bar{F}^{\v})\big].
\end{align}
	
For later use, we need to write $F_k(t,x,v)$ to be $F_k(t,r,\bar{v})$ in the polar coordinate. Since the scale of normal variables of $F_k$ and $\bar{F}_k$ are different, we take Taylor expansion for $F_k$ at $r=1$:
	\begin{align}
		\begin{split}
		\mu(t,r,\phi,v) 
		&= \mu_0 + \sum\limits_{l=1}^{\mathfrak{b}} \frac{(-\v y)^l}{l!}\partial_r^l \mu_0+ \frac{(-\v y)^{\mathfrak{b}+1}}{(\mathfrak{b}+1)!} \partial_r^{\mathfrak{b}+1}\tilde{\mu},\\
		F_k(t,r,\phi,v) &= F_k^0 + \sum\limits_{l=1}^{\mathfrak{b}}\frac{(-\v y)^l}{l!}\partial_r^l F_k^0 + \frac{(-\v y)^{\mathfrak{b}+1}}{(\mathfrak{b}+1)!} \partial_r^{\mathfrak{b}+1}\mathcal{F}_k,\quad k\geq 1,
		\end{split}
	\end{align}
	where we have denoted $\partial_r^l\mu_0:=\partial_r^l\mu|_{r=1}$ and $\partial_r^lF_k^0:=\partial_r^lF_k|_{r=1}, l=0,1,...$ It is clear that
	\begin{align}\label{mu}
		\mu_0(t,\phi,\bar{v})=\f{\rho^0}{2\pi T^0}e^{-\f{|\bar{v}+\mathfrak{u}^0|^2}{2T^0}}=\f{\rho^0}{2\pi T^0}e^{-\f{|\bar{v}-\bar{\mathfrak{u}}^0|^2}{2T^0}},
	\end{align}
	where $(\rho^0,\mathfrak{u}^0,T^0)=(\rho,\mathfrak{u},T)(t,r,\phi)|_{r=1}$ and $\bar{\mathfrak{u}}^0=-\mathfrak{u}^0$.  The notations $\partial_r^{\mathfrak{b}+1}\tilde{\mu}$ and $\partial_r^{\mathfrak{b}+1}\mathcal{F}_k$ are the corresponding Taylor remainders. 
	
	%Since $\f{1}{1-\v y}$ in \eqref{1-12} has singularity at $y=\f{1}{\v}$, we further consider 	
	%\begin{align}\label{1-13}
	%	\partial_t\bar{F}^{\v} +\frac1\v(\bar{v}\cdot \vec{n})\partial_y \bar{F}^{\v} -\mathcal{H}^\v(y)\Upsilon(\v^{\f12}y)(\bar{v}\cdot\vec{\tau})\partial_\phi \bar{F}^{\v} &= \frac{1}{\v^2}\big[Q(F^{\v},\bar{F}^{\v})+Q(\bar{F}^{\v},F^{\v})+Q(\bar{F}^{\v},\bar{F}^{\v})\big].
	%\end{align}
	Comparing the order of $\v$ in \eqref{1-12}, one gets that
	\begin{align}\label{vs}
		\begin{split}
		\v^{-1}:&\quad0 = Q(\mu_0, \bar{F}_1) + Q(\bar{F}_1, \mu_0),\\
		\v^0:&\quad(\bar{v}\cdot \vec{n})\partial_y \bar{F}_1 = Q(\mu_0, \bar{F}_2) + Q(\bar{F}_2 ,\mu_0) + Q(F_1^0,\bar{F}_1)+Q(\bar{F}_1,F_1^0) \\
		&\quad+ Q(-y\partial_r \mu_0,\bar{F}_1) + Q(\bar{F}_1,-y\partial_r \mu_0) + Q(\bar{F}_1,\bar{F}_1),\\
		&\quad\vdots\\
		\v^k:&\quad(\bar{v}\cdot \vec{n})\partial_y \bar{F}_{k+1} = Q(\mu_0, \bar{F}_{k+2}) +Q(\bar{F}_{k+2},\mu_0)\\
		&\quad+ \sum\limits_{l,j\geq1\atop l+j=k+2} (-1)^l \frac{y^l}{l!}[Q(\partial_r^l\mu_0,\bar{F}_j) +Q(\bar{F}_j,\partial_r^l\mu_0)] + \sum\limits_{l,j\geq1\atop l+j=k+2} [Q(F_l^0,\bar{F}_j)+Q(\bar{F}_j, F_l^0)]\\
		&\quad+\sum\limits_{i,l,j\geq1\atop i+l+j=k+2}(-1)^l \frac{y^l}{l!}[Q(\partial_r^lF_i^0,\bar{F}_j) + Q(\bar{F}_j,\partial_r^lF_i^0)] +\sum\limits_{i,j\geq1\atop i+j=k+2} Q(\bar{F}_j,\bar{F}_i)\\
		&\quad-\partial_t \bar{F}_k +\f{1}{1-\v y}(\bar{v}\cdot \vec{\tau})\partial_\phi \bar{F}_k. 
		\end{split}
	\end{align}
	
	Since $\bar{F}_k$ is mainly used to compensate the boundary condition, it is natural to assume % the zero far-field condition 
	\begin{align}\label{vs-1}
		\lim\limits_{y\rightarrow\infty}\bar{F}_k=0.
	\end{align}
	Noting that $\f{1}{1-\v y}$ has singularity at $y=\v^{-1}$, we define $\bar{\mathfrak{F}}_k$ to be the solution of the following approximate problem
	\begin{align}\label{1.16}
		&(\bar{v}\cdot \vec{n})\partial_y \bar{\mathfrak{F}}_{k+1} = Q(\mu_0, \bar{\mathfrak{F}}_{k+2}) +Q(\bar{\mathfrak{F}}_{k+2},\mu_0)\nonumber\\
		&\quad+ \sum\limits_{l,j\geq1\atop l+j=k+2} (-1)^l \frac{y^l}{l!}[Q(\partial_r^l\mu_0,\bar{\mathfrak{F}}_j) +Q(\bar{\mathfrak{F}}_j,\partial_r^l\mu_0)] + \sum\limits_{l,j\geq1\atop l+j=k+2} [Q(F_l^0,\bar{\mathfrak{F}}_j)+Q(\bar{\mathfrak{F}}_j, F_l^0)] \nonumber\\
		&\quad+\sum\limits_{i,l,j\geq1\atop i+l+j=k+2}(-1)^l \frac{y^l}{l!}[Q(\partial_r^lF_i^0,\bar{\mathfrak{F}}_j) + Q(\bar{\mathfrak{F}}_j,\partial_r^lF_i^0)] +\sum\limits_{i,j\geq1\atop i+j=k+2} Q(\bar{\mathfrak{F}}_j,\bar{\mathfrak{F}}_i)\nonumber\\
		&\quad-\partial_t \bar{\mathfrak{F}}_k +\mathcal{H}^\v(y)(\bar{v}\cdot \vec{\tau})\partial_\phi \bar{\mathfrak{F}}_k,
	\end{align}
	where $\mathcal{H}^\v(y)=\f{1}{1-\v y}\Upsilon(\v^{\f12}y)$ with $\Upsilon(z)$ being a smooth and monotone function such that
		\begin{align*}
		\Upsilon(z)=
		\begin{cases}
			1,\quad z\leq \f12,\\
			0,\quad z\geq 1.
		\end{cases}
	\end{align*}
	
	For later use, we define
	\begin{equation}\label{1.17-1}
		\begin{split}
		\mathbf{L}_0(t,\phi,\bar{v}):=\mathbf{L}(t,\phi,y,\bar{v})|_{y=0},\\
		\quad\mathbf{P}_0(t,\phi,\bar{v}):=\mathbf{P}(t,\phi,y,\bar{v})|_{y=0},\\
		\mathcal{N}_0(t,\phi,y,\bar{v}):=\mathcal{N}(t,\phi,y,\bar{v})|_{y=0}.
		\end{split}
	\end{equation}
	For $k\geq1$, we define $\bar{f}_k:=\frac{\bar{\mathfrak{F}}_k}{\sqrt{\mu_0}}$, and consider the decomposition
	\begin{align}\label{1.12-0}
		\bar{f}_k&\equiv \frac{\bar{\mathfrak{F}}_k}{\sqrt{\mu_0}}=\mathbf{P}_0\left\{\frac{\bar{\mathfrak{F}}_k}{\sqrt{\mu_0}}\right\}+\{\mathbf{I}-\mathbf{P}_0\}\left\{\frac{\bar{\mathfrak{F}}_k}{\sqrt{\mu_0}}\right\}\nonumber\\
		&=\left\{\frac{\bar{\rho}_k}{\rho^0}+\frac{\bar{u}_k}{T^0}\cdot(\bar{v}-\bar{\mathfrak{u}}^0)+\frac{\bar{\theta}_k}{4T^0}\cdot\left(\frac{|\bar{v}-\bar{\mathfrak{u}}^0|^2}{T^0}-2\right)\right\}\sqrt{\mu_0}+\{\mathbf{I}-\mathbf{P}_0\}\left\{\frac{\bar{\mathfrak{F}}_k}{\sqrt{\mu_0}}\right\},
	\end{align}
	where  %the macroscopic part $(\bar{\rho}_k,\bar{u}_k,\bar{\theta}_k)$ are defined by
\begin{align*}
	\bar{\rho}_k=\int_{\mathbb{R}^2}\bar{\mathfrak{F}}_k\, dv,\quad\rho^0 \bar{u}_k=\int_{\mathbb{R}^2}(\bar{v}-\bar{\mathfrak{u}}^0)\cdot \bar{\mathfrak{F}}_k\, dv,\quad\rho^0\bar{\theta}_k=\int_{\mathbb{R}^2}\left(|\bar{v}-\bar{\mathfrak{u}}^0|^2-2T^0\right)\bar{\mathfrak{F}}_k\, dv,
\end{align*}
	and
	\begin{align}\label{1.12}
		&\{\mathbf{I}-\mathbf{P}_0\}\left\{\frac{\bar{\mathfrak{F}}_{k+1}}{\sqrt{\mu_0}}\right\}=\mathbf{L}_0^{-1}\left\{\frac{-(\bar{v}\cdot \vec{n})\partial_y\bar{\mathfrak{F}}_k+\sum_{ i,j\geq1,\atop i+j=k+1}Q(\bar{\mathfrak{F}}_i,\bar{\mathfrak{F}}_j)}{\sqrt{\mu_0}}\right.\nonumber\\
		&\qquad \left. +\sum_{l,j\geq1\atop l+j=k+1} (-1)^l \frac{y^l}{l!}\f{[Q(\partial_r^l\mu_0,\bar{\mathfrak{F}}_j) +Q(\bar{\mathfrak{F}}_j,\partial_r^l\mu_0)]}{\sqrt{\mu_0}} +\sum_{l,j\geq1\atop l+j=k+1}\f{  [Q(F_l^0,\bar{\mathfrak{F}}_j)+Q(\bar{\mathfrak{F}}_j, F_l^0)]}{\sqrt{\mu_0}} \right.\nonumber\\
		&\qquad \left.+\sum_{i,l,j\geq1\atop i+l+j=k+1}(-1)^l \frac{y^l}{l!}\f{[Q(\partial_r^lF_i^0,\bar{\mathfrak{F}}_j) + Q(\bar{\mathfrak{F}}_j,\partial_r^lF_i^0)] }{\sqrt{\mu_0}}-\f{\partial_t \bar{\mathfrak{F}}_{k-1} -\mathcal{H}^\v(y)(\bar{v}\cdot \vec{\tau})\partial_\phi \bar{\mathfrak{F}}_{k-1}}{\sqrt{\mu_0}}\right\}.
	\end{align}
It is clear from \eqref{1.12} that $\{\mathbf{I}-\mathbf{P}_0\}\left\{\frac{\bar{\mathfrak{F}}_{k+1}}{\sqrt{\mu_0}}\right\} $ is determined by $F_i,0\leq i\leq k$ and $\bar{\mathfrak{F}}_j,1\leq j\leq k$. Noting $\eqref{vs}_1$, one has $\bar{f}_1\in \mathcal{N}_0$. 

Motivated by \cite{GHW-2021-ARMA}, we shall construct $\bar{\mathfrak{F}}_k$  by solving a parabolic system, see Sections 2.2 and \ref{sec3.2} for more details. Once we obtain $\bar{\mathfrak{F}}_k$, then we further define approximate the viscous layer  
\begin{align}
\bar{\mathfrak{F}}^\v(t,y,\phi,\bar{v})\sim\sum_{k=1}^\infty\bar{\mathfrak{F}}_k(t,y,\phi,\bar{v}).
\end{align}
We remark that, although the approximate viscous layer $\bar{\mathfrak{F}}_k$ does not exactly solve the original problem \eqref{vs}, but it is enough because the error has a very high order $\v$-decay rate due to the space decay.
%Now we denote the approximate viscous boundary layer to be $\bar{\mathscr{F}}^\v\sim\sum\limits_{k=1}^\infty\v^k\bar{\mathscr{F}}_k$ and it formally satisfies
%	\begin{align}\label{1.19}
%		&(\partial_t+v\cdot\nabla)\bar{\mathscr{F}}^\v+\f{1}{1-\v y}(1-\Upsilon(\v^{\f12} y))(\bar{v}\cdot\vec{\tau})\partial_\phi \bar{\mathscr{F}}^\v-\f{1}{\v}\partial_y\{\Upsilon_0(\v^{\f12}y)\}(\bar{v}\cdot \vec{n})\bar{\mathscr{F}}^\v\nonumber\\
%		&=\f{1}{\v^2}[Q(F^\v,\bar{\mathscr{F}}^\v)+Q(\mathscr{F}^\v,F^\v)+Q(\bar{\mathscr{F}}^\v,\bar{\mathscr{F}}^\v)].
%	\end{align} 

	\subsubsection{Formal Knudsen boundary layer expansion}
	From \eqref{1.10} and \eqref{1.12}, we notice that $F^\v+\bar{\mathfrak{F}}^\v$ may still not satisfy the specular boundary. In fact, both the internal solutions and the viscous boundary layers (governed by parabolic and hyperbolic systems respectively) only carry finite-dimensional information of $v$. To capture the remaining infinite dimensional information of $v$, it is essential to introduce the Knudsen layer, which is in fact a steady Boltzmann equation, see \cite{GHW-2021-ARMA,Sone}. 
	
For the Knudsen boundary layer, we define the new scaled normal coordinate as $$\eta:=\f{y}{\v}=\f{1-r}{\v^2}.$$ Following \cite{Wu-2020-Arxiv,Wu-2016-JDE}, due to the geometric effect of the disk and the poor normal regularity of Boltzmann equation,  it is necessary to select the essential tangential variable to establish the high order Hilbert expansion.	
%Then we have to separate out the real tangential term for high order expansion,  for more details. 
Thus we define $$\mv:=(\mv_1,\mv_2)=(\bar{v}\cdot \vec{n},\bar{v}\cdot\vec{\tau}),$$ 
which yields $\partial_\phi\rightarrow\partial_\phi+\mv_2\partial_{\mv_1}-\mv_1\partial_{\mv_2}$ and
	\begin{align*}
		&v\cdot\nabla_x=\frac{1}{\v^2}\mv_1\partial_\eta-\frac{1}{1-\v^2\eta}\left(\mv_2\partial_\phi+\mv_2^2\frac{\partial}{\partial{\mv_1}}-\mv_2 \mv_1\frac{\partial}{\partial {\mv_2}}\right).
	\end{align*}

	Inspired by \cite{GHW-2021-ARMA,Sone}, we consider the Knudsen layer of the following form 
	\begin{align}
	\hat{F}^\v(t,\eta,\phi, \mathfrak{v})\sim\sum\limits_{k=1}^\infty\v^k\hat{F}_k(t,\eta,\phi, \mathfrak{v}).
	\end{align}
	 Substituting $F^\v+\bar{\mathfrak{F}}^\v+\hat{F}^\v$ to \eqref{1.6}, using \eqref{1.7} and \eqref{1.16},  and ignoring some terms with high order of $\v$, one has
%	\begin{align*}
%		(\partial_t+v\cdot\nabla_x)\hat{F}^\v =\f{1}{\v^2}[Q(F^\v+\bar{\mathfrak{F}}^\v,\hat{F}^\v)+Q(\hat{F}^\v,F^\v+\bar{\mathfrak{F}}^\v)+Q(\hat{F}^\v,\hat{F}^\v)],
%	\end{align*}
%	which yields that
	\begin{align}\label{1-19}
		&\partial_t\hat{F}^\v+\frac{1}{\v^2}\mv_1\partial_\eta\hat{F}^\v-\frac{1}{1-\v^2\eta}\left(\mv_2\partial_\phi\hat{F}^\v+\mv_2^2\frac{\partial \hat{F}^\v}{\partial{\mv_1}}-\mv_2 \mv_1\frac{\partial\hat{F}^\v}{\partial {\mv_2}}\right)\nonumber\\
		&\qquad=\frac{1}{\v^2}[Q(\hat{F}^\v,F^\v+\bar{\mathfrak{F}}^\v)+Q(F^\v+\bar{\mathfrak{F}}^\v,\hat{F}^\v)+Q(\hat{F}^\v,\hat{F}^\v)].
	\end{align}
	Similarly, %to the discussion of viscous boundary layer expansion, 
	we need the Taylor expansions of $\bar{\mathfrak{F}}_k$ near $y=0$:
	\begin{align*}
		\bar{\mathfrak{F}}_k=\bar{\mathfrak{F}}_k^0+\sum\limits_{l=1}^{\mathfrak{b}}\frac{(\v\eta)^l}{l!}\partial_y^l\bar{\mathfrak{F}}_k^0+\frac{(\v\eta)^{\mathfrak{b}+1}}{(\mathfrak{b}+1)!} \partial_y^{\mathfrak{b}+1}\bar{\mathcal{F}}_k,
	\end{align*}
	where $\partial_y^l\bar{\mathfrak{F}}_k^0:=\partial_y^l\bar{\mathfrak{F}}_k|_{y=0}, l=0,1,\cdots$ and $\partial_y^{\mathfrak{b}+1}\bar{\mathcal{F}}_k$ is the remainder. Also, the internal solutions should also be rewritten as
	\begin{align*}
		&\mu=\mu_0+\sum\limits_{l=1}^{\mathfrak{b}}\frac{(-\v^2\eta)^l}{l!}\partial_r^l\mu_0+\frac{(-\v^2\eta)^{\mathfrak{b}+1}}{(\mathfrak{b}+1)!}\partial_r^{\mathfrak{b}+1}\tilde{\mu},\\
		&F_k=F_k^0+\sum\limits_{l=1}^{\mathfrak{b}}\frac{(-\v^2\eta)^l}{l!}\partial_r^lF_k^0+\frac{(-\v^2\eta)^{\mathfrak{b}+1}}{(\mathfrak{b}+1)!}\partial_r^{\mathfrak{b}+1}\mathcal{F}_k,
	\end{align*}
where we point out that $\mu_0$ in \eqref{mu} can also be written as $\mu_0=\f{\rho^0}{2\pi T^0}\exp\big\{-\f{|\mv-\bar{\fu}^0|^2}{2T^0}\big\}$ with $\bar{\fu}^0=(0,u_\tau^0)$ and $ u_\tau^0=\bar{\mathfrak{u}}^0\cdot\vec{\tau}$.
	
As pointed out in \cite{Wu-2016-JDE}, it is hard  to make high order expansion if the geometric correction is ignored. Hence, we compare the order of $\v$ in \eqref{1-19} and take the geometric effect into account,  one obtains that 
	\begin{align}\label{KL}
		\v^{-1}:&\quad \mv_1\partial_\eta\hat{F}_1-\f{\v^2}{1-\v^2\eta}\left\{\mv_2^2\f{\partial \hat{F}_1}{\partial \mv_1}-\mv_1 \mv_2 \f{\partial \hat{F}_1}{\partial \mv_2}\right\}=Q(\mu_0,\hat{F}_1)+Q(\hat{F}_1,\mu_0),\nonumber\\
		\v^0:&\quad \mv_1\partial_\eta\hat{F}_2-\f{\v^2}{1-\v^2\eta}\left\{\mv_2^2\f{\partial \hat{F}_2}{\partial \mv_1}-\mv_1 \mv_2 \f{\partial \hat{F}_2}{\partial \mv_2}\right\}=Q(\mu_0,\hat{F}_2)+Q(\hat{F}_2,\mu_0)\nonumber\\
		&\qquad\qquad+Q(F_1^0,\hat{F}_1)+Q(\bar{\mathfrak{F}}_1^0,\hat{F}_1)+Q(\hat{F}_1,\bar{\mathfrak{F}}_1^0)+Q(\hat{F}_1,\hat{F}_1),\nonumber\\
		\v^1:&\quad \mv_1\partial_\eta\hat{F}_3-\f{\v^2}{1-\v^2\eta}\left\{\mv_2^2\f{\partial \hat{F}_3}{\partial \mv_1}-\mv_1 \mv_2 \f{\partial \hat{F}_3}{\partial \mv_2}\right\}=Q(\mu_0,\hat{F}_3)+Q(\hat{F}_3,\mu_0)\nonumber\\
		&\qquad+Q(F_1^0,\hat{F}_2)+Q(\bar{\mathfrak{F}}_1^0,\hat{F}_2)+Q(\hat{F}_2,\bar{\mathfrak{F}}_1^0)+Q(\hat{F}_2,\hat{F}_1)\nonumber\\
		&\qquad+Q(F_2^0-\eta\partial_r\mu_0+\bar{\mathfrak{F}}_2^0+\eta\partial_y\bar{\mathfrak{F}}_1^0,\hat{F}_1)+Q(\hat{F}_1,F_2^0-\eta\partial_r\mu_0+\bar{\mathfrak{F}}_2^0+\eta\partial_y\bar{\mathfrak{F}}_1^0)\nonumber\\
		&\qquad+Q(\hat{F}_1,\hat{F}_2)+Q(\hat{F}_2,\hat{F}_1)-\partial_t\hat{F}_1+\f{1}{1-\v^2\eta}\mv_2\partial_\phi\hat{F}_1,\\
		...\nonumber\\
		\v^k:&\quad  \mv_1\partial_\eta\hat{F}_{k+2}-\f{\v^2}{1-\v^2\eta}\left\{\mv_2^2\f{\partial \hat{F}_{k+2}}{\partial \mv_1}-\mv_1 \mv_2 \f{\partial \hat{F}_{k+2}}{\partial \mv_2}\right\}=Q(\mu_0,\hat{F}_{k+2})+Q(\hat{F}_{k+2},\mu_0)\nonumber\\
		&\qquad\qquad+\sum\limits_{l,j\geq1\atop 2l+j=k+2}\frac{(-\eta)^l}{l!}[Q(\partial_r^l\mu_0,\hat{F}_j)+Q(\hat{F}_j,\partial_r^l\mu_0)]\nonumber\\
		&\qquad\qquad+\sum\limits_{i,l,j\geq1\atop 2l+i+j=k+2}\frac{(-\eta)^l}{l!}[Q(\partial_r^lF_i^0,\hat{F}_j)+Q(\hat{F}_j,\partial_r^lF_i^0)]\nonumber\\
		&\qquad\qquad+\sum\limits_{i,j\geq1\atop i+j=k+2}[Q(F_i^0+\bar{\mathfrak{F}}_i^0,\hat{F}_j)+Q(\hat{F}_j,F_i^0+\bar{\mathfrak{F}}_i^0)]\nonumber\\
		&\qquad\qquad+\sum\limits_{i,l,j\geq1\atop l+i+j=k+2}\frac{\eta^l}{l!}[Q(\partial_y^l\bar{\mathfrak{F}}_i^0,\hat{F}_j)+Q(\hat{F}_j,\partial_y\bar{\mathfrak{F}}_i^0)]+\sum\limits_{i,j\geq1\atop i+j=k+2}Q(\hat{F}_i,\hat{F}_j)\nonumber\\
		&\qquad\qquad-\partial_t\hat{F}_k+\frac{1}{1-\v^2\eta}\mv_2\partial_\phi\hat{F}_k.\nonumber
	\end{align}

For later use, we can write	 the equation of $\hat{F}_k$ in \eqref{KL} in the following general form
	\begin{align}\label{1.18}
		\mv_1\partial_\eta\hat{F}_k-\f{\v^2}{1-\v^2\eta}\left\{\mv_2^2\f{\partial \hat{F}_k}{\partial \mv_1}-\mv_1 \mv_2 \f{\partial \hat{F}_k}{\partial \mv_2}\right\}-[Q(\mu_0,\hat{F}_k)+Q(\hat{F}_k,\mu_0)]=\hat{\mathcal{S}}_{k},
	\end{align}
	where  
	\begin{align}\label{1.17-0}
		\hat{\mathcal{S}}_{k}&=\sum\limits_{\substack{i+j=k\\i,j\geq1}}[Q(F^0_i+\bar{\mathfrak{F}}_i^0,\hat{F}_j)+Q(\hat{F}_j,F^0_i+\bar{\mathfrak{F}}_i^0)+Q(\hat{F}_i,\hat{F}_j)]\nonumber\\
		&\quad+\sum\limits_{\substack{j+2l=k\\b\geq l\geq 1,j\geq 1}}\frac{(-\eta)^l}{l!}\cdot[Q(\partial_r^l\mu_0,\hat{F}_j)+Q(\hat{F}_j,\partial_r^l\mu_0)]\nonumber\\
		&\quad+\sum\limits_{\substack{i+j+2l=k\\i,j\geq1,b\geq l\geq 1}}\frac{(-\eta)^l}{l!}\cdot Q(\partial_r^lF^0_i,\hat{F}_j)+Q(\hat{F}_j,\partial_r^lF^0_i)]\nonumber\\
		&\quad+\sum\limits_{\substack{i+j+l=k\\i,j\geq1,b\geq l\geq 1}}\frac{\eta^l}{l!}[Q(\partial_y^l\bar{\mathfrak{F}}_i^0,\hat{F}_j)+Q(\hat{F}_j,\partial_y^l\bar{\mathfrak{F}}_i^0)]\nonumber\\
		&\quad-\partial_t\hat{F}_{k-2}+\frac{1}{1-\v^2\eta}\mv_2\partial_\phi\hat{F}_{k-2}.
	\end{align} 
From the last term on RHS of \eqref{1.17-0}, the estimate on $\partial_{\phi}\hat{F}_{k-2}$ is needed. 
It is noted that $\hat{\mathcal{S}}_k$ is determined by $F_i,0\leq i\leq k-1$ and $\bar{\mathfrak{F}}_j,\hat{F}_j, 1\leq j\leq k-1$, and thus it can be regarded as known functions when solving $\hat{F}_k$. 
To solve \eqref{1.18}, we denote $\hat{f}_k:=\f{\hat{F}_k}{\sqrt{\mu_0}}$, then \eqref{1.18} is rewritten as 
	\begin{align}\label{1.18-0}
		\mv_1\partial_\eta\hat{f}_k-\f{\v^2}{1-\v^2\eta}\left\{\mv_2^2\f{\partial \hat{f}_k}{\partial \mv_1}-\mv_1 \mv_2 \f{\partial \hat{f}_k}{\partial \mv_2}\right\}+\f{u_\tau^0}{2T^0}\f{\v^2}{1-\v^2\eta}\mv_1\mv_2\hat{f}_k+\FL_0\hat{f}_k=\f{\hat{\mathcal{S}}_{k}}{\sqrt{\mu_0}}=:\hat{S}_k.
	\end{align}
	
The Knudsen boundary layer problem \eqref{1.18-0} (or equivalent \eqref{1.18}) is indeed a steady Boltzmann equation with geometric correction. Extensive research has been conducted on classical Knudsen boundary layers(the curvature of domain is zero, then no geometric correction term arises). The seminal work by Bardos et al. \cite{Bardos-1986-CPAM} established existence, uniqueness, and asymptotic behavior properties for linearized Boltzmann equations under Dirichlet-type boundary condition, see also \cite{Coron} for a classification of Knudsen layer problems. Later, the existence and stability were proved in  \cite{UYY,UYY1} where the Mach number of at far field plays an important role.  For the diffuse boundary condition, the existence of steady Boltzmann solution is proved in \cite{DHWZ,EKGM,HW-2022-SIAM}. For the specular reflection condition, Golse-Perthame-Sulem \cite{GPS} proved the existence, uniqueness and asymptotic behavior in a certain functional space. Huang-Jiang-Wang \cite{HJW} obtained the existence and exponential space decay of steady boundary layer solution in $L^2_{x,v}\cap L^\infty_{x,v}.$ We also refer to \cite{BG} for the boundary layer  with transition,  \cite{DWY} for pointwise estimate.

For non-flat domain, the Knudsen boundary layer with geometric correction plays an important role in the hydrodynamic limit of Boltzmann equation. In fact, for the case of $u^0_{\tau}=0$,  Wu \cite{Wu-2016-JDE} established the existence and space decay of Knudsen boundary layer solution of \eqref{1.18-0} with inflow boundary condition, then successfully justified the hydrodynamic limit of Boltzmann equation to the incompressible Navier-Stokes equations in 2D disk.  

In the case of compressible Euler framework, it is not reasonable to assume  $u_\tau^0\equiv 0$ in general. We point out that it is very hard to solve the Knudsen boundary layer problem \eqref{1.18-0} in the case of $u_\tau^0\neq 0$ because a serious difficulty arises from $\f{u_\tau^0}{2T^0}\f{\v^2}{1-\v^2\eta}\mv_1\mv_2\hat{f}_k$ in the $L^2$ energy estimate.
%A key feature of this formulation is the inclusion of  $\f{u_\tau^0}{2T^0}\f{\v^2}{1-\v^2\eta}\mv_1\mv_2\hat{f}_k$,induced by the nonzero tangential velocity $u_\tau^0=\bar{\fu}^0\cdot\vec{\tau}$. 
Such term also appears in the study of shear flow \cite{Duan-2021-ARMA,Duan-2023-JEMS}, but only polynomial decay of particle velocity are obtained, and it is hard for us to construct the Knudsen boundary layers of  high order.
%To our knowledge, no prior results exist for this problem, primarily due to the analytical obstruction posed by the term  $\mv_1\mv_2\hat{f}_k$ in $L^2$--energy estimate. 
In fact, one of the core contributions of this paper is to establish  the well-posedness and space decay of above  Knudsen layer problem with $u_\tau^0\neq 0$.

Due to the new term $\f{u_\tau^0}{2T^0}\f{\v^2}{1-\v^2\eta}\mv_1\mv_2\hat{f}_k$, it is hard to solve it in $\eta\in (0,+\infty)$.
%Noting the singularity of $\f{1}{1-\v^2\eta}$ at $\eta=\f{1}{\v^2}$, 
By noting that the Knudsen layer only plays a key role near the boundary, then it should be enough to consider a truncated problem on $\eta\in (0,d)$.
% with  $d=\v^{-\mathfrak{a}}$ and $0<\mathfrak{a}<\f23$.
%we restrict the analysis of \eqref{1.18-0} on $\eta\in[0,d]$ with $d=\v^{-\mathfrak{a}}$, where $\mathfrak{a}$ can be any positive constant satisfying $0<\mathfrak{a}<\f23$. {\color{blue}This truncation aligns with the physical role of the Knudsen layer, which dominates only within a $\v^2$--scale neighborhood of the boundary, thereby justifying the domain restriction.}
%We now state one of the principal results of this work. Let $G(\eta):=-\f{\v^2}{1-\v^2\eta}$. To construct the Knudsen boundary layer solutions, we consider the following linear problem
That is 
\begin{align}\label{K}		
	\begin{cases}
		\dis \mv_1 \partial_\eta f+G(\eta)\left(\mv_2^2\f{\partial f}{\partial \mv_1}-\mv_1 \mv_2\f{\partial f}{\partial \mv_2}\right)-\f{ u_\tau^0}{2T^0}\cdot G(\eta)\mv_1 \mv_2 f+\FL_0 f=\mathfrak{S},\\
		\dis f(0,\mv)|_{\mv_1>0}=f(0,R\mv)+f_b(R\mv),\\
		\dis f(d,\mv)|_{\mv_1<0}=0,
	\end{cases}
\end{align}
where $G(\eta):=-\f{\v^2}{1-\v^2\eta}$, $\eta\in (0,d)$ with $d=\v^{-\mathfrak{a}}$, $ R\mv=(-\mv_1,\mv_2)$. The function $f_b(\mv$) is only defined for $\mv_1\textless0$, and we always assume that it is extended to be $0$ for $\mv_1>0$. Since we hope that $f(\eta,\mathfrak{v})$ has fast decay with respect to $\eta$, it is reasonable to give the boundary condition $f(d,\mv)|_{\mv_1<0}=0$, in fact it plays an important role in the space decay estimate.
		
%We write $F_k(r,\phi,v)$ and $\bar{\mathfrak{F}}_k(y,\phi,\bar{v})$ as $F_k(\eta,\phi,\mv_1,\mv_2)$ and $\bar{\mathfrak{F}}_k(\eta,\phi,\mv_1,\mv_2)$. To satisfy the specular reflection boundary \eqref{1.5} and solve equation as \eqref{1.18-0}, we adopt the boundary condition of \eqref{1.18} as
%	\begin{align*}
%		(F_k+\bar{\mathfrak{F}}_k+\hat{F}_k)(0,\phi,\mv_1,\mv_2)|_{\mv_1 >0}=(F_k+\bar{\mathfrak{F}}_k+\hat{F}_k)(0,\phi,-\mv_1,\mv_2),\quad k\geq 1,
%	\end{align*}
%	and 
%	\begin{align*}
%		\hat{F}_k(d,\mv)|_{\mv_1<0}=0.
%	\end{align*}
	
\smallskip
	
We shall apply the $L^2-L^\infty$-estimate developed in \cite{Guo-2010-ARMA} to study the well-posedness of Knudsen layer problem \eqref{K}. For later use,  we introduce a global Maxwellian
	\begin{align*}
		\mu_M:=\frac{1}{2\pi T_M}\exp\{-\frac{|v|^2}{2T_M}\},
	\end{align*}
	where %$T_M>0$ satisfies
	\begin{align}\label{ta}
	0<	T_M<\min_{x\in B_1}T(t,x)\leq\max_{x\in B_1}T(t,x)\textless2T_M.
	\end{align}
Under assumption \eqref{ta}, it is easy to know that there exist positive constant $C>0$ and   $\frac12<\iota<1$ such that
	\begin{align}\label{1.23}
		\frac{1}{C}\mu_M\leq\mu(t,x,v)\leq C\mu_M^{\iota}.
	\end{align}
We define weight function
\begin{align*}
	w_\beta=(1+|\mv|^2)^{\frac{\beta}{2}}e^{\zeta|\mv|^2},
\end{align*}
where $0<\zeta<\f{1}{4T_M}$. 
%We have the following theorem. 

\begin{Theorem}\label{thm2.2}
Let $\beta>3$. Assume $ \mathfrak{S}\in\mathcal{N}_0^{\perp}$ and there exists a positive constant $\sigma_0>0$ such that
\begin{align}\label{2.18}
&\Big\|\nu^{-1}w_\beta e^{\sigma_0\eta}\f{\sqrt{\mu_0}}{\sqrt{\mu_M}}\mathfrak{S} \Big\|_{L_{\eta,\mv}^\infty} + \Big|w_{\beta+1} \f{\sqrt{\mu_0}}{\sqrt{\mu_M}}f_b\Big|_{L_{\mv}^\infty}\nonumber\\
&+\Big|\nu^{-1}w_\beta\f{\sqrt{\mu_0}}{\sqrt{\mu_M}}(\mv_2^2\f{\partial}{\partial \mv_1}-\mv_1 \mv_2\f{\partial}{\partial \mv_2})f_b \Big|_{L^\infty_{\mv}} < \infty.
\end{align}
Suppose
\begin{align}\label{2.19}
\begin{split}
\int_{\mathbb{R}^2}\mv_1 f_b(\mv)\sqrt{\mu_0}\, d\mv\equiv0,\\
\int_{\mathbb{R}^2}\mv_1 \mv_2 f_b(\mv)\sqrt{\mu_0}\, d\mv\equiv0,\\
\int_{\mathbb{R}^2}\mv_1|\mv|^2f_b(\mv)\sqrt{\mu_0}\, d\mv\equiv0.
\end{split}
\end{align}
Then the problem \eqref{K} has a unique solution $f$ satisfying
\begin{align}\label{2.20}
&\Big\|w_\beta e^{\sigma\eta}\f{\sqrt{\mu_0}}{\sqrt{\mu_M}}f(\cdot,\cdot)\Big\|_{L^\infty_{\eta,\mv}} +\Big|w_\beta e^{\sigma\eta}\f{\sqrt{\mu_0}}{\sqrt{\mu_M}}f(0,\cdot)\Big|_{L^\infty_{\mv}}\nonumber\\
&\leq\frac{C}{\sigma(\sigma_0-\sigma)}\bigg\{ \Big\|\nu^{-1}w_\beta e^{\sigma_0\eta}\f{\sqrt{\mu_0}}{\sqrt{\mu_M}}\mathfrak{S}\Big\|_{L_{\eta,\mv}^\infty} + \Big|w_{\beta+1} \f{\sqrt{\mu_0}}{\sqrt{\mu_M}}f_b \Big|_{L_{\mv}^\infty} \nonumber\\
&\quad  +\Big|\nu^{-1}w_\beta\f{\sqrt{\mu_0}}{\sqrt{\mu_M}}(\mv_2^2\f{\partial}{\partial \mv_1}-\mv_1 \mv_2\f{\partial}{\partial \mv_2})f_b \Big|_{L^\infty_{\mv}}  \bigg\},
\end{align}
where $C>0$ is a constant independent of $\v$ and $\sigma \in (0,\sigma_0)$. Moreover, if $\mathfrak{S}$ is continuous in $[0,d]\times \R^2$ and $f_b$ is continuous in $\{\mv\in\R^2\}$, then $f$ is continuous away from grazing set $\{(\eta,\mv):\eta=0,\mv_1=0\}$.
\end{Theorem}
	%\begin{remark}
	%	In the study of equation \eqref{K-1}, the constraint on the domain length $d$ is natural since the geometric term $\f{1}{1-\v^2\eta}$ has singularity at $\eta=\v^{-2}$. Referring to \cite{Wu-2016-JDE}, the author  used the cutoff function and study equation as following: 
	%	\begin{align*}
	%		\begin{cases}
	%			\dis	\mv_1\partial_\eta f+G(\eta)\Upsilon(\v\eta)(\mv_2^2\f{\partial f}{\partial \mv_1}-\mv_1 \mv_2\f{\partial f}{\partial \mv_2})-\f{u_\tau^0}{2T^0}G(\eta)\Upsilon(\v \eta)\mv_1 \mv_2 f+\FL_0f =S,\\
	%			\dis	f(0,\mv)|_{\gamma_-}=f(0,-\mv_1,\mv_2),\\
	%			\dis	\lim_{\eta\to\infty}f(\eta,\mv)|_{\gamma_-}=0.
	%		\end{cases}
	%	\end{align*}
	%	However we can not control $-\f{u_\tau^0}{2T^0}G(\eta)\Upsilon(\v \eta)\mv_1 \mv_2 f$ by regular $L^2-L^\infty$ framework. 
	%\end{remark}
\begin{remark}\label{rmk2.4}
% Since the original problem \eqref{1.18} is intractable to solve directly, and noting that the boundary layer primarily serves to correct the behavior near the boundary, it suffices to consider an approximate problem instead. 
We also remark that the restriction $\mathfrak{a}<\f23$ is due to some technical requirement in the proof. With the help of Theorem \ref{thm2.2}, we can construct the Knudsen boundary layers $\hat{F}_{k}(t,\eta, \phi,\mathfrak{v})$ with  $\eta \in (0,d)$ or equivalently $r\in (1-\v^{2-\mathfrak{a}}, 1)$. That means the Knudsen layers are only defined in domain near the boundary of disk. For later use, we introduce a truncation so that it is meaningful in the whole disk. In fact, we define
\begin{align}\label{1.34}
\hat{\mathscr{F}}_k=\Upsilon(\v^{\mathfrak{a}}\eta)\hat{F}_k,
\end{align} 
and denote $\hat{\mathscr{F}}^\v\sim \sum_{k=1}^\infty \v^k\hat{\mathscr{F}}_k.$  Due to the exponential decay on $\eta$, the error generated from the truncation  has a very high order $\v$-decay rate, and it is enough for us to close remainder estimate.
\end{remark}
 
\begin{remark}
Based on our above constructions,  we hope that the specular reflection boundary condition holds for each order, that is
\begin{align*}
	(F_k+\bar{\mathfrak{F}}_k+\hat{F}_k)(0,\phi,\mv_1,\mv_2)|_{\mv_1 >0}=(F_k+\bar{\mathfrak{F}}_k+\hat{F}_k)(0,\phi,-\mv_1,\mv_2),\quad k\geq 1,
\end{align*}
which implies the form of boundary condition $\eqref{K}_2$.
\end{remark}
\begin{remark}
	The construction of Knudsen layer and Hilbert expansion requires the estimate for the derivative $\partial_t^l\hat{F}_k$, $\partial_\phi^j\hat{F}_k$, $l,j\in\mathbb{N}_+$. See \eqref{1.17-0} and \eqref{eq1.22} for details. Taking $\partial_t^l$ and $\partial_\phi^j$ relatively to \eqref{1.18-0}, since we separate out the real tangential term by considering geometric correction, we get the equations of $\partial_t^l\hat{f}_k$ and $\partial_\phi^j \hat{f}_k$ which keep the structure as the equation of $\hat{f}_k$. Then we apply Theorem \ref{thm2.2} to obtain the estimates for $\partial_t^l\hat{f}_k$ and $\partial_\phi^j \hat{f}_k$.
\end{remark}
%Then it formally satisfies
%\begin{align}
%		&(\partial_t+v\cdot\nabla)\hat{\mathscr{F}}^\v-\f{1}{\v^2}\partial_\eta\{\Upsilon(\v^{\mathfrak{a}\eta})\}\mv_1\hat{\mathscr{F}}^\v\nonumber\\ &=\f{1}{\v^2}[Q(F^\v+\bar{\mathfrak{F}}^\v,\hat{F}^\v)+Q(\hat{F}^\v,F^\v+\bar{\mathfrak{F}}^\v)+Q(\hat{F}^\v,\hat{F}^\v)].
%\end{align}

%%%%%%%%%%%%%%%%%%%%%%%%%%%%%%%%%%%%%%%%%%%%%%%%%%%%%%%%%%%%%%%%%%%%%%%%%%%%%%%%%%%%%%%%%%%%%%%%%%%%%%%%%%%%%%%%%%%%%%%%%%%%%%%%%%%%%%%%%%%%%%%%%%%%%%%%%%%%%
\subsection{Truncated Hilbert expansion and main result}
Motivated by \cite{GHW-2021-ARMA,Sone}, we consider the Boltzmann solution of \eqref{1.5}--\eqref{1.6} in the following Hilbert expansion with multi-scales 
\begin{align}
F^\v(t,x,v)&=\mu(t,x,v)+\sum\limits_{i=1}^N\v^iF_i(t,x,v)+\sum\limits_{i=1}^N\v^i\bar{\mathfrak{F}}_i(t,\phi,y,v) \nonumber\\
&\quad +\sum\limits_{i=1}^N\v^i\hat{\mathscr{F}}_i(t,\phi,\eta,\mv_1,\mv_2)+\v^5F_R^\v(t,x,v),
\end{align}
which, noting the constructions of $F_i, \bar{\mathfrak{F}}_i, \hat{\mathscr{F}}_i$ in section \ref{sec1.2}, yields that 
	\begin{align}\label{re}
		&\partial_tF_R^\v+v\cdot\nabla_x F_R^\v-\frac{1}{\v^2}[Q(F_R^\v,\mu)+Q(\mu,F_R^\v)]\\
		&=\v^3Q(F_R^\v,F_R^\v)+\sum\limits_{i=1}^N\v^{i-2}\{Q(F_i+\bar{\mathfrak{F}}_i+\hat{\mathscr{F}}_i,F_R^\v)+Q(F_R^\v,F_i+\bar{\mathfrak{F}}_i+\hat{\mathscr{F}}_i)\}\nonumber\\
		&\quad +R^\v+\bar{R}^\v+\hat{R}^\v,\nonumber
	\end{align}
	where
	\begin{align}\label{eq1.21}
		R^\v&=-\v^{N-6}\{\partial_tF_{N-1}+v\cdot\nabla_x F_{N-1}+\v\partial_tF_N+\v v\cdot\nabla_x F_N\}+\v^{N-6}\sum\limits_{1\leq i,j\leq N\atop i+j\geq N+1}\v^{i+j-N-1}Q(F_i,F_j),\nonumber\\
		\bar{R}^\v&=-\v^{N-6}\left\{\partial_t\bar{\mathfrak{F}}_{N-1}-\mathcal{H}^\v(y)(\bar{v}\cdot\vec{\tau})\partial_\phi\bar{\mathfrak{F}}_{N-1}+\v\partial_t\bar{\mathfrak{F}}_N-\v\mathcal{H}^\v(y)(\bar{v}\cdot\vec{\tau})\partial_\phi\bar{\mathfrak{F}}_N+(\bar{v}\cdot \vec{n})\partial_y\bar{\mathfrak{F}}_N\right\}\nonumber\\
		&\quad+\v^{N-6}\sum\limits_{1\leq l,j\leq N\atop l+j\geq N+1}\v^{l+j-N-1}\frac{(-y)l}{l!}[Q(\partial_r^l\mu_0,\bar{\mathfrak{F}}_j)+Q(\bar{\mathfrak{F}}_j,\partial_r^l\mu_0)]\nonumber\\
		&\quad+\v^{N-6}\sum\limits_{1\leq l,j\leq N\atop l+j\geq N+1}\v^{l+j-N-1}[Q(F_l^0,\bar{\mathfrak{F}}_j)+Q(\bar{\mathfrak{F}}_j,F_l^0)]\\
		&\quad+\v^{N-6}\sum\limits_{1\leq i,l,j\leq N\atop i+l+j\geq N+1}\v^{i+l+j-N-1}\frac{(-y)^l}{l!}[Q(\partial_r^lF_i^0,\bar{\mathfrak{F}}_j)+Q(\bar{\mathfrak{F}}_j,\partial_r^lF_i^0)]\nonumber\\
		&\quad+\v^{N-6}\sum\limits_{1\leq i,j\leq N\atop i+j\geq N+1}\v^{i+j-N-1}[Q(\bar{\mathfrak{F}}_i,\bar{\mathfrak{F}}_j)+Q(\bar{\mathfrak{F}}_j,\bar{\mathfrak{F}}_i)]\nonumber\\
		&\quad+\v^{\mathfrak{b}-5}\frac{(-y)^{\mathfrak{b}+1}}{(\mathfrak{b}+1)!}\sum\limits_{j=1}^N\v^{j-1}[Q(\partial_r^{\mathfrak{b}+1}\tilde{\mu},\bar{\mathfrak{F}}_j)+Q(\bar{\mathfrak{F}}_j,\partial_r^{\mathfrak{b}+1}\tilde{\mu})]\nonumber\\
		&\quad+\v^{\mathfrak{b}-4}\frac{(-y)^{\mathfrak{b}+1}}{(\mathfrak{b}+1)!}\sum\limits_{i,j=1}^N\v^{i+j-2}[Q(\partial_r^{\mathfrak{b}+1}\mathcal{F}_i,\bar{\mathfrak{F}}_j)+Q(\bar{\mathfrak{F}}_j,\partial_r^{\mathfrak{b}+1}\mathcal{F}_i)]\nonumber\\
		&\quad -\sum\limits_{i=1}^N\v^{i-5}\f{1}{1-\v y}(1-\Upsilon(\v y))(\bar{v}\cdot\vec{\tau})\partial_\phi \bar{\mathfrak{F}}_i,\nonumber
	\end{align}
	and
	\begin{align}\label{eq1.22}
		\hat{R}^\v&=-\v^{N-6}\Upsilon(\v^{\mathfrak{a}}\eta)\left(\partial_t\hat{F}_{N-1}-\frac{1}{1-\v^2\eta}\mv_2\partial_\phi\hat{F}_{N-1} +\v\partial_t\hat{F}_N-\frac{\v}{1-\v^2\eta}\mv_2\partial_\phi\hat{F}_N\right)\nonumber\\
		&\quad+\v^{N-6}\Upsilon(\v^{\mathfrak{a}}\eta)\sum\limits_{1\leq l,j\leq N\atop 2l+j\geq N+1}\v^{2l+j-N-1}\frac{(-\eta)^l}{l!}[Q(\partial_r^l\mu_0,\hat{F}_j)+Q(\hat{F}_j,\partial_r^l\mu_0)]\nonumber\\
		&\quad+\v^{N-6}\Upsilon(\v^{\mathfrak{a}}\eta)\sum\limits_{1\leq i,j,l\leq N\atop 2l+i+j\geq N+1}\v^{2l+i+j-N-1}\frac{(-\eta)^l}{l!}[Q(\partial_r^lF_i^0,\hat{F}_j)+Q(\hat{F}_j,\partial_r^lF_i^0)]\nonumber\\
		&\quad+\v^{N-6}\Upsilon(\v^{\mathfrak{a}}\eta)\sum\limits_{1\leq i,j\leq N\atop i+j\geq N+1}\v^{i+j-N-1}[Q(F_i+\bar{\mathfrak{F}}_i,\hat{F}_j)+Q(\hat{F_j},F_i+\bar{\mathfrak{F}}_i)+Q(\hat{F}_i,\hat{F}_j)]\\
		&\quad+\v^{N-6}\Upsilon(\v^{\mathfrak{a}}\eta)\sum\limits_{1\leq i,j,l\leq N\atop i+j+l\geq N+1}\v^{i+j+l-N-1}\frac{\eta^l}{l!}[Q(\partial_y^l\bar{\mathfrak{F}}_i^0,\hat{F}_j)+Q(\hat{F}_j,\partial_y^l\bar{\mathfrak{F}}_i^0)]\nonumber\\
		&\quad+\v^{2\mathfrak{b}-4}\Upsilon(\v^{\mathfrak{a}}\eta)\sum\limits_{j=1}^N\v^{j-1}\frac{(-\eta)^{\mathfrak{b}+1}}{(\mathfrak{b}+1)!}[Q(\partial_r^{\mathfrak{b}+1}\tilde{\mu},\hat{F}_j)+Q(\hat{F}_j,\partial_r^{\mathfrak{b}+1}\tilde{\mu})]\nonumber\\
		&\quad+\v^{2\mathfrak{b}-3}\Upsilon(\v^{\mathfrak{a}}\eta)\sum\limits_{i,j=1}^N\v^{i+j-2}\frac{(-\eta)^{\mathfrak{b}+1}}{(\mathfrak{b}+1)!}[Q(\partial_r^{\mathfrak{b}+1}\mathcal{F}_i,\hat{F}_j)+Q(\hat{F}_j,\partial_r^{\mathfrak{b}+1}\mathcal{F}_i)]\nonumber\\
		&\quad+\v^{\mathfrak{b}-4}\Upsilon(\v^{\mathfrak{a}}\eta)\sum\limits_{i,j=1}^N\v^{i+j-2}\frac{\eta^{\mathfrak{b}+1}}{(\mathfrak{b}+1)!}[Q(\partial_y^{\mathfrak{b}+1}\bar{\mathcal{F}}_i,\hat{F}_j)+Q(\hat{F}_j,\partial_y^{\mathfrak{b}+1}\bar{\mathcal{F}}_i)]\nonumber\\
		&\quad-\sum_{i=1}^N\v^{i-7+\mathfrak{a}}(\mv_1\partial_z\Upsilon(z))\hat{F}_k.\nonumber
	\end{align}
	
%	The equation \eqref{re} is a IBVP problem with given initial data $F_R^\v(0,x,v)$ and boundary conditions as 
%	\begin{align*}
%		F_R^\v(t,x,v)|_{\gamma_-}=F_R^\v(t,x,R_xv).
%	\end{align*}
	
For later use, we define
$\tilde{w}_{\mathbf{k}}:=(1+|v|^2)^{\f{\mathbf{k}}{2}}$ with $\mathbf{k}$ to be determined. Our main result is
%The main aim of the present paper is to establish the validity of the Hilbert expansion for the Boltzmann equation around the local Maxwellian 
	\begin{Theorem}\label{thm}
	 Let $\mathbf{k}\geq 6,N\geq6$ and $\mathfrak{b}\geq 5$. We assume the initial data 
		\begin{align*}
			F^\v(0,x,v)=&\mu(0,x,v)+\sum\limits_{i=1}^N\v^i\{F_i(0,x,v)+\bar{\mathfrak{F}}_i(0,y,\phi,\bar{v})+\hat{\mathscr{F}}_i(0,\eta,\phi,\mv_1,\mv_2)\}\nonumber\\
			&+\v^5F_R^\v(0)\geq0,
		\end{align*}
		and $F_i(0),\bar{\mathfrak{F}}_i(0),\hat{\mathscr{F}}_i(0),i=0,1,...,N$ satisfy the regularity and compatibility condition described in Proposition \ref{prop}, and
		\begin{align*}
			\|(\frac{F_R^\v}{\sqrt{\mu}})(0)\|_{L^2_{x,v}}+\v^3\|(\tilde{w}_{\mathbf{k}}\frac{F_R^\v}{\sqrt{\mu_M}})(0)\|_{L_{x,v}^\infty}\textless\infty.
		\end{align*}
		Then there exists a small positive constant $\v_0\textgreater0$ such that equation \eqref{1.5}--\eqref{1.6} has a unique solution for $\v\in(0,\v_0]$ over the time interval $t\in[0,\tau^\d]$ in the following form of expansion
		\begin{align*}
			F^\v(t,x,v)=&\mu(t,x,v)+\sum\limits_{i=1}^N\v^i\{F_i(t,x,v)+\bar{\mathfrak{F}}_i(t,\phi,\frac{1-r}{\v},\bar{v})+\hat{\mathscr{F}}_i(t,\phi,\frac{1-r}{\v^2},\mv_1,\mv_2)\}\nonumber\\
			&+\v^5F_R^\v\geq0,
		\end{align*}
		with
		\begin{align}\label{1.31}
			\sup_{t\in[0,\tau^\d]}\Big(\|(\frac{F_R^\v}{\sqrt{\mu}})(t)\|_{L^2_{x,v}}+\v^3\|(\tilde{w}_{\mathbf{k}}\frac{F_R^\v}{\sqrt{\mu_M}})(t)\|_{L_{x,v}^\infty}\Big)\leq C(\tau^\d)\textless\infty,
		\end{align}
		where $\tau^\d>0$ is the life-span of smooth solution of compressible Euler equation \eqref{1.9}--\eqref{1.9-1}. Here the functions $F_i(t,x,v),\bar{\mathfrak{F}}_i(t,\phi,y,v),\hat{\mathscr{F}}_i(t,\phi,\eta,\mv)$ are respectively the internal expansion, viscous and Knudsen boundary layers constructed in Section  \ref{sec5.1}.
	\end{Theorem}
	\begin{remark}
		From \eqref{1.31} and Section \ref{sec5.1}, we have 
		\begin{align*}
			\sup\limits_{t\in[0,\tau^\d]}\left\{\left\|\left(\f{F^\v-\mu}{\sqrt{\mu}}\right)(t)\right\|_{L^2(B_1\times\R^2)}+\left\|\tilde{w}_{\mathbf{k}}\left(\f{F^\v-\mu}{\sqrt{\mu_M}}\right)(t)\right\|_{L^\infty(B_1\times\R^2)}\right\}\leq C\v\to 0.
		\end{align*}
	That means we have justified the validity of  hydrodynamic limit from the Boltzmann equation to the compressible Euler system for the specular boundary problem in 2D disk.
	\end{remark}

	\subsection{Key points of the proof}
	We now make some comments on the key points of this paper.
	
	 In the disk, due to the coexistence of nonzero boundary curvature and nonzero tangential velocity $u_\tau^0\neq 0$ at the boundary,  it is highly nontrivial to construct the solution for geometrically modified Knudsen layers. 
	 %To address the problem, utilizing the smallness of $\v$, we employ a refined $L^2-L^\infty$ estimate method that systematically tracks the interplay between these norms. 
	 On the other hand, to control the remainder term, we would like to use the $L^2-L^\infty$ arguments in  \cite{Guo-2010-ARMA} where it is important to remove some singular region with small measure (it may depend on $\v$) so that the change of variable can be applied. Since the backward characteristics are complex, it is not an easy job to describe the $\v$-dependent singular region. 
	 %By removing this problematic region, the example can exhibit limited collisions within finite time, thereby allowing change of variables and establishing the construction of the $L^2-L^\infty$ framework. Nevertheless, the domain is related to $\v$ and difficult to get specific calculation. 
	 In this paper, by using the polar coordinate,  we give an explicit description for the backward characteristic of the remainder, revealing finite collision times under $|\bar{\mv}|\leq N, \ |\bar{\mv}_1/\bar{\mv}_2|\geq\bar{m}$,  which facilitates removing the singular part.
We explain the key ideas of our proof in the following.

\subsubsection{The existence of Knudsen layer with geometric correction and nonzero $u^0_{\tau}\neq0$.} We consider the following boundary value equation
\begin{align}\label{1-49}
	\begin{cases}
		\dis \mathcal{L}_\lambda f= S, \\
		\dis	f(0,\mv)|_{\gamma_{-}}=f\left(0,-\mv_1, \mv_2\right),\\
		\dis f(d, \mv)|_{\gamma_-}=0.
	\end{cases}
\end{align}
where
\begin{equation}\label{1-49-0}
\mathcal{L}_\lambda f:=	\mv_1\partial_\eta f +G(\eta)\left(\mv_2^2\f{\partial f}{\partial \mv_1}-\mv_1 \mv_2\f{\partial f}{\partial \mv_2}\right)+(1-\lambda)\nu f +\lambda \FL_0f -\f{u_\tau^0}{2T^0}G(\eta)\mv_1 \mv_2 f.
\end{equation}
In the following, we shall denote the solution of \eqref{1-49} as $f_{\lambda}:=\mathcal{L}_\lambda^{-1}S$. We aim to establish the existence of  $\mathcal{L}_\lambda^{-1}S$, especially $\mathcal{L}_1^{-1}S$ which is in fact our original Knudsen boundary layer problem \eqref{K}. 
%Firstly, we obtain the existence of $\mathcal{L}_0^{-1}S$. Subsequently, we derive $\lambda$--independent {\it a priori} estimate for $\mathcal{L}_\lambda^{-1}S$. Then we get the existence of $\mathcal{L}_1^{-1}S$ by a step-wise scheme.
Here we assume zero incoming boundary at $\eta=d$, which is essential for us to derive the exponential decay of $\eta$.

{\it Step 1.} To construct  $\mathcal{L}_1^{-1}S$, we first consider the existence of $\mathcal{L}_0^{-1}S$. It is a standard approach to consider the following approximate BVP: 
		\begin{align}\label{1-1}
			\begin{cases}
				\mathcal{L}_0 f =S,\\
				f(0,\mv_1,\mv_2)|_{\mv_1>0}=(1-\f1n)f(0,-\mv_1,\mv_2), \\
				f(d,\mv)|_{\mv_1<0}=0.
			\end{cases}
		\end{align}
We denote the solution of \eqref{1-1} as $f^n$ and define $\dis \mf^n=\f{\sqrt{\mu_0}}{\sqrt{\mu_M}}f^n$. In fact, for given $n\geq2$, we can obtain the solution  $f^n$ relatively easy by an iterative arguments. The coefficient plays an important role in the construction of solution $f^n$, see steps 1 \& 2 in the proof of Lemma \ref{lem1.1}.
	
Then we consider the limit $n\to \infty$ and it is better to use the equation of $\mathbf{f}^n$. For $n_1>1,n_2>1$, it holds that 
		\begin{align}
			\begin{cases}
				\dis	\mv_1 \partial_\eta (\mf^{n_1}-\mf^{n_2})+G(\eta)\left(\mv_2^2 \frac{\partial\left(\mf^{n_1}-\mf^{n_2}\right)}{\partial \mv_1}-\mv_1 \mv_2 \frac{\partial\left(\mf^{n_1}-\mf^{n_2}\right)}{\partial \mv_2}\right) +\nu\left(\mf^{n_1}-\mf^{n_2}\right)=0, \\
				\dis	\left(\mf^{n_1}-\mf^{n_2}\right)(0, \mv)|_{\gamma_{-}}=\left(1-\frac{1}{n_1}\right)\left(\mf^{n_1}-\mf^{n_2}\right)\left(0,-\mv_1, \mv_2\right)+\left(\frac{1}{n_2}-\frac{1}{n_1}\right) \mf^{n_2}\left(0,-\mv_1, \mv_2\right), \\
				\dis	\left(\mf^{n_1}-\mf^{n_2}\right)(d, \mv)|_{\gamma_-}=0.
			\end{cases}
		\end{align}
		Along its backward characteristic   in \eqref{5-6}, one has 
		\begin{align}\label{1.32}
			w_\beta(\mf^{n_1}-\mf^{n_2})(\eta,\mv)=&(1-\f{1}{n_1})^k
			w_\beta(\mf^{n_1}-\mf^{n_2})(0,v_k)e^{-\nu(\mathfrak{t}-\mathfrak{t}_k)}\nonumber\\
			&+\sum_{l=1}^k(1-\f{1}{n_1})^{l-1}(\f{1}{n_2}-\f{1}{n_1})w_\beta\mf^{n_2}(0,v_l)e^{-\nu(\mathfrak{t}-\mathfrak{t}_l)}.
		\end{align}
		%\begin{align}
		%	&\|w(\mf^{n_1}-\mf^{n_2})\|_{L^\infty_{\eta,v}}+|w(\mf^{n_1}-\mf^{n_2})|_{L^\infty(\gamma_+)}\nonumber\\
		%	&\leq  (1-\f1n)^k|w(\mf^{n_1}-\mf^{n_2})|_{L^\infty(\gamma_+)}e^{-\nu(t-t_k)}+Ck(\f{1}{n_1}+\f{1}{n_2})|w\mf^{n_2}|_{L^\infty(\gamma_+)}.		
Due to the geometric correction in \eqref{1-49-0}, its backward characteristics are very complex  which brings  difficulty in the $L^\infty$ estimate of \eqref{1.32}.
In fact, the particles may oscillate in $[0,\eta_+]$ where $\eta_+$ is the one defined in \eqref{22.4} (see also \eqref{22.5} for more detail).

It is noted that $\eta_+$ may approach zero if $\eta\ll 1$ and $|\mv_1/\mv_2|\ll 1$. Thus, we have to take the number of collision $k$ sufficiently large so that $\dis e^{-\nu(\mathfrak{t}-\mathfrak{t}_k)}$ will provide a small coefficient for the first term on RHS of \eqref{1.32}. However, for $k\gg1$, difficulty arises in   the second term on RHS of \eqref{1.32}. Fortunately, by noting the normal velocity $V_{cl,1}(X(s))$ decreases with $X(s)$, then it holds that $|v_{k,1}|\geq \tilde{m}$ if $|\mv_1|\geq \tilde{m}$. Also, we find that $(1+|\mv|^2)\eta_+$ has a lower bound if $|\mv_1|\geq \tilde{m}$. Thus by choosing $k=(1+|\mv|^2)k_0(\tilde{m})$ with $k_0(\tilde{m})\gg 1$, we have
 \begin{align*}
|\mathbf{1}_{\{|\mv_1|\geq \tilde{m}\}}w_{\beta}(\mf^{n_1}-\mf^{n_2})(0,v_k)|	e^{-\nu(\mathfrak{t}-\mathfrak{t}_k)}%\nonumber\\
%  &\leq |\mathbf{1}_{\{|\mv_1|\geq \tilde{m}\}}w_{\beta}(\mf^{n_1}-\mf^{n_2})|_{L^\infty(\gamma_+)} e^{-(k-1)\int_{\mathfrak{t}_2}^{\mathfrak{t}_1}|V_1|\,d\tau }\nonumber\\
  &\leq |\mathbf{1}_{\{|\mv_1|\geq \tilde{m}\}}w_{\beta}(\mf^{n_1}-\mf^{n_2})|_{L^\infty(\gamma_+)} e^{-2(k-1)\eta_+}\nonumber\\
  &\leq \f18  |\mathbf{1}_{\{|\mv_1|\geq \tilde{m}\}}w_{\beta}(\mf^{n_1}-\mf^{n_2})|_{L^\infty(\gamma_+)},
  \end{align*}
% thereby controlling the first term of \eqref{1.32}. We remark that the assumption of the zero incoming condition at $\eta=d$ is important in our analysis. With this choice of $k$, we obtain the $L^\infty$ estimate as 
which yields that 
		\begin{align}\label{1.52}
			&\|\mathbf{1}_{\{|\mv_1|\geq \tilde{m}\}} w_{\beta}(\mf^{n_1}-\mf^{n_2})\|_{L^\infty_{\eta,\mv}}+|\mathbf{1}_{\{|\mv_1|\geq \tilde{m}\}}w_{\beta}(\mf^{n_1}-\mf^{n_2})|_{L^\infty(\gamma_+)}\nonumber\\
			&\leq  C_{\tilde{m}}(\f{1}{n_1}+\f{1}{n_2})|w_{\beta+2}\mf^{n_2}|_{L^\infty(\gamma_+)}\leq C_{\tilde{m}}(\f{1}{n_1}+\f{1}{n_2})\|\nu^{-1}w_{\beta+2 }\bar{S}\|_{L^\infty_{\eta,\mv}}\,\to 0,
		\end{align} 
	as $ n_1,n_2\to \infty$.
Noting $\tilde{m}$ can be any positive constant, we get the limit function $\mf_0$ and the continuity of $\mf_0$ in $[0,d]\times \{\R^2\backslash {\mv_1=0}\}$. 
	
Next we consider the limit at $(0,d]\times \R^2$. It is better to rewrite \eqref{1.32} in the following form
	\begin{align}\label{1.51}
		w_\beta(\mf^{n_1}-\mf^{n_2})(\eta,\mv)&=(1-\f1n)^kw_\beta(\mf^{n_1}-\mf^{n_2})(\eta,V_{cl}(\mathfrak{t}_c))e^{-\nu(\mathfrak{t}-\mathfrak{t}_k)-\nu(\mathfrak{t}_k-\mathfrak{t}_c)}\nonumber\\
		&+\sum_{l=1}^k(1-\f1n)^{l-1}w_\beta\mf^{n_2}(0,v_l)e^{-\nu(\mathfrak{t}-\mathfrak{t}_l)},
	\end{align}
	where we have denoted $\mathfrak{t}_c\in(\mathfrak{t}_{k+1},\mathfrak{t}_k)$ to be the time so that the backward characteristic moves from $X(\mathfrak{t}_k)=0$ to $X(\mathfrak{t}_c)=\eta$. 
	% without reaching $\eta_+$. 
	For $\eta\geq \tilde{m}'$, it is key to note that $\eta_+\geq\eta\geq \tilde{m}'$, then we choose $k=k_0'(\tilde{m}')$ with $k_0'(\tilde{m}')\gg 1$ such that
	\begin{align*}
		e^{-\nu(\mathfrak{t}-\mathfrak{t}_k)}\leq e^{-(k-1)\int_{\mathfrak{t}_2}^{\mathfrak{t}_1}|V_{cl,1}|\,d\tau }= e^{-2(k-1)\eta_+}\leq \f18.
	\end{align*}
	By some detailed analysis, one obtains from \eqref{1.51} that 
	\begin{align}\label{1.53}
		\|\mathbf{1}_{\{\eta\geq \tilde{m}'\}} w_\beta (\mf^{n_1}-\mf^{n_2})\|_{L^\infty_{\eta,\mv}}&\leq \f18\|\mathbf{1}_{\{\eta\geq \tilde{m}'\}} w_\beta (\mf^{n_1}-\mf^{n_2})\|_{L^\infty_{\eta,\mv}}+C_{\tilde{m}'}(\f{1}{n_1}+\f{1}{n_2})|w_\beta \mf^{n_2}|_{L^\infty(\gamma_+)}\nonumber\\
		&\leq C_{\tilde{m}'}(\f{1}{n_1}+\f{1}{n_2})|w_\beta\mf^{n_2}|_{L^\infty(\gamma_+)}\, \to 0,\quad n_1,n_2\to\infty.
	\end{align}
	It follows from \eqref{1.52} and \eqref{1.53} that $w_\beta\mf^{n} \to w_\beta\mathbf{f}_0$ in $L^\infty(\tilde{m}',d]\times \R^2)$ and  $w_{\beta-2}\mf^{n} \to w_{\beta-2}\mathbf{f}_0$ in $L^\infty([0,d]\times \{\mv\in\R^2:|\mv_1|\geq \tilde{m}\})$.	Since $\tilde{m},\tilde{m}'$ can be arbitrary small, it is clear to know that $w_{\beta-2}\mf^{n} \to w_{\beta-2}\mathbf{f}_0$ a.e. in $\Omega_d\times\R^2$ and also $f^n-f_0 \to 0$ in $L^2(\Omega_d\times\R^2)$ with $f_0=\f{\sqrt{\mu_M}}{\sqrt{\mu_0}}\mathbf{f}_0$. 
	
	Finally, due to the arbitrary smallness of $\tilde{m},\tilde{m}'$ and the continuity of $\mathbf{f}^n$, by noting \eqref{1.52} and \eqref{1.53}, it is clear to see that $\mf_0$ (or $f_0$) is continuous away from grazing set $\{(\eta,\mv):\eta=0,\mv_1=0\}$. 
	
\smallskip

{\it Step 2.} To extend $\mathcal{L}_0^{-1}S$ to $\mathcal{L}_1^{-1}S$, we need some $L^2$	and $L^\infty$ uniform-in-$\lambda$ estimates for $\lambda\in[0,1]$.

{\it Step 2.1. The a priori $L^2$-estimate.} We point out that it is almost impossible to control $\dis G(\eta)\f{u_\tau^0}{2T^0}\mv_1\mv_2 f_\lambda$ in pure $L^2$-estimate. By noting the smallness of $G(\eta)\sim \v^2$ for $\eta\in [0,d]$ with $d=\v^{-\mathfrak{a}}$, the key idea is to control it by  weighted $L^\infty$-norm. In fact, multiplying \eqref{1-49} by $f_\lambda$, one has 
\begin{align}\label{1-50}
	&\f12|f_\lambda(d)|_{L^2(\gamma_+)}^2+(1-\lambda)\int_0^d\|f_\lambda\|_{L^2_\mv}^2\, d\eta+\lambda c_0\int_0^d\|(\FI-\FP_0)f_\lambda\|_{\nu}^2\, d\eta\nonumber\\
	&\leq C\v^{2-\mathfrak{a}}\|w\mf_\lambda\|_{L^\infty_{\eta,\mv}}^2+C\int_0^d\int_{\R^2}|f_\lambda\cdot S|\, d\mv \, d\eta,
\end{align}
where the small coefficient $\v^{2-\mathfrak{a}}$ plays an important role.

To derive uniform-in-$\lambda$ $L^2$-estimate, we divide the discussion into two cases:

{\it Case 1. $\lambda \in [0,1-\v]$.} It is clear that  $1-\lambda\geq \v$. By utilizing the smallness of $\v$, we directly have from \eqref{1-50} that
\begin{align}\label{1-53}
	\left\|f_\lambda\right\|_{L_{\eta,\mv}^2}\leq C\v^{\f12(1-\mathfrak{a})}\left\|w_\beta \mf_\lambda\right\|_{L_{\eta,\mv}^{\infty}}+C_d\|S\|_{L_{\eta,\mv}^2}.
\end{align}

{\it Case 2. $\lambda \in [1-\v,1]$.} Since $\lambda\geq 1-\v$, we shall utilize  the good term $\dis \int_0^d\|(\FI-\FP_0)f_\lambda\|_{\nu}^2\, d\eta$ to derive the desired $L^2$ estimate. In fact, by taking some suitable test functions as in \cite{GPS,HJW}, we can control the macroscopic part  $\|\mathbf{P}_0f_\lambda\|_{L^2}$ by the microscopic part, and then we obtain
\begin{align}\label{1-52-0}
	\|f_\lambda\|_{L^2_{\eta,\mv}}\leq C\v^{\f12(\min\{2-3\mathfrak{a},1-\mathfrak{a}\})}\|w_\beta\mf_\lambda\|_{L^\infty_{\eta,\mv}}+C_{\v,d}\|S\|_{L^2_{\eta,\mv}}.
\end{align}
We refer to Lemma \ref{lem1.2-1} for details.

\smallskip
	
{\it Step 2.2. The a priori $L^\infty$-estimate.} We shall use $\mf_\lambda\equiv \f{\sqrt{\mu_0}}{\sqrt{\mu_M}}f_\lambda$ to consider the $L^\infty$-estimate. By some detailed analysis, we have 
%and denote $\|g\|_{L^\infty L^\infty_{\beta,\zeta}}$ as $\|(1+|v|^2)^{\f{\beta}{2}}e^{\zeta|\mv|^2}g \|_{L^\infty_{\eta,\mv}}+|(1+|\mv|^2)^{\f{\beta}{2}}e^{\zeta|\mv|^2}g |_{L^\infty(\gamma_+)}$. Inspired by \cite{Wu-2016-JDE}, one has 
	\begin{align*}
	\|w_{\beta} \mathbf{f}_{\lambda} \|_{L^\infty_{\eta,\mv}}+|w_{\beta} \mathbf{f}_{\lambda}|_{L^\infty(\gamma_+)} 
 \leq C \Big|\int_{\R^2}e^{2\zeta|\mv|^2}|\mf_\lambda|^2(\eta,\mv)d\mv\Big|_{L^\infty_{\eta}} + C\|\nu^{-1}w_{\beta}\bar{S}\|_{L^\infty_{\eta,\mv}}.
	\end{align*} 
	Then we need to estimate
	$\dis \int_{\R^2}e^{2\zeta|\mv|^2}|\mf_\lambda|^2(\eta,\mv)d\mv$. Noting the mild formulation of $\mathbf{f}_{\lambda}$ in \eqref{4.31-1}--\eqref{5.10-0}, we have 
		\begin{align}\label{1-52-1}
		& \int_{\R^2}e^{2\zeta|\mv|^2}|\mf_\lambda|^2(\eta,\mv)d\mv \lesssim o(1) \big(\|w_{\beta} \mathbf{f}_{\lambda} \|_{L^\infty_{\eta,\mv}}+|w_{\beta} \mathbf{f}_{\lambda}|_{L^\infty(\gamma_+)} \big)^2 + \|\nu^{-1}w_\beta\bar{S}\|^2_{L^\infty_{\eta,\mv}}\nonumber\\
		&+ \lambda \int_{\R^2}\mathbf{1}_{\{|\mv|\leq \mathfrak{N},|\mv_1|\geq m,|\mv_1/\mv_2|\geq m\}} e^{2\zeta|\mv|^2}\left\{\int_{\mathfrak{t}_k}^\mathfrak{t}e^{-\nu(\mv)(\mathfrak{t}-s)}|K_M\mf_\lambda|(X_{cl}(s),V_{cl}(s))ds\right\}^2\,d\mv,
	\end{align}
	where we have removed some singular region with small measure. We remark that $o(1)$ on RHS of \eqref{1-52-1} depends on the smallness of such singular region, and $k>0$ is a number depending on $m^{-1}$.
%	Along the backward characteristic, we obtain
%	\begin{align*}
%		\mf_\lambda(\eta,\mv)&=\mf_\lambda(\eta_k,v_k)e^{-\nu(\mv)(\mathfrak{t}-\mathfrak{t}_k)}+\lambda  \sum\limits_{l=0}^{k-1}\int_{\mathfrak{t}_{l+1}}^{\mathfrak{t}_l}e^{-\nu(\mv)(\mathfrak{t}-s)}K_M\mf_\lambda\,ds\nonumber\\
%		&+\sum\limits_{l=0}^{k-1}\int_{\mathfrak{t}_{l+1}}^{\mathfrak{t}_l}e^{-\nu(\mv)(\mathfrak{t}-s)}\bar{S}ds.
%	\end{align*}	
	%Under the assumption that the collision time $k$ is finite and after investigating certain small localized regions, the crucial task now is to evaluate the integral 
	%$$\int_{\R^2}\mathbf{1}_{\{|\mv|\leq N,|\mv_1|\geq m,|\mv_1/\mv_2|\geq m\}} e^{2\zeta|\mv|^2}\left\{\int_{\mathfrak{t}_k}^\mathfrak{t}e^{-\nu(\mv)(\mathfrak{t}-s)}|K_M\mf_\lambda|(X_{cl}(s),V_{cl}(s))ds\right\}^2\,d\mv.$$ 
 
To control the last term on RHS of \eqref{1-52-1}, we shall use  the change of variables $ds d\mv\to dXdV$ where we denote $(X_{cl},V_{cl})$ as $(X,V)$. It is a challenging task to directly calculate the Jacobian determinant
$\dis\f{\partial (X,V)}{\partial(s,\mv)}$, and thus it is advisable to compute the Jacobian in a stepwise manner. We note 
	\begin{align*}
		X(s)=\eta_{l}-\int_s^{\mathfrak{t}_{l}}|V_{cl,1}(\tau)|d\tau,\quad s\in(\mathfrak{t}_{l+1},\mathfrak{t}_{l}].
	\end{align*}
Since $\eta_l=0$ or $d$, it holds that 
	\begin{align*}
		\left|\f{dX(s)}{ds}\right|=|V_{cl,1}(s)|=\sqrt{|\mv|^2-\mv_2^2\left(e^{-2W(\eta)+2W(X(s))}\right)}.
	\end{align*}
Then, in $\{\mv\in\R^2:|\mv|\leq \mathfrak{N},|\mv_1|\geq m,|\mv_1/\mv_2|\geq m\}$, we obtain from above formula that 
	\begin{align}\label{1-52-2}
	\left|\f{dX(s)}{ds}\right|^2 =	|V_{cl,1}|^2\geq m^2+|\mv_2|^2\left\{1-\left(\f{1-\v^2\eta}{1-\v^2X(s)}\right)^2\right\} 
		\geq m^2-\mathfrak{N}^2\v^{2-\mathfrak{a}}\geq \f{m^2}{4},
	\end{align}
where the smallness of $\v>0$ is used.
	
Using \eqref{1-52-2}, we first perform the change of variable $ds\to dX$ and express $V_{cl}$ as a function of $X$ and $\mv$, i.e., $V_{cl}=V_{cl}(X,\mv)$. Now it can be checked that $d\mv\sim dV$. 

Combining above estimates, we can  establish 
	\begin{align}\label{1-57} 
		\|w_\beta\mf_\lambda\|_{L^\infty_{\eta,\mv}}+|w_\beta\mf_\lambda|_{L^\infty(\gamma_+)}\leq C\|f_\lambda\|_{L^2_{\eta,\mv}}+C\|\nu^{-1}w_\beta \bar{S}\|_{L^\infty_{\eta,\mv}}.
	\end{align}
We remark that the constant on RHS of \eqref{1-57} is independent of $d>0$. We refer to Lemma \ref{lem1.2} for details of proof in {\it Step 2.2}.

\smallskip

{\it Step 2.3.} From {\it Steps 2.1} \& {\it 2.2}, we can obtain the uniform-in-$\lambda$ {\it a priori} estimate
\begin{align}
\|w_\beta\mf_\lambda\|_{L^\infty_{\eta,\mv}}+|w_\beta\mf_\lambda|_{L^\infty(\gamma_+)}\leq C_{\v,d}\|\nu^{-1}w_\beta \bar{S}\|_{L^\infty_{\eta,\mv}}.
\end{align}
Then, by the continuity arguments, we can extend the existence result and energy estimate from $\mathcal{L}_0^{-1}$ to $\mathcal{L}_1^{-1}$, see Lemma \ref{lem1.3} for details.

\smallskip

{\it Step 3. Exponential decay of Knudsen layer solutions.} For the solution $f=\mathcal{L}_1^{-1}S$ obtained in Step 2, 
%	We consider equation as follows:
%\begin{align}\label{1-51}
%	\begin{cases}
%		\dis \mathcal{L}_1 f
%		%	\mv_1\partial_\eta f+G(\eta)(\mv_2^2\frac{\partial f}{\partial \mv_1}-\mv_1 \mv_2 \frac{\partial f}{\partial \mv_2})-\frac{u_\tau^0}{2T^0}G(\eta)\mv_1 \mv_2 f+\FL_0f
%		=S,\\
%		\dis	f(0,\mv)|_{\gamma_-}=f(0,-\mv_1,\mv_2),\\
%		\dis	f(d,\mv)|_{\gamma_-}=0,
%	\end{cases}
%\end{align}
%Since the structure of \eqref{1-51} is complex, it is hard to compensate the far-field with some macroscopic functions as in \cite{HJW} when obtaining exponential spatial decay. 
using the zero incoming boundary condition at $\eta=d$ and the special test functions in \cite{GPS,HJW}, we obtain the exponential decay estimate
\begin{align}
	\|e^{\sigma\eta}f\|_{L^2_{\eta,\mv}}+e^{\sigma d}|f(d)|_{L^2(\gamma_+)}\leq \frac{C}{\sqrt{\sigma}}\v^{1-\f{\mathfrak{a}}{2}}\|e^{\sigma\eta}w\mf\|_{L^\infty_{\eta,\mv}}+\f{C}{\sqrt{\sigma(\sigma_2-\sigma)}}\|e^{\sigma_2\eta}S\|_{L^2_{\eta,\mv}},
\end{align}
which yields immediately that 
\begin{align}
	\|e^{\sigma\eta}w\mf\|_{L^\infty_{\eta,\mv}}+|e^{\sigma\eta}w\mf|_{L^\infty(\gamma_+)}\leq \f{C}{\sigma(\sigma_0-\sigma)}\|e^{\sigma_0\eta}\nu^{-1}w\bar{S}\|_{L^\infty_{\eta,\mv}},
\end{align}
see Section 4.2 for details. 
We point out that above decay estimate is crucial for us to close the estimation of remainder of Hilbert expansion.
	
%\Blue{The intricate structure of \eqref{K} poses a critical challenge in establishing exponential spatial decay of Knudsen layers, as standard compensation techniques from \cite{HJW,GPS} become inapplicable for far-field behavior analysis.}
%Multiplying \eqref{1-51} by $\dis f\cdot e^{2\sigma \eta}e^{-W(\eta)}$, under the solvable conditions, one obtains
%\begin{align}\label{1-52}
%		\int_0^d\|(\FI-\FP_0)f\|_\nu^2\cdot e^{2\sigma\eta}\,d\eta + e^{2\sigma d}|f(d)|_{L^2(\gamma_+)}^2\leq 	C\v^{2-\mathfrak{a}} \|e^{\sigma\eta}w\mathbf{f}\|_{L^\infty_{\eta,\mv}}^2+C\|e^{\sigma\eta}S\|_{L^2_{\eta,\mv}}^2.
%	\end{align}
%We can see from \eqref{1-52} that the zero incoming boundary at $\eta=d$ is very important. If we consider specular boundary condition or non-zero incoming boundary at $\eta=d$, we can not obtain the exponential spatial decay of $f$.  

\subsubsection{ The remainder estimate.} 
To control the remainder term, since the collisions are very complicate, we use the polar coordinate to study its backward characteristics, see \eqref{6.4}. By  tedious calculations, we derive 
\begin{align}\label{1-6-32}
	\dis	\mathcal{X}\left |\f{\partial(\mathcal{X},\Phi)}{\partial \bar{\mv}}(s)\right | 
	=&\f{|\bar{\mv}_1|}{|\bar{\mv}_2|}(s-t_{1,0})(t-s) \cdot\left|\f{2}{\sqrt{\f{1}{r^2}\f{|\bar{\mv}|^2}{|\bar{\mv}_2|^2}-1}}-\f{|\bar{\mv}_2|}{|\bar{\mv}_1|}-\f{1}{\sqrt{\f{\mathcal{X}^2}{r^2}\f{|\bar{\mv}|^2}{|\bar{\mv}_2|^2}-1}}\right|.
\end{align}
See Lemma \ref{lem6.2-0} for details.
%the Jacobian determinant $\dis\f{\partial(\mathcal{X} ,\Phi)}{\partial\bar{\mv}}$
%	\begin{align}\label{1-54}
%		&\int_{\R^2}\f{1}{\v^4}\left\{\int_0^t\int_{\R^2}k_R(\mathcal{V}_{cl}(s),v')\mf^\v_R(s,\mathcal{X}_{cl}(s),\Phi_{cl}(s),v')dv'\cdot e^{-\f{1}{\v^2}\tilde{\nu} (t-s)}ds\right\}^2d\mv\nonumber\\
%		&\leq C\sup_{0\leq s\leq t}\|w\mf_R^\v\|_{L^\infty}^2+C_\v\sup_{0\leq s\leq t}\|f_R^\v\|_{L^2}^2.
%	\end{align}

With the help of \eqref{1-6-32}, we divide the $L^2-L^\infty$ estimate into several cases, and remove some  regions with small measure which is singular for \eqref{1-6-32}.
%We analyze five distinct cases, each governed by the inequalities \eqref{6.11}--\eqref{6.15}. After excluding specific small integral regions, we impose the constraint
Then we need only to  consider the change of variable in the domain 
\begin{align}\label{1.62}
\Big\{(s,\bar{\mv},v',\mathcal{X})&\in[0,t]\times\R^2\times\R^2\times[0,1]:\, s-t_{1,0}\geq m_1, t-s\geq m_1, \nonumber\\
& |\bar{\mv}|\leq \mathfrak{N},|v'|\leq 2\mathfrak{N}, |\bar{\mv}_{1}/\bar{\mv}_{2}|\geq \bar{m}, |\bar{\mv}_2|\geq \bar{m},\, \text{and}\quad
|\f{\mathcal{X}}{r}-1|\geq m_2\Big\}.
\end{align}
%  and $$\f{\mathcal{X}}{r}\leq 1-m_2\quad  \text{or} \quad \f{\mathcal{X}}{r}\geq 1+m_2.$$ 
Under \eqref{1.62}, it holds that  
\begin{align}
	\mathcal{X}\left |\f{\partial(\mathcal{X},\Phi)}{\partial \bar{\mv}}(s)\right |\geq  \f{m_2}{2}m_1^2.
\end{align}
Then, by the change of variables $d\mv\to \mathcal{X}d\mathcal{X} d\Phi$, we have 
\begin{align*}
	\sup_{0\leq s\leq t}\|h^\v_R(s)\|_{L^\infty}^2\leq &C(\|h^\v_R(0)\|_{L^\infty}^2+1)+ \left\{\f{1}{\mathfrak{N}}+C\left(\bar{m}\mathfrak{N}+Ce^{-\f{1}{3\v^2}\bar{m}^2}\right)\right\}\sup_{0\leq s\leq t}\|h^\v_R(s)\|_{L^\infty}^2\nonumber\\
	&+\left(Cm_1^2+\f{C\mathfrak{N}t}{\bar{m}^6}m_2^2\right)\cdot\f{1}{\v^4}\sup_{0\leq s\leq t}\|h^\v_R(s)\|_{L^\infty}^2+\f{C_\mathfrak{N}}{ m_1^2m_2}\sup_{0\leq s\leq t}\|f^\v_R(s)\|_{L^2}^2.
\end{align*}
Choosing $\bar{m}=\f{1}{\mathfrak{N}^2}$ with $\mathfrak{N}>1$ suitably large, and $m_1=\bar{m}\v^2, m_2=\bar{m}^4\v^2$, one obtains that
\begin{align}
\sup_{0\leq s\leq t}\|\v^3h^\v_R(s)\|_{L^\infty}\leq  C(t)\{\|\v^3h^\v_R(0)\|_{L^\infty}+\v^{N-1}+\v^{\mathfrak{b}}\}+C\sup_{0\leq s\leq t}\|f^\v_R(s)\|_{L^2},
\end{align}
which, together with the $L^2$ energy estimate, yields that 
\begin{align}
	\sup_{0\leq s\leq t}(\v^3\|h_R^\v(s)\|_{L^\infty})\leq C(t).
\end{align}
	See Section 5.2 for more details.
	%We define the characteristics of \eqref{6.05} as \eqref{6.4}. If $\mv_2 \cdot r=0$, the backward characteristic line moves from $(x,\bar{\mv})$ to its counterpoint and back straightly with velocity $|\mv_1|$. This case is easy. Otherwise, if $\mv_2\cdot r\neq 0$, the line collides without reaching the center. Then there is a $x_+>0$ that $\mathcal{X}$ cycles between $x_+$ and $1$.
	
	\smallskip
	
	%\subsection{Structure} 
	The paper is organized as follows:
	In Section \ref{sec2}, we reformulate the interior, viscous and Knudsen boundary layers. Also we derive the corresponding boundary conditions so that the formulations of interior expansion, the viscous and Knudsen boundary layers are all well-posed. Section \ref{sec3} is devoted to existence theories for a linear hyperbolic system with characteristic boundary and a linear parabolic system with degenerate viscosity, which are used for the construction of $F_i$ and $\bar{\mathfrak{F}}_i$. Section \ref{sec4} provides a complete proof of Theorem \ref{thm2.2}, where we establish the existence of Knudsen boundary layer problem and obtain spatial exponential decay. In Section \ref{sec5}, we construct each layer and control the remainder estimate, which completes the proof of Theorem \ref{thm}.
	
	\smallskip
	
	%\subsection{Notations}
\noindent{\bf Notations.} Throughout this paper,  $C>0$ denotes a constant. It is
	referred to as universal and can change from one inequality to another. When we write $C(z)$ or $C_z$, it means a certain positive constant depending on the quantity $z$. We use $\langle\cdot ,\cdot \rangle$ to denote the standard $L^2$ inner product in $\R^2_v$ and $\|\cdot\|_{\nu}=\|\sqrt{\nu}\cdot\|_{L^2_v}.$
	% while we use $(\cdot ,\cdot )$ to denote the $L^2$ inner product
$\|\cdot\|_{L^2}$ denotes the standard $L^2(B_1\times\mathbb{R}^2_v)$-norm, and $\|\cdot\|_{L^\infty}$ denotes the $L^\infty(B_1\times\mathbb{R}^2_v)$-norm.
	
	\section{Reformulation and boundary conditions}\label{sec2}
	Firstly we introduce the existence result on the compressible Euler equations.
	\begin{Lemma}[\cite{Schochet}]\label{lem3.1}
		Let $s_0\geq 2$ be some positive number. Consider the IBVP of compressible Euler equations \eqref{1.9}--\eqref{1.9-1}. Choose $\delta\in(0,\delta_1]$ so that for any $\d\in(0,\d_1],$ the positivity of $\rho_0$ and $T_0$ is guaranteed. Then for $\d\in(0,\d_1]$, there is a family of classical solutions $(\rho^\d,\mathfrak{u}^\d,T^\d)\in C([0,\tau^\d];H^{s_0}(B_1))\cap C^1([0,\tau^\d];H^{s_0-1}(B_1))$ of IBVP \eqref{1.9}--\eqref{1.9-1} such that $\rho^\d>0$ and $T^\d>0$, and the following estimate holds:
		\begin{align}\label{3.0}
			\|(\rho^\d-1,\mathfrak{u}^\d,T^\d-1)\|_{C([0,\tau^\d];H^{s_0-}(B_1))\cap C^1([0,\tau^\d];H^{s_0-1-}(B_1))}\leq C_0.
		\end{align}
		The constant $C_0$ is independent of $\d$, depending only on the $H^{s_0}$ norm of $(\phi_0,\Phi_0,\vartheta_0)$.
	\end{Lemma}

We refer \cite{Schochet} for the local existence of the IBVP of compressible Euler equation \eqref{1.9}-\eqref{1.9-1}, see also \cite{Chen} and the references cited therein. 
\begin{remark}
	By the analysis in Section 3.1, we know that the following holds for 2D-disk problem:
	 	\begin{align}
	 	\|(\rho^\d-1,\mathfrak{u}^\d,T^\d-1)\|_{C([0,\tau^\d];H^{s_0}(B_1))\cap C^1([0,\tau^\d];H^{s_0-1}(B_1))}\leq C_0.
	 \end{align}
\end{remark}
	\subsection{Reformulation of internal solutions} We introduce the linear hyperbolic systems for the macroscopic variables $(\rho_{k+1},u_{k+1},\theta_{k+1})$ of internal solutions $f_{k+1}$.
	\begin{Lemma}[\cite{GHW-2021-ARMA}]\label{lem2.2}
	Let $(\rho,\mathfrak{u},T)(t)$ be some smooth solution of compressible Euler equation \eqref{1.9}--\eqref{1.9-1}. For each given non-negative integer $k$, assume $F_k$ is found. Then the macroscopic part $(\rho_{k+1},u_{k+1},\theta_{k+1})$ satisfy the following: 
	\begin{align}\label{2.1}
			\begin{cases}
			\partial_t\rho_{k+1}+\text{\rm div}(\rho u_{k+1}+\rho_{k+1}\mathfrak{u})=0,\\
				\dis 	\rho\{\partial_tu_{k+1}+(u_{k+1}\cdot\nabla_x)\mathfrak{u}+(\mathfrak{u}\cdot\nabla_x)u_{k+1}\}-\frac{\rho_{k+1}}{\rho}\nabla(\rho T)+\nabla_x\left(\frac{\rho\theta_{k+1}+2\rho_{k+1}T}{2}\right)=h_k,\\
				\dis	\rho\{\partial_t\theta_{k+1}+(\theta_{k+1}\text{\rm div}u+2T\text{\rm div}u_{k+1})+\mathfrak{u}\cdot\nabla_x\theta_{k+1}+2u_{k+1}\cdot\nabla_xT\}=g_k,
			\end{cases}
		\end{align}
		where
		\begin{align*}
			h_{k,i}&=-\sum\limits_{j=1}^2\partial_{x_j}\left(T\int_{\mathbb{R}^2}\left[\frac{(v_i-\mathfrak{u}_i)(v_j-\mathfrak{u}_j)}{T}-\frac{1}{2T}|v-\mathfrak{u}|^2\delta_{ij}\right]F_{k+1}dv\right),\\
			g_k&=-\sum\limits_{j=1}^2\partial_{x_j}\left\{\int_{\R^2}(v_j-\mathfrak{u}_j)(|v-\mathfrak{u}|^2-4T)F_{k+1}dv+2\mathfrak{u}_i\int_{\R^2}[(v_i-\mathfrak{u}_i)(v_j-\mathfrak{u}_j)\right.\nonumber\\
			&\qquad\qquad\qquad\qquad-\left.\frac12|v-\mathfrak{u}|^2\delta_{ij}]F_{k+1}dv\right\}+2\mathfrak{u}\cdot h_k.
		\end{align*}
		We have used  the subscript $k$ for the source terms $h_{k}$ and $g_{k}$ in order to emphasize that the right hand side depends only on $F_i$'s for $0\leq i\leq k$.
	\end{Lemma}
 The detailed calculations leading to \eqref{2.1} are omitted here for brevity; we refer readers to \cite{GHW-2021-ARMA} for a comprehensive derivation. The microscopic part $(\FI-\FP)f_{k+1}$ can be determined by \eqref{1.10}.
 
 \begin{remark}\label{rmk-ie}
 	To solve \eqref{ie}, the problem equivalently reduces to solving the linear hyperbolic system \eqref{2.1}. As an initial-boundary value problem, proper boundary condition must be specified for \eqref{2.1}. However, the boundary data cannot be arbitrarily assigned: to ensure that each component $F_k+\bar{F}_k+\hat{F}_k$ satisfies the specular reflection boundary condition and guarantees solvability of the Knudsen layer, careful constraints are required.
 	Determining the boundary condition constitutes a highly technical procedure, the details of which are systematically addressed in Section 2.4. 
 \end{remark}
\begin{remark}
	With suitable boundary condition, the existence of smooth solution of hyperbolic equations \eqref{2.1} can be found in Section 3.1.
\end{remark}
  
	\subsection{Reformulation of viscous boundary layers}
	By $\eqref{vs}_1$, we have $\bar{f}_1 \in \mathcal{N}_0$, i.e.
	\begin{equation}
		\bar{f}_1 = \mathbf{P}_0\bar{f}_1 = \left\{\frac{\bar{\rho}_1}{\rho^0} + \bar{u}_1 \cdot \frac{\bar{v}-\bar{\mathfrak{u}}^0}{T^0} + \frac{\bar{{\theta}}_1}{4T^0}(\frac{|\bar{v}-\bar{\mathfrak{u}}^0|^2}{T^0} -2) \right\}\sqrt{\mu_0}.
	\end{equation}
	The far-field assumption \eqref{vs-1} becomes 
	\begin{equation}\label{2.3}
		\lim\limits_{y\rightarrow \infty} \bar{f}_k(t,y,\phi,\bar{v}) = 0.
	\end{equation} 
	Multiplying $\eqref{vs}_2$ by $1,(\bar{v}\cdot \vec{n})$  and integrating on $v$, one has 
\begin{align*}
	&\partial_y \langle\bar{v}\cdot \vec{n},\bar{F}_1\rangle  = \partial_y(\rho^0\bar{u}_1\cdot \vec{n}) = \rho^0\partial_y(\bar{u}_1\cdot \vec{n}) = 0,\\
	&\partial_y \langle|\bar{v}\cdot \vec{n}|^2,\bar{F}_1\rangle  =\f12 \partial_y(\bar{\rho}_1T^0 + 2\rho^0\bar{\theta}_1) = \partial_y \bar{p}_1 = 0,
\end{align*}
which, together with \eqref{2.3}, yields that
\begin{equation}\label{2.4}
	\bar{u}_1\cdot \vec{n}\equiv 0,\quad \bar{p}_1(t,y,\phi,\bar{v})\equiv 0,\quad \forall (t,y,\phi,\bar{v})\in[0,\tau^\d) \times [0,\infty)\times[0,2\pi) \times \mathbb{R}^2.
\end{equation}

We now formulate the linear parabolic system for $(\bar{u}_k\cdot \vec{n},\bar{\theta}_k)$ and the equations for $(\bar{u}_{k+1}\cdot \vec{n},\bar{p}_{k+1})$.
\begin{Lemma}[\cite{GHW-2021-ARMA}]\label{lem2.3}
	The parabolic system on $(\bar{u}_k\cdot\vec{\tau},\bar{\theta}_k)$ is as following:
	\begin{align}\label{bu-0}
		\begin{cases}
			\rho^0\partial_t(\bar{u}_k\cdot\vec{\tau})-\kappa_1(T^0)\partial_{yy}(\bar{u}_k\cdot\vec{\tau})\\
			\quad -\mathcal{H}^\v(y)\rho^0u_\tau^0\partial_\phi(\bar{u}_k\cdot\vec{\tau})+\rho^0(-y\partial_r\bar{\mathfrak{u}}^0\cdot \vec{n}-u^0_1\cdot \vec{n})\partial_y(\bar{u}_k\cdot\vec{\tau})\\
			\quad +(\bar{u}_k\cdot\vec{\tau})\cdot\left\{\left(1-\mathcal{H}^\v(y)\right)u_\tau^0\partial_\phi\rho^0+\left(1-2\mathcal{H}^\v(y)\right)\rho^0\partial_\phi u_\tau^0 \right\}\\
			\quad -\f{\rho^0}{2T^0}\bar{\theta}_k\cdot\left\{\left(1-\mathcal{H}^\v(y)\right) u_\tau^0\partial_\phi  u_\tau^0+\partial_\phi T^0+\f{T^0}{\rho^0}\partial_\phi\rho^0\right\}=\bar{h}_{k-1},\\
			\\
			\rho^0\partial_t\bar{\theta}_k-\kappa_2(T^0)\partial_{yy}\bar{\theta}_k+\rho^0(-y\partial_r\bar{\mathfrak{u}}^0\cdot \vec{n}-u_1^0\cdot \vec{n})\partial_y\bar{\theta}_k-\mathcal{H}^\v(y)\rho^0u_\tau^0\partial_\phi\bar{\theta}_k\\
			\quad +\bar{\theta}_k\left\{\left(1-\mathcal{H}^\v(y)\right)\partial_\phi(\rho^0u_\tau^0)-\f{\rho^0}{T^0}\left(1-\mathcal{H}^\v(y)\right)u_\tau^0\partial_\phi T^0-\rho^0\partial_\phi u_\tau^0-\rho^0\partial_r\bar{\mathfrak{u}}^0\cdot \vec{n}\right\}\\
			\quad +\rho^0(\bar{u}_k\cdot \vec{\tau})\left\{\left(1-2\mathcal{H}^\v(y)\right)\partial_\phi T^0+\f{T^0}{\rho^0}\partial_\phi\rho^0+\left(1-\mathcal{H}^\v(y)\right)u_\tau^0\partial_\phi u_\tau^0\right\}=\bar{g}_{k-1},
		\end{cases}
	\end{align}
	where
	\begin{align*}
		\bar{h}_{k-1}:=	&-\f{\bar{p}_k}{T^0}\cdot\left\{\left(1-\mathcal{H}^\v(y)\right) u_\tau^0\partial_\phi  u_\tau^0+\partial_\phi T^0+\f{T^0}{\rho^0}\partial_\phi\rho^0\right\}\nonumber\\
		&+\mathcal{H}^\v(y)\partial_\phi\bar{p}_k+\left(\mathcal{H}^\v(y)u_\tau^0+\partial_r\bar{\mathfrak{u}}^0\cdot\vec{\tau}\right)\rho^0(\bar{u}_k\cdot \vec{n})\nonumber\\
		&-\rho^0(-y\partial_r \bar{\mathfrak{u}}^0\cdot \vec{\tau}-u^0_1\cdot\vec{\tau}+\bar{u}_1\cdot\vec{\tau})\partial_y(\bar{u}_k\cdot \vec{n})+\bar{W}_{k-1}-T^0\partial_y\langle \mathcal{A}_{12}^0,\bar{J}_{k-1}\rangle,\\
		\bar{g}_{k-1}:=&\, \partial_t\bar{p}_k-\mathcal{H}^\v(y)u_\tau^0\partial_\phi\bar{p}_k-\f{\bar{p}_k}{T^0}\left\{\left(2-2\mathcal{H}^\v(y)\right)u_\tau^0\partial_\phi T^0+2T^0\text{\rm div}\bar{\mathfrak{u}}^0\right\}\nonumber\\
		&-\rho^0\left\{\mathcal{H}^\v(y)|u_\tau^0|^2-2\partial_rT^0\right\}(\bar{u}_k\cdot \vec{n})-\rho^0(-2y\partial_rT^0+\theta_1^0+\bar{\theta}_1)\partial_y(\bar{u}_k\cdot \vec{n})\nonumber\\
		&+\f{1}{2}\bar{H}_{k-1}-\f{1}{2}(2T^0)^{\f32}\partial_y\langle \mathcal{B}_1^0,\bar{J}_{k-1}\rangle,
	\end{align*}
	and $\bar{W}_{k-1},\bar{H}_{k-1}$, $\bar{J}_{k-1}$ and $\kappa_1(T^0),\kappa_2(T^0)$ are defined by \eqref{7.4}, \eqref{7.12}, \eqref{7.55}, \eqref{7.88} and \eqref{6.1-1}. 
	
	One also has the equations for $\bar{u}_{k+1}\cdot \vec{n}$ and $\bar{p}_{k+1}$ as
	\begin{align}
		&\partial_y(\bar{u}_{k+1}\cdot \vec{n})=-\frac{1}{\rho^0}\partial_t\bar{\rho}_k+\mathcal{H}^\v(y)\frac{1}{\rho^0}\partial_\phi(\rho^0\bar{u}_k\cdot\vec{\tau}+\bar{\rho}_ku_\tau^0)+\mathcal{H}^\v(y)\bar{u}_k\cdot \vec{n}, \label{2.8}\\
		\smallskip
		&\partial_y\bar{p}_{k+1}=-\partial_t(\rho^0\bar{u}_k\cdot \vec{n})+\mathcal{H}^\v(y)\partial_\phi\left\{\rho^0u_\tau^0(\bar{u}_k\cdot \vec{n})+T^0\langle\mathcal{A}_{12}^0,(\FI-\FP_0)\bar{f}_k\rangle \right\}\nonumber\\
		&\qquad\qquad-\mathcal{H}^\v(y)\left\{(2\rho^0u_\tau^0\bar{u}_k\cdot\vec{\tau}+\bar{\rho}_k|u_\tau^0|^2)+2T^0\langle\mathcal{A}_{22}^0,(\FI-\FP_0)\bar{f}_k\rangle\right\}\nonumber\\
		&\qquad\qquad-T^0\partial_y\langle \mathcal{A}_{11}^0,(\FI-\FP_0)\bar{f}_{k+1}\rangle. \label{2.9}
	\end{align}
	The boundary condition of \eqref{2.8} and \eqref{2.9} is given
	\begin{align}\label{2.9-1}
		\lim_{y\to\infty}\bar{u}_{k+1}\cdot \vec{n}=0,\quad \lim_{y\to\infty}\bar{p}_{k+1}\cdot \vec{n}=0.
	\end{align}
\end{Lemma}

For brevity, the explicit calculations leading to \eqref {bu-0} are omitted here. As the derivation involves geometric terms, the full details are provided in Appendix 6.1 to ensure completeness.
 \begin{remark}
	To solve \eqref{1.16}, the problem equivalently reduces to solving the linear parabolic system \eqref{bu-0}--\eqref{2.9-1}. As an initial-boundary value problem, proper boundary condition must be specified for \eqref{bu-0}. As mentioned in Remark \ref{rmk-ie}, the boundary data cannot be arbitrarily assigned. One can see the details on obtaining boundary condition in Section 2.4. 
\end{remark}
\begin{remark}
	With suitable boundary condition, the existence of smooth solution of parabolic equations \eqref{bu-0} can be found in Section 3.2. Subsequently, with prior knowledge of $(\bar{u}_k\cdot \vec{n})$ and $\bar{p}_k$, the residual component $(\mathbf{I-P}_0)\bar{f}_{k+1}$ can be uniquely determined by \eqref{1.12}.
\end{remark}
 
	\subsection{Reformulation of Knudsen boundary layers}
    To solve \eqref{1.18-0}, we decompose $\hat{S}_k = \hat{S}_{k,1}+\hat{S}_{k,2}$ with
	\begin{align}\label{2.11}
		\hat{S}_{k,1}=-\mathbf{P}_0\left\{\f{1}{\sqrt{\mu_0}}\left(\partial_t\hat{F}_{k-2}-\frac{1}{1-\v^2\eta}\mv_2\frac{\partial}{\partial \phi}\hat{F}_{k-2}\right)\right\},
	\end{align}
	and
	\begin{align}
		\hat{S}_{k,2}&=\sum\limits_{\substack{i+j=k\\i,j\geq1}}\frac{1}{\sqrt{\mu_0}}[Q(F^0_i+\bar{\mathfrak{F}}_i^0,\sqrt{\mu_0}\hat{f}_j)+Q(\sqrt{\mu_0}\hat{f}_j,F^0_i+\bar{\mathfrak{F}}_i^0)+Q(\sqrt{\mu_0}\hat{f}_i,\sqrt{\mu_0}\hat{f}_j)]\nonumber\\
		&\quad+\sum\limits_{\substack{j+2l=k\\b\geq l\geq 1,j\geq 1}}\frac{(-\eta)^l}{l!}\cdot\frac{1}{\sqrt{\mu_0}}[Q(\partial_r^l\mu_0,\sqrt{\mu_0}\hat{f}_j)+Q(\sqrt{\mu_0}\hat{f}_j,\partial_r^l\mu_0)]\nonumber\\
		&\quad+\sum\limits_{\substack{i+j+2l=k\\i,j\geq1,b\geq l\geq 1}}\frac{(-\eta)^l}{l!}\cdot\frac{1}{\sqrt{\mu_0}}[Q(\partial_r^l\bar{\mathfrak{F}}^0_i,\sqrt{\mu_0}\hat{f}_j)+Q(\sqrt{\mu_0}\hat{f}_j,\partial_r^l\bar{\mathfrak{F}}^0_i)]\nonumber\\
		&\quad+\sum\limits_{\substack{i+j+l=k\\i,j\geq1,b\geq l\geq 1}}\frac{\eta^l}{l!}\frac{1}{\sqrt{\mu_0}}[Q(\partial_y^l\bar{\mathfrak{F}}_i^0,\sqrt{\mu_0}\hat{f}_j)+Q(\sqrt{\mu_0}\hat{f}_j,\partial_y^l\bar{\mathfrak{F}}_i^0)]\nonumber\\
		&\quad-(\mathbf{I}-\FP_0)\left\{\f{1}{\sqrt{\mu_0}}\left(\partial_t\hat{F}_{k-2}-\frac{1}{1-\v^2\eta}\mv_2\frac{\partial}{\partial \phi}\hat{F}_{k-2}\right)\right\}.
	\end{align} 
	For later use, we define $F_k=0, \bar{\mathfrak{F}}_k=0$ and  $\hat{F}_k=0$ for $k<0.$ It is easy to know that $\hat{S}_{k,1}\in\mathcal{N}_0$, $\hat{S}_{k,2}\in\mathcal{N}_0^{\perp}$. As $\hat{F}_{-1}=\hat{F}_0=0$, it also holds that
	\begin{equation}
		\hat{S}_1=\hat{S}_{1,1}=\hat{S}_{1,2}=\hat{S}_{2,1}=0.
	\end{equation}

		Noting \eqref{2.11}, we can rewrite $\hat{S}_{k,1}$ as 
		\begin{align}\label{eq2.22}
			\hat{S}_{k,1}=\left\{\hat{a}_k+\f{\hat{b}_{k,1}}{T^0}\mv_1+\f{\hat{b}_{k,2}}{T^0}(\mv_2-u_\tau^0)+\f{\hat{c}_k}{T^0}(|v-\bar{\mathfrak{u}}^0|^2-2T^0)\right\}\sqrt{\mu_0}.
		\end{align}
		\begin{Lemma}\label{2.2}
			Assume 	there exists some positive constant $\sigma_1>0$ such that
			\begin{align*}
				\left|e^{\sigma_1\eta}(\hat{a}_k,\hat{b}_k,\hat{c}_k)(t,\phi,\eta)\right|<\infty,
			\end{align*}
			then there exists a function
			\begin{align}\label{eq2.23}
				\hat{f}_{k,1}=\left\{\f{\hat{A}_k}{T^0}\mv_1+\f{\hat{B}_k}{T^0}\mv_1 (\mv_2-u_\tau^0)+\f{\hat{C}_k}{T^0}\mv_1(|\mv-\bar{\mathfrak{u}}^0|^2-4T^0)+\hat{D}_k\right\}\sqrt{\mu_0},
			\end{align}
			such that
			\begin{align}\label{eq2.24}
				\mv_1\partial_\eta\hat{f}_{k,1}+G(\eta)\left(\mv_2^2\frac{\partial\hat{f}_{k,1}}{\partial \mv_1}-\mv_1 \mv_2\frac{\partial\hat{f}_{k,1}}{\partial \mv_2}-\f{ u_\tau^0}{2T^0}\mv_1 \mv_2 \hat{f}_{k,1}\right)-\hat{S}_{k,1}\in\mathcal{N}_0^\perp,
			\end{align}
		and 
		\begin{align*}
		 &\left|\mv_1\partial_\eta\hat{f}_{k,1}+G(\eta)(\mv_2^2\frac{\partial\hat{f}_{k,1}}{\partial \mv_1}-\mv_1 \mv_2\frac{\partial\hat{f}_{k,1}}{\partial \mv_2})-G(\eta)\f{ u_\tau^0}{2T^0}\mv_1 \mv_2 \hat{f}_{k,1}+\FL_0\hat{f}_{k,1}-\hat{S}_{k,1}\right|\nonumber\\
		 &\leq C\int_\eta^d|(\hat{a}_k,\hat{b}_k,\hat{c}_k)|(t,\phi,z)dz\cdot (1+|\mv|)^4\sqrt{\mu_0},|\hat{f}_{k,1}(t,\phi,\eta,\mv)|\nonumber\\ &\leq C\int_\eta^d|(\hat{a}_k,\hat{b}_k,\hat{c}_k)|dz\cdot (1+|\mv|)^3\sqrt{\mu_0},
		\end{align*}
		where $(\hat{A}_k,\hat{B}_k,\hat{C}_k,\hat{D}_k)(t,\phi,r)$ solves a ODEs and has explicit expression decided by $(\hat{a}_k,\hat{b}_{k},\hat{c}_k)$. See Appendix 6.2 for details.
		\end{Lemma}

		\begin{remark}\label{rmk2.7}
			Based on \eqref{2.11}, we observe that $\hat{S}_{k,1}$ is exclusively determined by $\hat{f}_{k-2}.$ This implies that $\hat{f}_{k,1}$ only depends on $\hat{f}_{k-2}$, which is known in our existence analysis for $\hat{f}_k$. We now investigate the equation for  $\hat{f}_{k,2}$:
			\begin{align}\label{2-24}
				&\mv_1\frac{\partial\hat{f}_{k,2}}{\partial\eta}+G(\eta)(\mv_2^2\frac{\partial\hat{f}_{k,2}}{\partial \mv_1}-\mv_1 \mv_2\frac{\partial\hat{f}_{k,2}}{\partial \mv_2})-\f{ u_\tau^0}{2T^0}G(\eta)\mv_1 \mv_2 \hat{f}_{k,2} +\mathbf{L}_0\hat{f}_{k,2}\\
				&=\hat{S}_{k,2}-\mathbf{L}_0\hat{f}_{k,1}-\left\{\mv_1\frac{\partial\hat{f}_{k,1}}{\partial \eta}+G(\eta)(\mv_2^2\frac{\partial\hat{f}_{k,1}}{\partial \mv_1}-\mv_1 \mv_2\frac{\partial\hat{f}_{k,1}}{\partial \mv_2})-\f{ u_\tau^0}{2T^0}G(\eta)\mv_1 \mv_2 \hat{f}_{k,2}-\hat{S}_{k,1}\right\}
				\in\mathcal{N}_0^\perp.\nonumber
			\end{align}
			With suitable boundary condition, one solves $\hat{f}_{k,2}$ by applying Theorem \ref{thm2.2}. Then
			\begin{equation}\label{2.17}
				\hat{f}_k=\hat{f}_{k,1}+\hat{f}_{k,2}
			\end{equation}
			is the solution to Knudsen layer problem \eqref{1.18-0}.
		\end{remark}
		\subsection{Boundary conditions}
	To construct the solutions for  interior expansion, viscous and Knudsen boundary layers, the remain problem is to determine suitable boundary conditions for well-posedness. As mentioned in Remark \ref{rmk-ie}, we require that each $F_k+\bar{F}_k+\hat{F}_k$ satisfies the specular reflection boundary condition, i.e.,
		\begin{equation}
			(\sqrt{\mu_0}f_k+\sqrt{\mu_0}\bar{f}_k+\sqrt{\mu_0}\hat{f}_k)|_{\gamma_-}=(\sqrt{\mu_0}f_k+\sqrt{\mu_0}\bar{f}_k+\sqrt{\mu_0}\hat{f}_k)(t,x,R_xv),
		\end{equation}
		which, together with \eqref{2.17}, yields 
		\begin{align}
			\hat{f}_{k,2}(t,0,\phi,\mv_1,\mv_2)|_{\mv_1>0}&=\hat{f}_{k,2}(t,0,\phi,-\mv_1,\mv_2)+f_k(t,1,\phi,R_xv)+\bar{f}_k(t,0,\phi,R_y\bar{v})\nonumber\\
			&+\hat{f}_{k,1}(t,0,\phi,-\mv_1,\mv_2)-f_k(t,1,\phi,v)-\bar{f}_k(t,0,\phi,\bar{v})-\hat{f}_{k,1}(t,0,\phi,\mv_1,\mv_2).\nonumber
		\end{align}
		Denote
		\begin{align}\label{2.25}
			\hat{g}_k(t,\phi,\mv_1,\mv_2)=
			\begin{cases}
				0,\quad \mv_1\textgreater0,\\
				f_k(t,1,\phi,v)+\bar{f}_k(t,0,\phi,\bar{v})-f_k(t,1,\phi,R_xv)-\bar{f}_k(t,0,\phi,R_y\bar{v})\\\
				+\hat{f}_{k,1}(t,0,\phi,\mv)-\hat{f}_{k,1}(t,0,\phi,-\mv_1,\mv_2),\quad \mv_1<0.
			\end{cases}
		\end{align}
		We derive the boundary condition for \eqref{2-24} as
		\begin{align}\label{2-18}
			\hat{f}_{k,2}(0,\mv)|_{\mv_1>0}=\hat{f}_{k,2}(0,R\mv)+\hat{g}_k(R\mv),\quad \hat{f}_{k,2}(d,\mv)|_{\mv_1<0}=0.
		\end{align}
	
	 Noting from Theorem \ref{thm2.2}, to solve \eqref{2-24} and \eqref{2-18}, we need $\hat{g}_k$ to satisfy \eqref{2.19}, i.e.
		\begin{equation}\label{2.23}
			\int_{\mathbb{R}^2}\mv_1\hat{g}_k\sqrt{\mu_0}\, d\mv=\int_{\mathbb{R}^2}\mv_1 \mv_2\hat{g}_k\sqrt{\mu_0}\, d\mv=\int_{\mathbb{R}^2}\mv_1|\mv|^2\hat{g}_k\sqrt{\mu_0}\, d\mv=0,
		\end{equation}
	which yields the following lemma on boundary conditions.
		\begin{Lemma}[\cite{GHW-2021-ARMA}]
			The boundary condition for \eqref{2.1} is 
			\begin{align}\label{2.26}
				(u_k\cdot \vec{n})|_{r=1}&=\int_0^\infty \Big\{\frac{1}{\rho^0}\partial_t\bar{\rho}_{k-1}-\mathcal{H}^\v(y)[\partial_\phi(\rho^0\bar{u}_{k-1}\cdot\vec{\tau})+\partial_\phi(\bar{\rho}_{k-1}u_\tau^0)]\nonumber\\
				&\qquad\qquad -\mathcal{H}^\v(y)(\bar{u}_{k-1}\cdot \vec{n})\Big\}(y)dy +\hat{A}_k(0),
			\end{align} 
			which is decided by $f_i,\bar{f}_i(i\leq k-1),\hat{f}_j(j\leq k-2).$

			The boundary condition for \eqref{bu-0} is
			\begin{align}\label{2.27}
				\begin{cases}
				\dis 	\partial_y(\bar{u}_{k-1}\cdot\vec{\tau})(t,\phi,0)=\frac{1}{\kappa_1(T^0)}
					\Big\{\left[\rho^0(\bar{u}_{k-1}\cdot \vec{n})(-u^0_1\cdot\vec{\tau}+\bar{u}_1\cdot\vec{\tau})\right]|_{y=0}+T^0\langle\mathcal{A}_{12}^0,\bar{J}_{k-2}\rangle|_{y=0}\\
				\dis 	\quad\quad\quad\quad\quad\quad\quad\quad\quad\quad\quad\quad+T^0\langle\mathcal{A}_{12}^0,(\mathbf{I-P}_0)f_k\rangle|_{r=1} +\rho^0T^0\hat{B}_{k}(0)\Big\},\\
				\dis 	\partial_y\bar{\theta}_{k-1}(t,\phi,0)=\frac{1}{\kappa_2(T^0)}\Big\{\rho^0(\bar{u}_{k-1}\cdot \vec{n})(\theta^0_1+\bar{\theta}_1)|_{y=0}+\f12(2T^0)^{\frac32}\langle\mathcal{B}_1^0,\bar{J}_{k-2}\rangle|_{y=0}\\
				\dis	\quad\quad\quad\quad\quad\quad\quad\quad\quad\quad-\f12(2T^0)^{\frac32}\langle\mathcal{B}_1^0,(\FI-\FP_0)f_k\rangle  |_{r=1} +4\rho^0(T^0)^2\hat{C}_k(0)\Big\}.
				\end{cases}
			\end{align}
			The RHS of \eqref{2.27} is decided by $f_i(i\leq k-1),\bar{f}_j,\hat{f}_l(j,l\leq k-2)$. The notations $\hat{A}_k(0),\hat{B}_k(0),\hat{C}_k(0)$ are defined in \eqref{eq2.23}.
		\end{Lemma}
		\begin{proof}
			A direct calculation shows that  
			\begin{align*}
				\begin{split}
					\int_{\mv_1<0}\mv_1\sqrt{\mu_0}\left\{f_k(R_xv)-f_k(v)\right\}|_{r=1}\, d\mv=-\int_{v\cdot \vec{n}>0}(v\cdot \vec{n})\sqrt{\mu_0}\left\{f_k(R_xv)-f_k(v)\right\}|_{r=1}\, dv\nonumber\\
					\qquad\qquad\qquad\qquad\qquad\qquad\qquad\qquad\quad=\int_{\R^2}(v\cdot \vec{n})\sqrt{\mu_0}f_k(v)|_{r=1}\, dv,\nonumber\\
					\int_{\mv_1<0}\mv_1\sqrt{\mu_0}\left\{\bar{f}_k(R_y\bar{v})-\bar{f}_k(\bar{v})\right\}|_{y=0}\, d\mv=-\int_{\R^2}(\bar{v}\cdot \vec{n})\sqrt{\mu_0}f_k(\bar{v})|_{y=0}\, d\bar{v},\\
					\int_{\mv_1<0}\mv_1\sqrt{\mu_0}\left\{\hat{f}_{k,1}(-\mv_1,\mv_2)-\hat{f}_k(\mv_1,\mv_2)\right\}|_{\eta=0}\, d\mv=-\int_{\R^2} \mv_1\sqrt{\mu_0}\hat{f}_{k,1}(\mv)|_{\eta=0}\, d\mv.
				\end{split}
			\end{align*}
			which,, together with \eqref{2.23}, yields
			\begin{align*}
				\int_{\R^2}(v\cdot \vec{n})\sqrt{\mu_0}f_k(v)|_{r=1}\, dv-\int_{\R^2}(\bar{v}\cdot \vec{n})\sqrt{\mu_0}\bar{f}_k(\bar{v})|_{y=0}\, d\bar{v}-\int_{\R^2} \mv_1\sqrt{\mu_0}\hat{f}_{k,1}(\mv)|_{\eta=0}\, d\mv=0.
			\end{align*}
		Similarly, we have from $\eqref{2.23}_{2,3}$ that
		\begin{align*}
			\begin{split}
				\int_{\R^2}(v\cdot \vec{n})(v\cdot\vec{\tau} -\mathfrak{u}^0\cdot\vec{\tau})\sqrt{\mu_0}f_k(v)|_{r=1}\, dv+\int_{\R^2}(\bar{v}\cdot \vec{n})(\bar{v}\cdot\vec{\tau}-u_\tau^0)\sqrt{\mu_0}\bar{f}_k(\bar{v})|_{y=0}\, d\bar{v}\\
				\qquad\qquad\qquad\qquad\qquad\qquad\qquad\qquad\qquad+\int_{\R^2} \mv_1(\mv_2-u_\tau^0)\sqrt{\mu_0}\hat{f}_{k,1}(\mv)|_{\eta=0}\, d\mv=0,\\
				\int_{\R^2}(v\cdot \vec{n})(|v-\mathfrak{u}^0|^2-4T^0)\sqrt{\mu_0}f_k(v)|_{r=1}\, dv-\int_{\R^2}(\bar{v}\cdot \vec{n})(|\bar{v}-\bar{\mathfrak{u}}^0|^2-4T^0)\sqrt{\mu_0}\bar{f}_k(\bar{v})|_{y=0}\, d\bar{v}\\
				\qquad\qquad\qquad\qquad\qquad\qquad\qquad\qquad\qquad-\int_{\R^2} \mv_1(|v-\bar{\mathfrak{u}}^0|^2-4T^0)\sqrt{\mu_0}\hat{f}_{k,1}(\mv)|_{\eta=0}\, d\mv=0.
			\end{split}
			\end{align*}
	Then, a delicate calculation shows \eqref{2.26} and \eqref{2.27}. See arguments in \eqref{7.8}--\eqref{7.13} for details.
		\end{proof}
		
			The condition \eqref{2.18} in Theorem \ref{thm2.2} requires 
		\begin{align}\label{2-27}
			|w_{\beta+1}\mv_1 \mv_2 \hat{g}_k|_{L^\infty_{\mv}}+\left|\nu^{-1}w_\beta (\mv_2^2\f{\partial}{\partial \mv_1}-\mv_1 \mv_2\f{\partial}{\partial \mv_2})\hat{g}_k(t,\phi,0,\cdot)\right|_{L^\infty_{\mv}}<\infty.
		\end{align}
		This will be proved in Proposition \ref{prop}.
		%	and
		%	\begin{align}\label{bc-2}
			%		&-\kappa_1(T^0)\partial_y(\bar{u}_{k-1}\cdot\vec{\tau})|_{y=0}+\rho^0(\bar{u}_{k-1}\cdot \vec{n})(-u_1^0\cdot \vec{\tau} +\bar{u}_1\cdot\vec{\tau})|_{y=0}+T^0\langle \mathcal{A}_{12}^0,\bar{J}_{k-1}\rangle |_{y=0}\nonumber\\
			%		&=T^0\langle \mathcal{A}_{12}^0,(\FI-\FP_0)f_k\rangle |_{r=1} -\rho^0T^0\hat{B}_k(0),\nonumber\\
			%		\smallskip
			%		&-2\kappa_2(T^0)\partial_y\bar{\theta}_{k-1}|_{y=0}+2\rho^0(\bar{u}_{k-1}\cdot \vec{n})(\theta_1^0+\bar{\theta}_1)|_{y=0}+(2T^0)^{\f32}\langle \mathcal{B}_1^0,\bar{J}_{k-1}\rangle|_{y=0} \nonumber\\
			%		&=(2T^0)^{\f32}\langle \mathcal{B}_1^0,(\FI-\FP_0)f_k\rangle |_{r=1}-8\rho^0(T^0)^2\hat{C}_k(0).
			%	\end{align}

		\section{Estimate on linear hyperbolic and parabolic systems}\label{sec3} 
		\subsection{Estimate on linear hyperbolic system}\label{sec3.1}
		To study the  existence of interior expansion, by noting Lemma \ref{lem2.2}, we consider the following linear problem for $(\tilde{\rho},\tilde{u},\tilde{\theta})(t,x)$:
		\begin{align}\label{3.1}
			\begin{cases}
				\dis	\partial_t\tilde{\rho}+\text{div}(\rho\tilde{u}+\tilde{\rho}\mathfrak{u})=0,\\
				\dis	\rho\{\partial_t\tilde{u}+(\tilde{u}\cdot\nabla)\mathfrak{u}+(\mathfrak{u}\cdot\nabla)\tilde{u}\}-\frac{\tilde{\rho}}{\rho}\nabla(\rho T)+\nabla(\frac{\rho\tilde{\theta}+2\tilde{\rho}T}{2})=g_1,\\
				\dis	\rho\{\partial_t\tilde{\theta}+(\tilde{\theta}\text{div}\mathfrak{u}+2T\text{div}\tilde{u})+\mathfrak{u}\cdot\nabla\tilde{\theta}+2\tilde{u}\cdot\nabla T\}=g_2,
			\end{cases}
		\end{align}
		where $(t,x)\in[0,\tau^\d]\times  B_1$. For \eqref{3.1}, we impose 
		\begin{align}\label{3.2}
			\tilde{u}\cdot \vec{n}(t,x)=d_0(t,x),\quad |x|=1, 
		\end{align}
		and 
		\begin{align}\label{3.3}
			(\tilde{\rho},\tilde{u},\tilde{\theta})(0,x)=(\tilde{\rho}_0,\tilde{u}_0,\tilde{\theta}_0)(x).
		\end{align}
		\begin{Lemma}\label{lem3.2}
			Let $(\rho,\mathfrak{u},T)$ is the smooth solution of compressible Euler equations obtained in Lemma \ref{lem3.1}. We assume that
			\begin{align*}
				\|(\tilde{\rho}_0,\tilde{u}_0,\tilde{\theta}_0)\|_{H^k(B_1)}^2+\sup_{t\in(0,\tau^\d]}\{\|(g_1,g_2)(t)\|_{H^{k+1}(B_1)}^2+\|d_0(t)\|_{H^{k+2}(B_1)}^2\}\textless\infty,
			\end{align*}
			with $k\geq 3$. If \eqref{3.2}--\eqref{3.3} satisfy the compatibility condition(the boundary condition \eqref{3.2} and the time-derivatives of initial data $(\tilde{\rho}_0,\tilde{u}_0,\tilde{\theta}_0)$ are defined through system \eqref{3.1} inductively.) Then there exists a unique smooth solution $(\tilde{\rho},\tilde{u},\tilde{\theta})$ of \eqref{3.1}--\eqref{3.3} satisfying
			\begin{align}\label{H.4}
				\sup_{t\in[0,\tau^\d]}	\|(\tilde{\rho},\tilde{u},\tilde{\theta})\|_{H^k(B_1)}^2\leq & C(\tau^\d,E_{k+2})\left\{\|(\tilde{\rho}_0,\tilde{u}_0,\tilde{\theta}_0)\|_{H^k(B_1)}^2\right.\nonumber\\
				&\left.+\sup_{t\in(0,\tau^\d]}\left[\|(g_1,g_2)(t)\|_{H^{k+1}(B_1)}^2+\|d_0(t)\|_{H^{k+2}}^2\right]\right\}.
			\end{align}
			where $$E_k:=\sup_{t\in[0,\tau^\d]}\|(\rho-1,\mathfrak{u},T-1)(t)\|_{H^k}=\sum_{0\leq i+j+l\leq k}\|(\partial_t^i\partial_{x_1}^j\partial_{x_2}^l)(\rho-1,\mathfrak{u},T-1)\|_{L^2}.$$
		\end{Lemma}
		\begin{proof} We divide the proof into several steps.
			
			{\noindent \it Step 1: Reformulation.}	Define $\tilde{p}=\frac{\rho\tilde{\theta}+2\tilde{\rho}T}{2}$. To deal with the boundary terms, it is more convenient to use the variables $(\tilde{p},\tilde{u},\tilde{\theta})$. Then \eqref{3.1} becomes
			\begin{align*}
				\begin{cases}
					\dis \partial_t\tilde{p}+\mathfrak{u}\cdot\nabla\tilde{p}+2p\,\text{div}\tilde{u}+2\tilde{p}\, \text{div}\mathfrak{u}+\nabla{p}\cdot\tilde{u}=\frac12g_2,\\
					\dis \rho\partial_t\tilde{u}+\rho \mathfrak{u}\cdot\nabla\tilde{u}+\nabla\tilde{p}-\frac{\nabla p}{p}\tilde{p}+\rho\tilde{u}\cdot\nabla \mathfrak{u}+\frac{\nabla p}{2T}\tilde{\theta}=g_1,\\
					\dis \rho\partial_t\tilde{\theta}+\rho \mathfrak{u}\cdot\nabla\tilde{\theta}+2p\,\text{div}\tilde{u}+2\rho\tilde{u}\cdot\nabla T+\rho\tilde{\theta}\, \text{div}\mathfrak{u}=g_2.
				\end{cases}
			\end{align*}
			Let $\chi(s)$ be a monotone smooth cut-off function such that
			\begin{align}\label{H.4-1}
				\chi(s)=
				\begin{cases}
					1,\quad s\in[\f34,1],\\
					0,\quad s\in[0,\f12].
				\end{cases}
			\end{align}
			Define
			\begin{align*}
				u_d(t,x)\cdot \vec{n}=d_0(t,x)\chi(|x|), \quad u_d(t,x)\cdot\vec{\tau}=0,
			\end{align*}
			and $\tilde{w}=\tilde{u}-u_d.$ Then \eqref{3.1} can be rewritten as 
			\begin{align}\label{H.5}
				\begin{cases}
					\partial_t\tilde{p}+\mathfrak{u}\cdot\nabla\tilde{p}+2p\,\text{div}\tilde{w}+2\tilde{p}\,\text{div}\mathfrak{u}+\nabla{p}\cdot\tilde{w}=G_0,\\
					\rho\partial_t\tilde{w}+\rho (\mathfrak{u}\cdot\nabla)\tilde{w}+\nabla\tilde{p}-\frac{\nabla p}{p}\tilde{p}+\rho(\tilde{w}\cdot\nabla) \mathfrak{u}+\frac{\nabla p}{2T}\tilde{\theta}=G_1,\\
					\rho\partial_t\tilde{\theta}+\rho \mathfrak{u}\cdot\nabla\tilde{\theta}+2p\, \text{div}\tilde{w}+2\rho\tilde{w}\cdot\nabla T+\rho\tilde{\theta}\,\text{div}\mathfrak{u}=G_2,
				\end{cases}
			\end{align}
			where 
			\begin{align*}
				&G_0=\frac12g_2-2p\,\text{div}u_d-\nabla p\cdot u_d,\\
				&G_1=g_1-\rho\partial_tu_d-\rho (\mathfrak{u}\cdot\nabla) u_d-\rho (u_d\cdot\nabla) \mathfrak{u},\\
				&G_2=g_2-2p\,\text{div}u_d-2\rho\, u_d\cdot\nabla T.
			\end{align*}
			The boundary condition of \eqref{H.5} becomes
			\begin{align}\label{H.6}
				\tilde{w}\cdot \vec{n}=0,\quad |x|=1.
			\end{align}
			
			Now we display \eqref{H.5} with symmetric structure:
			\begin{align}\label{H.7}
				A_0\partial_tU+\sum\limits_{i=1}^2A_i\partial_iU+A_3U=\mathcal{G},
			\end{align}
			where
			\begin{align*}
				&U=
				\begin{pmatrix}
					\tilde{p}\\
					\tilde{w}\\
					\tilde{\theta}
				\end{pmatrix}
				\quad
				,\quad
				A_0=
				\begin{pmatrix}
					\frac32 & 0 & 0 & -\rho\\
					0 & \rho p & 0 & 0\\
					0 & 0 & \rho p & 0\\
					-\rho & 0 & 0 & \rho^2
				\end{pmatrix}
				,\\
				&A_1=
				\begin{pmatrix}
					\frac32\mathfrak{u}_1 & p & 0 & -\rho \mathfrak{u}_1\\
					p & \rho p\mathfrak{u}_1 & 0 & 0\\
					0 & 0 & \rho p\mathfrak{u}_1 & 0\\
					-\rho \mathfrak{u}_1 & 0 & 0 & \rho^2\mathfrak{u}_1
				\end{pmatrix}
				\quad
				,\quad
				A_2=
				\begin{pmatrix}
					\frac32\mathfrak{u}_2 & 0 & p & -\rho \mathfrak{u}_2\\
					0 & \rho p\mathfrak{u}_2 & 0 & 0\\
					p & 0 & \rho p\mathfrak{u}_2 & 0\\
					-\rho \mathfrak{u}_2 & 0 & 0 & \rho^2\mathfrak{u}_2
				\end{pmatrix}.
				\nonumber
			\end{align*}
				The matrix $A_3$ and column vector $\mathcal{G}$ can be easily write down, and we do not give the details here. It is easy to check that $A_0$ is positive.

				Since $A_n:=\sum_{i=1}^2A_ix_i,$ is singular at the boundary $\partial B_1$ due to $\mathfrak{u}\cdot \vec{n}|_{\partial B_1}=0$, hence the problem \eqref{H.7} with \eqref{H.6} is a linear hyperbolic system with characteristic boundary. We refer to \cite{Chen,Schochet,Se} for the local existence of smooth solutions. 
				
				%				Let $U=U^\chi+U^{1-\chi}$, where 
				%	\begin{align*}
					%		U^\chi=U\cdot\chi(|x|),\quad U^{1-\chi}=U\cdot\left(1-\chi(|x|)\right).
					%	\end{align*}
				%	Then one obtains
				%	\begin{align}\label{H.10}
					%		A_0\partial_tU^\chi+\sum_{i=1}^2A_1\partial_iU^\chi+\left(A_i\partial_i\chi\right)U+A_3U^\chi=\chi(|x|)\cdot G,
					%	\end{align}
				%	and 
				%	\begin{align}\label{H.11}
					%		A_0\partial_tU^{1-\chi}+\sum_{i=1}^2A_i\partial_iU^{1-\chi}-\left(\sum_{i=1}^2A_i\partial_i\chi\right)U+A_3U^{1-\chi}=(1-\chi(|x|))\cdot G.
					%	\end{align}			
				It follows from Newtonian-Leibniz formula that
				\begin{align}\label{H.12}
					\|U(t)\|_{H^{k-1}}^2&\leq \|U_0\|_{H^{k-1}}^2+2\int_0^t\|\partial_tU(s)\|_{H^{k-1}}\|U(s)\|_{H^{k-1}}ds\nonumber\\
					&\leq \|U_0\|_{H^{k-1}}^2+\int_0^t\|U(s)\|_{H^{k}}^2ds.
				\end{align}
				Hence we need only to close the highest order derivatives estimates. 
				
				{\noindent \it Step 2: Internal estimate.}
				Let $l+|\alpha|=k$ and applying $\partial_t^l\partial_x^\alpha$ to \eqref{H.7}, we obtain 
				\begin{align}\label{H.13}
					&A_0 \partial_t\partial_t^l\partial_x^\alpha U+\sum_{i=1}^2 A_i \partial_i \partial_t^l\partial_x^\alpha U\nonumber\\
					&=\partial_t^l\partial_x^\alpha \mathcal{G}-\partial_t^l\partial_x^\alpha (A_3U)-[\partial_t^l\partial_x^\alpha ,A_0] \partial_t U
					-\sum_{i=1}^2[\partial_t^l\partial_x^\alpha,A_i] \partial_i U,
				\end{align}
				where and whereafter the notation $[\cdot, \, \cdot]$ denote the commutator operator, i.e.,
				\begin{align}\nonumber
					[\partial^{\alpha},f]g=\partial^\alpha (fg)-f\partial^\alpha g=\sum_{|\alpha_1|+|\alpha_2|=\alpha, |\alpha_1|\geq1} C_{\upsilon,\gamma}\partial^{\alpha_1} f\cdot \partial^{\alpha_2} g.
				\end{align}
				Multiplying \eqref{H.13} by $(\partial_t^l\partial_x^\alpha (U^\top))\cdot(1-\chi(|x|))^2$ and integrating the resultant equation over $[0,t]\times B_1$,  we obtain
				\begin{align}\label{H.14}
					\|((1-\chi)\partial_t^l\partial_x^\alpha U)(t)\|_{L^2}^2&\leq C\|\partial_t^l\partial_x^\alpha U(0)\|_{L^2}^2+C(E_{k+1})\int_0^t\|(U,\mathcal{G})(s)\|_{H^k}^2ds,
				\end{align}
				where the boundary term vanishes since $1-\chi=0$ on the boundary.  %Using Gronwall inequality, one has 
				%\begin{align}\label{H.15}
				%	\| U^{1-\chi}(t)\|_{H^k}^2&\leq C\| U^{1-\chi}(0)\|_{H^k}^2+C(E_{k+1},t)\sup_{0\leq s\leq t}\|G^{1-\chi}(s)\|_{H^k}^2.
				%\end{align}
				
				%Now we consider \eqref{H.10}. 	
				{\noindent \it Step 3: Estimate near the boundary.}
				Taking polar coordinate for \eqref{H.7}, then it is rewritten as
				\begin{align}\label{H.16}
					\tilde{A}_0\partial_t\tilde{U}+\tilde{A}_1\partial_r\tilde{U}+\tilde{A}_2\partial_\phi\tilde{U}+\tilde{A}_3\tilde{U}= \tilde{\mathcal{G}}.
				\end{align}
				Here $\tilde{A}_0=A_0, \tilde{\mathcal{G}}=(G_0,G_1\cdot \vec{n},G_1\cdot \vec{\tau}, G_2)^\top$ and $\tilde{U}=(\tilde{p},\tilde{w}\cdot \vec{n}, \tilde{w}\cdot\vec{\tau},\tilde{\theta})^\top$. By direct calculation, one obtains
				\begin{align}\label{H.17}
					\tilde{A}_1=
					\begin{pmatrix}
						\f32(\mathfrak{u}\cdot \vec{n}) & p & 0 & -\rho(\mathfrak{u}\cdot \vec{n})\\
						p & \rho p(\mathfrak{u}\cdot \vec{n}) & 0 & 0\\
						0 & 0 & \rho p (\mathfrak{u}\cdot \vec{n}) & 0\\
						-\rho(\mathfrak{u}\cdot \vec{n}) & 0 & 0 &\rho^2(\mathfrak{u}\cdot \vec{n})
					\end{pmatrix},
				\end{align}
				and 
				\begin{align*}
					\tilde{A}_2=
					\begin{pmatrix}
						\f{3}{2r}(\mathfrak{u}\cdot\vec{\tau}) & 0 & \f{p}{r} & -\f{1}{r}\rho(\mathfrak{u}\cdot\vec{\tau})\\
						0 & \f{1}{r}\rho p(\mathfrak{u}\cdot\vec{\tau}) & 0 & 0\\
						\f{p}{r} & 0 & \f{1}{r}\rho p(\mathfrak{u}\cdot\vec{\tau}) & 0\\
						-\f{1}{r}\rho(\mathfrak{u}\cdot\vec{\tau}) & 0 & 0 & \f{1}{r}\rho^2(\mathfrak{u}\cdot\vec{\tau})
					\end{pmatrix}.
				\end{align*}
				
				We notice that 
				\begin{align*}
					\tilde{A}_3=A_3+
					\begin{pmatrix}
						0 &  \f{p}{r} & 0 & 0\\
						0 & 0 & -\f{1}{r}\rho p(\mathfrak{u}\cdot\vec{\tau}) & 0\\
						0 &  -\f{1}{r}\rho p(\mathfrak{u}\cdot\vec{\tau}) & 0 & 0\\
						0 & 0 & 0 & 0
					\end{pmatrix}.
				\end{align*}
				The expression of $\tilde{A}_2$ and $\tilde{A}_3$ has singularity of $r^{-1}$. We deal with it by using the cut-off function $\chi(r)$ which is  $0$ when  $r\leq\f12$. We also notice that when $r>\f12$, one has 
				\begin{align}\label{H.45}
					\left|\frac{\partial(x_1,x_2)}{\partial(r,\phi)}\right|=r\sim1.
				\end{align}
				
			The procedure to get energy estimate for \eqref{H.16} is almost the same as half space problem in \cite{GHW-2021-ARMA} and we omit it here for brevity. We obtain 
			\begin{align}\label{H.44}
				\sum_{|\alpha|=k-l}\|\chi(\bp^\alpha\partial_r^l \tilde{U})(t)\|_{L^2}^2
				&\leq C(E_{k+1}) \int_0^t\|U(s)\|_{H^k}^2+\|\mathcal{G}(s)\|_{H^{k+1}}^2ds \nonumber\\
				&\quad+C(E_{k+1}) \|(U,\mathcal{G})(0)\|_{H^k}^2.
			\end{align}
		Hence we conclude from \eqref{H.14} and \eqref{H.44} that \eqref{H.4} is true and the proof of Lemma \ref{lem3.2} is finished.
			\end{proof}

			\subsection{Estimate on linear parabolic system}\label{sec3.2}
			To construct viscous boundary layer solutions, noting Lemma \ref{lem2.3}, we consider linear parabolic system for $(u_\tau,\theta)$:
			\begin{align}\label{P.1}
				\begin{cases}
					\rho^0\partial_tu_\tau-\kappa_1(T^0)\partial_{yy}u_\tau\\
					\quad -\mathcal{H}^\v(y)\rho^0u_\tau^0\partial_\phi u_\tau+\rho^0(-y\partial_r\bar{\mathfrak{u}}^0\cdot \vec{n})\partial_yu_\tau\\
					\quad +u_\tau\cdot\left\{\left(1-\mathcal{H}^\v(y)\right)u_\tau^0\partial_\phi\rho^0+\left(1-2\mathcal{H}^\v(y)\right)\rho^0\partial_\phi u_\tau^0 \right\}\\
					\quad-\f{\rho^0}{2T^0}\theta\cdot\left\{\left(1-\mathcal{H}^\v(y)\right) u_\tau^0\partial_\phi  u_\tau^0+\partial_\phi T^0+\f{T^0}{\rho^0}\partial_\phi\rho^0\right\}=\bar{g}_1,\\
					\\
					\rho^0\partial_t\theta-\kappa_2(T^0)\partial_{yy}\theta-\mathcal{H}^\v(y)\rho^0u_\tau^0\partial_\phi\theta+\rho^0(-y\partial_r\bar{\mathfrak{u}}^0\cdot \vec{n})\partial_y\theta\\
					\quad+\theta\cdot\left\{\left(1-\mathcal{H}^\v(y)\right)\partial_\phi(\rho^0u_\tau^0)-\f{\rho^0}{T^0}\left(1-\mathcal{H}^\v(y)\right)u_\tau^0\partial_\phi T^0-\rho^0\partial_\phi u_\tau^0-\rho^0\partial_r\bar{\mathfrak{u}}^0\cdot \vec{n}\right\}\\
					\quad+\rho^0u_\tau\cdot\left\{\left(1-2\mathcal{H}^\v(y)\right)\partial_\phi T^0+\f{T^0}{\rho^0}\partial_\phi\rho^0+\left(1-\mathcal{H}^\v(y)\right)u_\tau^0\partial_\phi u_\tau^0\right\}=\bar{g}_2,
				\end{cases}
			\end{align}
			where $\mathcal{H}^\v(y)=\f{1}{1-\v y}\Upsilon(\v^{\f12 }y)$ and$(t,\phi,y)\in[0,\tau^\d)\times[0,2\pi]\times\mathbb{R}_{+}$, $(\rho^0,u_\tau^0,\partial_r\bar{\mathfrak{u}}^0,\partial_rp^0,\partial_\phi p^0)$ represents values at $r=1$, and are independent of $y$. We impose non-homogeneous Neumann boundary condition for \eqref{P.1}, i.e.
			\begin{align}
				\begin{cases}
					\partial_yu_\tau(t,\phi,y)|_{y=0}=b(t,\phi),\quad\partial_y\theta(t,\phi,y)|_{y=0}=a(t,\phi),\\
					\lim\limits_{y\rightarrow\infty}(u_\tau,\theta)(t,\phi,y)=0,\nonumber
				\end{cases}
			\end{align}
			We also impose the initial data
			\begin{align}\label{P.2}
				u_\tau(t,\phi,y)|_{t=0}=u_0(\phi,y),\quad\theta(t,\phi,y)|_{t=0}=\theta_0(\phi,y).
			\end{align}
			The initial data $(u_0,\theta_0)$ should satisfy the compatible condition.
			
			Let $l\geq 0$, denote 
			\begin{align*}
				&\|f\|_{L^2_l} = \iint (1+y)^l|f(\phi,y)|^2d\phi dy\nonumber\\
				&\bar{x}:=(\phi,y),\quad\nabla_{\bar{x}}=(\partial_\phi,\partial_y),
			\end{align*}
			we have the following lemma. For the proof of Lemma \ref{lem3.3}, we refer the readers to \cite{GHW-2021-ARMA}, which establishes existence and regularity for parabolic systems in 3D half-space. While the original work focuses on three dimensions, the framework extends to the 2D case \eqref{P.1} with minor geometric adjustments. These modifications preserve the core analytical structure of the proof; full details are omitted here for brevity.
			\begin{Lemma}[\cite{GHW-2021-ARMA}]\label{lem3.3}
				Assume $l\geq0,k\geq2$, and the compatibility condition for the initial data \eqref{P.1} is satisfied. Assume
				\begin{equation*}
					\sup_{t\in[0,\tau^\d]}\left\{\sum\limits_{\alpha+2i\leq k+2}\|\partial_\phi^\alpha\partial_t^i(a,b)(t)\|_{L^2(\mathbb{R}^2)}^2+\sum\limits_{j=0}^k\sum\limits_{|\alpha|+2i=j}\|\nabla_{\bar{x}}^\alpha\partial_t^i(\bar{g}_1,\bar{g}_2)(t)\|_{L_{l_j}^2}^2\right\}<\infty,
				\end{equation*}
				where $l_j:=l+2(k-j),0\leq j\leq k.$ Then there is a unique smooth solution $(u_\tau,\theta)$ of \eqref{P.1}--\eqref{P.2} which satisfies
				\begin{align}
					&\sum\limits_{j=0}^k\sum\limits_{|\alpha|+2i=j}\sup_{t\in[0,\tau^\d]}\left\{\|\partial_t^i\nabla_{\bar{x}}^\alpha(u_\tau,\theta)(t)\|_{L^2_{l_j}}^2+\int_0^t\|\partial_t^i\nabla_{\bar{x}}^\alpha\partial_y(u_\tau,\theta)\|_{L_{l_j}^2}^2ds\right\}\nonumber\\
					&\leq C(\tau,E_{k+3})\left\{\sum\limits_{j=0}^k\sum\limits_{|\alpha|+2i=j}\|\partial_t^i\nabla_{\bar{x}}^\alpha(u_\tau,\theta)(0)\|_{L^2_{l_j}}^2\right\}\nonumber\\
					&+\sup_{t\in(0,\tau^\d]}\left\{\sum\limits_{\alpha+2i\leq k+2}\|\partial_\phi^\alpha\partial_t^i(a,b)(t)\|_{L^2(\mathbb{R}^2)}^2+\sum\limits_{j=0}^k\sum\limits_{|\alpha|+2i=j}\|\nabla_{\bar{x}}^\alpha\partial_t^i(\bar{g}_1,\bar{g}_2)(t)\|_{L_{l_j}^2}^2\right\},
				\end{align}
				where the notation $E_{k+3}$ is defined in Lemma \ref{lem3.2}.
			\end{Lemma}

			\section{The Existence of Knudsen boundary layer problem}\label{sec4}
			In this section, we recall that $\mv=(\mv_1,\mv_2)$ and $\bar{\fu}^0=(0,u_\tau^0)$. 
			For \eqref{K}, we define $f^c(\eta,\mv):=f(\eta,\mv)+\Upsilon(\eta)f_b(\mv)$, then the equation for $f^c$ is
			\begin{align}\label{K-1}
				\begin{cases}
					\dis \mv_1 \partial_\eta f^c+G(\eta)\left(\mv_2^2\f{\partial f^c}{\partial \mv_1}-\mv_1 \mv_2\f{\partial f^c}{\partial \mv_2}\right)-\f{ u_\tau^0}{2T^0}\cdot G(\eta)\mv_1 \mv_2 f^c+\FL_0 f^c=S,\\
					\dis f^c(0,\mv)|_{\mv_1>0}=f^c(0,R\mv),\\
					f^c(d,\mv)|_{\mv_1<0}=0,
				\end{cases}
			\end{align} 
			where 
			\begin{align}\label{K-1.1}
				S=&\mathfrak{S}+\partial_\eta\Upsilon(\eta)\mv_1 f_b+G(\eta)\Upsilon(\eta)\left(\mv_2^2\f{\partial f_b}{\partial \mv_1}-\mv_1 \mv_2\f{\partial f_b}{\partial \mv_2}\right)\nonumber\\
				&-\f{u_\tau^0}{2T^0}G(\eta)\Upsilon(\eta)\mv_1 \mv_2 f_b+\Upsilon(\eta)\FL_0f_b.
			\end{align}
		 The solvable condition \eqref{2.19} and $\mathfrak{S}\in\mathcal{N}_0^\perp$ in Theorem \ref{thm2.2} are actually
		 	\begin{align}\label{5.1-1}
		 	\langle S,(1,\mv_2,|\mv|^2)\sqrt{\mu_0}\rangle \equiv0.
		 \end{align}
		Then one considers
			\begin{align}\label{5.1}
				\begin{cases}
					\dis	\mv_1\partial_\eta f+G(\eta)(\mv_2^2\frac{\partial f}{\partial \mv_1}-\mv_1 \mv_2 \frac{\partial f}{\partial \mv_2})-\frac{u_\tau^0}{2T^0}G(\eta)\mv_1 \mv_2 f+\FL_0f=S,\\
					\dis	f(0,\mv)|_{\gamma_-}=f(0,-\mv_1,\mv_2),\\
					\dis	f(d,\mv)|_{\gamma_-}=0,
				\end{cases}
			\end{align}
			where $ (\eta,\mv)\in\Omega_d:=[0,d]\times\R^2$ with $d=\v^{-\mathfrak{a}}$ and $\mathfrak{a}<\f23$, and $G(\eta)=-\f{\v^2}{1-\v^2\eta}$. The source term $S$ satisfies \eqref{5.1-1}. The notation $\gamma_-$ here denotes $\{(0,\mv)|\mv_1>0\}\cup\{(d,\mv)|\mv_1<0\}$. 
			
%		 {\color{blue}When $u_\tau^0\equiv0$, Wu \cite{Wu-2016-JDE} has proved the existence of \eqref{5.1}. However, $u_\tau^0$, the tangential boundary velocity of Euler solutions, is generally not zero. New difficulty arises from $\dis-\frac{u_\tau^0}{2T^0}G(\eta) \mv_1 \mv_2 f$when we try to get $L^2$ estimates. As far as we know, Duan and Liu have tried the Caflisch decomposition in \cite{Duan-2021-ARMA} to deal with terms like $\mv_1\mv_2f$. In this way, they can obtain the solution existence, but only with polynomial velocity decay for $\sqrt{\mu_0}f$. This is not enough for us to get the remainder estimates.}
			
			We proceed to specify how the refined $L^2-L^\infty$ framework, combined with the smallness of $\v$ and the constructed zero incoming condition at $\eta=d$, systematically establishes the existence of solutions to \eqref{5.1}.
			
			Noting the definition of $\FL_0$ in \eqref{1.17-1}, one has that
			\begin{align*}
				\FL_0f&=\nu(\mv) f-Kf=\nu(\mv) f-(K_1f-K_2f)\\
				&=\nu(\mv) f-\left(\int_{\R^2}k_1(\mv, u)f(u)\, du-\int_{\R^2}k_2(\mv,u)f(u)\, du\right),
			\end{align*} 
			where $	\nu(\mv)=\int_{\R^2}|\mv-u|\mu_0(u)\,du\cong 1+|\mv|$ and
			\begin{align*}
				K_1f&=\int_{\R^2}k_1(\mv,u)f(u)\, du=\int_{\R^2}|\mv-u|\sqrt{\mu_0(\mv)}\sqrt{\mu_0(u)}f(u)\, du\\
				K_2f&=\int_{\R^2}k_2(\mv,u)f(u)\, du=2\int_{\R^2}|\mv-u|\sqrt{\mu_0(u)}\sqrt{\mu_0(v')}f(u')\, du.
			\end{align*}	
			One has from \cite{Gl} that
			\begin{align}\label{5.0-1}
				\begin{split}
				k_1(\mv,u)\lesssim |\mv-u|\exp\left\{-\f{1}{4T^0}(|\mv-\bar{\mathfrak{u}}^0|^2+|u-\bar{\mathfrak{u}}^0|^2)\right\},\\
				k_2(\mv,u)\lesssim \exp\left\{-\f{1}{8T^0}\left(|\mv-u|^2+\f{\left||\mv-\bar{\mathfrak{u}}^0|^2-|u-\bar{\mathfrak{u}}^0|^2\right|^2}{|\mv-u|^2}\right)\right\},
			    \end{split}
			\end{align}
			and $k_i(\mv,u)\in L_u^1\cap L^2_u,\, i=1,2$.
			
			For $L^\infty$ estimate, we define $\mf=\f{\sqrt{\mu_0}}{\sqrt{\mu_M}}f$, then \eqref{5.1} is rewritten as
			\begin{align}\label{5.2}
				\begin{cases}
					\dis	\mv_1\partial_\eta \mathbf{f}+G(\eta)(\mv_2^2\frac{\partial \mathbf{f}}{\partial \mv_1}-\mv_1 \mv_2 \frac{\partial \mathbf{f}}{\partial \mv_2})+\FL_M\mathbf{f}=\f{\sqrt{\mu_0}}{\sqrt{\mu_M}}S=:\bar{S},\\
					\dis 	\mathbf{f}(0,\mv)|_{\gamma_-}=\mathbf{f}(0,-\mv_1,\mv_2),\\
					\dis	\mathbf{f}(d,\mv)|_{\gamma_-}=0,
				\end{cases}
			\end{align}
			where
			\begin{align*}
				\FL_M \mathbf{f}=-\f{1}{\sqrt{\mu_M}}\left[Q(\mu_0,\sqrt{\mu_M}\mathbf{f})+Q(\sqrt{\mu_M}\mathbf{f},\mu_0)\right]=\nu \mathbf{f}-K_M\mathbf{f}.
			\end{align*}
			%In the following discussion, we assume 
			%\begin{align*}
			%	\|\nu^{-1}w\bar{S}\|_{W^{1,\infty}_\eta L^\infty_v}+|w\partial_\eta\mf|_{\eta=d}<\infty.
			%\end{align*}
			
			From the definition of $f$ and $\mathbf{f}$, one has
			\begin{align*}
				K_M\mathbf{f}=&K_{M,1}\mathbf{f}-K_{M,2}\mathbf{f}\nonumber\\
				=&\int_{\R^2}k_{M,1}(\mv,u)\mathbf{f}(u)du-\int_{\R^2}k_{M,2}(\mv,u)\mathbf{f}(u)du,
			\end{align*}
			with
			\begin{align*}
				k_{M,1}(\mv,u)=k_1(\mv,u)\f{\sqrt{\mu_0(\mv)}\sqrt{\mu_M(u)}}{\sqrt{\mu_0(u)}\sqrt{\mu_M(\mv)}},\quad k_{M,2}(\mv,u)=k_2(\mv,u)\f{\sqrt{\mu_0(\mv)}\sqrt{\mu_M(u)}}{\sqrt{\mu_0(u)}\sqrt{\mu_M(\mv)}}.
			\end{align*}
			\begin{Lemma}\label{lem4.1}
				For $k_M(\mv,u)$, one has that
				\begin{align}\label{5.2-0}
					\int_{\R^2}k_M(\mv,u)(1+|u|^2)^{\f{\beta}{2}}e^{\zeta|u|^2}du\leq C(1+|\mv|^2)^{\f{\beta-1}{2}}e^{\zeta|\mv|^2},\quad \beta\geq 0,\zeta<\f{1}{4T_M}.
				\end{align}
			\end{Lemma}
		\begin{remark}
			Since we have not found a direct reference which gives Lemma \ref{lem4.1}, so we present details of calculation in the Appendix 6.3.
		\end{remark} 
		It is known that $\FL_0$ is a self adjoint linear operator and for $g\in\mathcal{N}_0^{\perp}$, there is some positive constant $c_0>0$ such that
			\begin{align*}
				\langle g,\FL_0g\rangle \geq c_0\|(\FI-\FP_0)g\|_{\nu}^2.
			\end{align*}
			Define $\FL_0^{-1}:\mathcal{N}_0^\perp\to \mathcal{N}_0^\perp$ as the pseudo-inverse operator of $\FL_0$. In the following, we will write $w_\beta$ as $w$ for simplicity when no confusion arises.
			\subsection{  Existence of solution}
			\begin{Definition}
				We define the backward characteristic for \eqref{5.2} as 
				\begin{align}\label{5-6}
					\begin{cases}
						\dis	\f{dX}{ds}=V_1,\quad\f{dV_1}{ds}=G(X(s))V_2^2,\quad \f{dV_2}{ds}=-G(X(s))V_1V_2,\\
						\dis	(X,V)(\mathfrak{t})=(\eta,\mv).
					\end{cases}
				\end{align}
			\end{Definition}
			For late use, we define an increasing function $W(\eta)$ as
			\begin{align*}
				W(\eta)=\int_0^\eta -G(z)dz= -\ln(1-\v^2\eta).
			\end{align*}
			It is clear that
			\begin{align*}
				0\leq W(\eta)	\leq \f{\v^2\eta}{1-\v^2\eta},\quad e^{W(\eta)}-1=\f{1}{1-\v^2\eta}-1=\f{\v^2\eta}{1-\v^2\eta}.
			\end{align*}
			As in \cite{Wu-2016-JDE}, along the characteristics, it holds that
			\begin{align}\label{5-7}
				E_1:=V_1^2+V_2^2\equiv \mv_1^2+\mv_2^2=|\mv|^2,\quad E_2:=V_2(s)e^{-W(X(s))}\equiv \mv_2\cdot e^{-W(\eta)}.
			\end{align}
			For each $(\mathfrak{t},\eta,\mv)$ with $\eta\in\overline{\Omega}_d,\mv_1\neq0$, we define its backward time $\mathfrak{t}_{\mathbf{b}}(\mathfrak{t},\eta,\mv)\geq0$, 
			\begin{align*}
				\mathfrak{t}_{\mathbf{b}}(\mathfrak{t},\eta,\mv):=\min\left\{z\geq0: \left(\eta-\int_{\tau}^{\mathfrak{t}}V_1(s) ds,V(\tau)\right)\in\gamma_-,\  \mathfrak{t}-\tau=z\right\}.
			\end{align*}
			We also define $\eta_{\mathbf{b}}(\mathfrak{t},\eta,\mv)=\eta(\mathfrak{t})-\int_{\mathfrak{t}-\mathfrak{t}_{\mathbf{b}}}^{\mathfrak{t}} V_1(s) ds\in\partial\overline{\Omega}_d $. Let $\eta\in\overline{\Omega}_d,\ (\eta,\mv)\notin\{\gamma_0\cup\gamma_-\},\ (\mathfrak{t}_0,\eta_0,v_0)=(\mathfrak{t},\eta,\mv)$. We inductively define for $k\geq0$ that 
			\begin{equation*}
				(\mathfrak{t}_{k+1},\eta_{k+1})=(\mathfrak{t}_{k}-\mathfrak{t}_{\mathbf{b}}(\eta_{k},v_{k}),\eta_{\mathbf{b}}(\eta_{k},v_k)).% \quad k=0,1,2,\cdots.
			\end{equation*}
			The definition of $v_{k+1}$ can be obtained in the following backtracking process. In fact, we divide the discussion into two cases.\\
			
			{\it Case 1. $E_2>|V_2(d)|^2=\mv_2^2\cdot e^{2W(d)-2W(\eta)}$.} It would be discussed for $\mv_1>0$ and $\mv_1<0$, respectively.\\
			{\it Case 1.1. $E_1> \mv_2^2\cdot e^{2W(d)-2W(\eta)}, \mv_1<0.$} In this type, the backward characteristic line is traced back from $(\eta,\mv)$ to $X(s)=d$. Then one has 
			\begin{align*}
				\begin{cases}
				\dis X_{cl}(s)=\eta-\int_{s}^{\mathfrak{t}}V_{cl,1}(z)dz,\\
				V_{cl}=\left(-\sqrt{E_1-E_2^2e^{2W(X(s))}},E_2e^{W(X(s))}\right).
				\end{cases}
			\end{align*}
			{\it Case 1.2. $E_1> \mv_2^2\cdot e^{2W(d)-2W(\eta)}, \mv_1>0.$} At this time, the backward characteristic line is traced back to $X(s)=0$ first, and then traced back to $X(s)=d$ due to the specular reflection condition. Then we obtain from \eqref{5-6}--\eqref{5-7} that
			\begin{align*}
				X_{cl}(s)&=\mathbf{1}_{(\mathfrak{t}_{1},\mathfrak{t}]}(s)\left\{\eta-\int_{s}^{\mathfrak{t}}V_{cl,1}(z)dz\right\}+\mathbf{1}_{(\mathfrak{t}_2,\mathfrak{t}_1]}(s)\left\{-\int_{s}^{\mathfrak{t}_1}V_{cl,1}(z)dz\right\},\\
				V_{cl}(s)&=\mathbf{1}_{(\mathfrak{t}_{1},\mathfrak{t}]}(s)\left(\sqrt{E_1-E_2^2e^{2W(X(s))}},E_2e^{W(X(s))}\right)\nonumber\\
				&\quad+\mathbf{1}_{(\mathfrak{t}_{2},\mathfrak{t}_1]}(s)\left(-\sqrt{E_1-E_2^2e^{2W(X(s))}},E_2e^{W(X(s))}\right),
			\end{align*}
			and the backward velocity $v_1=(-\sqrt{E_1-E_2^2}, E_2)$ representing the velocity after collision with $\eta_1=0$. See the figures below for {\it Case 1}.
			\begin{figure}[H]\label{fig1}
				\centering
				\subfigure[Case 1.1 $\mv_1<0$]{
					\begin{tikzpicture}
						\draw (0,0.5)--(2,0.5);
						\draw (0,3)--(2,3);
						\draw[black,->] (1,2)--(1,1.5);
						\draw (0.2,3)--(0.6,3.2);
						\draw (0.6,3)--(1.0,3.2);
						\draw (1.0,3)--(1.4,3.2);
						\draw (1.4,3)--(1.8,3.2);
						\draw (0.2,0.3)--(0.6,0.5);
						\draw (0.6,0.3)--(1.0,0.5);
						\draw (1.0,0.3)--(1.4,0.5);
						\draw (1.4,0.3)--(1.8,0.5);
						\draw[dashed] (1,2)--(1,3);
						\fill[black] (1,2)circle (.04);
						\node [left] at (1,2) {$(\eta,\mathfrak{v})$};
						\node [right] at (2.2,0.5) {$\eta=0$};
						\node [right] at (2.2,3) {$\eta=d$};						
						\draw[->] (1,2.6)--(1,2.4);
				\end{tikzpicture}}
				\qquad \qquad \qquad
				\subfigure[Case 1.2 $\mv_1>0$]{
					\begin{tikzpicture}
						\draw (0,0.5)--(2,0.5);
						\draw (0,3)--(2,3);
						\draw[black,->] (1,1.5)--(1,2);
						\draw (0.2,3)--(0.6,3.2);
						\draw (0.6,3)--(1.0,3.2);
						\draw (1.0,3)--(1.4,3.2);
						\draw (1.4,3)--(1.8,3.2);
						\draw (0.2,0.3)--(0.6,0.5);
						\draw (0.6,0.3)--(1.0,0.5);
						\draw (1.0,0.3)--(1.4,0.5);
						\draw (1.4,0.3)--(1.8,0.5);
						\draw[dashed] (1,0.5)--(1,1.5);
						\draw[dashed] (1.5,0.5)--(1.5,3);
						\draw[blue,->] (1.5,0.5)--(1.5,0.2);
						\fill[black] (1,1.5)circle (.04);
						\node [left] at (1,1.5) {$(\eta,\mathfrak{v})$};
						\fill (1.5,0.5)circle (.04);
						\node [right] at (1.5,0.7) {$(\eta_1,\mathfrak{v}_1)$};
						\draw[->] (1,1)--(1,1.2);
						\draw[->] (1.5,2.3)--(1.5,2.1);
				\end{tikzpicture}}
				\caption{The backward progress for {\it Case 1}}
			\end{figure}
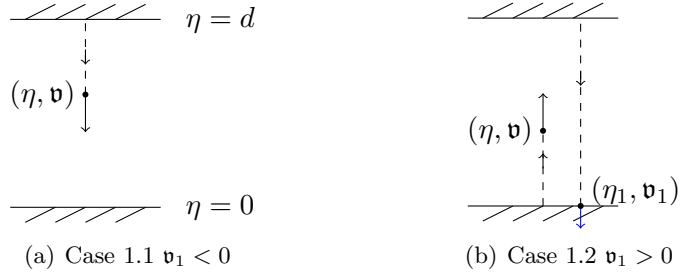

			{\it Case 2. $E_1=V_2^2(\eta_+)$ for some $0<\eta_+\leq d$, which is equivalent to
			\begin{align}\label{22.4}
				\mv_1^2+\mv_2^2=\mv_2^2\cdot e^{2W(\eta_+)-2W(\eta)}.
		\end{align}}
			Thus, the characteristic line cannot reach $\gamma_-$ on $X(s)=d$, then it will cycle between $0$ and $\eta_+$. We have $\eta_k=0$ for all $k\geq 1$, then define $v_k=v_1=(-\sqrt{E_1-E_2^2}, E_2)$.\\
			{\it Case 2.1.} For $\mv_1>0$, we define the backward cycle as
			\begin{equation}\label{Eq3.11}
				\left\{\begin{aligned}
					X_{cl}(s;t,\eta,\mv)&=\sum\limits_{k\geq 0}\mathbf{1}_{(\mathfrak{t}_{k+1},\mathfrak{t}_k]}(s)\left\{\eta_k-\int_{s}^{\mathfrak{t}_k} V_{cl,1}(z) dz\right\},\\
					V_{cl}(s;t,\eta,\mv)&=\mathbf{1}_{(\mathfrak{t}_{1},\mathfrak{t}]}(s)\left(\sqrt{E_1-E_{2}^{2}e^{2W(X(s))}},E_{2}e^{W(X(s))}\right)\nonumber\\
					&\quad +\sum\limits_{k\geq 1}\mathbf{1}_{(\mathfrak{t}_{k,0},\mathfrak{t}_{k}]}(s)\left(-\sqrt{E_1-E_{2}^{2}e^{2W(X(s))}},E_{2}e^{W(X(s))}\right)\nonumber\\	
					&\quad +\sum\limits_{k\geq 1}\mathbf{1}_{(\mathfrak{t}_{k+1},\mathfrak{t}_{k,0}]}(s)\left(\sqrt{E_1-E_{2}^{2}e^{2W(X(s))}},E_{2}e^{W(X(s))}\right).
				\end{aligned}\right.
			\end{equation}
			The notation $\mathfrak{t}_{k,0}$ means that on the $k$-th cycle, $V_{1}(s)=0$ at $s=\mathfrak{t}_{k,0}\in[\mathfrak{t}_{k+1},\mathfrak{t}_k]$.\\
			{\it Case 2.2.} For $\mv_1<0$, the backward cycle is
				\begin{equation}
				\left\{\begin{aligned}
					X_{cl}(s;t,\eta,\mv)&=\sum\limits_{k\geq 0}\mathbf{1}_{(\mathfrak{t}_{k+1},\mathfrak{t}_k]}(s)\left\{\eta_k-\int_{s}^{\mathfrak{t}_k} V_{cl,1}(z) dz\right\},\\
					V_{cl}(s;t,\eta,\mv)&=\sum\limits_{k\geq 0}\mathbf{1}_{(\mathfrak{t}_{k,0},\mathfrak{t}_{k}]}(s)\left(-\sqrt{E_1-E_{2}^{2}e^{2W(X(s))}},E_{2}e^{W(X(s))}\right)\nonumber\\
					&\quad+\sum\limits_{k\geq 0}\mathbf{1}_{(\mathfrak{t}_{k+1},\mathfrak{t}_{k,0}]}(s)\left(\sqrt{E_1-E_{2}^{2}e^{2W(X(s))}},E_{2}e^{W(X(s))}\right).
				\end{aligned}\right.
			\end{equation}
			See the following figure for {\it Case 2}:
			\begin{figure}[H]\label{fig2}
				\centering
				\subfigure[Case 2.1 $\mv_1>0$]{
					\begin{tikzpicture}
						\draw (0,0.5)--(3,0.5);
						\draw (0,4)--(3,4);
						\node at (3.5,0.5) {$X=0$};
						\node at (3.5,4) {$X=d$};
						\draw (0.2,4)--(0.6,4.2);
						\draw (0.6,4)--(1.0,4.2);
						\draw (1.0,4)--(1.4,4.2);
						\draw (1.4,4)--(1.8,4.2);
						\draw (1.8,4)--(2.2,4.2);
						\draw (2.2,4)--(2.6,4.2);
						\draw (2.6,4)--(3.0,4.2);
						\draw (0.2,0.3)--(0.6,0.5);
						\draw (0.6,0.3)--(1.0,0.5);
						\draw (1.0,0.3)--(1.4,0.5);
						\draw (1.4,0.3)--(1.8,0.5);
						\draw (1.8,0.3)--(2.2,0.5);
						\draw (2.2,0.3)--(2.6,0.5);
						\draw (2.6,0.3)--(3.0,0.5);
						\draw[densely dashed] (0,3.7)--(3,3.7);
						\node at (3.6,3.6) {$X=\eta_+$};%the basic pattern
						\draw[->] (0.5,2)--(0.5,2.5);		
						\fill[black] (0.5,2)circle (.04);
						\node [left] at (0.5,2) {$(\eta,\mathfrak{v})$};%starting point
						\draw[dashed] (0.5,0.5)--(0.5,2);
						\draw[dashed] (0.8,0.5)--(0.8,3.7);
						\draw[blue,->] (0.8,0.5)--(0.8,0.2);
						\fill (0.8,0.5)circle (.04);
						\node  at (0.8,0.1) {$(\eta_1,v_1)$};%the first collision
						\draw[->] (1.1,3.5)--(1.1,3.7);
						\draw[dashed] (1.1,0.5)--(1.1,3.5);
						\draw[dashed] (1.4,0.5)--(1.4,3.7);
						\draw[blue,->] (1.4,0.5)--(1.4,0.2);
						\fill (1.4,0.5)circle (.04);	%\node at (2,0.2) {$(\eta_2,v_2)$};% the second collision
						\draw[loosely dotted] (1.4,2)--(2.2,2);
						\draw[blue,->] (2.2,0.5)--(2.2,0.2);
						\fill (2.2,0.5)circle (.04);
						\node  at (2.2,0.1) {$(\eta_k,v_k)$};%the kth collision
						\draw[dashed] (2.2,0.5)--(2.2,3.7);
						\draw[->] (2.5,3.5)--(2.5,3.7);
						\draw[dashed] (2.5,0.5)--(2.5,3.5);
						\draw[->] (0.5,1)--(0.5,1.2);
						\draw[->] (0.8,2.3)--(0.8,2.1);
						\draw[->] (1.1,1)--(1.1,1.2);
						\draw[->] (1.4,2.3)--(1.4,2.1);
						\draw[->] (2.2,2.3)--(2.2,2.1);
						\draw[->] (2.5,1)--(2.5,1.2);
				\end{tikzpicture}}
				\qquad\qquad
				\subfigure[Case 2.2 $\mv_1<0$]{
					\begin{tikzpicture}
						\draw (0,0.5)--(3,0.5);
						\draw (0,4)--(3,4);
						\node at (3.5,0.5) {$X=0$};
						\node at (3.5,4) {$X=d$};
						\draw (0.2,4)--(0.6,4.2);
						\draw (0.6,4)--(1.0,4.2);
						\draw (1.0,4)--(1.4,4.2);
						\draw (1.4,4)--(1.8,4.2);
						\draw (1.8,4)--(2.2,4.2);
						\draw (2.2,4)--(2.6,4.2);
						\draw (2.6,4)--(3.0,4.2);
						\draw (0.2,0.3)--(0.6,0.5);
						\draw (0.6,0.3)--(1.0,0.5);
						\draw (1.0,0.3)--(1.4,0.5);
						\draw (1.4,0.3)--(1.8,0.5);
						\draw (1.8,0.3)--(2.2,0.5);
						\draw (2.2,0.3)--(2.6,0.5);
						\draw (2.6,0.3)--(3.0,0.5);
						\draw[densely dashed] (0,3.7)--(3,3.7);
						\node at (3.6,3.6) {$X=\eta_+$};%the basic pattern
						\draw[->] (0.2,2.5)--(0.2,2);		
						\fill[black] (0.2,2.5)circle (.04);
						\node [left] at (0.2,2.5) {$(\eta,\mathfrak{v})$};%starting point
						\draw[dashed] (0.2,2.5)--(0.2,3.7);
						\draw[->] (0.5,3.5)--(0.5,3.7);
						\draw[dashed] (0.5,0.5)--(0.5,3.5);
						\draw[blue,->] (0.8,0.5)--(0.8,0.2);
						\fill (0.8,0.5)circle (.04);
						\node  at (0.8,0.1) {$(\eta_1,v_1)$};%the first collision
						\draw[dashed] (0.8,0.5)--(0.8,3.7);
						\draw[->] (1.1,3.5)--(1.1,3.7);
						\draw[dashed] (1.1,0.5)--(1.1,3.7);
						\draw[blue,->] (1.4,0.5)--(1.4,0.2);
						\draw[dashed] (1.4,0.5)--(1.4,3.7);
						\fill (1.4,0.5)circle (.04);	%\node at (2,0.2) {$(\eta_2,v_2)$};% the second collision
						\draw[loosely dotted] (1.4,2)--(2.2,2);
						\draw[blue,->] (2.2,0.5)--(2.2,0.2);
						\fill (2.2,0.5)circle (.04);
						\node  at (2.2,0.1) {$(\eta_k,v_k)$};%the kth collision
						\draw[dashed] (2.2,0.5)--(2.2,3.7);
						\draw[->] (2.5,3.5)--(2.5,3.7);
						\draw[dashed] (2.5,0.5)--(2.5,3.5);
						\draw[->] (0.5,1)--(0.5,1.2);
						\draw[->] (0.8,2.3)--(0.8,2.1);
						\draw[->] (1.1,1)--(1.1,1.2);
						\draw[->] (1.4,2.3)--(1.4,2.1);
						\draw[->] (2.2,2.3)--(2.2,2.1);
						\draw[->] (2.5,1)--(2.5,1.2);
				\end{tikzpicture}}
				\caption{The backward progress for {\it Case 2}}
			\end{figure}
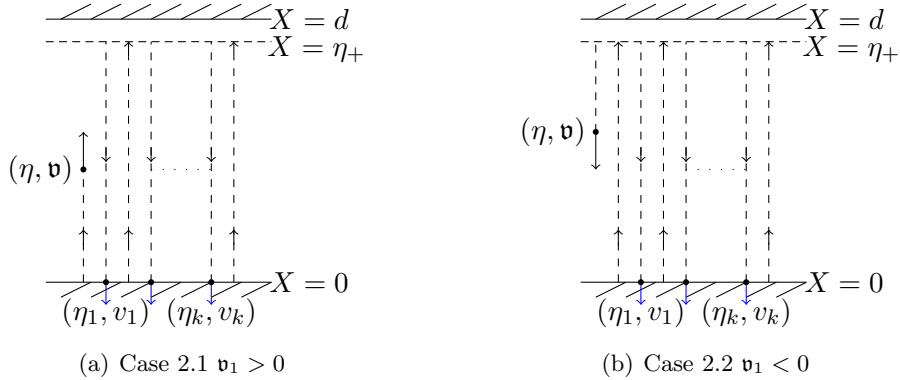
			
			For later use, we define:
			\begin{align*}
				\mathcal{L}_\lambda f_\lambda:= & 	\mv_1\partial_\eta f_\lambda+G(\eta)\left(\mv_2^2\f{\partial f_\lambda}{\partial \mv_1}-\mv_1 \mv_2\f{\partial f_\lambda}{\partial \mv_2}\right)+\nu f_\lambda -\lambda Kf_\lambda -\f{u_\tau^0}{2T^0}G(\eta)\mv_1 \mv_2 f_\lambda.
			\end{align*}
			In this subsection, we consider
			\begin{align}\label{5.23-1}
				\begin{cases}
					\mathcal{L}_\lambda f =S,\\
					f(0,\mv_1,\mv_2)|_{\mv_1>0}=f(0,-\mv_1,\mv_2), \\
					f(d,\mv)|_{\mv_1<0}=0.
				\end{cases}
			\end{align}
			We denote the solution of \eqref{5.23-1} as $\mathcal{L}_\lambda^{-1}S.$ 
			We start with the existence and uniform estimate for $\mathcal{L}_0^{-1}S$. 
			\begin{Lemma}\label{lem1.1}
				Let $\beta>3$ and  $\|\nu^{-1}w\bar{S}\|_{L^\infty_{\eta,\mv}}<\infty$. For
				\begin{align}\label{5.13-1}
					\begin{cases}
						\dis	\mathcal{L}_0f_0= 	\mv_1\partial_\eta f_0+G(\eta)\left(\mv_2^2\f{\partial f_0}{\partial \mv_1}-\mv_1 \mv_2\f{\partial f_0}{\partial \mv_2}\right)+\nu f_0-\f{u_\tau^0}{2T^0}G(\eta)\mv_1 \mv_2 f_0=S,\\
						\dis	f_0(0,\mv)|_{\gamma_-}=f_0(0,-\mv_1,\mv_2),\quad f_0(d,\mv)|_{\gamma_-}=0,
					\end{cases}
				\end{align}
				there is a unique solution $f_0$ of \eqref{5.13-1} satisfying 
				\begin{align*}
					\left\|\f{\sqrt{\mu_0}}{{\sqrt{\mu_M}}}wf_0\right\|_{L^\infty_{\eta,\mv}}+\left|\f{\sqrt{\mu_0}}{{\sqrt{\mu_M}}}wf_0\right|_{L^\infty(\gamma_+)}<C\left\|\nu^{-1}w\bar{S}\right\|_{L^\infty_{\eta,\mv}},
				\end{align*}
				where the constant $C$ is independent of $d$. Moreover, if $S$ is continuous in $[0,d]\times \R^2$, then $f_0$ is continuous away from grazing set $\{(\eta,\mv):\eta=0,\mv_1=0\}$.
			\end{Lemma}
			\begin{proof}
				Since the proof is very long, we divide it into several steps.\\
				 {\noindent \it Step 1.} Consider the iterative scheme for \eqref{5.13-1} with $n\in\mathbb{N}$:
				 	\begin{align*}
				 	\begin{cases}
				 		\dis	 \mv_1 \partial_\eta f^{i+1}+G(\eta)\left(\mv_2^2 \frac{\partial f^{i+1}}{\partial \mv_1}-\mv_1 \mv_2 \frac{\partial f^{i+1}}{\partial \mv_2}\right)+\nu f^{i+1}-\f{u_\tau^0}{2T^0}G(\eta)\mv_1 \mv_2 f^{i+1}=S, \\
				 		\dis	f^{i+1}(0, \mv)|_{\gamma_{-}}=\left(1-\frac{1}{n}\right) f^i\left(0,-\mv_1, \mv_2\right), \quad f^{i+1}(d, \mv)|_{\gamma_-}=0 .
				 	\end{cases}
				 \end{align*}
			 where $f^0\equiv 0$ and $i=0,1,2,...$ Let $\dis \mf^i=\f{\sqrt{\mu_0}}{\sqrt{\mu_M}}f^i$, the equation for $\mf^{i+1}$ is 
				\begin{align}\label{5.16}
					\begin{cases}
						\dis	 \mv_1 \partial_\eta \mathbf{f}^{i+1}+G(\eta)\left(\mv_2^2 \frac{\partial \mathbf{f}^{i+1}}{\partial \mv_1}-\mv_1 \mv_2 \frac{\partial \mathbf{f}^{i+1}}{\partial \mv_2}\right)+\nu \mathbf{f}^{i+1}=\bar{S}, \\
						\dis	\mathbf{f}^{i+1}(0, \mv)|_{\gamma_{-}}=\left(1-\frac{1}{n}\right) \mathbf{f}^i\left(0,-\mv_1, \mv_2\right), \quad \mathbf{f}^{i+1}(d, \mv)|_{\gamma_-}=0 .
					\end{cases}
				\end{align}
			Let $h^i=w \mf^i$. Along the characteristic \eqref{5-6}, we have from $\eqref{5.16}_1$ that
				\begin{align*}
					\f{d}{ds}\left\{h^{i+1}e^{\nu s}\right\}=w\bar{S}e^{\nu s}.
				\end{align*} 
				For {\it Case 1.1}, one has 
				\begin{align}\label{4.11-1}
					h^{i+1}(\eta,\mv)= \int_{\mathfrak{t}_1}^{\mathfrak{t}} w \bar{S} \cdot e^{-\nu(\mathfrak{t}-s)} d s .
				\end{align}
				In {\it Case 1.2}, one obtains
				\begin{align}\label{4.11-2}
					h^{i+1}(\eta,\mv)= \sum_{l=0}^1(1-\f1n)^l\int_{\mathfrak{t}_{l+1}}^{\mathfrak{t}_l} w \bar{S} \cdot e^{-\nu(\mathfrak{t}-s)} d s .
				\end{align}
				Hence, for {\it Case 1}, we have from \eqref{4.11-1}--\eqref{4.11-2} that
				\begin{align}\label{5.16-1}
					h^{i+1}(\eta,\mv)\equiv h^i(\eta,\mv)\equiv ...\equiv h^1(\eta,\mv).
				\end{align}
				
				For {\it Case 2}, one has
				\begin{align*}
					h^{i+1}(\eta,\mv)e^{\nu t}&=(1-\f1n)h^i(\eta_1,v_1)e^{\nu \mathfrak{t}_1}+\int_{\mathfrak{t}_1}^{\mathfrak{t}}w\bar{S}\cdot e^{\nu s}\, ds\nonumber\\
					&=...\nonumber\\
					&=(1-\f1n)^kh^{i+1-k}(\eta_k,v_k)e^{\nu \mathfrak{t}_k}+\sum_{l=1}^k(1-\f1n)^{l-1}\int_{\mathfrak{t}_l}^{\mathfrak{t}_{l-1}}w\bar{S}\cdot e^{\nu s}\, ds,
				\end{align*}
				which yields that
				\begin{align}\label{5.17}
					h^{i+1}(\eta,\mv)=&\left(1-\frac{1}{n}\right)^k h^{i+1-k}\left(0, v_k\right) e^{-\nu\left(\mathfrak{t}-\mathfrak{t}_k\right)}\nonumber\\
					&+\sum_{l=1}^k\left(1-\frac{1}{n}\right)^{l-1} \int_{\mathfrak{t}_l}^{\mathfrak{t}_{l-1}} w \bar{S} \cdot e^{-\nu(\mathfrak{t}-s)} d s .
				\end{align}
				Choosing $k=i+1$, then the first term on the RHS of \eqref{5.17} is zero, thus we have 
				\begin{align}\label{4.13-1}
					|h^{i+1}|\leq \left|\int_{\mathfrak{t}_k}^{\mathfrak{t}}w\bar{S}\cdot e^{-\nu (\mathfrak{t}-s)}\, ds\right|\leq \|\nu^{-1}w\bar{S}\|_{L^\infty_{\eta,\mv}}.
				\end{align}
				Combining \eqref{4.11-1}, \eqref{4.11-2} and \eqref{4.13-1}, we obtain
				\begin{align}\label{5.17-2}
					\left\|h^{i+1}\right\|_{L_{\eta, \mv}^{\infty}}+\left|h^{i+1}\right|_{L^{\infty}\left(\gamma_{+}\right)} \leq C\left\|\nu^{-1} w \bar{S}\right\|_{L_{\eta, \mv}^{\infty}}, \quad i=0,1,2,\dots,
				\end{align}
				where $C>0$ does not depend on $i,n$ and $\v$.
				
				{\noindent \it Step 2.} Now we consider $h^{i+1}-h^i$. In {\it Case 1}, we have from \eqref{5.16-1} that $$h^{i+1}-h^i\equiv 0.$$ For {\it Case 2}, taking $k=i$, one has from \eqref{5.17} that
				\begin{align*}
					(h^{i+1}-h^i)(\eta,\mv)=\left(1-\frac{1}{n}\right)^ih^1(0,v_i) e^{-\nu\left(\mathfrak{t}-\mathfrak{t}_i\right)},
				\end{align*}
				which, together with \eqref{5.17-2}, yields that
				\begin{align*}
					\|h^{i+1}-h^i\|_{L_{\eta,\mv}^{\infty}}+|h^{i+1}-h^i|_{L^{\infty}(\gamma_{+})}\leq C\left(1-\frac{1}{n}\right)^i\|\nu^{-1} w \bar{S}\|_{L_{\eta, \mv}^{\infty}} .
				\end{align*}
				It follows that $\left\{h^i\right\}$ is a Cauchy sequence in $L^{\infty}$, hence there exists a $\mf^n\in L^\infty$ such that $h^i\to w\mf^n$ as $i\to \infty$. It is clear that $\mf^n$ solves
				\begin{align}\label{4.19-1}
					\begin{cases}
					\dis	 \mv_1 \partial_\eta \mathbf{f}^n+G(\eta)\left(\mv_2^2 \frac{\partial \mathbf{f}^n}{\partial \mv_1}-\mv_1 \mv_2 \frac{\partial \mathbf{f}^n}{\partial \mv_2}\right)+\nu \mathbf{f}^n=\bar{S}, \\
					\dis	\mathbf{f}^n(0, \mv)|_{\gamma_{-}}=\left(1-\frac{1}{n}\right) \mathbf{f}^n\left(0,-\mv_1, \mv_2\right), \quad \mf^n(d, \mv)|_{\gamma_-}=0,
					\end{cases}
				\end{align}
			which is 
			\begin{align}
					\begin{cases}
					\mathcal{L}_0 f^n=S, \\
					f^n(0, \mv)|_{\gamma_{-}}=\left(1-\frac{1}{n}\right) f^n\left(0,-\mv_1, \mv_2\right), \quad f^n(d, \mv)|_{\gamma_-}=0,
				\end{cases}
			\end{align}
		where $f^n=\f{\sqrt{\mu_M}}{\sqrt{\mu_0}}\mf^n$. It follows from \eqref{5.17-2} that
				\begin{align}\label{5.17-3}
					\left\|\f{\sqrt{\mu_0}}{\sqrt{\mu_M}}wf^n\right\|_{L^\infty_{\eta,\mv}}+\left|\f{\sqrt{\mu_0}}{\sqrt{\mu_M}}wf^n\right|_{L^\infty(\gamma_+)}\leq C\|\nu^{-1}w\bar{S}\|_{L^\infty_{\eta,\mv}},
				\end{align}
				where $C>0$ is a constant independent of $n,d$ and $\v$. 

				{\noindent \it Step 3.}	Now we consider the limit $n\to \infty$. Without loss of generality, we take $n_1\leq n_2$. The equation of $(\mf^{n_1}-\mf^{n_2})$ is
				%\begin{align}\label{5.20}
				%	\begin{cases}
					%	\dis	\mv_1 \partial_\eta (f^{n_1}-f^{n_2})+G(\eta)\left(\mv_2^2 \frac{\partial\left(f^{n_1}-f^{n_2}\right)}{\partial \mv_1}-\mv_1 \mv_2 \frac{\partial\left(f^{n_1}-f^{n_2}\right)}{\partial \mv_2}\right) \\
					%	\dis	\quad\quad\quad\quad\quad+\nu\left(f^{n_1}-f^{n_2}\right)-\frac{u_\tau^0}{2 T^0} G(\eta) \mv_1 \mv_2\left(f^{n_1}-f^{n_2}\right)=0, \\
					%	\dis	\left(f^{n_1}-f^{n_2}\right)(0,\mv)|_{\gamma_{-}}=\left(1-\frac{1}{n_1}\right)\left(f^{n_1}-f^{n_2}\right)\left(0,-\mv_1, \mv_2\right)+\left(\frac{1}{n_2}-\frac{1}{n_1}\right) f^{n_2}\left(0,-\mv_1, \mv_2\right), \\
					%	\dis	\left(f^{n_1}-f^{n_2}\right)(d, v)|_{\gamma_-}=0,
					%	\end{cases}
				%\end{align}
				%and
				\begin{align}\label{5.20-1}
					\begin{cases}
						\dis	\mv_1 \partial_\eta (\mf^{n_1}-\mf^{n_2})+G(\eta)\left(\mv_2^2 \frac{\partial\left(\mf^{n_1}-\mf^{n_2}\right)}{\partial \mv_1}-\mv_1 \mv_2 \frac{\partial\left(\mf^{n_1}-\mf^{n_2}\right)}{\partial \mv_2}\right) +\nu\left(\mf^{n_1}-\mf^{n_2}\right)=0, \\
						\dis	\left(\mf^{n_1}-\mf^{n_2}\right)(0,\mv)|_{\gamma_{-}}=\left(1-\frac{1}{n_1}\right)\left(\mf^{n_1}-\mf^{n_2}\right)\left(0,-\mv_1, \mv_2\right)+\left(\frac{1}{n_2}-\frac{1}{n_1}\right) \mf^{n_2}\left(0,-\mv_1, \mv_2\right), \\
						\dis	\left(\mf^{n_1}-\mf^{n_2}\right)(d, \mv)|_{\gamma_-}=0.
					\end{cases}
				\end{align}
				
				To estimate \eqref{5.20-1}, we first consider the following problem.
				\begin{align}\label{4.20-1}
					\begin{cases}
						\dis	\mv_1 \partial_\eta g+G(\eta)\left(\mv_2^2\f{\partial g}{\partial \mv_1}-\mv_1 \mv_2\f{\partial g}{\partial \mv_2}\right)+\nu g=0,\\
						\dis 	g(0,\mv)|_{\gamma_-}=(1-\f1n)g(0,-\mv_1,\mv_2)+r(-\mv_1,\mv_2),\\
						\dis	g(d,\mv)|_{\gamma_-}=0.
					\end{cases}
				\end{align}
			Integrating along the backward characteristic, we have
				\begin{align*}
					g(\eta,\mv)=
					\begin{cases}
						0,\quad \text{for}\,\text{\it Case 1.1},\\
						r(v_1)e^{-\nu(\mathfrak{t}-\mathfrak{t}_1)},\quad \text{for} \,\text{\it Case 1.2}.
					\end{cases}
				\end{align*}
				For {\it Case 2}, one obtains 
				\begin{align}\label{4.24}
					g(\eta,\mv)=(1-\f1n)^kg(0,v_k)e^{-\nu(\mathfrak{t}-\mathfrak{t}_k)}+\sum_{l=1}^k(1-\f1n)^{l-1}r(v_l)e^{-\nu(\mathfrak{t}-\mathfrak{t}_l)},
				\end{align}
				where the second term on the RHS brings the main difficulty since there is not a uniform bounded $k$ to bound the first term for general $(\eta,\mv)$. 
				
				In {\it Case 2}, it is clear that
				\begin{align*}
					\int_{\mathfrak{t}_{l+1}}^{\mathfrak{t}_l}|V_1|\,ds =2\eta_+.
				\end{align*}
				It follows from \eqref{22.4} that
				\begin{align}\label{22.5}
					\eta_+&=\f{1}{\v^2}\left\{1-\f{|\mv_2|}{|\mv|}\right\}+\f{|\mv_2|}{|\mv|}\eta 
					 =\f{1}{\v^2}\left\{\f{|\mv_1|^2}{|\mv|\cdot(|\mv|+|\mv_2|)}\right\}+\f{|\mv_2|}{|\mv|}\eta.
				\end{align}
				If $|\mv_1|\geq \tilde{m}>0$, one has from \eqref{22.5} that $(1+|\mv|)^2\eta_+\geq \f{C\tilde{m}^2}{\v^2}>0$, so we choose $k=(1+|\mv|^2)k_0(\tilde{m})$ with $k_0(\tilde{m})>0$ suitably large such that
				\begin{align*}
					e^{-\nu(\mathfrak{t}-\mathfrak{t}_k)}\leq e^{-(k-1)\int_{\mathfrak{t}_2}^{\mathfrak{t}_1}|V_1|\,d\tau }= e^{-2(k-1)\eta_+}\leq \f18.
				\end{align*}
				It follows from \eqref{5-7} that $V_2(X(s))$ increases with $X(s)$, then one has $|v_{k+1,1}|\geq \tilde{m}$ for $|\mv_1|\geq \tilde{m}$.	
				Combining {\it Case 1} and {\it Case 2}, one obtains
				\begin{align}\label{5.20-2}
					&\|\mathbf{1}_{\{|\mv_1|\geq \tilde{m}\}} g\|_{L^\infty_{\eta,\mv}}+|\mathbf{1}_{\{|\mv_1|\geq \tilde{m}\}}g|_{L^\infty(\gamma_+)}\nonumber\\
					&\leq \f18|\mathbf{1}_{\{|\mv_1|\geq \tilde{m}\}}g|_{L^\infty(\gamma_+)}+C_{\tilde{m}}|w_2r|_{L^\infty_\mv}\leq C_{\tilde{m}}|w_2r|_{L^\infty_\mv}.
				\end{align}		
				
				For $\eta\geq \tilde{m}'$, one has from \eqref{22.5} that 
				\begin{align}\label{22.5-1}
					\eta_+=\eta+\left(\f{1}{\v^2}-\eta\right)\cdot\left(1-\f{|\mv_2|}{|\mv|}\right)\geq \tilde{m}'.
				\end{align}
			 So we choose $k=k_0'(\tilde{m}')$ with $k_0'(\tilde{m}')>0$ suitably large such that
			 \begin{align*}
			 	e^{-\nu(\mathfrak{t}-\mathfrak{t}_k)}\leq e^{-(k-1)\int_{\mathfrak{t}_2}^{\mathfrak{t}_1}|V_1|\,d\tau }= e^{-2(k-1)\eta_+}\leq \f18.
			 \end{align*}
		  In {\it Case 2}, we have from \eqref{4.24} that 
		  \begin{align}\label{4.24-1}
		    g(\eta,\mv)&=(1-\f1n)^kg(0,v_k)e^{-\nu(\mathfrak{t}-\mathfrak{t}_k)}+\sum_{l=1}^k(1-\f1n)^{l-1}r(v_l)e^{-\nu(\mathfrak{t}-\mathfrak{t}_l)}\nonumber\\
		    &=(1-\f1n)^kg(\eta,V_{cl}(\mathfrak{t}_c))e^{-\nu(\mathfrak{t}-\mathfrak{t}_k)-\nu(\mathfrak{t}_k-\mathfrak{t}_c)}+\sum_{l=1}^k(1-\f1n)^{l-1}r(v_l)e^{-\nu(\mathfrak{t}-\mathfrak{t}_l)},
		  \end{align}
		 where $\mathfrak{t}_c\in(\mathfrak{t}_{k+1},\mathfrak{t}_k)$ is the time that the backward characteristic moves from $X(\mathfrak{t}_k)=0$ to $X(\mathfrak{t}_c)=\eta$ without reaching $\eta_+$. Combining {\it Case 1} and {\it Case 2}, one obtains from \eqref{4.24-1} that
		 	\begin{align}\label{5.20-3}
		 	\|\mathbf{1}_{\{\eta\geq \tilde{m}'\}} g\|_{L^\infty_{\eta,\mv}}\leq \f18\|\mathbf{1}_{\{\eta\geq \tilde{m}'\}} g\|_{L^\infty_{\eta,\mv}}+C_{\tilde{m}'}|r|_{L^\infty_\mv}\leq C_{\tilde{m}'}|r|_{L^\infty_\mv}.
		 \end{align}		
				Applying \eqref{5.20-2}, \eqref{5.20-3} to \eqref{5.20-1}, and using \eqref{5.17-3}, we have
				\begin{align}\label{5.21-1}
					&\|\mathbf{1}_{\{|\mv_1|\geq \tilde{m}\}} w_{\beta-2}(\mf^{n_1}-\mf^{n_2})\|_{L^\infty_{\eta,\mv}}+|\mathbf{1}_{\{|\mv_1|\geq \tilde{m}\}}w_{\beta-2}(\mf^{n_1}-\mf^{n_2})|_{L^\infty(\gamma_+)}\\
					&\leq C_{\tilde{m}}(\f{1}{n_1}+\f{1}{n_2})|w_{\beta}\mf^{n_2}|_{L^\infty(\gamma_+)}\leq C_{\tilde{m}}(\f{1}{n_1}+\f{1}{n_2})\|\nu^{-1}w_{\beta }\bar{S}\|_{L^\infty_{\eta,\mv}}\,\to 0, \, n_1,n_2\to \infty ,\nonumber
				\end{align}
			and 
			\begin{align}\label{5.21-2}
					&\|\mathbf{1}_{\{\eta\geq \tilde{m}'\}} w_{\beta}(\mf^{n_1}-\mf^{n_2})\|_{L^\infty_{\eta,\mv}}\leq C_{\tilde{m}'}(\f{1}{n_1}+\f{1}{n_2})|w_{\beta}\mf^{n_2}|_{L^\infty(\gamma_+)}\nonumber\\
					&\leq C_{\tilde{m}'}(\f{1}{n_1}+\f{1}{n_2})\|\nu^{-1}w_{\beta }\bar{S}\|_{L^\infty_{\eta,\mv}}\,\to 0, \, n_1,n_2\to \infty .
			\end{align}
					It follows from \eqref{5.20-2} and \eqref{5.21-2} that $w_\beta\mf^{n} \to w_\beta\mathbf{f}_0$ in $L^\infty(\tilde{m}',d]\times \R^2)$ and  $w_{\beta-2}\mf^{n} \to w_{\beta-2}\mathbf{f}_0$ in $L^\infty([0,d]\times \{\mv\in\R^2:|\mv_1|\geq \tilde{m}\})$. 	Since $\tilde{m},\tilde{m}'$ can be arbitrary small, it is clear to know that $w_{\beta-2}\mf^{n} \to w_{\beta-2}\mathbf{f}_0$ a.e. in $\Omega_d\times\R^2$ and also $f^n-f_0 \to 0$ in $L^2(\Omega_d\times\R^2)$ with $f_0=\f{\sqrt{\mu_M}}{\sqrt{\mu_0}}\mathbf{f}_0$. Noting $\tilde{m}$ and $\tilde{m}'$ can be any positive constants, one has from \eqref{5.21-1}--\eqref{5.21-2} that $\{\mf^n\}$ converges almost everywhere on $\Omega_d\backslash\{\gamma_0\cup\gamma_-\}$. Hence there exists a limit function $\mf_0\in L^\infty_{w_{\beta}}$ and $\dis f_0=\f{\sqrt{\mu_M}}{\sqrt{\mu_0}}\mf_0$ solves \eqref{5.13-1}. It follows immediately that $f_0\in C([0,d]\times\{\R^2\backslash\{|\mv_1|=0\}\})\cup((0,d]\times\R^2)$, which shows that $f_0$ is continuous away from $\{(\eta,\mv):\eta=0,\mv_1=0\}$. Then Lemma \ref{lem1.1} is true and the proof is finished.
			\end{proof}
			Now we give the {\it a priori} estimate for $\mathcal{L}_\lambda^{-1}S$. 	To establish the $L^2$ estimate of $\mathcal{L}_\lambda^{-1}S$, 
			we highlight a critical analytical challenge arising from the geometric curvature term $\dis -\f{u_\tau^0}{2T^0}G(\eta)\mv_1 \mv_2 f_\lambda$ , where the velocity product $\mv_1\mv_2$ resists standard energy methods, see \cite{Duan-2021-ARMA}. In fact, this term  originates from the geometric effect of the disk. By utilizing the smallness associated with $\v$ in $G(\eta)$, we use $L^\infty$--estimate to overcome the difficulty.
			\begin{Lemma}\label{lem1.2-1}
				Assume $\beta>3$. For any given $\lambda\in[0,1]$, let $f_\lambda$ be the solution of \eqref{5.23-1} with $\dis \|f_\lambda\|_{L^2_{\eta,\mv}}<\infty$. Let $\dis\mf_\lambda = \f{\sqrt{\mu_0}}{\sqrt{\mu_M}}f_\lambda$, then it holds that 
				\begin{align}\label{5.24-0}
					\left\|f_\lambda\right\|_{L^2_{\eta,\mv}}^2 \leq C\v^{\min\{2-3\mathfrak{a},1-\mathfrak{a}\}}\left\|w \mf_\lambda\right\|_{L_{\eta,\mv}^{\infty}}^2+C\max\{\v^{-2}, d^4\}\|S\|_{L_{\eta,\mv}^2}^2, 
				\end{align}
				where the constant $C>0$ is independent of $\lambda$.
			\end{Lemma}
			\begin{proof}
				The equation \eqref{5.23-1} can be rewritten as 
				\begin{align}\label{5.23-0}
					\begin{cases}
						\dis	\mv_1\partial_\eta f_\lambda+G(\eta)\left(\mv_2^2\f{\partial f_\lambda}{\partial \mv_1}-\mv_1 \mv_2\f{\partial f_\lambda}{\partial \mv_2}\right)+(1-\lambda)\nu f_\lambda +\lambda \FL_0f_\lambda -\f{u_\tau^0}{2T^0}G(\eta)\mv_1 \mv_2 f_\lambda=S, \\
						\dis	f_\lambda(0,\mv)|_{\gamma_{-}}=f_\lambda\left(0,-\mv_1, \mv_2\right), \quad f_\lambda(d, \mv)|_{\gamma_-}=0.
					\end{cases}
				\end{align}
				Multiplying \eqref{5.23-0} by  $f_\lambda$ and integrating on $\R^2$, noting the non-negativity of $\FL_0$, one obtains
				\begin{align}\label{5.31}
					&\f12\partial_\eta\left\{\int_{\R^2}\mv_1 |f_\lambda|^2\,d\mv\cdot e^{-W(\eta)}\right\}+(1-\lambda)\|f_{\lambda}\|_{\nu}^2\cdot e^{-W(\eta)}+\lambda c_0 \|(\FI-\FP_0)f_\lambda\|_{\nu}^2\cdot e^{-W(\eta)}\nonumber\\
					&\leq C \left|G(\eta)e^{-W(\eta)}\int_{\R^2}\mv_1 \mv_2 |f_\lambda|^2 \, d\mv\right|+C\left|e^{-W(\eta)}\int_{\R^2}S\cdot f_\lambda\, d\mv \right|.
				\end{align}
				A direct calculation shows that  
				\begin{align*}
					&\int_0^d\left|G(\eta)e^{-W(\eta)}\int_{\R^2}\mv_1 \mv_2 |f_\lambda|^2 \, d\mv\right|\, d\eta \leq \|w\mf_\lambda\|_{L^\infty_{\eta,\mv}}^2\cdot \int_0^d |G(\eta)e^{-W(\eta)}|\, d\eta\\
					&\leq C\v^{2-\mathfrak{a}}\|w\mf_\lambda\|_{L^\infty_{\eta,\mv}}^2.
				\end{align*}
				Then, integrating \eqref{5.31} over $[0,d]$ and observing that $e^{-W(\eta)}\sim 1$, we obtain
				\begin{align}\label{5.31-2}
					&\f12|f_\lambda(d)|_{L^2(\gamma_+)}^2+(1-\lambda)\int_0^d\|f_\lambda\|_{L^2_\mv}^2\, d\eta+\lambda c_0\int_0^d\|(\FI-\FP_0)f_\lambda\|_{\nu}^2\, d\eta\nonumber\\
					&\leq C\v^{2-\mathfrak{a}}\|w\mf_\lambda\|_{L^\infty_{\eta,\mv}}^2+C\int_0^d\int_{\R^2}|f_\lambda\cdot S|\, d\mv \, d\eta.
				\end{align}
				The rest proof are divided into several steps.\\
				{\it Step 1.} For $1-\lambda\geq \v$, we have from \eqref{5.31-2} that
				\begin{align}\label{5.32-0}
					\left\|f_\lambda\right\|_{L_{\eta,\mv}^2}^2+|f_\lambda(d)|_{L^2(\gamma_+)}^2  \leq C\v^{-2}\|S\|_{L_{\eta,\mv}^2}^2+C\v^{1-\mathfrak{a}}\left\|w \mf_\lambda\right\|_{L_{\eta,\mv}^{\infty}}^2.
				\end{align}
				
				{\noindent \it Step 2.} For $1-\lambda\leq \v$,
				denote $$\FP_0 f_\lambda=\left\{a_\lambda+b_\lambda\cdot(\mv-\bar{\mathfrak{u}}^0)+c_\lambda(|\mv-\bar{\mathfrak{u}}^0|^2-2T^0)\right\}\sqrt{\mu_0},$$
				then it holds that 
				\begin{align*}
					\mv_1\partial_\eta(\FP_0 f_\lambda)=&\, \partial_\eta a_\lambda\cdot \mv_1\sqrt{\mu_0}+\partial_\eta b_{1,\lambda}\cdot \mv_1^2\sqrt{\mu_0}+\partial_\eta b_{2,\lambda}\cdot \mv_1(\mv_2-u_\tau^0)\sqrt{\mu_0}\nonumber\\
					&+\partial_\eta c_\lambda \cdot \mv_1(|\mv-\bar{\mathfrak{u}}^0|^2-2T^0)\sqrt{\mu_0}.
				\end{align*}
				{\noindent \it Step 2.1 Estimate on $b_{2,\lambda}$.} Taking $\psi_{b,2}=\mv_1(\mv_2-u_\tau^0) \sqrt{\mu_0}$, it is clear that
				\begin{align*}
					\langle \mv_1 f_{\lambda}, \psi_{b,2}\rangle &=	\langle \mv_1 \FP_0f_{\lambda}, \psi_{b,2}\rangle+	\langle \mv_1 (\FI-\FP_0)f_{\lambda}, \psi_{b,2}\rangle\\
					&=\rho^0(T^0)^2b_{2,\lambda}+\langle \mv_1\psi_{b,2}, (\mathbf{I}-\mathbf{P}_0)f_\lambda \rangle,
				\end{align*}  
				where we have used
				\begin{align*}
					\int_{\R^2}\mv_1^2(\mv_2-u_\tau^0)^2\mu_0(\mv) \, d\mv=\rho^0(T^0)^2.
				\end{align*}
				
				Multiplying \eqref{5.23-0} by $\psi_{b,2}$ and integrating on $\R^2$, one has that
				\begin{align}\label{5.32}
					&\partial_\eta \left\{\rho^0(T^0)^2b_{2,\lambda}+\langle \mv_1\psi_{b,2}\sqrt{\mu_0},(\FI-\FP_0)f_\lambda\rangle \right\}\nonumber\\
					%&=-G(\eta)\rho^0\left\{T^0u_\tau^0 a_\lambda +(T^0)^2b_{2,\lambda}+T^0|u_\tau^0|^2b_{2,\lambda}+2(T^0)^2u_\tau^0c_{\lambda}\right\}+\lambda\langle \FL_0f,\mv_1(\mv_2-u_\tau^0)\sqrt{\mu_0}\rangle \nonumber\\
					&=G(\eta)\left\langle \mv_2^2\f{\partial \psi_{b,2}}{\partial \mv_1}-\mv_1\f{\partial (\mv_2 \psi_{b,2})}{\partial \mv_2},f_\lambda\right\rangle -(1-\lambda)\langle \psi_{b,2}, \nu f_\lambda\rangle-\lambda\langle \FL_0f_\lambda,\psi_{b,2}\rangle \nonumber\\
					&\quad +\f{u_\tau^0}{2T^0}G(\eta)\langle \mv_1\mv_2\psi_{b,2},f_\lambda\rangle +\langle \psi_{b,2},S\rangle=:\sum_{i=1}^5\mathcal{I}_{b,i}.
				\end{align}
				It is clear that
				\begin{align}\label{4.54}
					&\int_\eta^d \partial_\eta \left\{\rho^0(T^0)^2b_{2,\lambda}+\langle \mv_1\psi_{b,2},(\FI-\FP_0)f_\lambda\rangle \right\}\, dz\nonumber\\
					&=\langle \mv_1^2(\mv_2-u_\tau^0)\sqrt{\mu_0},f_\lambda(d)\rangle-[\rho^0(T^0)^2b_{2,\lambda}(\eta)+\langle \mv_1^2(\mv_2-u_\tau^0)\sqrt{\mu_0},(\FI-\FP_0)f_\lambda(\eta) \rangle],
				\end{align}
				and 
				\begin{align*}
					|\langle \mv_1^2(\mv_2-u_\tau^0)\sqrt{\mu_0},f_\lambda(d)\rangle|\leq C|f_\lambda(d)|_{L^2(\gamma_+)}^2.
				\end{align*}
				A direct calculation shows 
				\begin{align*}
					\begin{split}
						\left|\int_\eta^d(\mathcal{I}_{b,1}+\mathcal{I}_{b,4})\,dz\right|\leq \int_\eta^d|G(z)|\cdot \|f_\lambda\|_{L^2_\mv}\,dz\leq \left(\int_\eta^d G^2(z)\, dz\right)^{\f12}\cdot \|f_\lambda \|_{L^2_{\eta,\mv}},\\
						\left|\int_\eta^d\mathcal{I}_{b,2}\,dz\right|\leq (1-\lambda) \int_\eta^d \|f_\lambda\|_{L^2_\mv}\,dz\leq (1-\lambda)d^{\f12}\cdot \|f_\lambda\|_{L^2_{\eta,\mv}},\\
						\left|\int_\eta^d\mathcal{I}_{b,3}\,dz\right|\leq \int_\eta^d \|(\FI-\FP_0)f_\lambda\|_{L^2_\mv}\,dz\leq d^{\f12}\cdot \left\{\int_0^d\|(\FI-\FP_0)f_\lambda\|_\nu^2\, d\eta\right\}^{\f12},	
					\end{split}
				\end{align*}
				and 
				\begin{align}\label{4.49}
					\left|\int_\eta^d\mathcal{I}_2^5\,dz\right|\leq \int_\eta^d \|S\|_{L^2_\mv}\,dz\leq d^{\f12}\cdot \|S\|_{L^2_{\eta,\mv}}.
				\end{align}
				
				Thus, by integrating \eqref{5.32} over $[\eta,d]$, it follows from \eqref{4.54}--\eqref{4.49} that 
				\begin{align}\label{4.49-1}
					|b_{2,\lambda}(\eta)|&\leq C|f_\lambda(d)|_{L^2(\gamma_+)}+C\|(\FI-\FP_0)f_\lambda(\eta)\|_\nu\nonumber\\
					&\quad +Cd^{\f12}\cdot \|S\|_{L^2_{\eta,\mv}}+Cd^{\f12}\cdot \left\{\int_0^d\|(\FI-\FP_0)f_\lambda\|_\nu^2\, d\eta\right\}^{\f12}\nonumber\\
					&\quad +C\left\{(1-\lambda)d^{\f12}+ \left(\int_\eta^d G^2(z)\, dz\right)^{\f12}\right\}\|f_\lambda\|_{L^2_{\eta,\mv}}.
				\end{align}
				Subsequently, integrating \eqref{4.49-1} over $[0,d]$ and noting $1-\lambda\leq \v$, one gets
				\begin{align}\label{5.33}
					\int_0^d|b_{2,\lambda}|^2\,d\eta&\leq\, Cd|f_\lambda(d)|_{L^2(\gamma_+)}^2+Cd^2\int_0^d\|(\FI-\FP_0)f_\lambda\|_\nu^2\,d\eta\nonumber\\
					&\quad +Cd^2\|S\|_{L^2_{\eta,\mv}}^2+C\v^{2-2\mathfrak{a}}\|f_{\lambda}\|_{L^2_{\eta,\mv}}^2,
				\end{align}
				where one has used $d=\v^{-\mathfrak{a}}$ and 
				\begin{align*}
					\int_0^d\int_\eta^dG^2(z)dzd\eta\leq C\v^{4-2\mathfrak{a}}.
				\end{align*}
				
				{\noindent \it Step 2.2 Estimate on $c_\lambda$.} Taking $\psi_c=\mv_1(|\mv-\bar{\mathfrak{u}}^0|^2-4T^0)\sqrt{\mu_0}$, we obtain
				\begin{align*}
					\langle \mv_1 f_{\lambda}, \psi_c\rangle &=	\langle \mv_1 \FP_0f_{\lambda}, \psi_c\rangle+	\langle \mv_1 (\FI-\FP_0)f_{\lambda}, \psi_c\rangle\\
					&=8\rho^0(T^0)^3c_{\lambda}+\langle \mv_1 \psi_c, (\mathbf{I}-\mathbf{P}_0)f_\lambda \rangle,
				\end{align*}  
				where we have used
				\begin{align*}
					&\int_{\R^2}\mv_1^2(|\mv-\bar{\mathfrak{u}}^0|^2-4T^0)\mu_0(\mv)d\mv=0,\\
					&\int_{\R^2}\mv_1^2(|\mv-\bar{\mathfrak{u}}^0|^2-2T^0)(|\mv-\bar{\mathfrak{u}}^0|^2-4T^0)\mu_0(\mv)d\mv=8\rho^0(T^0)^3.
				\end{align*}
				
				Multiplying \eqref{5.23-0} by $\psi_c$ and integrating on $\R^2$, one has that
				\begin{align}\label{5.34}
					&\partial_\eta \left\{8\rho^0(T^0)^3c_{\lambda}+\langle \mv_1\psi_c,(\FI-\FP_0)f_\lambda\rangle \right\}\nonumber\\
					%&=-G(\eta)\rho^0\left\{T^0u_\tau^0 a_\lambda +(T^0)^2b_{2,\lambda}+T^0|u_\tau^0|^2b_{2,\lambda}+2(T^0)^2u_\tau^0c_{\lambda}\right\}+\lambda\langle \FL_0f,\mv_1(\mv_2-u_\tau^0)\sqrt{\mu_0}\rangle \nonumber\\
					&=G(\eta)\left\langle \mv_2^2\f{\partial \psi_c}{\partial \mv_1}-\mv_1\f{\partial (\mv_2 \psi_c)}{\partial \mv_2},f_\lambda\right\rangle -(1-\lambda)\langle \psi_c, \nu f_\lambda\rangle-\lambda\langle \FL_0f_\lambda,\psi_c\rangle \nonumber\\
					&\quad +\f{u_\tau^0}{2T^0}G(\eta)\langle \mv_1\mv_2\psi_c,f_\lambda\rangle+\langle \psi_c,S\rangle=:\sum_{i=1}^5\mathcal{I}_{c,i}.
				\end{align}
				It is clear that
				\begin{align}\label{4.52}
					&\int_\eta^d \partial_\eta \left\{8\rho^0(T^0)^3c_{\lambda}+\langle \mv_1\psi_c ,(\FI-\FP_0)f_\lambda\rangle \right\}\, dz\nonumber\\
					&=\langle \mv_1\psi_c,f_\lambda(d)\rangle-[8\rho^0(T^0)^3c_{\lambda}(\eta)+\langle \mv_1\psi_c,(\FI-\FP_0)f_\lambda(\eta)\rangle ].
				\end{align}
				A direct calculation shows 
				\begin{align*}
					\begin{split}
						\left|\int_\eta^d(\mathcal{I}_{c,1}+\mathcal{I}_{c,4})\,dz\right|\leq \int_\eta^d|G(z)|\cdot \|f_\lambda\|_{L^2_\mv}\,dz\leq \left(\int_\eta^d G^2(z)\, dz\right)^{\f12}\cdot \|f_\lambda \|_{L^2_{\eta,\mv}},\\
						\left|\int_\eta^d\mathcal{I}_{c,2}\,dz\right|\leq (1-\lambda) \int_\eta^d \|f_\lambda\|_{L^2_\mv}\,dz\leq (1-\lambda)d^{\f12}\cdot \|f_\lambda\|_{L^2_{\eta,\mv}},\\
						\left|\int_\eta^d\mathcal{I}_{c,3}\,dz\right|\leq \int_\eta^d \|f_\lambda\|_{L^2_\mv}\,dz\leq d^{\f12}\cdot \left\{\int_0^d\|(\FI-\FP_0)f_\lambda\|_\nu^2\, d\eta\right\}^{\f12},
					\end{split}			
				\end{align*}
				and 
				\begin{align}\label{4.53}
					\left|\int_\eta^d\mathcal{I}_{c,5}\,dz\right|\leq \int_\eta^d \|S\|_{L^2_\mv}\,dz\leq d^{\f12}\cdot \|S\|_{L^2_{\eta,\mv}}.
				\end{align}
				
				Then, by integrating \eqref{5.34} over $[\eta,d]$, it follows from \eqref{4.52}--\eqref{4.53} that
				\begin{align}\label{4.53-1}
					|c_{\lambda}(\eta)|&\leq C|f_\lambda(d)|_{L^2(\gamma_+)}+C\|(\FI-\FP_0)f_\lambda(\eta)\|_\nu\nonumber\\
					&\quad +Cd^{\f12}\cdot \|S\|_{L^2_{\eta,\mv}}+Cd^{\f12}\cdot \left\{\int_0^d\|(\FI-\FP_0)f_\lambda\|_\nu^2\, d\eta\right\}^{\f12}\nonumber\\
					&\quad +C\left\{(1-\lambda)d^{\f12}+ \left(\int_\eta^d G^2(z)\, dz\right)^{\f12}\right\}\|f_\lambda\|_{L^2_{\eta,\mv}}.
				\end{align}
				Next, integrating \eqref{4.53-1} over $[0,d]$ and noting $1-\lambda\leq \v$, one gets
				\begin{align}\label{5.35}
					\int_0^d|c_{\lambda}|^2\,d\eta&\leq\, Cd|f_\lambda(d)|_{L^2(\gamma_+)}^2+Cd^2\int_0^d\|(\FI-\FP_0)f_\lambda\|_\nu^2\,d\eta\nonumber\\
					&\quad +Cd^2\|S\|_{L^2_{\eta,\mv}}^2+C\v^{2-2\mathfrak{a}}\|f_{\lambda}\|_{L^2_{\eta,\mv}}^2.
				\end{align}
				{\noindent \it Step 2.3 Estimate on $a_\lambda$.} Taking $\psi_a =-\mv_1(|\mv-\bar{\mathfrak{u}}^0|^2-8T^0)\sqrt{\mu_0}$, one has that 
				\begin{align*}
					\langle \mv_1 f_{\lambda}, \psi_a\rangle &=	\langle \mv_1 \FP_0f_{\lambda}, \psi_a\rangle+	\langle \mv_1 (\FI-\FP_0)f_{\lambda}, \psi_a\rangle\\
					&=4\rho^0(T^0)^2a_{\lambda}+\langle \mv_1 \psi_a, (\mathbf{I}-\mathbf{P}_0)f_\lambda \rangle,
				\end{align*} 
				% to \eqref{5.23-0}, a direct calculation shows
				%		\begin{align}\label{5.36}
					%		&\partial_\eta\left\{4\rho^0(T^0)^2a_\lambda-\langle \mv_1^2(|\mv-\bar{\mathfrak{u}}^0|^2-8T^0)\sqrt{\mu_0},(\FI-\FP_0)f_\lambda\rangle\right\}\nonumber\\
					%			&-G(\eta)\rho^0\{4T^0|u_\tau^0|^2a_\lambda+6(T^0)^2u_\tau^0b_{2,\lambda}+8(T^0)^2|u_\tau^0|^2c_{\lambda}\}-\lambda\langle \FL_0f,\mv_1(|\mv-\bar{\mathfrak{u}}^0|^2-8T^0)\sqrt{\mu_0}\rangle\nonumber\\
					%			&+G(\eta)\left\langle \mv_2^2\f{\partial \{\mv_1(|\mv-\bar{\mathfrak{u}}^0|^2-8T^0)\sqrt{\mu_0}\}}{\partial \mv_1}-\f{\partial \{\mv_1^2 \mv_2(|\mv-\bar{\mathfrak{u}}^0|^2-8T^0)\sqrt{\mu_0}\}}{\partial \mv_2},(\FI-\FP_0)f_\lambda\right\rangle \nonumber\\
					%			&+\f{u_\tau^0}{2T^0}G(\eta)\langle \mv_1^2\mv_2(|\mv-\bar{\mathfrak{u}}^0|^2-8T^0)\sqrt{\mu_0},(\FI-\FP_0)f_\lambda\rangle \nonumber\\
					%			&=-\langle \mv_1(|\mv-\bar{\mathfrak{u}}^0|^2-8T^0) \sqrt{\mu_0},S\rangle+(1-\lambda)\langle \mv_1(|\mv-\bar{\mathfrak{u}}^0|^2-8T^0) \sqrt{\mu_0},\nu f_\lambda\rangle,
					%		\end{align}
				where one has used
				\begin{align*}
					&\int_{\R^2}\mv_1^2(|\mv-\bar{\mathfrak{u}}^0|^2-8T^0)\mu_0(\mv)d\mv=-4\rho^0(T^0)^2,\\
					&\int_{\R^2}\mv_1^2(|\mv-\bar{\mathfrak{u}}^0|^2-2T^0)(|\mv-\bar{\mathfrak{u}}^0|^2-8T^0)\mu_0(\mv)d\mv=0.
				\end{align*}
				%		Rewrite the major term in \eqref{5.36} as 
				%		\begin{align*}
					%			4(T^0)^2\partial_\eta\left\{a_\lambda \cdot e^{\f{|u_\tau^0|^2}{T^0}W(\eta)}\right\},
					%		\end{align*}
				%		then integrating \eqref{5.36} on $[\eta,d]$, one has that
				%		\begin{align*}
					%			|a_\lambda(\eta)|&\leq C|f_\lambda(d)|_{L^2(\gamma_+)}+C\|(\FI-\FP_0)f_\lambda\|_\nu(\eta)+C\int_\eta^d\|(\FI-\FP_0)f_\lambda\|_\nu\,dz\nonumber\\
					%			&\quad +C\int_\eta^d\|S\|_{L^2_v}dz+C\int_\eta^d|G(z)\cdot |(b_{2,\lambda},c_{\lambda})(z)|dz+C(1-\lambda)\int_\eta^d |(a_\lambda,b_\lambda,c_{\lambda})|dz.
					%		\end{align*}
				Similarly, one obtains
				\begin{align}\label{5.37}
					\int_0^d|a_{\lambda}|^2\,d\eta&\leq\, Cd|f_\lambda(d)|_{L^2(\gamma_+)}^2+Cd^2\int_0^d\|(\FI-\FP_0)f_\lambda\|_\nu^2\,d\eta\nonumber\\
					&\quad +Cd^2\|S\|_{L^2_{\eta,\mv}}^2+C\v^{2-2\mathfrak{a}}\|f_{\lambda}\|_{L^2_{\eta,\mv}}^2.
				\end{align}
				
				{\noindent \it Step 2.4 Estimate on $b_{1,\lambda}$.}	Multiplying \eqref{5.23-0} by $\sqrt{\mu_0}$, one has that
				\begin{align*}
					\rho^0T^0\partial_\eta b_{1,\lambda}+G(\eta)\rho^0T^0 b_{1,\lambda}=-(1-\lambda)\langle \sqrt{\mu_0},\nu f_\lambda\rangle +\int_{\R^2}S\sqrt{\mu_0}\, d\mv,
				\end{align*}
				which yields that
				\begin{align}\label{4.56-1}
					\rho^0T^0\partial_\eta\left\{b_{1,\lambda} e^{-W(\eta)}\right\}=&-(1-\lambda)\langle \sqrt{\mu_0},\nu f_\lambda\rangle e^{-W(\eta)}\nonumber\\
					&+e^{-W(\eta)}\int_{\R^2}S \sqrt{\mu_0}\, d\mv  .
				\end{align}
				Due to the specular boundary condition at $\eta=0$, it holds that $b_{1,\lambda}(0)=0$. Then, integrating \eqref{4.56-1} over $[0,\eta]$, one gets that 
				\begin{align*}
					e^{-W(\eta)}\rho^0T^0b_{1,\lambda}(\eta)=&-(1-\lambda)\int_0^\eta\langle \sqrt{\mu_0},\nu f_\lambda\rangle e^{-W(z)}dz\\
					&+\int_0^\eta e^{-W(z)}\int_{\R^2}S(z)\cdot \sqrt{\mu_0}\, d\mv \, dz,
				\end{align*}
				which implies that
				\begin{align*}
					|b_{1,\lambda}(\eta)|\leq C(1-\lambda)\eta^{\f12}\|f_\lambda\|_{L^2_{\eta,\mv}}+C\eta^{\f12}\|S\|_{L^2_{\eta,\mv}}.
				\end{align*}
				Hence we have 
				\begin{align}\label{5.39}
					\int_0^d|b_{1,\lambda}|^2d\eta\leq C(1-\lambda)^2d^2\|f_\lambda\|_{L^2_{\eta,\mv}}^2+Cd^2\|S\|_{L^2_{\eta,\mv}}^2.
				\end{align}
				{\noindent \it Step 2.5.} Noting $1-\lambda\leq \v$ and $d=\v^{-\mathfrak{a}}$ with $\mathfrak{a}<\f23,\v\ll1$, one has from \eqref{5.33}, \eqref{5.35}, \eqref{5.37} and \eqref{5.39} that
				\begin{align}\label{5.40}
					\int_0^d|(a_\lambda,b_{\lambda},c_{\lambda})|^2\,d\eta \leq& \, Cd|f(d)|_{L^2(\gamma_+)}^2+Cd^2\int_0^d\|(\FI-\FP_0)f_\lambda\|_\nu^2\,d\eta +Cd^2\|S\|_{L^2_{\eta,\mv}}^2\nonumber\\
					&+C\v^{2-2\mathfrak{a}}\int_0^d|(a_\lambda,b_{\lambda},c_{\lambda})|^2\,d\eta\nonumber\\
					\leq &\, Cd|f(d)|_{L^2(\gamma_+)}^2+Cd^2\int_0^d\|(\FI-\FP_0)f_\lambda\|_\nu^2\,d\eta +Cd^2\|S\|_{L^2_{\eta,\mv}}^2.
				\end{align}
				
				Noting $\lambda\geq 1-\v$, we have from \eqref{5.31-2} that
				\begin{align*}
					\int_0^d\left\|\left(\mathbf{I}-\mathbf{P}_0\right) f_\lambda\right\|_\nu^2\, d\eta+|f_\lambda(d)|_{L^2(\gamma_+)}^2 \leq C \v^{2-\mathfrak{a}}\left\|w \mf_\lambda\right\|_{L_{\eta,\mv}^{\infty}}^2+C\|f_\lambda\|_{L^2_{\eta,\mv}}\cdot \|S\|_{L^2_{\eta,\mv}},
				\end{align*}
				which, together with \eqref{5.40}, yields that
				\begin{align}\label{5.19-1}
					\left\|f_\lambda\right\|_{L^2_{\eta,\mv}}^2 \leq C \v^{2-3\mathfrak{a}} \cdot \left\|w \mf_\lambda\right\|_{L_{\eta,\mv}^{\infty}}^2+Cd^4\|S\|_{L_{\eta,\mv}^2}^2, \quad 1-\lambda \leq \v.
				\end{align}
				{\noindent \it Step 3. Conclusion.} Combining \eqref{5.19-1} and  \eqref{5.32-0}, we conclude \eqref{5.24-0} holds. Therefore the proof of that Lemma \ref{lem1.2-1} is completed.
			\end{proof}
		
		Then we aim to get the $\lambda$-independent $L^\infty$ estimate.	Define $K_{M,w}:=w(\mv)K_Mw(u)^{-1}.$
			\begin{Lemma}\label{lem1.2}
				Assume $\beta>3,\|\nu^{-1}w\bar{S}\|_{L^\infty_{\eta,\mv}}<\infty$. For any given $\lambda\in[0,1]$, let $f_\lambda$ be the solution of \eqref{5.23-1}. Let $\mf_\lambda=\f{\sqrt{\mu_0}}{\sqrt{\mu_M}}f_\lambda$ with $\|w\mf_\lambda\|_{L^\infty_{\eta,\mv}}+|w\mf_\lambda|_{L^\infty(\gamma_+)}<\infty$. Then it holds that 
				\begin{align}\label{4.27}
					\|w\mf_\lambda\|_{L^\infty_{\eta,\mv}}+|w\mf_\lambda|_{L^\infty(\gamma_+)}\leq C\|f_\lambda\|_{L^2_{\eta,\mv}}+C\|\nu^{-1}w\bar{S}\|_{L^\infty_{\eta,\mv}},
				\end{align}
				where the constant $C>0$ is independent of $\lambda$ and $d$.
			\end{Lemma}
			\begin{proof}
				We divide the proof into several steps. \\
				{\it Step 1.} The equation for $\mf_\lambda$ is 
				\begin{align}\label{5.23-2}
					\begin{cases}
					 \dis \mv_1 \partial_\eta \mathbf{f}_\lambda+G(\eta)\left(\mv_2^2 \frac{\partial \mf_\lambda}{\partial \mv_1}-\mv_1 \mv_2 \frac{\partial \mf_\lambda }{\partial \mv_2}\right)+\nu \mathbf{f}_\lambda -\lambda K_M \mathbf{f}_\lambda=\bar{S},\\
					 \dis \mf_\lambda(0,\mv_1,\mv_2)|_{\mv_1>0}=\mf_\lambda (0,-\mv_1,\mv_2), \quad \mf_\lambda(d,\mv)|_{\mv_1<0}=0.
					 \end{cases}
				\end{align}
			Along the characteristics, \eqref{5.23-2} is rewritten as
				\begin{align*}
					\f{d\mf_\lambda }{ds}+\nu \mf_\lambda-\lambda K_M\mf_\lambda=\bar{S}.
				\end{align*}
				Let $h=w\mf_\lambda$, the equation of $h$ is
				\begin{align*}
					\f{dh}{ds}+\nu h-\lambda K_{M,w} h=w\bar{S},
				\end{align*}
				which yields
				\begin{align*}
					\f{d}{ds}\left\{he^{\nu s}\right\}-\lambda K_{M,w}h\cdot e^{\nu s}=w\bar{S}\cdot e^{\nu s}.
				\end{align*}
				
				For {\it Case 1.1}, we have
				\begin{align}\label{22.1}
					h(\eta,\mv)=&\int_{\mathfrak{t}_1}^{\mathfrak{t}} \lambda K_{M,w}h(X_{cl}(s),V_{cl}(s))\cdot e^{-\nu(\mathfrak{t}-s)}ds\nonumber\\
					&+\int_{\mathfrak{t}_1}^{\mathfrak{t}}w\bar{S}(X_{cl}(s),V_{cl}(s))\cdot e^{-\nu(\mathfrak{t}-s)}ds.
				\end{align} 
				For {\it Case 1.2}, we obtain
				\begin{align}\label{22.2}
					h(\eta,\mv)=&\int_{\mathfrak{t}_2}^{\mathfrak{t}} \lambda K_{M,w}h(X_{cl}(s),V_{cl}(s))\cdot e^{-\nu(\mathfrak{t}-s)}ds\nonumber\\
					&+\int_{\mathfrak{t}_2}^{\mathfrak{t}}w\bar{S}(X_{cl}(s),V_{cl}(s))\cdot e^{-\nu(\mathfrak{t}-s)}ds.
				\end{align}
				For {\it Case 2}, one has  $E_1= V_2^2(\eta_+),0<\eta_+\leq d.$ After $k$ times collision with $\eta=0$, one has that
				\begin{align*}
					h(\eta,\mv)e^{\nu \mathfrak{t}}&=h(0,v_1)e^{\nu \mathfrak{t}_1}+\int_{\mathfrak{t}_1}^{\mathfrak{t}}(\lambda K_{M,w}h+w\bar{S})e^{\nu s}ds\nonumber\\
					&=h(0,v_2)e^{\nu \mathfrak{t}_2}+\int_{\mathfrak{t}_2}^{\mathfrak{t}}(\lambda K_{M,w}h+w\bar{S})e^{\nu s}ds\nonumber\\
					&=...\nonumber\\
					&=h(0,v_k)e^{\nu \mathfrak{t}_k}+\int_{\mathfrak{t}_k}^{\mathfrak{t}}(\lambda K_{M,w}h+w\bar{S})e^{\nu s}ds,
				\end{align*}
				which yields that
				\begin{align}\label{22.3}
					h(\eta,\mv)=h(0,v_k)e^{-\nu(\mathfrak{t}-\mathfrak{t}_k)}+\sum_{l=1}^k\int_{\mathfrak{t}_l}^{\mathfrak{t}_{l-1}}(\lambda K_{M,w}h+w\bar{S})e^{-\nu(\mathfrak{t}-s)}ds.
				\end{align} 
				
				To control the first term on the RHS of \eqref{22.3}, we choose $k\sim \f{1}{\eta_+(\eta,\mv)}k_0$ with $k_0\gg 1$, then one obtains
				\begin{align*}
					\dis 	e^{-\nu(\mathfrak{t}-\mathfrak{t}_k)}\leq e^{-(k-1)\nu(\mathfrak{t}_{l}-\mathfrak{t}_{l+1})}\leq \exp\{-(k-1)\int_{\mathfrak{t}_{l+1}}^{\mathfrak{t}_l}|V_1|ds\} = e^{-2(k-1)\eta_+}\leq \f18,
				\end{align*}
				which, together with \eqref{22.1}--\eqref{22.3}, yields that
				\begin{align}\label{4.31}
					\|h\|_{L^\infty_{\eta,\mv}}+|h|_{L^\infty(\gamma_+)}\leq \f18|h|_{L^\infty(\gamma_+)}+C\lambda\| K_{M,w}h\|_{L^\infty_{\eta,\mv}}+C\|\nu^{-1}w\bar{S}\|_{L^\infty_{\eta,\mv}}.
				\end{align}
				Noting \eqref{5.2-0}, one gets from \eqref{4.31} that
				\begin{align}\label{22.6}
					\|\mf_\lambda\|_{L^\infty L^\infty_{\beta,\zeta}}&\leq C\lambda  \|K_M\mf_\lambda\|_{L^\infty L^\infty_{\beta,\zeta}}+C\|\nu^{-1}w\bar{S}\|_{L^\infty_{\eta,\mv}}\nonumber\\
					&\leq C\lambda \|\mf_\lambda\|_{L^\infty L^\infty_{\beta-1,\zeta}}+C\|\nu^{-1}w\bar{S}\|_{L^\infty_{\eta,\mv}}\nonumber\\
					&...\nonumber\\
					&\leq C\lambda \|K_M\mf_\lambda\|_{L^\infty L^\infty_{0,\zeta}}+C\|\nu^{-1}w\bar{S}\|_{L^\infty_{\eta,\mv}}\nonumber\\
					&\leq C\lambda\|\mf_\lambda\|_{L^\infty L^2_\zeta}+C\|\nu^{-1}w\bar{S}\|_{L^\infty_{\eta,\mv}},
				\end{align}
				where we have used the notations $\dis \|\mf_\lambda\|_{L^\infty L^\infty_{\beta,\zeta}}=\|(1+|\mv|^2)^{\f{\beta}{2}}e^{\zeta|\mv|^2}\mf_\lambda\|_{L^\infty_{\eta,\mv}}+|(1+|\mv|^2)^{\f{\beta}{2}}e^{\zeta|\mv|^2}\mf_\lambda|_{L^\infty(\gamma_+)}$ and  $$\dis \|\mf_\lambda\|_{L^\infty L^2_\zeta}^2=\left| \int_{\R^2}e^{2\zeta|\mv|^2}|\mf_\lambda|^2(\eta,\mv)d\mv\right|_{L^\infty_\eta}.$$

				{\it Step 2.} By the mild formulation, one has
				\begin{align}
					\mf_\lambda(\eta,\mv)&=\int_{\mathfrak{t}_1}^{\mathfrak{t}}e^{-\nu(\mv)(\mathfrak{t}-s)}\lambda K_M\mf_\lambda\,ds+\int_{\mathfrak{t}_1}^{\mathfrak{t}}e^{-\nu(\mv)(\mathfrak{t}-s)}\bar{S}\, ds\quad \text{in \it Case 1.1, }\label{4.31-1}\\
					\mf_\lambda(\eta,\mv)&=\int_{\mathfrak{t}_2}^{\mathfrak{t}}e^{-\nu(\mv)(\mathfrak{t}-s)}\lambda K_M\mf_\lambda\,ds+\int_{\mathfrak{t}_2}^{\mathfrak{t}}e^{-\nu(\mv)(\mathfrak{t}-s)}\bar{S}\, ds\quad \text{in \it Case 1.2, }\label{4.31-2}
				\end{align}
				and 
				\begin{align}\label{5.10-0}
					\mf_\lambda(\eta,\mv)&=\mf_\lambda(\eta_k,v_k)e^{-\nu(\mv)(\mathfrak{t}-\mathfrak{t}_k)}+ \int_{\mathfrak{t}_k}^{\mathfrak{t}}e^{-\nu(\mv)(\mathfrak{t}-s)}\lambda K_M\mf_\lambda\,ds\nonumber\\
					&\quad+\int_{\mathfrak{t}_k}^{\mathfrak{t}}e^{-\nu(\mv)(\mathfrak{t}-s)}\bar{S}ds,\quad \text{in \it Case 2}.
				\end{align}		
				
				In {\it Case 2}, for the first term on the RHS of \eqref{5.10-0}, $\eta_k$ is actually $0$ and we have to control
				$\dis	\int_{\R^2}e^{2\zeta|\mv|^2}|\mf_\lambda|^2(0,v_k)e^{-2\nu(\mathfrak{t}-\mathfrak{t}_k)}d\mv.$ We know from \eqref{22.5} that
				\begin{align}
					\eta_+=\f{1}{\v^2}\f{|\mv_1|^2/|\mv_2|^2}{\left(1+\sqrt{1+|\mv_1|^2/|\mv_2^2}\right)\sqrt{1+|\mv_1|^2/|\mv_2|^2}}+\f{|\mv_2|}{|\mv|}\eta.
				\end{align} 
				We observe that $\eta_+$ increases with $\dis|\mv_1/\mv_2|.$ For $|\mv_1/\mv_2|\geq m$, it holds that $\eta_+\geq \f{m^2}{4\v^2}$. Then we choose $k=C(m)k_0$ with $k_0\gg1$ such that
				\begin{align}\label{4.37}
					e^{-\nu(\mathfrak{t}-\mathfrak{t}_k)}\leq m.
				\end{align}  
				Now we take $k=C(m)k_0$ for {\it Case 2}  in \eqref{5.10-0}. If $|\mv|\geq \mathfrak{N}$, it is clear that 
				\begin{align}\label{5.11-0}
					\int_{\R^2}\mathbf{1}_{\{|\mv|\geq \mathfrak{N}\}}e^{2\zeta|\mv|^2}|\mf_\lambda|^2(0,v_k)e^{-2\nu(\mathfrak{t}-\mathfrak{t}_k)}d\mv\leq \f{1}{\mathfrak{N}^\beta}\|\mf_\lambda\|_{L^\infty L^\infty_{\beta,\zeta}}^2.
				\end{align}
				For $|\mv_1/\mv_2|\geq m, |\mv|\leq \mathfrak{N}$, one has from \eqref{4.37} that
				\begin{align*}
					\int_{\R^2}\mathbf{1}_{\{|\mv_1/\mv_2|\geq m,|\mv|\leq \mathfrak{N}\}}e^{2\zeta|\mv|^2}|\mf_\lambda|^2(0,v_k)e^{-2\nu(\mathfrak{t}-\mathfrak{t}_k)}d\mv\leq m\|\mf_\lambda\|_{L^\infty L^\infty_{\beta,\zeta}}^2.
				\end{align*}
				For $|\mv_1/\mv_2|\leq m, |\mv|\leq \mathfrak{N}$, one has $|\mv_1|\leq m\mathfrak{N}$, the integral domain is small and
				\begin{align}\label{5.11-1}
					\int_{\R^2}\mathbf{1}_{\{|\mv_1/\mv_2|\leq m, |\mv|\leq \mathfrak{N}\}}e^{2\zeta|\mv|^2}|\mf_\lambda|^2(0,v_k)e^{-2\nu(\mathfrak{t}-\mathfrak{t}_k)}d\mv\leq m\mathfrak{N}\|\mf_\lambda\|_{L^\infty L^\infty_{\beta,\zeta}}^2.
				\end{align}
				Thus we obtain from \eqref{4.31-1}--\eqref{5.11-1} that 
				\begin{align}\label{4.40}
					\|\mf_\lambda\|_{L^\infty L^2_\zeta}&\leq\left(\sqrt{m}+\sqrt{m\mathfrak{N}}+\f{1}{\mathfrak{N}^{\f{\beta}{2}}}\right)\|\mf_\lambda\|_{L^\infty L^\infty_{\beta,\zeta}}+\lambda \left\|\int_{\mathfrak{t}_k}^{\mathfrak{t}}e^{-\nu(\mv)(\mathfrak{t}-s)}|K_M\mf_\lambda|\,ds\right\|_{L^\infty L^2_\zeta}\nonumber\\
					&+\left\|\int_{\mathfrak{t}_k}^{\mathfrak{t}}e^{-\nu(\mv)(\mathfrak{t}-s)}|\bar{S}|ds\right\|_{L^\infty L^2_\zeta},
				\end{align}
				where $k=1$ for {\it Case 1.1,} $k=2$ for {\it Case 1.2} and $k=C(m)k_0$ for {\it Case 2.} 
				
				A direct calculation shows 
				\begin{align*}
					\left\|\int_{\mathfrak{t}_{k}}^{\mathfrak{t}}e^{-\nu(\mv)(\mathfrak{t}-s)}|\bar{S}|ds\right\|_{L^2_\zeta}&=\left\{\int_{\R^2}e^{2\zeta|\mv|^2}\left[\int_{\mathfrak{t}_k}^{\mathfrak{t}}e^{-\nu(\mv)(\mathfrak{t}-s)}|\bar{S}|ds\right]^2d\mv\right\}^{\f12}\nonumber\\
					&\leq \|\nu^{-1}w\bar{S}\|_{L^\infty_{\eta,\mv}}\left\{\int_{\R^2}e^{2\zeta|\mv|^2}\left[\int_{\mathfrak{t}_k}^{\mathfrak{t}}e^{-\nu(\mv)(\mathfrak{t}-s)}\nu w^{-1}ds\right]^2d\mv\right\}^{\f12}\\
					&\leq C\|\nu^{-1}w\bar{S}\|_{L^\infty_{\eta,\mv}}\left\{\int_{\R^2}(1+|\mv|^2)^{-\beta}d\mv\right\}^{\f12}\\
					&\leq C\|\nu^{-1}w\bar{S}\|_{L^\infty_{\eta,\mv}},
				\end{align*}
				which, together with \eqref{4.40}, yields that
				\begin{align}\label{5.12}
					\|\mf_\lambda\|_{L^\infty L^2_\zeta}\leq &\left(\sqrt{m}+\sqrt{m\mathfrak{N}}+\f{1}{\mathfrak{N}^{\f{\beta}{2}}}\right)\|\mf_\lambda\|_{L^\infty L^\infty_{\beta,\zeta}}+C\|\nu^{-1}w\bar{S}\|_{L^\infty_{\eta,\mv}}\nonumber\\
					&+C\lambda\left|\left\{\int_{\R^2}e^{2\zeta|\mv|^2}\left\{\int_{\mathfrak{t}_k}^{\mathfrak{t}}e^{-\nu(\mv)(\mathfrak{t}-s)}|K_M\mf_\lambda|ds\right\}^2d\mv\right\}^{\f12}\right|_{L^\infty_{\eta}}.
				\end{align}
				
				 {\it Step 3.} We now denote $\dis\mathcal{J}:=\int_{\R^2}e^{2\zeta|\mv|^2}\left\{\int_{\mathfrak{t}_k}^{\mathfrak{t}}e^{-\nu(\mv)(\mathfrak{t}-s)}|K_M\mf_\lambda|(X_{cl}(s),V_{cl}(s))ds\right\}^2d\mv$ and split the integral into 4 parts.\\
				\noindent{\it Case i.} $|\mv|\geq \mathfrak{N}$. In this case, we have
				\begin{align}\label{22.7}
					\mathcal{J}_1=&\int_{\R^2}\mathbf{1}_{|\mv|\geq \mathfrak{N}}e^{2\zeta|v|^2}\{\cdot\cdot\cdot \}^2\, d\mv\nonumber\\
					\leq&C\|K_M\mf_\lambda\|^2_{L^\infty L^\infty_{\beta,\zeta}}\int_{\R^2}\mathbf{1}_{\{|\mv|\geq \mathfrak{N}\}}\f{1}{(1+|\mv|^2)^\beta}\left\{\int_{\mathfrak{t}_k}^{\mathfrak{t}}e^{-\nu(\mv)(\mathfrak{t}-s)}ds\right\}^2d\mv\nonumber\\
					\leq&\f{C}{\mathfrak{N}^{\beta}}\|K_M\mf_\lambda\|^2_{L^\infty L^\infty_{\beta,\zeta}}\int_{\R^2}\f{1}{(1+|\mv|^2)^{\f{\beta}{2}}\nu^2(\mv)}d\mv\nonumber\\
					\leq&\f{C}{\mathfrak{N}^{\beta}}\|K_M\mf_\lambda\|^2_{L^\infty L^\infty_{\beta,\zeta}}\leq \f{C}{\mathfrak{N}^{\beta}}	\|\mf_\lambda\|_{L^\infty L^\infty_{\beta,\zeta}}^2.
				\end{align}
				\noindent{\it Case ii.} $|\mv|\leq \mathfrak{N}, |\mv_1|\geq m,|\mv_1/\mv_2|\geq m$. In this case, one has that
				\begin{align}\label{22.8}
					\mathcal{J}_2=&\int_{\R^2}\mathbf{1}_{\{|\mv|\leq \mathfrak{N}, |\mv_1|\geq m,,|\mv_1/\mv_2|\geq m\}}e^{2\zeta|v|^2}\{\cdot\cdot\cdot \}^2\, d\mv\\
					&\leq C\mathfrak{N}e^{2\zeta \mathfrak{N}^2}\sum_{l=1}^k\int_{\R^2}\int_{\mathfrak{t}_l}^{\mathfrak{t}_{l-1}}\mathbf{1}_{\{|\mv|\leq \mathfrak{N}, |\mv_1|\geq m,,|\mv_1/\mv_2|\geq m\}}\f{(K_M\mf_\lambda)^2}{\nu^2}(X_{cl}(s),V_{cl}(s))dsd\mv.\nonumber
					%	&\qquad\qquad\qquad\qquad  \times \left\{\int_{\mathfrak{t}_l}^{\mathfrak{t}_{l-1}}e^{-2\nu(\mv)(\mathfrak{t}-s)}\nu ds\right\}\nonumber\\
				\end{align} 
				Let $z=X_{cl}(s)$, i.e.
				\begin{align*}
					z=\eta_{l-1}-\int_s^{\mathfrak{t}_{l-1}}V_1d\tau,\quad s\in(\mathfrak{t}_l,\mathfrak{t}_{l-1}].
				\end{align*}
				Noting $\eta_k=0$ for all $k\geq 1$ and $\eta_0=\eta$, one has that
				\begin{align*}
					\left|\f{dz}{ds}\right|=|V_1|,
				\end{align*}
				where
				\begin{align*}
					|V_1|=\sqrt{|\mv|^2-\mv_2^2\left(e^{-2W(\eta)+2W(X(s))}\right)}, \quad V_2=\mv_2 e^{-W(\eta)+W(X(s))}.
				\end{align*}
				Since $|\mv_1|\geq m, |\mv|\leq \mathfrak{N},$ noting the smallness of $\v$, we obtain
				\begin{align*}
					|V_1|^2\geq m^2+|\mv_2|^2\left\{1-\left(\f{1-\v^2\eta}{1-\v^2X(s)}\right)^2\right\} 
					\geq m^2-3\mathfrak{N}^2\v^{2-\mathfrak{a}}\geq \f{m^2}{4}. 
				\end{align*}
				Then we first take the change of variable $ds\to dz$ and have 
				\begin{align*}
					\left|\f{ds}{dz}\right| \leq \f{2}{m}.
				\end{align*}

				A direct calculation shows
				\begin{align*}
					\left|\f{\partial V_{cl}}{\partial \mv}\right|&=\f{|\mv_1|e^{-W(\eta)+W(z)}}{\sqrt{|\mv|^2-\mv_2^2e^{-2W(\eta)+2W(z)}}}=\f{|\mv_1|e^{-W(\eta)+W(z)}}{|V_1|}\nonumber\\
					&=\f{e^{-W(\eta)+W(z)}}{\sqrt{1+|\mv_2/\mv_1|^2\left\{1-\left(\f{1-\v^2\eta}{1-\v^2z}\right)^2\right\} }}\geq \f{1}{2\sqrt{1+3\f{\v^{2-\mathfrak{a}}}{m^2}}}\geq \f{1}{3},
				\end{align*}
				which further yields that
				\begin{align*}
					\left|\f{\partial \mv}{\partial V}\right|=\f{|V_1|}{|\mv_1|e^{-W(\eta)+W(z)}}\leq  3.
				\end{align*}
				Then we take $d\mv\to dV$ and obtain from \eqref{22.8} that
				\begin{align}\label{22.9}
					\mathcal{J}_2&\leq C(m)\f{\mathfrak{N}e^{2\zeta \mathfrak{N}^2}}{m}\left\{\int_{\R^2}\int_0^d\f{(K_M\mf_\lambda)^2(z,V)}{\nu^2}dzdV\right\}\nonumber\\
					&\leq C(m)\f{\mathfrak{N}e^{2\zeta \mathfrak{N}^2}}{m}\|\nu^{-1}K_M\mf_\lambda\|_{L^2L^2}^2.
				\end{align}
				%\noindent{\it Case iii.} $|\mv|\leq N,|V_1|\geq m, |\mv_1/\mv_2|\leq m$. In this case, one has 
				%\begin{align*}
				%	|\mv_1|\leq mN,
				%\end{align*}
				%which means the integral domain is small, then 
				%\begin{align*}
				%	\mathcal{J}_3\leq CmN\|K_M\mf_\lambda\|_{L^\infty L^\infty_{\beta,\zeta}}^2.
				%\end{align*}
				\noindent{\it Case iii.}  $|\mv|\leq \mathfrak{N},\ |\mv_1|\geq m,|\mv_1/\mv_2|\leq m$. In this case, the integral domain of $\mv_1$ is small, then one gets
				\begin{align*}
					\mathcal{J}_3\leq Cm\mathfrak{N}\|K_M\mf_\lambda\|^2_{L^\infty L^\infty_{\beta,\zeta}}\leq Cm\mathfrak{N}	\|\mf_\lambda\|_{L^\infty L^\infty_{\beta,\zeta}}^2.
				\end{align*}
				\noindent{\it Case iv.} $|\mv|\leq \mathfrak{N},\ |\mv_1|\leq m$. In this case, the integral domain of $\mv_1$ is also small, then one has
				\begin{align}\label{22.10}
					\mathcal{J}_4\leq Cm\|K_M\mf_\lambda\|^2_{L^\infty L^\infty_{\beta,\zeta}}\leq Cm	\|\mf_\lambda\|_{L^\infty L^\infty_{\beta,\zeta}}^2.
				\end{align}
				
			One gets from \eqref{22.7} and  \eqref{22.9}--\eqref{22.10} that
				\begin{align}\label{22.11}
					\mathcal{J}\leq C_{m,\mathfrak{N}}\|\nu^{-1}K_M\mf_\lambda\|^2_{L^2_{\eta,\mv}}+C(m+m\mathfrak{N}+\f{1}{\mathfrak{N}^\beta})	\|\mf_\lambda\|_{L^\infty L^\infty_{\beta,\zeta}}^2.
				\end{align}
			Substituting \eqref{22.11} to \eqref{5.12} and \eqref{22.6}, one has
				\begin{align*}
					\|\mf_\lambda\|_{L^\infty L^\infty_{\beta,\zeta}}\leq& C_{m,\mathfrak{N}}\lambda \|\nu^{-1}K_M\mf_\lambda\|_{L^2_{\eta,\mv}}\nonumber\\
					&+C\left(\sqrt{m}+\sqrt{m\mathfrak{N}}+\f{1}{\mathfrak{N}^{\f{\beta}{2}}}\right)	\|\mf_\lambda\|_{L^\infty L^\infty_{\beta,\zeta}}+C\|\nu^{-1}w\bar{S}\|_{L^\infty_{\eta,\mv}},
				\end{align*}
				which, together with choosing $\mathfrak{N}\gg 1$ and $m\ll\f{1}{\mathfrak{N}}$, yields that
				\begin{align*}
					\|w\mf_\lambda\|_{L^\infty_{\eta,\mv}}+|w\mf_\lambda|_{L^\infty(\gamma_+)}&\leq C\|\nu^{-1}K_M\mf_\lambda\|_{L^2_{\eta,\mv}}+C\|\nu^{-1}w\bar{S}\|_{L^\infty_{\eta,\mv}}\nonumber\\
					&\leq C\|f_\lambda\|_{L^2_{\eta,\mv}}+C\|\nu^{-1}w\bar{S}\|_{L^\infty_{\eta,\mv}}.
				\end{align*}
				Therefor the proof of Lemma \ref{lem1.2} is finished.
			\end{proof}
		
			Combining Lemmas \ref{lem1.2-1}--\ref{lem1.2} and noting $\v\ll1,\mathfrak{a}<\f23$, we have the following result.
			\begin{Corollary}
				Assume $\|\nu^{-1}w\bar{S}\|_{L^\infty_{\eta,\mv}}<\infty, \beta>3$. For any given $\lambda\in[0,1]$, let $f_\lambda$ be the solution of \eqref{5.23-1} with $\dis \|f_\lambda\|_{L^2_{\eta,\mv}}<\infty$. Let $\mf_\lambda=\f{\sqrt{\mu_0}}{\sqrt{\mu_M}}f_{\lambda}$ with $\|w\mf_\lambda\|_{L^\infty_{\eta,\mv}}+|w\mf_\lambda|_{L^\infty(\gamma_+)}<\infty$. It holds that		
				\begin{align}\label{5.24}
					\left\|f_\lambda\right\|_{L_{\eta,\mv}^2}+ \left\|w \mf_\lambda\right\|_{L_{\eta,\mv}^{\infty}}+|w\mf_\lambda|_{L^\infty(\gamma_+)}\leq C_{ \v,d}\left\|\nu^{-1} w \bar{S}\right\|_{L_{\eta,\mv}^\infty},
				\end{align}
				where $C_{\v,d}=\max\{\v^{-1}d,d^3\}$ does not depend on $\lambda$.
			\end{Corollary} 
			\begin{Lemma}\label{lem1.3}
				Let $\beta>3$ and $ \|\nu^{-1}w\bar{S}\|_{L^\infty_{\eta,\mv}}<\infty$. Then there is a unique solution $f$ to equation \eqref{5.1} (or equivalently a unique solution $\mf$ to \eqref{5.2} with $\mf=\f{\sqrt{\mu_0}}{\sqrt{\mu_M}}f$) satisfying 
				\begin{align}\label{4.71-1}
					\left\|\f{\sqrt{\mu_0}}{\sqrt{\mu_M}}wf\right\|_{L^\infty_{\eta,\mv}}+\left|\f{\sqrt{\mu_0}}{\sqrt{\mu_M}}wf\right|_{L^\infty(\gamma_+)}<C_{\v,d}\|\nu^{-1}w\bar{S}\|_{L^\infty_{\eta,\mv}},
				\end{align}
				where the constant $C_{\v,d}$ only depends on $\v$ and $d$. Moreover, if $S$ is continuous in $[0,d]\times \R^2$, then $f$ is continuous away from grazing set $\{(\eta,\mv):\eta=0,\mv_1=0\}$.
			\end{Lemma}
			\begin{proof}
				Define Banach space:
				\begin{align*}
					\mathbf{X}:=&\left\{f=f(\eta,\mv)| \f{\sqrt{\mu_0}}{\sqrt{\mu_M}}wf\in L^\infty([0,d]\times \R^2)\cap L^\infty(\gamma_+)\right.\nonumber\\
					&\left.\text{and} \quad f(0,\mv)|_{\gamma_-}=f(0,-\mv_1,\mv_2),\, f(d,\mv)|_{\gamma_-}=0 \right\}.
				\end{align*}
				We denote a linear operator $T_\lambda$ as
				\begin{align*}
					T_\lambda f :=\mathcal{L}_0^{-1}\{\lambda K f+S\}.
				\end{align*}
				For any $f_1,f_2\in \mathbf{X}$, we have from Lemma \ref{lem1.1} that
				\begin{align*}
					\left\|\f{\sqrt{\mu_0}}{\sqrt{\mu_M}}wT_\lambda (f_1-f_2)\right\|_{L^\infty_{\eta,\mv}}=&\left\|\f{\sqrt{\mu_0}}{\sqrt{\mu_M}}\left[w\mathcal{L}_0^{-1}\{\lambda K f_1+S\}-w\mathcal{L}_0^{-1}\{\lambda K f_2+S\}\right]\right \|_{L^\infty_{\eta,\mv}}\nonumber\\
					=&\lambda\|\f{\sqrt{\mu_0}}{\sqrt{\mu_M}}w\mathcal{L}_0^{-1}\{ K (f_1-f_2)\}\|_{L^\infty_{\eta,\mv}}\nonumber\\
					\leq&C \lambda \|\f{\sqrt{\mu_0}}{\sqrt{\mu_M}}\nu^{-1}w K(f_1-f_2)\|_{L^\infty_{\eta,\mv}}\nonumber\\
					\leq &C\lambda  \|\f{\sqrt{\mu_0}}{\sqrt{\mu_M}}w (f_1-f_2)\|_{L^\infty_{\eta,\mv}}.
				\end{align*}
				Take $\lambda_*$ small such that $C\lambda_*\leq \f12$, then $T_\lambda:\mathbf{X}\to\mathbf{X}$ is a contraction mapping for $\lambda\in[0,\lambda_*]$. Thus $T_\lambda$ has a fixed point, i.e., $\exists f_\lambda\in \mathbf{X}$, such that
				\begin{align*}
					f_\lambda =T_\lambda f_\lambda =\mathcal{L}_0^{-1}\{\lambda K f_\lambda+S\},\quad \forall \lambda\in[0,\lambda_*],
				\end{align*}
				which yields immediately
				\begin{align*}
					\mathcal{L}_\lambda f_\lambda= \mv_1 \partial_\eta f_\lambda+G(\eta)\left(\mv_2^2 \frac{\partial f_\lambda}{\partial \mv_1}-\mv_1 \mv_2 \frac{\partial f_\lambda}{\partial \mv_2}\right)+\nu f_\lambda-\lambda K f_\lambda=S.
				\end{align*}
				Hence, for any $\lambda\in[0,\lambda_*]$, we have solved \eqref{5.23-1} with $f_\lambda = \mathcal{L}_\lambda^{-1}S \in\mathbf{X}.$  Moreover $f_\lambda$ satisfies \eqref{5.24}.

				Noting the constant $C_{\v,d}$ in \eqref{5.24} does not depend on $\lambda$. For any $f_1,f_2\in \mathbf{X}$, we have from \eqref{5.24} that
				\begin{align*}
					\left\|\f{\sqrt{\mu_0}}{\sqrt{\mu_M}}wT_{\lambda_*+\lambda} (f_1-f_2)\right\|_{L^\infty_{\eta,\mv}}=&\left\|\f{\sqrt{\mu_0}}{\sqrt{\mu_M}}\left[w\mathcal{L}_{\lambda_*}^{-1}\{\lambda Kf_1+S\}-w\mathcal{L}_{\lambda_*}^{-1}\{\lambda K f_2+S\}\right]\right\|_{L^\infty_{\eta,\mv}}\nonumber\\
					=&\lambda\|\f{\sqrt{\mu_0}}{\sqrt{\mu_M}}w\mathcal{L}_{\lambda_*}^{-1}\{ K (f_1-f_2)\}\|_{L^\infty_{\eta,\mv}}\nonumber\\
					\leq&C_{\v,d} \lambda \left\|\f{\sqrt{\mu_0}}{\sqrt{\mu_M}}\nu^{-1}w K (f_1-f_2)\right\|_{L^\infty_{\eta,\mv}}\nonumber\\
					\leq &C_{\v,d}\lambda  \left\|\f{\sqrt{\mu_0}}{\sqrt{\mu_M}}w (f_1-f_2)\right\|_{L^\infty_{\eta,\mv}}.
				\end{align*} 
				By similar arguments, we prove there exists a $\lambda_1>0$, such that $T_{\lambda_*+\lambda}$ is a contraction mapping for $\lambda\in[0,\lambda_1].$  Then we obtain the existence of
				operator $\mathcal{L}_{\lambda_*+\lambda}^{-1}$.  Step by step, we can finally obtain the
				existence of operator $\mathcal{L}_\lambda^{-1},\, \forall \lambda\in[0,1]$. The estimate \eqref{4.71-1} follows directly from \eqref{5.24}. Therefore the proof of Lemma \ref{lem1.3} is completed.
			\end{proof}
			\subsection{Exponential spatial decay}
			We have obtained the solution of \eqref{5.1}(or equivalently \eqref{5.2}) from Lemma \ref{lem1.3} and get some $L^2,L^\infty$ estimates relying on $d$. However, to construct the Hilbert expansion, spatial decay estimate is crucial. 
			
			Multiplying \eqref{5.2} by $e^{\sigma\eta}$, we have
			\begin{align}\label{4.72}
				&\mv_1\partial_\eta \left(e^{\sigma\eta}\mf\right)+G(\eta)\left(\mv_2^2\f{\partial}{\partial \mv_1}-\mv_1 \mv_2\f{\partial}{\partial \mv_2}\right)\left(e^{\sigma\eta}\mf\right)\nonumber\\
				&+\FL_M \left(e^{\sigma\eta}\mf\right)=e^{\sigma\eta}\bar{S}+\sigma \mv_1 \left(e^{\sigma\eta}\mf\right).
			\end{align}
			Noting $\nu\sim 1+|\mv|$ and $\sigma\ll1$, applying Lemma \ref{lem1.2} to \eqref{4.72}, one obtains that
			\begin{align}\label{5.3}
				\|e^{\sigma\eta}w\mathbf{f}\|_{L^\infty_{\eta,\mv}}+|e^{\sigma\eta}w\mf|_{L^\infty(\gamma_+)}\leq C\left\{\|e^{\sigma\eta}f\|_{L^2_{\eta,\mv}}+\|e^{\sigma\eta}\nu^{-1}w\bar{S}\|_{L^\infty_{\eta,\mv}}\right\}.
			\end{align}
			 Thus we need to establish the uniform $L^2$ estimate $\|e^{\sigma\eta} f\|_{L^2_{\eta,\mv}}$.
			
			Denote 
			\begin{align}\label{4.73-1}
				\FP_0 f=a\sqrt{\mu_0}+b_1\mv_1 \sqrt{\mu_0}+b_2(\mv_2-u_\tau^0)\sqrt{\mu_0}+c(|\mv-\bar{\mathfrak{u}}^0|^2-2T^0)\sqrt{\mu_0}.
			\end{align}
			Multiplying \eqref{5.1} by $\sqrt{\mu_0}$ and integrating on $\R^2$, one has from \eqref{5.1-1} that
			\begin{align*}
				\partial_\eta b_1+G(\eta)b_1=0,
			\end{align*}
			which, with the specular property $b_1(0)=0$, yields that
			\begin{align}\label{5.5}
				b_1\equiv0.
			\end{align}
			Similarly, multiplying $\mv_2\sqrt{\mu_0}$ and $|\mv|^2\sqrt{\mu_0}$ respectively with \eqref{5.1}, using \eqref{5.1-1} and \eqref{5.5}, one gets that
			\begin{align*}
				&\partial_\eta \left\{\int_{\R^2}\mv_1 \mv_2 \sqrt{\mu_0}f \, d\mv\right\}+2G(\eta)\left\{\int_{\R^2}\mv_1 \mv_2 \sqrt{\mu_0}f \, d\mv\right\}=0,\\
				&\partial_\eta \left\{\int_{\R^2}\mv_1 |\mv|^2 \sqrt{\mu_0}f \, d\mv\right\}+G(\eta)\left\{\int_{\R^2}\mv_1 |\mv|^2 \sqrt{\mu_0}f \, d\mv\right\}=0,
			\end{align*}
			which yields that
			\begin{align}\label{5.5-1}
				\int_{\R^2}\mv_1 \mv_2 \sqrt{\mu_0}f\,d\mv\equiv \int_{\R^2}\mv_1 |\mv|^2\sqrt{\mu_0}f\,d\mv\equiv0.
			\end{align}
			Using \eqref{5.5}--\eqref{5.5-1} and noting \eqref{5.1-1}, one deduces that
			\begin{align}\label{5.6}
				\int_{\R^2}\mv_1 f^2\,d\mv=\int_{\R^2}\mv_1|(\FI-\FP_0)f|^2\,d\mv,\quad \int_{\R^2}f\cdot S\, d\mv=\int_{\R^2}(\FI-\FP_0)f\cdot S\, d\mv.
			\end{align}
			\begin{Lemma}\label{lem1.4}
				Let $\beta > 3, \|e^{\sigma_2\eta}S\|_{L^2_{\eta,\mv}}<\infty$ for some $\sigma_2>0$ and \eqref{5.1-1} hold. Let $f$ be the solution of \eqref{5.1} constructed in Lemma \ref{lem1.3}, it holds that
				\begin{align}\label{4.70}
					\|e^{\sigma\eta}f\|_{L^2_{\eta,\mv}}+e^{\sigma d}|f(d)|_{L^2(\gamma_+)}\leq \frac{C}{\sqrt{\sigma}}\v^{1-\f{\mathfrak{a}}{2}}\|e^{\sigma\eta}w\mf\|_{L^\infty_{\eta,\mv}}+\f{C}{\sqrt{\sigma(\sigma_2-\sigma)}}\|e^{\sigma_2\eta}S\|_{L^2_{\eta,\mv}},
				\end{align}
				where $\sigma\in(0,\sigma_2)$ and the constant $C> 0$ is independent of $d$ and $\sigma$. 
			\end{Lemma}
			\begin{proof}
				We divide the proof into several steps.\\
				{\noindent \it Step 1. Decay of microscopic part.} Multiplying \eqref{5.1} by $f\cdot e^{-W(\eta)}$, noting \eqref{5.6}, we obtain
				\begin{align}\label{4.71}
					&\f{d}{d\eta}\left\{\f12\int_{\R^2}\mv_1 f^2\,d\mv\cdot e^{-W(\eta)}\right\}-\f{u^0_\tau}{2T^0}G(\eta)e^{-W(\eta)}\int_{\R^2}\mv_1 \mv_2 f^2\,d\mv+\langle f,\FL_0f\rangle e^{-W(\eta)} \nonumber\\
					&=\langle S,(\FI-\FP_0)f\rangle e^{-W(\eta)}.
				\end{align}
				Multiplying \eqref{4.71} by $e^{2\sigma\eta}$, one has
				\begin{align*}
					&\f{d}{d\eta}\left\{\f12\int_{\R^2}\mv_1 f^2\,d\mv\cdot e^{-W(\eta)}e^{2\sigma\eta}\right\}-\f{u^0_\tau}{2T^0}G(\eta)e^{-W(\eta)}e^{2\sigma\eta}\int_{\R^2}\mv_1 \mv_2 f^2\,d\mv\nonumber\\
					&+\langle f,\FL_0f\rangle e^{-W(\eta)}e^{2\sigma\eta} \nonumber\\
					&=\langle S\cdot (\FI-\FP_0)f\rangle \cdot e^{-W(\eta)}e^{2\sigma\eta}+\sigma\int_{\R^2}\mv_1 f^2\,d\mv \cdot e^{-W(\eta)}e^{2\sigma\eta}.
				\end{align*}
				Then, integrating the above equation over $[0,d]$ and using \eqref{5.6}, we obtain 
				\begin{align*}
					&c_0\int_0^d \|(\FI-\FP_0)\|_\nu^2\cdot e^{-W(\eta)}e^{2\sigma\eta}d\eta +\f12 e^{2\sigma d}e^{-W(d)}|f(d)|_{L^2(\gamma_+)}^2\nonumber\\
					&\leq \f{u_\tau^0}{2T^0}\int_0^dG(\eta)e^{-W(\eta)}e^{2\sigma\eta}\int_{\R^2}\mv_1 \mv_2 f^2 d\mv d\eta +\int_0^d\langle S,(\FI-\FP_0)f\rangle e^{-W(\eta)}e^{2\sigma\eta}d\eta\nonumber\\
					&\quad +\sigma\int_0^d \int_{\R^2}\mv_1|(\FI-\FP_0)f|^2\cdot e^{-W(\eta)}e^{2\sigma\eta}d\mv d\eta.
				\end{align*}
				which, together with $\sigma\ll1$ and $e^{-W(\eta)}\sim 1$, yields
				\begin{align}\label{5.7}
					\int_0^d\|(\FI-\FP_0)f\|_\nu^2\cdot e^{2\sigma\eta}\,d\eta + e^{2\sigma d}|f(d)|_{L^2(\gamma_+)}^2\leq 	C\v^{2-\mathfrak{a}} \|e^{\sigma\eta}w\mathbf{f}\|_{L^\infty_{\eta,\mv}}^2+C\|e^{\sigma\eta}S\|_{L^2_{\eta,\mv}}^2.
				\end{align}
				
				{\noindent \it Step 2. Decay of macroscopic part.}\\
				%Formally, a direct calculation shows that
				%	\begin{align*}
					%		&\mv_1\partial_\eta \FP_0f+G(\eta)\left(\mv_2^2\f{\partial \FP_0f}{\partial \mv_1}-\mv_1 \mv_2\f{\partial \FP_0f}{\partial \mv_2}\right)-\f{u^0_\tau}{2T^0}G(\eta)\mv_1 \mv_2 \FP_0f \nonumber\\
					%		&=(\partial_\eta a+2T^0\partial_\eta c)\cdot \mv_1 \sqrt{\mu_0}+\partial_\eta b_2\cdot \mv_1(\mv_2-u_\tau^0)\sqrt{\mu_0}+\partial_\eta c\cdot \mv_1(|\mv-\bar{\mathfrak{u}}^0|^2-4T^0)\sqrt{\mu_0}\nonumber\\
					%		&\quad -G(\eta)a\cdot\left\{\f{|u_\tau^0|^2}{T^0}\mv_1 \sqrt{\mu_0}+\f{u_\tau^0}{T^0}\mv_1(\mv_2-u_\tau^0)\sqrt{\mu_0}\right\}\nonumber\\
					%		&\quad -G(\eta )b_2\cdot\left\{u_\tau^0 \mv_1\sqrt{\mu_0}+\left(1+\f{|u_\tau^0|^2}{T^0}\right)\mv_1(\mv_2-u^0_\tau)\sqrt{\mu_0}+\f{u_\tau^0}{T^0}\mv_1(\mv_2-u_\tau^0)^2\sqrt{\mu_0}\right\}\nonumber\\
					%		&\quad -G(\eta)c\cdot\left\{-2u_\tau^0 \mv_1(\mv_2-u_\tau^0)\sqrt{\mu_0}+\f{u_\tau^0}{T^0}\mv_1(\mv_2-u_\tau^0)(|\mv-\bar{\mathfrak{u}}^0|^2-2T^0)\sqrt{\mu_0}\right.\nonumber\\
					%		&\qquad\qquad\qquad\left.+\f{|u_\tau^0|^2}{T^0}\mv_1(|\mv-\bar{\mathfrak{u}}^0|^2-4T^0)\sqrt{\mu_0}\right\}.
					%	\end{align*}
				{\noindent \it Step 2.1.} Taking $\varphi_b(\mv)=\FL_0^{-1}\{\mv_1(\mv_2-u_\tau^0)\sqrt{\mu_0}\}$, one has from \eqref{5.5-1} that 
				\begin{align}\label{4.73}
					\langle \mv_1 f,\varphi_b\rangle =T^0\kappa_1(T^0)b_2+\langle \mv_1 \varphi_b,(\FI-\FP_0)f\rangle,
				\end{align}	
				and 
				\begin{align}\label{4.74}
					\langle \FL_0f, \varphi_b\rangle=\langle \FL_0f, \FL_0^{-1}\{\mv_1(\mv_2-u_\tau^0)\sqrt{\mu_0}\}\rangle=\langle f, \mv_1(\mv_2-u_\tau^0)\sqrt{\mu_0}\rangle=0,
				\end{align}
				where $\kappa_1$ is the one defined in \eqref{6.1-1}. Multiplying \eqref{5.1} by $\varphi_b$ and integrating over $\R^2$, one has from \eqref{4.73}--\eqref{4.74} that
				\begin{align}\label{5.9}
					&\partial_\eta [T^0\kappa_1(T^0)b_2+\langle \mv_1 \varphi_b,(\FI-\FP_0)f\rangle]\nonumber\\
					&=-G(\eta)\left\langle \mv_2^2\f{\partial\varphi_b}{\partial v_1}-\mv_1\f{\partial (\mv_2\varphi_b)}{\partial \mv_1},f\right\rangle+\f{u_\tau^0}{2T^0}G(\eta)\langle \mv_1\mv_2f,\varphi_b\rangle-\langle \FL_0f, \varphi_b\rangle+\langle S,\varphi_b\rangle\nonumber\\
					&=-G(\eta)\left\langle \mv_2^2\f{\partial\varphi_b}{\partial v_1}-\mv_1\f{\partial (\mv_2\varphi_b)}{\partial \mv_1},f\right\rangle+\f{u_\tau^0}{2T^0}G(\eta)\langle\mv_1\mv_2f,\varphi_b\rangle+\langle S,\varphi_b\rangle.
				\end{align}
				%	\begin{align*}
					%		&\rho^0\left\{T^0\partial_\eta a +2(T^0)^2\partial_\eta c-G(\eta)|u_\tau^0|^2a-2G(\eta)T^0u_\tau^0b_2\right\}\nonumber\\
					%		&\leq -\partial_\eta \left\{\int_{\R^2}\mv_1^2\sqrt{\mu_0} (\FI-\FP_0)f\,d\mv \right\}-G(\eta)\|(\FI-\FP_0)f\|_\nu+\int_{\R^2}S\cdot \mv_1\sqrt{\mu_0}\,d\mv ,
					%	\end{align*}
				%	which is
				%	\begin{align}\label{5.8}
					%		&T^0\partial_\eta\left\{a\cdot e^{\f{|u_\tau^0|^2}{T^0}W(\eta)} \right\} +2(T^0)^2\partial_\eta \left\{c\cdot e^{\f{|u_\tau^0|^2}{T^0}W(\eta)}\right\}-b_2\cdot\left\{2G(\eta)T^0u_\tau^0\cdot e^{\f{|u_\tau^0|^2}{T^0}W(\eta)}\right\}\nonumber\\
					%		&+c\cdot\left\{2G(\eta)T^0|u_\tau^0|^2e^{\f{|u_\tau^0|^2}{T^0}W(\eta)}\right\}\nonumber\\
					%		&\leq -\f{1}{\rho^0}\partial_\eta \left\{\int_{\R^2}\mv_1^2\sqrt{\mu_0} (\FI-\FP_0)f\,d\mv \cdot e^{\f{|u_\tau^0|^2}{T^0}W(\eta)}\right\}-CG(\eta)\|(\FI-\FP_0)f\|_\nu\cdot e^{\f{|u_\tau^0|^2}{T^0}W(\eta)}\nonumber\\
					%		&\quad+C\int_{\R^2}S\cdot \mv_1\sqrt{\mu_0}\,d\mv \cdot e^{\f{|u_\tau^0|^2}{T^0}W(\eta)}.
					%	\end{align}
				%	Multiplying $\FL_0^{-1}\{\mv_1 (\mv_2-u_\tau^0)\sqrt{\mu_0}\}$ to \eqref{5.1} and integrating on $\R^2$, one obtains
				%	\begin{align*}
					%		&\kappa_1\partial_\eta b_2-G(\eta)\left\{\kappa_1\f{ u_\tau^0}{T^0} a+\kappa_1(1+\f{|u_\tau^0|^2}{T^0})b_2-\kappa_1\epsilon_1u_\tau^0c\right\}\nonumber\\
					%		&\leq -\partial_\eta
					%		\left\{\int_{\R^2} \left[\mv_1\FL_0^{-1}\{\mv_1 (\mv_2-u_\tau^0)\sqrt{\mu_0}\}\right](\FI-\FP_0)f\,d\mv \right\}\nonumber\\
					%		&\quad -CG(\eta)\|(\FI-\FP_0)f\|_{\nu}+\int_{\R^2}S\cdot \FL_0^{-1}\{\mv_1 (\mv_2-u_\tau^0)\sqrt{\mu_0}\}d\mv ,
					%	\end{align*}
				
				A direct calculation shows that 
				\begin{align}\label{4.76}
					\begin{split}
						\int_\eta^d-G(z)\left\langle \mv_2^2\f{\partial\varphi_b}{\partial v_1}-\mv_1\f{\partial (\mv_2\varphi_b)}{\partial \mv_1},f(z)\right\rangle\, dz\leq \int_\eta^d|G(z)|\cdot\|f(z)\|_{L^2_\mv}\, dz,\\
						\int_\eta^d \f{u_\tau^0}{2T^0}G(z)\langle\mv_1\mv_2f(z),\varphi_b\rangle\, dz\leq\int_\eta^d|G(z)|\cdot\|f(z)\|_{L^2_\mv}\, dz.
					\end{split}
				\end{align}
				Integrating \eqref{5.9} on $[\eta,d]$, noting $f(d,\mv)|_{\gamma_-}=0$, one obtains from \eqref{4.76} that
				\begin{align*}
					|b_2(\eta)|\leq &C|f(d)|_{L^2(\gamma_+)}+C\|(\FI-\FP_0)f(\eta)\|_\nu\\
					&+C\int_\eta^d|G(z)|\cdot\|f(z)\|_{L^2_\mv}\, dz+C\int_\eta^d\|S\|_{L^2_\mv}\, dz,
				\end{align*}
				which yields that
				\begin{align}\label{5.10}
					\int_0^d|b_2(\eta)|^2e^{2\sigma \eta}\, d\eta\leq &\f{C}{\sigma}|f(d)|_{L^2(\gamma_+)}^2\cdot e^{2\sigma d}+C\int_0^d\|(\FI-\FP_0)f(\eta)\|_\nu^2\cdot e^{2\sigma \eta}\,d\eta \nonumber\\
					&+\v^{4-2\mathfrak{a}}\int_0^d\|f(\eta)\|_{L^2_\mv}^2\cdot e^{2\sigma \eta}\, d\eta+\frac{C}{\sigma_2(\sigma_2-\sigma)}\|e^{\sigma_2\eta}S\|_{L^2_{\eta,\mv}}^2,
				\end{align}
				where we have used 
				\begin{align*}
					&\int_0^d \left\{\int_\eta^d|G(z)|\cdot\|f(z)\|_{L^2_\mv}\, dz\right\}^2\cdot e^{2\sigma \eta}\, d\eta \\
					&\leq \int_0^de^{2\sigma\eta}\int_\eta^dG^2(z)e^{-2\sigma z}dzd\eta\times \int_0^d\|f(z)\|_{L^2_\mv}^2\cdot e^{2\sigma z}\, dz\\
					&\leq \v^{4-2\mathfrak{a}}\int_0^d\|f(\eta)\|_{L^2_\mv}^2\cdot e^{2\sigma \eta}\, d\eta,
				\end{align*}
				and
				\begin{align*}
					&\int_0^d\left\{\int_\eta^d\|S(z)\|_{L^2_\mv}\, dz\right\}^2\cdot e^{2\sigma \eta}d\eta\leq \left\{\int_0^de^{2\sigma\eta}\int_\eta^d e^{-2\sigma_2z}dz\right\}\cdot \|e^{\sigma_2\eta}S\|_{L^2_{\eta,\mv}}^2\\
					&\leq  \f{1}{2\sigma_2}\int_0^de^{-2(\sigma_2-\sigma)\eta}d\eta\cdot \|e^{\sigma_2\eta}S\|_{L^2_{\eta,\mv}}^2\leq \f{1}{4\sigma_2(\sigma_2-\sigma)}\cdot  \|e^{\sigma_2\eta}S\|_{L^2_{\eta,\mv}}^2.
				\end{align*}
				{\noindent \it Step 2.2.} Taking $\varphi_c(\mv)=\FL_0^{-1}\{\mv_1(|\mv-\bar{\mathfrak{u}}^0|^2-4T^0)\sqrt{\mu_0}\}$, one has from \eqref{5.5-1} that
				\begin{align*}
					\langle \mv_1 f,\varphi_c\rangle =8(T^0)^2\kappa_2(T^0)c+\langle \mv_1 \varphi_c,(\FI-\FP_0)f\rangle,
				\end{align*}	
				and 
				\begin{align*}
					\langle \FL_0f, \varphi_c\rangle=&\langle \FL_0f, \FL_0^{-1}\{\mv_1(|\mv-\bar{\mathfrak{u}}^0|^2-4T^0)\sqrt{\mu_0}\}\rangle\\
					=&\langle f, \mv_1(|\mv-\bar{\mathfrak{u}}^0|^2-4T^0)\sqrt{\mu_0}\rangle=0,
				\end{align*}
				where $\kappa_2$ is the one defined in \eqref{6.1-1}. Similar to \eqref{5.10}, one gets that
				\begin{align}\label{5.11}
					\int_0^d|c(\eta)|^2e^{2\sigma \eta}\, d\eta\leq &\f{C}{\sigma}|f(d)|_{L^2(\gamma_+)}^2\cdot e^{2\sigma d}+C\int_0^d\|(\FI-\FP_0)f(\eta)\|_\nu^2\cdot e^{2\sigma \eta}\,d\eta \nonumber\\
					&+\v^{4-2\mathfrak{a}}\int_0^d\|f(\eta)\|_{L^2_\mv}^2\cdot e^{2\sigma \eta}\, d\eta+\frac{C}{\sigma_2(\sigma_2-\sigma)}\|e^{\sigma_2\eta}S\|_{L^2_{\eta,\mv}}^2.
				\end{align}
				{\noindent \it Step 2.3.} %We estimate $\dis\int_0^d e^{2\sigma\eta}|a(\eta)|^2\, d\eta$. 
				Taking  $\varphi_a(\mv)=\mv_1\sqrt{\mu_0}$, one has
				\begin{align*}
					\langle \mv_1f,\varphi_a\rangle =&\langle \mv_1\FP_0f,\varphi_a\rangle +\langle \mv_1(\FI-\FP_0)f,\varphi_a\rangle \\
					=&\rho^0T^0a+2\rho^0(T^0)^2c+\langle \mv_1(\FI-\FP_0)f,\varphi_a\rangle.
				\end{align*} 
				Multiplying \eqref{5.1} by $\varphi_a$, we have that 
				\begin{align*}
					|a(\eta)|\leq& 2|c(\eta)|+C|f(d)|_{L^2(\gamma_+)}+C\|(\FI-\FP_0)f(\eta)\|_\nu\\
					&+C\int_\eta^d|G(z)|\cdot\|f(z)\|_{L^2_\mv}\, dz+C\int_\eta^d\|S\|_{L^2_\mv}\, dz,
				\end{align*}
				which, together with \eqref{5.11}, yields that
				\begin{align}\label{5.12-1}
					\int_0^d|a(\eta)|^2e^{2\sigma \eta}\, d\eta\leq &\f{C}{\sigma}C|f(d)|_{L^2(\gamma_+)}^2\cdot e^{2\sigma d}+C\int_0^d\|(\FI-\FP_0)f(\eta)\|_\nu^2\cdot e^{2\sigma \eta}\,d\eta \nonumber\\
					&+\v^{4-2\mathfrak{a}}\int_0^d\|f(\eta)\|_{L^2_\mv}^2\cdot e^{2\sigma \eta}\, d\eta+\frac{C}{\sigma_2(\sigma_2-\sigma)}\|e^{\sigma_2\eta}S\|_{L^2_{\eta,\mv}}^2.
				\end{align}
				
				{\noindent \it Step 3. Conclusion.} It follows from \eqref{5.10}, \eqref{5.11} and \eqref{5.12-1} that 
				\begin{align}\label{5.13}
					\int_0^d |(a,b_2,c)|^2\cdot e^{2\sigma\eta}d\eta\leq &\f{C}{\sigma} e^{2\sigma d}|f(d)|_{L^2(\gamma_+)}^2+C\int_0^d\|(\FI-\FP_0)f(\eta)\|_\nu^2\cdot e^{2\sigma \eta}\,d\eta \nonumber\\
					&+\frac{C}{\sigma_2(\sigma_2-\sigma)}\|e^{\sigma_2\eta}S\|_{L^2_{\eta,\mv}}^2,
				\end{align}
				where $0<\sigma<\sigma_2.$ From \eqref{5.5}, \eqref{5.7} and \eqref{5.13}, we conclude \eqref{4.70}. Therefore the proof of Lemma \ref{lem1.4} is completed.
			\end{proof}
			\begin{Lemma}\label{lem1.5}
				Let $\beta > 3, \|e^{\sigma_0\eta}\nu^{-1}w\bar{S}\|_{L^\infty_{\eta,\mv}}<\infty$ for some $\sigma_0>0$ and \eqref{5.1-1} hold. Let $\mf$ be the solution of \eqref{5.2} established in Lemma \ref{lem1.3}, it holds that
				\begin{align*}
					\|e^{\sigma\eta}w\mf\|_{L^\infty_{\eta,\mv}}+|e^{\sigma\eta}w\mf|_{L^\infty(\gamma_+)}\leq \f{C}{\sigma(\sigma_0-\sigma)}\|e^{\sigma_0\eta}\nu^{-1}w\bar{S}\|_{L^\infty_{\eta,\mv}},
				\end{align*}
				where $\sigma\in(0,\sigma_0)$ and the constant $C> 0$  is independent of $d$ and $\sigma$.
			\end{Lemma}
			\begin{proof}
				Combining \eqref{5.3} and \eqref{4.70}, we have
				\begin{align*}
					\|e^{\sigma\eta}w\mf\|_{L^\infty_{\eta,\mv}}+|e^{\sigma\eta}w\mf|_{L^\infty(\gamma_+)}+\|e^{\sigma\eta}f\|_{L^2_{\eta,\mv}}\leq \f{C}{\sigma(\sigma_0-\sigma)}\|e^{\sigma_0\eta}\nu^{-1}w\bar{S}\|_{L^\infty_{\eta,\mv}},
				\end{align*}
				where we have chosen $\sigma_2=(\sigma+\sigma_0)/2$ and used
				\begin{align*}
					\|e^{\sigma_2\eta}S\|_{L^2_{\eta,\mv}}=\left\{\int_0^d\int_{\R^2} e^{2\sigma_2\eta}S^2\, d\mv d\eta\right\}^{\f12}\leq \f{C}{\sqrt{\sigma_0-\sigma_2}}\|e^{\sigma_0\eta}\nu^{-1}w\bar{S}\|_{L^\infty_{\eta,\mv}}.
				\end{align*}
				Therefore the proof of Lemma \ref{lem1.5} is completed.
			\end{proof}
		{\noindent \bf Proof of Theorem \ref{thm2.2}:} Applying Lemma \ref{lem1.5} to \eqref{5.1} which is in fact \eqref{K-1}, noting \eqref{K-1.1}, we now have proved Theorem \ref{thm2.2}. 	The source term conditions (solvable conditions) \eqref{5.1-1} for \eqref{K-1} is actually \eqref{2.19} and $\mathfrak{S}\in\mathcal{N}_0^\perp$. $\hfill\Box$

%%%%%%%%%%%%%%%%%%%%%%%%%%%%%%%%%%%%%%%%%%%%%%%%%%%%%%%%%%%%%%%%%%%%%%%%%%%%%%%%%%%%%%%%%%%%%%%%%%%%%%%%%%%%%%%%%%%%%%%%%%%%%		
			\section{Construction of Hilbert expansion and remainder estimate}\label{sec5}
			\subsection{Construction of $F_i, \bar{\mathfrak{F}}_i, \hat{\mathscr{F}}_i, i=1,2,\cdots, N$}\label{sec5.1}
			%Due to the cut-off domain in Knudsen layer, the equation \eqref{re} is only valid for $\eta\in[0,\v^{-\f23}]$, which equals to $|x|\in[1-\v^{\f43},1]$. When $|x|\in[0,1-\v^{\f43}]$, terms about $\hat{F}_k$ in \eqref{re} formally becomes 0. 
			Combining Section \ref{sec2} to Section \ref{sec4}, each layer can be constructed via a process similar as  in \cite{GHW-2021-ARMA}. Define the velocity weight functions
			\begin{align}\label{6.07}
				\varpi_{\kappa_i}(v)=\tilde{w}_{\kappa_i}(v)\mu^{\mathbf{-a}},\quad\bar{\varpi}_{\bar{\kappa}_i}(\bar{v})=\tilde{w}_{\bar{\kappa}_i}(v)\mu_0^{-\mathbf{a}} \quad\mbox{and}\quad \hat{\varpi}_{\hat{\kappa}_i}(\mv)=\tilde{w}_{\hat{\kappa}_i}\mu_M^{-\mathbf{a}} ,
			\end{align}
			for constants $\kappa_i,\bar{\kappa}_i,\hat{\kappa}_i\geq0,1\leq i\leq N$ and $0\leq\mathbf{a}<\frac12$. Denote $\bar{x}=(y,\phi),\hat{x}=(\eta,\phi)$ and  $\nabla_{\bar{x}}:=(\partial_y,\partial_\phi)$ in viscous boundary coordinate.
			\begin{Proposition}\label{prop}
				Let $0\leq\mathbf{a}\textless\frac12$ in \eqref{6.07}. Let $s_0,s_i,\bar{s}_i,\hat{s}_i\in\mathbb{N}_+,\kappa_i,\bar{\kappa}_i,\ \hat{\kappa}_i\in\mathbb{R}_+$ for $1\leq i\leq N$; and define $l_j^i:=\bar{l}_i+2(\bar{s}_i-j)$ for $1\leq i\leq N,0\leq j\leq\bar{s}_i$. For these parameters, we assume the restrictions
				\begin{align}\label{eq6.2}
					&s_0\geq s_1+\mathfrak{b}+6,\quad s_1=\bar{s}_1=\hat{s}_1\gg1;\nonumber\\
					&s_1\textgreater s_i\textgreater\bar{s}_i\textgreater\hat{s}_i\geq s_{i+1}\textgreater\bar{s}_{i+1}\textgreater\hat{s}_{i+1}\geq...\gg1, \ \text{for}\quad i=2,...,N-2;\\
					&s_{i+1}\leq \min\{\hat{s}_i,\frac12\bar{s}_i-j\},\bar{s}_{i+1}\leq s_{i+1}-8-\mathfrak{b},\hat{s}_{i+1}\leq\frac12\bar{s}_{i+1}-2-\mathfrak{b},\ \text{for}\quad i=1,...,N-1;\nonumber\\
					&l_j^N\gg2\mathfrak{b}\quad and\quad l_j^i\geq2l_j^{i+1}+18+2\mathfrak{b},\ \text{for}\quad 1\leq i\leq N-1;\nonumber\\
					&\kappa_i\gg\bar{\kappa}_i\gg\hat{\kappa}_i\gg\kappa_{i+1}\gg\bar{\kappa}_{i+1}\gg\hat{\kappa}_{i+1}\gg1,\nonumber
				\end{align}
				hold. Let the initial data $(\rho_i,u_i,\theta_i)(0)$ of IBVP \eqref{2.1}, and initial data $(\bar{u}_i\cdot\vec{\tau},\bar{\theta}_i)(0)$ of IBVP \eqref{bu-0} satisfy 
				\begin{align}
					\sum\limits_{i=0}^N\Big\{\sum\limits_{l+|\alpha|\leq s_i}\|\partial_t^l\nabla_x^\alpha(\rho_i,u_i,\theta_i)(0)\|_{L_x^2}+\sum\limits_{j=0}^{\bar{s}_i}\sum\limits_{2l+|\alpha|=j}\||\partial_t^l\nabla_{\bar{x}}^\alpha(\bar{u}_i\cdot\vec{\tau},\bar{\theta}_i)(0)\|_{L_{l_j^i}^2}^2\Big\}\textless\infty.
				\end{align}
				We also assume the compatibility conditions for initial data $(\rho_i,u_i,\theta_i)(0)$ and $(\bar{u}_i\cdot\vec{\tau},\bar{\theta}_i)(0)$ are satisfied at the boundary. Then there exist solutions $\dis F_i=\sqrt{\mu}f_i,\bar{\mathfrak{F}}_i=\sqrt{\mu_0}\bar{f}_i,\hat{F}_i=\sqrt{\mu_M}\hat{\mathbf{f}}_i$(or equivalently $\sqrt{\mu_0}\hat{f}_i)$ to \eqref{ie}, \eqref{1.16}, \eqref{KL} over the time interval $t\in[0,\tau^\d]$, respectively.  Moreover, we have the following estimates
				\begin{align}
					&\sup_{t\in[0,\tau]}\sum\limits_{i=1}^{N}\Big(\sum\limits_{l+|\alpha|\leq s_i}\|\varpi_{\kappa_i}\partial_t^l\nabla_x^\alpha f_i(t)\|_{L^2_xL^\infty_v}+\sum\limits_{j=0}^{\bar{s}_i}\sum\limits_{j=2l+|\alpha|}\|\bar{\varpi}_{\bar{\kappa}_i}\partial_t^l\nabla_{\bar{x}}^\alpha \bar{f}_i(t)\|_{L^2_{l_j^i}L^\infty_{\bar{v}}}\nonumber\\
					&+\sum\limits_{l+|\alpha|\leq\hat{s}_i}\|e^{q_i\eta}\hat{\varpi}_{\hat{\kappa}_i}\partial_t^l\partial_\phi^\alpha \hat{\mathbf{f}}_i(t)\|_{L_{\hat{x},v}^\infty\cap L_{\phi}^2L_{\eta,\mv}^\infty}\Big)\nonumber\\
					&\leq C\Big(\tau,\|(\phi_0,\Phi_0,\vartheta_0)\|_{H^{s_0}}+	\sum\limits_{i=0}^N\big\{\sum\limits_{l+|\alpha|\leq s_i}\|\partial_t^l\nabla_x^\alpha(\rho_i,u_i,\theta_i)(0)\|_{L_x^2}\nonumber\\
					&\qquad\qquad\qquad\qquad\qquad\qquad\qquad+\sum\limits_{j=0}^{\bar{s}_i}\sum\limits_{2l+|\alpha|=j}\||\partial_t^l\nabla_{\bar{x}}^\alpha(\bar{u}_i,\bar{\theta}_i)(0)\|_{L_{l_j^i}^2}^2\big\}\Big).
				\end{align}
				where the positive constants $q_i\textgreater0(i=1,...,N)$ satisfying $q_{i+1}=\frac12q_i$ and $q_1=1$.
			\end{Proposition}
			\begin{proof}
				Firstly, we aim to obtain the tangential and time derivatives estimates for $\hat{\mf}_k$ (or equivalently $\hat{f}_k$). We study the equation of $\partial_t^{i}\partial_\phi^{j}(\sqrt{\mu_0}\hat{f}_k).$ Using similar methods in Section \ref{sec4}, one gets that
				\begin{align}\label{K-2}
					& \sum_{i+j\leq \hat{s}_k}\|\hat{\varpi}_{\hat{\kappa}_k} e^{q_k\eta}\partial_t^{i}\partial_\phi^{j}\hat{\mf}_k(t,\phi,\cdot,\cdot)\|_{L^\infty_{\eta,\mv}}
					+\|\hat{\varpi}_{\hat{\kappa}_k} \partial_t^{i}\partial_\phi^{j}\hat{\mf}_k(t,\phi,0,\cdot)\|_{L^\infty_\mv}<\infty.
				\end{align}
				Such an  estimate \eqref{K-2} is enough for us to establish the higher order Hilbert expansion. Moreover, taking $L^\infty_{\phi}\cap L^2_{\phi}$ over \eqref{K-2}, one obtains
				\begin{align*}
					& \sum_{i+j\leq \hat{s}_k} \sup_{t\in[0,\tau^\d]} \Big\{\|\hat{\varpi}_{\hat{\kappa}_k}e^{q_k\eta}\partial_t^{i}\partial_\phi^{j}\hat{\mf}_k(t)\|_{L^\infty_{\phi}\cap L^2_{\phi}L^\infty_{\eta,\mv}}+\|\hat{\varpi}_{\hat{\kappa}_k} \partial_t^{i}\partial_\phi^{j}\hat{\mf}_k(t,\cdot,0,\cdot)\|_{L^\infty_{\phi}\cap L^2_{\phi}L^\infty_{\mv}}\Big\}<\infty.
				\end{align*}
				
				Secondly, we explain how \eqref{2-27} holds when constructing the Knudsen layer $\hat{f}_k=\hat{f}_{k,1}+\hat{f}_{k,2}.$ Assume we have obtained $f_k,\bar{f}_k$ and $\hat{f}_{k,1}$. %It is clear the expression for $\hat{g}_k$ consists of internal solution $f_k$, viscous boundary layer $\bar{f}_k$, and $\hat{f}_{k,1}.$ The condition \eqref{2-27} needs the velocity-weighted norm and the velocity derivative norm. 
				We can regularly obtain the velocity-weighted norm of  $f_k,\bar{f}_k$ and $\hat{f}_{k,1}$ and also the velocity derivative norm of $\FP f_k, \FP_0\bar{f}_k$ and $\hat{f}_{k,1}$. Noting \eqref{1.10} and \eqref{1.12}, we also get the weighted velocity derivative norm of $(\FI-\FP)f_k$ and $(\FI-\FP_0)\bar{f}_k$ using the estimate of $\FL^{-1}$ in \cite{Jiang}. Then from the expression of $\hat{g}_k$ in \eqref{2.25}, we have \eqref{2-27} for each layer.
				
				The rest proof of Proposition \ref{prop} is almost the same as \cite{GHW-2021-ARMA} and we omit the details here. Then the proof of Proposition \ref{prop} is finished.
			\end{proof}

	  \begin{remark}
	  	Recall Remark \ref{rmk2.4}, since the Knudsen layer
	  	solutions $\hat{F}_i$ are only defined in domain near the boundary of disk, it is necessary to extend $\hat{F}_i$ into the whole disk by introducing $\hat{\mathscr{F}}_i=\Upsilon(\v^{\mathfrak{a}}\eta)\hat{F}_i$, see also \eqref{1.34}. Due to the exponential space decay on $\eta$ and the cut-off function, it is direct to know that the error has a very high order $\v$-decay rate.
	  \end{remark}
	%	\Blue{Why not add some details on how to construct $F_k,\bar{\mathfrak{F}}_k, \hat{\mathscr{F}}_k$ ?}

		\subsection{$L^2-L^\infty$ estimate of remainder}
	For later use,	we first recall $F_i,\bar{\mathfrak{F}}_i, \hat{\mathscr{F}}_i$ constructed in section \ref{sec5.1},  and the equation \eqref{re} for  remainder  $F_R^\v$. 
	
	Define 
	\begin{align}
	 f_R^\v:=\f{1}{\sqrt{\mu}}F^\v_R\quad \mbox{and}\quad  h^\v_R:=\f{\tilde{w}_\mathbf{k}}{\sqrt{\mu_M}}F_R^\v=\tilde{w}_\mathbf{k}\f{\sqrt{\mu}}{\sqrt{\mu_M}}f_R^\v,
	\end{align}
	then we have from \eqref{re} that
		\begin{align}\label{6.04}
			\begin{cases}
			\dis	\partial_t f^\v_R+v\cdot\nabla_x f^\v_R +\f{1}{\v^2}\FL f^\v_R=-\f{\left\{\partial_t+v\cdot\nabla_x\right\}\sqrt{\mu}}{\sqrt{\mu}}f^\v_R+S_R,\\
			\dis	f^\v_R(x,v)|_{\gamma_-}=f^\v_R(x,R_xv),
			\end{cases}
		\end{align}
		and 
		\begin{align}\label{6.05}
			\begin{cases}
			\dis	\partial_t h^\v_R+v\cdot\nabla_x h^\v_R +\f{1}{\v^2}\tilde{\nu} h^\v_R-\f{1}{\v^2}\tilde{K}_{M,\tilde{w}}h^\v_R=\tilde{w}_{\mathbf{k}}\bar{S}_R,\\
			\dis	h^\v_R(x,v)|_{\gamma_-}=h^\v_R(x,R_xv).
			\end{cases}
		\end{align}
		Recall $\FL$ in \eqref{1.11-1}, we have  
		\begin{align*}
			\FL g&=\tilde{\nu}(v) g-\tilde{K}g=\tilde{\nu}(v) g-(\tilde{K}_1g-\tilde{K}_2g)\\
			&=\tilde{\nu}(v) g(v)-\left(\int_{\R^2}\tilde{k}_1(v, u)g(u)\, du-\int_{\R^2}\tilde{k}_2(v,u)g(u)\, du\right),
		\end{align*} 
		where $	\tilde{\nu}(v)=\int_{\R^2}|v-u|\mu(u)\,du\sim 1+|v|$ and $\tilde{\nu}\geq \tilde{\nu}_0$ with $\tilde{\nu}_0$ a positive constant, and
		\begin{align*}
			\tilde{K}_1g&=\int_{\R^2}\tilde{k}_1(v,u)g(u)\, du=\int_{\R^2}|v-u|\sqrt{\mu(v)}\sqrt{\mu(u)}g(u)\, du\\
			\tilde{K}_2g&=\int_{\R^2}\tilde{k}_2(v,u)g(u)\, du=2\int_{\R^2}|v-u|\sqrt{\mu(u)}\sqrt{\mu(v')}g(u')\, du.
		\end{align*}	
%Similar to \eqref{5.0-1}, one has 
%		\begin{align}
%			\begin{split}
%				\tilde{k}_1(v,u)\lesssim |v-u|\exp\left\{-\f{1}{4T}(|v-\mathfrak{u}|^2+|u-\mathfrak{u}|^2)\right\},\\
%				\tilde{k}_2(v,u)\lesssim \exp\left\{-\f{1}{8T}\left(|v-u|^2+\f{\left||v-\mathfrak{u}|^2-|u-\mathfrak{u}|^2\right|^2}{|v-u|^2}\right)\right\},
%			\end{split}
%		\end{align}
%		and $\tilde{k}_i(v,u)\in L_u^1\cap L^2_u,\, i=1,2$.\\		
		Denote $\dis \tilde{K}_Mg=\f{\sqrt{\mu}}{\sqrt{\mu_M}}\tilde{K}(\f{\sqrt{\mu_M}}{\sqrt{\mu}}g)=\int_{\R^2}\tilde{k}_M(v,u)g(u)\, du$ and  $\dis\tilde{K}_{M,\tilde{w}}h=\f{\sqrt{\mu}}{\sqrt{\mu_M}}\tilde{w}_{\mathbf{k}}\tilde{K}(\f{\sqrt{\mu_M}}{\sqrt{\mu}}\f{h}{\tilde{w}_{\mathbf{k}}})$. Similar to Lemma \ref{lem4.1}, one has
		\begin{align}\label{6-7}
			\int_{\R^2}\tilde{k}_M(v,u)(1+|u|^2)^{\f{\beta}{2}}e^{\zeta|u|^2}\,du\leq C(1+|v|^2)^{\f{\beta-1}{2}}e^{\zeta|v|^2},\quad \beta\geq 0,\zeta<\f{1}{4T}.
		\end{align}
		We divide $S_R$ into three parts as $S_R=S_{R,1}+S_{R,2}+S_{R,3}$:
		\begin{align}\label{6.06}
			S_{R,1}&=\v^3\Gamma (f^\v_R,f^\v_R),\nonumber\\
			S_{R,2}&=\sum_{i=1}^N\v^{i-2}\left\{\Gamma\left(f^\v_R,\f{F_i+\bar{\mathfrak{F}}_i+\hat{\mathscr{F}}_i}{\sqrt{\mu}}\right)+\Gamma\left(\f{F_i+\bar{\mathfrak{F}}_i+\hat{\mathscr{F}}_i}{\sqrt{\mu}},f^\v_R\right)\right\},\\
			S_{R,3}&=\f{R^\v+\bar{R}^\v+\hat{R}^\v}{\sqrt{\mu}},\nonumber
		\end{align}
		where 
		\begin{align*}
			\Gamma(g_1,g_2)=\f{1}{\sqrt{\mu}}Q(\sqrt{\mu}g_1,\sqrt{\mu}g_2).
		\end{align*}
	Furthermore, $\bar{S}_R=\f{\sqrt{\mu}}{\sqrt{\mu_M}}S_R$ can be relatively divided into $\bar{S}_{R,1},\bar{S}_{R,2}$ and $\bar{S}_{R,3}$.\\ 
		
		We intend to establish the estimate of remainder $F_R^\v$ by using the $L^2-L^\infty$ framework.
		\subsubsection{$L^2$-estimate}
		\begin{Lemma}\label{lem6.1}
			Let $0<\f{1}{2\iota}(1-\iota)<\mathbf{a}<\f12,\mathbf{k}\geq 6,N\geq6$ and $\mathfrak{b}\geq 5.$ Let $\tau^\d>0$ be the life span of compressible Euler solution obtained in Lemma \ref{lem3.1}. For $t\in[0,\tau^\d]$, there exists a suitably small constant $\v_0 > 0$ such that for all $\v\in  (0, \v_0)$, the following estimate holds:
			\begin{align}\label{6.1}
				&\f{d}{dt}\|f^\v_R\|_{L^2_{x,v}}^2+\f{c_0}{\v^2}\int_{B_1}\|(\FI-\FP)f^\v_R\|_{\tilde{\nu}}^2\, dx\nonumber\\
				&\leq C\left\{1+\v^{\min\{2\mathbf{k}-5,8\}}\|h^\v_R(t)\|_{L^\infty_{x,v}}^2\right\}\cdot(\|f^\v_R(t)\|_{L^2_{x,v}}^2+1).
			\end{align}
		\end{Lemma}
		\begin{proof}
			Multiplying \eqref{6.04} by $f^\v_R$ and integrating on $B_1\times\R^2$, one has that
			\begin{align}\label{6.2}
				&\f{d}{dt}\|f^\v_R\|_{L^2_{x,v}}^2+\f{c_0}{\v^2}\int_{B_1}\|(\FI-\FP)f^\v_R\|_{\tilde{\nu}}^2\, dx\nonumber\\
				&\leq \int_{B_1}\int_{\R^2}(1+|v|)^3|f^\v_R(x,v)|^2dvdx+\int_{B_1}\int_{\R^2}|f^\v_R\cdot S_R(x,v)|dvdx.
			\end{align} 
			As in \cite{GHW-2021-ARMA}, we have
			\begin{align}\label{6.2-0}
				&\int_{B_1}\int_{\R^2}(1+|v|)^3|f^\v_R(x,v)|^2dvdx\nonumber\\
				&\leq \int_{B_1}\int_{|v|\leq \f{\lambda}{\v}}(1+|v|)^3|f^\v_R(x,v)|^2dvdx+ \int_{B_1}\int_{|v|\geq \f{\lambda}{\v}}(1+|v|)^3|f^\v_R(x,v)|^2dvdx\nonumber\\
				&\leq \f{\lambda}{\v^2}\int_{B_1}\|(\FI-\FP)f^\v_R\|_{\tilde{\nu}}^2\, dx+C_\lambda(\|f^\v_R\|_{L^2_{x,v}}^2+\v^{2\mathbf{k}-5}\|h^\v_R\|_{L^\infty_{x,v}}^2).
			\end{align}
			It is clear that
			\begin{align*}
				&\int_{B_1}\int_{\R^2}|f^\v_R\cdot S_{R,1}|dvdx=\v^3\int_{B_1}\int_{\R^2}|f^\v_R\cdot\Gamma(f^\v_R,f^\v_R)|dvdx\nonumber\\
				&=\v^3\int_{B_1}\int_{\R^2}|(\FI-\FP)f^\v_R\cdot\Gamma(f^\v_R,f^\v_R)|dvdx\nonumber\\
				&\leq \f{\lambda}{\v^2}\int_{B_1}\|(\FI-\FP)f_R^\v\|_{\tilde{\nu}}^2\, dx+C_\lambda\v^8\|h^\v_R\|_{L^\infty_{x,v}}^2\|f^\v_R\|_{L^2_{x,v}}^2.
			\end{align*}
			Noting the discussion of $F_i,\bar{\mathfrak{F}}_i$ and $\hat{F}_i$ in Proposition \ref{prop} and recall \eqref{1.34}, we get 
			\begin{align*}
				&\int_{B_1}\int_{\R^2}|f^\v_R\cdot S_{R,2}|dvdx\nonumber\\
				&=\sum_{i=1}^N\v^{i-2}\int_{B_1}\int_{\R^2}\left|(\FI-\FP)f^\v_R\cdot \Gamma\left(f^\v_R,\f{F_i+\bar{\mathfrak{F}}_i+\hat{\mathscr{F}}_i}{\sqrt{\mu}}\right)\right|\,dvdx\nonumber\\
				&\leq \f{C}{\v}\int_{B_1}\|(\FI-\FP)f^\v_R\|_{\tilde{\nu}}\cdot\|f^\v_R\|_{L^2_v}\, dx\leq\f{\lambda}{\v^2}\int_{B_1}\|(\FI-\FP)f_R^\v\|_{\tilde{\nu}}^2\, dx+C_\lambda\|f^\v_R\|_{L^2_{x,v}}^2,
			\end{align*}
			where we have used \eqref{1.23} and obtained the following facts
			\begin{align*}
				\left|\tilde{w}_{\mathbf{k}}\f{\bar{\mathscr{F}_i}}{\sqrt{\mu}}\right|&\leq |\tilde{w}_{\mathbf{k}}\mu_0^{-\mathbf{a}}\bar{f}_i|\cdot \f{\mu_0^{\f12+\mathbf{a}}}{\mu^{\f12}}\leq  |\tilde{w}_{\mathbf{k}}\mu_0^{-\mathbf{a}}\bar{f}_i|\cdot\mu_M^{(\f12+\mathbf{a})\iota-\f12},\\
				\left|\tilde{w}_{\mathbf{k}}\f{\hat{F}_i}{\sqrt{\mu}}\right|&\leq |\tilde{w}_{\mathbf{k}}\mu_M^{-\mathbf{a}}\hat{\mathbf{f}}_i|\cdot \f{\mu_M^{\f12+\mathbf{a}}}{\mu^{\f12}}\leq  |\tilde{w}_{\mathbf{k}}\mu_M^{-\mathbf{a}}\hat{\mathbf{f}}_i|\cdot\mu_M^{\mathbf{a}}.
			\end{align*}
			From Proposition \ref{prop}, noting \eqref{eq1.21} and \eqref{eq1.22}, one deduces that $S_{R,3}\in L^2(B_1\times\R^2)$ and
			\begin{align}\label{6.2-1}
				\int_{B_1}\int_{\R^2}|f^\v_R\cdot S_{R,3}|dvdx\leq C(\v^{N-6}+\v^{\mathfrak{b}-5})\|f^\v_R\|_{L^2_{x,v}}.
			\end{align}
			Combining \eqref{6.2}--\eqref{6.2-1}, taking $\lambda$ small enough, one concludes Lemma \ref{lem6.1}. Therefore the proof of Lemma \ref{lem6.1} is completed.
		\end{proof}
		
\subsubsection{$L^\infty$-estimate}
To study \eqref{6.05}, we consider the change of variables. Let $$x=(r\cos\phi,r\sin\phi),\quad  \bar{\mv}=(\bar{\mv}_{1},\bar{\mv}_{2})=(v\cdot \vec{n},v\cdot\vec{\tau})$$ where  $\vec{n}=(\cos\phi,\sin\phi),\vec{\tau}=(-\sin\phi,\cos\phi)$. Then it holds that
\begin{align*}
	v\cdot \nabla_x=\bar{\mv}_{1}\partial_r+\f{\bar{\mv}_{2}}{r}(\partial_\phi+\bar{\mv}_{2}\partial_{\bar{\mv}_{1}}-\bar{\mv}_{1}\partial_{\bar{\mv}_{2}}).
\end{align*}
Now \eqref{6.05} is rewritten as
\begin{align}\label{6.3}
	\partial_t h^\v_R+\bar{\mv}_{1}\partial_rh^\v_R+\f{\bar{\mv}_{2}}{r}(\partial_\phi+\bar{\mv}_{2}\partial_{\bar{\mv}_{1}}-\bar{\mv}_{1}\partial_{\bar{\mv}_{2}})h^\v_R+\f{1}{\v^2}\tilde{\nu} h^\v_R-\f{1}{\v^2}\tilde{K}_{M,\tilde{w}}h^\v_R=\tilde{w}_{\mathbf{k}}\bar{S}_R,
\end{align}
and the boundary is divided into three parts:
\begin{align*}
\begin{split}
	\tilde{\gamma}_-=\{(r,\phi,\bar{\mv}_{1},\bar{\mv}_{2}):r=1,\bar{\mv}_{1}<0\},\\
	\tilde{\gamma}_0=\{(r,\phi,\bar{\mv}_{1},\bar{\mv}_{2}):r=1,\bar{\mv}_{1}=0\},\\
	\tilde{\gamma}_+=\{(r,\phi,\bar{\mv}_{1},\bar{\mv}_{2}):r=1,\bar{\mv}_{1}>0\}.
\end{split}	
\end{align*}
\begin{Definition}\label{def6.1}
	Define the backward characteristic for \eqref{6.3} as
	\begin{align}\label{6.4}
		\begin{cases}
		\dis	\f{d\mathcal{X}}{ds}=\mathcal{V}_1,\f{d\Phi}{ds}=\f{\mathcal{V}_2}{\mathcal{X}}, \f{d \mathcal{V}_1}{ds}=\f{\mathcal{V}_2^2}{\mathcal{X}},\f{d \mathcal{V}_2}{ds}=-\f{\mathcal{V}_1 \mathcal{V}_2}{\mathcal{X}},\\
		\smallskip
		\smallskip
		\dis	(\mathcal{\mathcal{X}},\Phi,\mathcal{V}_1,\mathcal{V}_2)(t)=(r,\phi,\bar{\mv}_{1},\bar{\mv}_{2}).
		\end{cases}
	\end{align}
\end{Definition}
Along the backward characteristic line, we have
\begin{align}\label{6.4-0}
	\mathcal{V}_2\cdot \mathcal{X}=\bar{\mv}_{2} r,\quad |\mathcal{V}_1|^2+|\mathcal{V}_2|^2=|\bar{\mv}|^2=|v|^2,
\end{align}
which yields 
\begin{align}\label{6.5}
	\mathcal{V}_2(s)=\bar{\mv}_{2}\f{r}{\mathcal{X}(s)},\quad |\mathcal{V}_1|=\sqrt{|v|^2-\bar{\mv}_{2}^2\f{r^2}{\mathcal{X}^2(s)}}.
\end{align}

 For each $(t,r,\phi,\bar{\mv}_{1},\bar{\mv}_{2})\notin \{\tilde{\gamma}_0\cup\tilde{\gamma}_-\}$ and $|\bar{\mv}|\neq 0$, we define its backward time $t_{\mathbf{b}}(t,r,\phi,\bar{\mv}_{1},\bar{\mv}_{2})\geq0$ as
\begin{align*}
	t_{\mathbf{b}}(t,r,\phi,\bar{\mv}_{1},\bar{\mv}_{2}):=\sup\left\{z\geq0,\, r-\int_{\tau}^{t}\mathcal{V}_1(s) ds<1,\ 0\leq t-\tau\leq z\right\}.
\end{align*}
%We also define $\phi_{\mathfrak{b}}(t,r,\phi,y,\psi):=\phi-\int_{t-t_{\mathbf{b}}}^{t} \f{\mathcal{V}_2}{\mathcal{X}} ds,\,\text{and}\,\psi_{\mathfrak{b}}(t,r,\phi,y,\psi):=\psi+\int_{t-t_{\mathbf{b}}}^{t} \f{y\sin\Psi}{\mathcal{X}} ds$. 
Let $(r,\phi,\bar{\mv}_{1},\bar{\mv}_{2})\in [0,1]\times[0,2\pi)\times \R^2\backslash\{\gamma_0\cup\gamma_-\},\ (t_0,r_0,\phi_0,\mathbf{v}_0)=(t,r,\phi,\bar{\mv})$, and inductively define for $k\geq0$ that
\begin{equation*}
	(t_{k+1},r_{k+1})=(t_{k}-t_{\mathbf{b}}(r_k,\mathbf{v}_k),1).
\end{equation*}
It follows from \eqref{6.5} that $|\mathcal{V}_1|$ increases with $\mathcal{X}$, and 
\begin{align*}
	\mathbf{v}_{k+1,1}=\sqrt{|v|^2-\bar{\mv}_{2}^2r^2},\quad |\mathbf{v}_{k+1,2}|=|\bar{\mv}_{2}| r.
\end{align*}
Define the cycle as
\begin{equation}\label{6.6}
	\left\{\begin{aligned}
		&\mathcal{X}_{cl}(s;t,r,\phi,\bar{\mv})=\mathbf{1}_{(t_1,t]}\left\{r-\int_s^t\mathcal{V}_{cl,1}\, d\tau\right\}+\sum\limits_{k\geq 1}\mathbf{1}_{(t_{k+1},t_k]}(s)\left\{1-\int_{s}^{t_k} \mathcal{V}_{cl,1}d\tau\right\},\\
		&\Phi_{cl}(s;t,r,\phi,\bar{\mv})=\left(\phi-\int_s^t\f{\mathcal{V}_{cl,2}}{\mathcal{X}_{cl}}(\tau)\, d\tau\right)\mod 2\pi,\\
		&\mathcal{V}_{cl,2}(s;t,r,\phi,\bar{\mv})=\bar{\mv}_{2}\f{r}{\mathcal{X}_{cl}(s)}.
	\end{aligned}\right.
\end{equation}

Next, we make some analysis on the behavior of above backward characteristic line.\\

{\it Case a. $\bar{\mv}_{2}\cdot r=0$.} Then one has $\mathcal{V}_{cl,2}\equiv 0$ and $|\mathcal{V}_{cl,1}|\equiv |\bar{\mv}_{1}|$. The backward characteristic line cycles between $\mathcal{X}=0$ and $1$ with constant speed $|\bar{\mv}_{1}|$. If $\bar{\mv}_{1}>0$, the backward characteristic line moves to $\mathcal{X}=0$ first and 
\begin{align*}
	\mathcal{V}_{cl,1}(s;t,r,\phi,\bar{\mv})=\sum_{k\geq 0}\mathbf{1}_{(t_{k,0},t_k]}(s)\bar{\mv}_{1}+\sum_{k\geq 0}\mathbf{1}_{(t_{k+1},t_{k,0}]}(s)(-\bar{\mv}_{1}).
\end{align*}
If $\bar{\mv}_{1} \leq 0$, the line comes to $\mathcal{X}=1$ first and 
\begin{align*}
	\mathcal{V}_{cl,1}(s;t,r,\phi,\bar{\mv})=\mathbf{1}_{(t_1,t]}\bar{\mv}_{1}+\sum_{k\geq 1}\mathbf{1}_{(t_{k,0},t_k]}(s)\bar{\mv}_{1}+\sum_{k\geq 1}\mathbf{1}_{(t_{k+1},t_{k,0}]}(s)(-\bar{\mv}_{1}),
\end{align*}
where we denote $t_{k,0}\in[t_{k+1},t_k]$ to be the time of backward characteristic reaching $\mathcal{X}=0$. The reason $\bar{\mv}_{1}$ changes its sign at this point is actually that the normal vector $n$ turns to the opposite direction. \\

{\it Case b. $\bar{\mv}_{2}\cdot r\neq 0$.} Then there exists a $x_+\in(0,1]$ such that $|\mathcal{V}_1|=0$ at this position and the backward characteristic line will cycle in $x\in[x_+,1]$. We have from \eqref{6.5} that
\begin{align}\label{6.5-1}
	\f{r^2}{x_+^2}=1+\f{\bar{\mv}_{1}^2}{\bar{\mv}_{2}^2}.
\end{align}
For later use, we denote $t_{k,0}\in(t_{k+1},t_k]$ to be the time so that $\mathcal{V}_1(t_{k,0})=0$ or equivalently $\mathcal{X}=x_+$. \\
If $\bar{\mv}_{1}>0$, one has 
\begin{align*}
	\mathcal{V}_{cl,1}(s;t,r,\phi,\bar{\mv})&=\sum\limits_{k\geq 0}\mathbf{1}_{(t_{k,0},t_{k}]}(s)\sqrt{|v|^2-\bar{\mv}_{2}^2\f{r^2}{\mathcal{X}^2(s)}}
	\nonumber\\
	&\quad+ \sum\limits_{k\geq 0}\mathbf{1}_{(t_{k+1},t_{k,0}]}(s)(-1)\sqrt{|v|^2-\bar{\mv}_{2}^2\f{r^2}{\mathcal{X}^2(s)}}.
\end{align*}
If $\bar{\mv}_{1}\leq 0$, one obtains
\begin{align*}
	&\mathcal{V}_{cl,1}(s;t,r,\phi,\bar{\mv})=\mathbf{1}_{(t_1,t]}(-1)\sqrt{|v|^2-\bar{\mv}_{2}^2\f{r^2}{\mathcal{X}^2(s)}}\nonumber\\
	&\quad+\sum\limits_{k\geq 1}\mathbf{1}_{(t_{k,0},t_{k}]}(s)\sqrt{|v|^2-\bar{\mv}_{2}^2\f{r^2}{\mathcal{X}^2(s)}}+\sum\limits_{k\geq 1}\mathbf{1}_{(t_{k+1},t_{k,0}]}(s)(-1)\sqrt{|v|^2-\bar{\mv}_{2}^2\f{r^2}{\mathcal{X}^2(s)}}.
\end{align*}
The backward characteristic in polar coordinates are indeed consistent with those in original coordinates. See the figures below for instance.
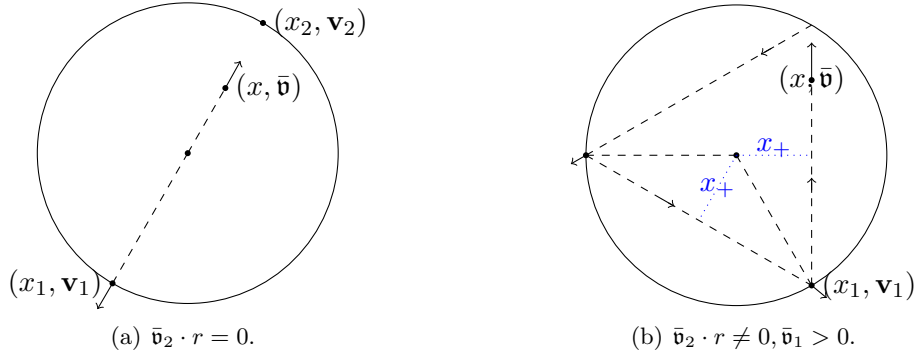
\begin{figure}[H]\label{fig3}
	\centering 
	\subfigure[$\bar{\mathfrak{v}}_2\cdot r=0$.]{
		\begin{tikzpicture}
			\draw (0,0) circle (2);
			\fill[black]  (0,0) circle (.04);% the circle
			\fill[black] (0.5,{sqrt(0.75)}) circle (.04);
			\node [right] at (0.5,{sqrt(0.75)}) {$(x,\bar{\mathfrak{v}})$};
			\draw[->]  (0.5,{sqrt(0.75)})--(0.7,{sqrt(1.5)});
			\draw[dashed]  (0,0)--(0.5,{sqrt(0.75)});
			\draw[dashed] (-1,{-sqrt(3)})--(0,0);
			\fill (-1,{-sqrt(3)}) circle (.04);
			\draw[->] (-1,{-sqrt(3)})--(-1.2,{-sqrt(4.3)});
			\node [left] at (-1,{-sqrt(3)}) {$(x_1,\mathbf{v}_1)$};
			\fill (1,{sqrt(3)}) circle (.04);
			\node [right] at (1,{sqrt(3)}) {$(x_2,\mathbf{v}_2)$};
	\end{tikzpicture}}
	\qquad\qquad\qquad
	\subfigure[$\bar{\mathfrak{v}}_2\cdot r\neq0,\bar{\mv}_1>0$.]{	\begin{tikzpicture}
			\draw (0,0) circle (2);
			\fill[black] (0,0) circle (.04);% the circle
			\fill[black]  (1,1) circle (.04);
			\draw[->]  (1,1)--(1,1.5);
			\node at (1,1) {$(x,\bar{\mathfrak{v}})$};% the starting point
			\draw[dashed]  (1,{-sqrt(3)})--(1,1);
			\draw[dashed] (0,0)--(1,{-sqrt(3)});
			\draw[->] (1,{-sqrt(3)})--(1.2,{-sqrt(3.6)});
			\fill[black]  (1,{-sqrt(3)}) circle (.04);
			\node[right] at (1,{-sqrt(3)}) {$(x_1,\mathbf{v}_1)$};% the first collision
			\draw[dashed] (-2,0)--(1,{-sqrt(3)});
			\draw[dashed] (-2,0)--(0,0);
			\draw[->] (-2,0)--(-2.2,-0.115);
			\draw[dashed] (-2,0)--(1,{sqrt(3)});
			\draw[blue,dotted] (0,0)--(1,0);
			\node at (0.5,0.1) {$\color{blue}x_+$};
			\draw[blue,dotted] (0,0)--({-0.5},{-sqrt(0.75)});
			\node at ({-0.25},{-0.43}) {$\color{blue}x_+$};
			\draw[->] (1,-0.5)--(1,-0.3);
			\draw[->] (-1,-0.577)--(-0.827,-0.677);
			\draw[->] (0.5,1.443)--(0.327,1.343);
			\fill (-2,0) circle(.04);
	\end{tikzpicture}}
	\caption{Backward characteristics in unit disk.}
\end{figure}
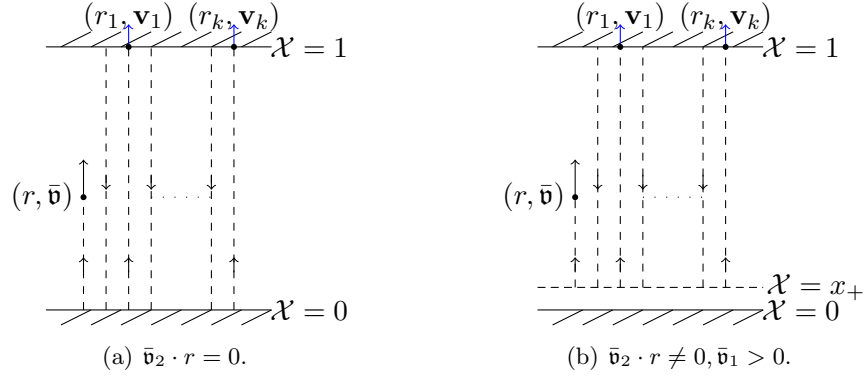
\begin{figure}[H]\label{fig3-1}
	\centering
	\subfigure[$\bar{\mv}_2\cdot r=0.$]{
		\begin{tikzpicture}
			\draw (0,0.5)--(3,0.5);
			\draw (0,4)--(3,4);
			\node at (3.5,0.5) {$\mathcal{X}=0$};
			\node at (3.5,4) {$\mathcal{X}=1$};
			\draw (0.2,4)--(0.6,4.2);
			\draw (0.6,4)--(1.0,4.2);
			\draw (1.0,4)--(1.4,4.2);
			\draw (1.4,4)--(1.8,4.2);
			\draw (1.8,4)--(2.2,4.2);
			\draw (2.2,4)--(2.6,4.2);
			\draw (2.6,4)--(3.0,4.2);
			\draw (0.2,0.3)--(0.6,0.5);
			\draw (0.6,0.3)--(1.0,0.5);
			\draw (1.0,0.3)--(1.4,0.5);
			\draw (1.4,0.3)--(1.8,0.5);
			\draw (1.8,0.3)--(2.2,0.5);
			\draw (2.2,0.3)--(2.6,0.5);
			\draw (2.6,0.3)--(3.0,0.5);%the pasic pattern;
			\draw[->] (0.5,2)--(0.5,2.5);		
			\fill[black] (0.5,2)circle (.04);
			\node [left] at (0.5,2) {$(r,\bar{\mathfrak{v}})$};%starting point
			\draw[dashed] (0.5,0.5)--(0.5,2);
			\draw[dashed] (0.8,0.5)--(0.8,4);
			\draw[blue,->] (1.1,4)--(1.1,4.3);	
			\fill (1.1,4) circle (.04);
			\node  at (1.1,4.4) {$(r_1,\mathbf{v}_1)$};%the first collision
			\draw[dashed] (1.1,0.5)--(1.1,4);
			\draw[dashed] (1.4,0.5)--(1.4,4);
			\draw[loosely dotted] (1.4,2)--(2.2,2);
			\draw[dashed] (2.2,0.5)--(2.2,4);
			\draw[blue,->] (2.5,4)--(2.5,4.3);
			\fill (2.5,4)circle (.04);
			\node  at (2.5,4.4) {$(r_k,\mathbf{v}_k)$};%the kth collision
			\draw[dashed] (2.5,0.5)--(2.5,4);
			\draw[->] (0.5,1)--(0.5,1.2);
			\draw[->] (0.8,2.3)--(0.8,2.1);
			\draw[->] (1.1,1)--(1.1,1.2);
			\draw[->] (1.4,2.3)--(1.4,2.1);
			\draw[->] (2.2,2.3)--(2.2,2.1);
			\draw[->] (2.5,1)--(2.5,1.2);
	\end{tikzpicture}}
	\qquad\qquad
	\subfigure[$\bar{\mv}_2\cdot r\neq0, \bar{\mv}_1>0.$]{
		\begin{tikzpicture}
			\draw (0,0.5)--(3,0.5);
			\draw (0,4)--(3,4);
			\node at (3.5,0.5) {$\mathcal{X}=0$};
			\node at (3.5,4) {$\mathcal{X}=1$};
			\draw (0.2,4)--(0.6,4.2);
			\draw (0.6,4)--(1.0,4.2);
			\draw (1.0,4)--(1.4,4.2);
			\draw (1.4,4)--(1.8,4.2);
			\draw (1.8,4)--(2.2,4.2);
			\draw (2.2,4)--(2.6,4.2);
			\draw (2.6,4)--(3.0,4.2);
			\draw (0.2,0.3)--(0.6,0.5);
			\draw (0.6,0.3)--(1.0,0.5);
			\draw (1.0,0.3)--(1.4,0.5);
			\draw (1.4,0.3)--(1.8,0.5);
			\draw (1.8,0.3)--(2.2,0.5);
			\draw (2.2,0.3)--(2.6,0.5);
			\draw (2.6,0.3)--(3.0,0.5);
			\draw[densely dashed] (0,0.8)--(3,0.8);
			\node at (3.7,0.8) {$\mathcal{X}=x_+$};%the basic pattern
			\draw[->] (0.5,2)--(0.5,2.5);		
			\fill[black] (0.5,2)circle (.04);
			\node [left] at (0.5,2) {$(r,\bar{\mathfrak{v}})$};%starting point
			\draw[dashed] (0.5,0.8)--(0.5,2);
			\node  at (1.1,4.4) {$(r_1,\mathbf{v}_1)$};%the first collision
			\draw[dashed] (0.8,0.8)--(0.8,4);
			\draw[dashed] (1.1,0.8)--(1.1,4);
			\draw[blue,->] (1.1,4)--(1.1,4.3);
			\fill (1.1,4)circle (.04);	%\node at (2,0.2) {$(\eta_2,v_2)$};% the second collision\draw[dashed] (1.1,0.5)--(1.1,4);
			\draw[dashed] (1.4,0.8)--(1.4,4);
			\draw[loosely dotted] (1.4,2)--(2.2,2);
			\draw[dashed] (2.2,0.8)--(2.2,4);
			\draw[blue,->] (2.5,4)--(2.5,4.3);
			\fill (2.5,4)circle (.04);
			\node  at (2.5,4.4) {$(r_k,\mathbf{v}_k)$};%the kth collision
			\draw[dashed] (2.5,0.8)--(2.5,4);
			\draw[->] (0.5,1)--(0.5,1.2);
			\draw[->] (0.8,2.3)--(0.8,2.1);
			\draw[->] (1.1,1)--(1.1,1.2);
			\draw[->] (1.4,2.3)--(1.4,2.1);
			\draw[->] (2.2,2.3)--(2.2,2.1);
			\draw[->] (2.5,1)--(2.5,1.2);
	\end{tikzpicture}}
	\caption{The backward progress in a slab.}
\end{figure}
%\begin{Lemma}\label{lem6.2-0}
%It holds that 
%\begin{align}\label{6-20}
%	\dis \f{\partial \mathcal{X}(s)}{\partial \bar{\mv}_1}=
%	\dis \begin{cases}
	%		-\f{\bar{\mv}_1|\bar{\mv}_2|r}{|\bar{\mv}|^4}|\mathcal{V}_1(s)|\cdot\left[\left(\f{|\mathcal{V}_2(s)|}{|\mathcal{V}_1(s)|}-\f{|\mathcal{V}_1(s)|}{|\mathcal{V}_2(s)|}\right)-\left(\f{|\bar{\mv}_2|}{|\bar{\mv}_1|}-\f{|\bar{\mv}_1|}{|\bar{\mv}_2|}\right)\right], \quad \bar{\mv}_1>0,s\geq t_{0,0},\\
	%		-\f{\bar{\mv}_1|\bar{\mv}_2|r}{|\bar{\mv}|^4}|\mathcal{V}_1(s)|\cdot\left(\f{|\mathcal{V}_2(s)|}{|\mathcal{V}_1(s)|}-\f{|\mathcal{V}_1(s)|}{|\mathcal{V}_2(s)|}\right),\quad \bar{\mv}_{1}>0,t_1\leq s\leq t_{0,0},\\
	%		-\f{\bar{\mv}_1|\bar{\mv}_2|r}{|\bar{\mv}|^4}|\mathcal{V}_1(s)|\cdot\left[\left(\f{|\mathcal{V}_2(s)|}{|\mathcal{V}_1(s)|}-\f{|\mathcal{V}_1(s)|}{|\mathcal{V}_2(s)|}\right)-\left(\f{|\bar{\mv}_2|}{|\bar{\mv}_1|}-\f{|\bar{\mv}_1|}{|\bar{\mv}_2|}\right)\right], \quad \bar{\mv}_{1}<0,s\geq t_1,\\
	%		-\f{\bar{\mv}_1|\bar{\mv}_2|r}{|\bar{\mv}|^4}|\mathcal{V}_1(s)|\cdot\left(\f{|\mathcal{V}_2(s)|}{|\mathcal{V}_1(s)|}-\f{|\mathcal{V}_1(s)|}{|\mathcal{V}_2(s)|}\right),\quad \bar{\mv}_{1}>0,t_{1,0}\leq s\leq t_1.
	%	\end{cases}
%\end{align}
%\end{Lemma}

\begin{Lemma}\label{lem6.2-0}
	For $\bar{\mv}_1\leq 0$, $\bar{\mv}_2\cdot r\neq 0$ and $t_{1,0}\leq s\leq t_1$, one has 
	\begin{align}\label{6-32}
		\dis	\mathcal{X}\left |\f{\partial(\mathcal{X},\Phi)}{\partial \bar{\mv}}(s)\right | 
		=&\f{|\bar{\mv}_1|}{|\bar{\mv}_2|}(s-t_{1,0})(t-s) \cdot\left|\f{2}{\sqrt{\f{1}{r^2}\f{|\bar{\mv}|^2}{|\bar{\mv}_2|^2}-1}}-\f{|\bar{\mv}_2|}{|\bar{\mv}_1|}-\f{1}{\sqrt{\f{\mathcal{X}^2}{r^2}\f{|\bar{\mv}|^2}{|\bar{\mv}_2|^2}-1}}\right|.
	\end{align}
\end{Lemma}
\begin{proof}
	We divide the proof into several steps.\\
	{\noindent \it Step 1. Calculation on $\dis \f{\partial \mathcal{X}(s)}{\partial \bar{\mv}_i}$.} It follows from \eqref{6.4}--\eqref{6.4-0} that
	\begin{align}\label{6-1}
		\f{d}{ds} [\mathcal{X}(s)\mathcal{V}_1(s)]=\mathcal{V}_1^2(s)+\mathcal{V}_2^2(s)=|\bar{\mv}|^2,
	\end{align}
which yields that
	\begin{align}\label{6-2}
	r\bar{\mv}_1+|\mathcal{V}_1(t_1)|=|\bar{\mv}|^2(t-t_1)\quad \text{and}\quad |\mathcal{V}_1(t_1)|=|\bar{\mv}|^2(t_1-t_{1,0}). 
	\end{align}
	For $s\in(t_{1,0},t_1]$, it holds that $\mathcal{V}_1(s)\geq 0$. Then one has from  \eqref{6-1} that
	\begin{align}\label{6-4}
		\mathcal{X}(s)\mathcal{V}_1(s)=|\bar{\mv}|^2(s-t_{1,0}),
	\end{align}
	which, together with \eqref{6.5}, yields that
	\begin{align}\label{6-5}
		\sqrt{|\bar{\mv}|^2\mathcal{X}^2(s)-\bar{\mv}_2^2r^2}=|\bar{\mv}|^2(s-t_{1,0}).
	\end{align}
	Thus we have from \eqref{6-5} that
	\begin{align}\label{6-6}
		\mathcal{X}^2(s)=|\bar{\mv}|^2(s-t_{1,0})^2+\f{\bar{\mv}_2^2}{|\bar{\mv}|^2}r^2.
	\end{align}
	
	Applying $\partial_{\bar{\mv}_1}$ to \eqref{6-6}, one has 
	\begin{align*}
		2\mathcal{X}(s)\f{\partial \mathcal{X}(s)}{\partial \bar{\mv}_1}=2\bar{\mv}_1(s-t_{1,0})^2-2|\bar{\mv}|^2\f{\partial t_{1,0}}{\partial \bar{\mv}_1}(s-t_{1,0})-2\f{\bar{\mv}_1\bar{\mv}_2^2}{|\bar{\mv}|^4}r^2,
	\end{align*}
	which yields that 
	\begin{align}\label{6-8}
		\mathcal{X}(s)\f{\partial \mathcal{X}(s)}{\partial \bar{\mv}_1}=-|\bar{\mv}_1|(s-t_{1,0})^2-|\bar{\mv}|^2\f{\partial t_{1,0}}{\partial \bar{\mv}_1}(s-t_{1,0})+\f{|\bar{\mv}_1|\bar{\mv}_2^2}{|\bar{\mv}|^4}r^2.
	\end{align}
	Applying $\partial_{\bar{\mv}_2}$ to \eqref{6-6}, one gets
	\begin{align*}
		2\mathcal{X}\f{\partial \mathcal{X}}{\partial \bar{\mv}_2}&=2\bar{\mv}_2(s-t_{1,0})^2-2|\bar{\mv}|^2\f{\partial t_{1,0}}{\partial \bar{\mv}_2}(s-t_{1,0})+2\f{\bar{\mv}_2}{|\bar{\mv}|^2}r^2-2\f{\bar{\mv}_2^3}{|\bar{\mv}|^4}r^2\\
		&=2\text{sgn}\bar{\mv}_2|\bar{\mv}_2|(s-t_{1,0})^2-2|\bar{\mv}|^2\f{\partial t_{1,0}}{\partial \bar{\mv}_2}(s-t_{1,0})+2\text{sgn}\bar{\mv}_2\f{|\bar{\mv}_1|^2|\bar{\mv}_2|}{|\bar{\mv}|^4}r^2,
	\end{align*}
	which yields that 
	\begin{align}\label{6-9}
		\mathcal{X}\f{\partial \mathcal{X}}{\partial \bar{\mv}_2}=\text{sgn}\bar{\mv}_2|\bar{\mv}_2|(s-t_{1,0})^2-|\bar{\mv}|^2\f{\partial t_{1,0}}{\partial \bar{\mv}_2}(s-t_{1,0})+\text{sgn}\bar{\mv}_2\f{|\bar{\mv}_1|^2|\bar{\mv}_2|}{|\bar{\mv}|^4}r^2.
	\end{align}
	{\noindent \it Step 2. Calculation on $\dis \f{\partial\Phi}{\partial \bar{\mv}_1}.$} It follows from \eqref{6.6} that 
	\begin{align}\label{6-25}
		\Phi(s)&=\phi-\int_s^t\f{\mathcal{V}_2}{\mathcal{X}}(\tau)\, d\tau=\phi-\int_{t_1}^t\f{\mathcal{V}_2}{\mathcal{X}}(\tau)\, d\tau-\int_s^{t_1}\f{\mathcal{V}_2}{\mathcal{X}}(\tau)\, d\tau\nonumber\\
		&=\phi-\int_1^r\f{\mathcal{V}_2}{\mathcal{X}}\cdot\f{-1}{|\mathcal{V}_1|}\, d\mathcal{X}-\int_{\mathcal{X}(s)}^1\f{\mathcal{V}_2}{\mathcal{X}}\cdot\f{1}{|\mathcal{V}_1|}\, d\mathcal{X}\nonumber\\
		&=\phi-\int_r^1\f{\mathcal{V}_2}{\mathcal{X}}\cdot\f{1}{|\mathcal{V}_1|}\, d\mathcal{X}-\int_{\mathcal{X}(s)}^1\f{\mathcal{V}_2}{\mathcal{X}}\cdot\f{1}{|\mathcal{V}_1|}\, d\mathcal{X}.
	\end{align}
	To calculate \eqref{6-25}, we consider the change of variable  $\dis y=y(z)=\f{z}{r}\sqrt{1+\f{\bar{\mv}_{1}^2}{\bar{\mv}_{2}^2}}=z\cdot\f{|\bar{\mv}|}{|\bar{\mv}_2|r},$ and define $(y(r),y(1),y(\mathcal{X}(s)))=\f{|\bar{\mv}|}{|\bar{\mv}_2|r}(r,1,\mathcal{X}(s))$. Then one has 
	\begin{align}\label{5.29-1}
		&\f{\mathcal{V}_2(z)}{z}\cdot\f{1}{|\mathcal{V}_1(z)|}dz =\f{\bar{\mv}_2\cdot r}{z^2}\cdot \f{1}{\sqrt{|\bar{\mv}|^2-\f{\bar{\mv}_2^2\cdot r^2}{z^2}}}dz\nonumber\\
		&=\f{\bar{\mv}_2\cdot r}{y^2r^2|\bar{\mv}_2|^2}\cdot |\bar{\mv}|^2\cdot \f{1}{\sqrt{|\bar{\mv}|^2-\f{|\bar{\mv}|^2}{y^2}}}\cdot \f{|\bar{\mv}_2|r}{|\bar{\mv}|}dy\nonumber\\
		&=\text{sgn}\bar{\mv}_{2}\f{1}{y\sqrt{y^2-1}}\,dy.
	\end{align}
	Applying \eqref{5.29-1} to \eqref{6-25}, one obtains that
	\begin{align}\label{6-26}
		\Phi(s)=\phi-\text{sgn} \bar{\mv}_2\int_{y(r)}^{y(1)}\f{1}{y\sqrt{y^2-1}}\, dy-\text{sgn} \bar{\mv}_2\int_{y(\mathcal{X}(s))}^{y(1)}\f{1}{y\sqrt{y^2-1}}\, dy.
	\end{align}
	
	Thus we have from \eqref{6-26} that 
	\begin{align}\label{6-27}
		\f{\partial\Phi(s)}{\partial \bar{\mv}_1}=-\text{sgn}\bar{\mv}_2&\left\{2\f{\partial y(1)}{\partial \bar{\mv}_1}\cdot \f{1}{y(1)\sqrt{y(1)^2-1}}-\f{\partial y(r)}{\partial \bar{\mv}_1}\cdot \f{1}{y(r)\sqrt{y(r)^2-1}}\right.\nonumber\\
		&\quad\left.-\f{\partial y(\mathcal{X}(s))}{\partial \bar{\mv}_1}\cdot \f{1}{y(\mathcal{X}(s))\sqrt{y(\mathcal{X}(s))^2-1}}\right\}.
	\end{align}
	A direct calculation shows
	\begin{align*}
		\begin{split}
		&\f{\partial y(1)}{\partial \bar{\mv}_1}\cdot \f{1}{y(1)\sqrt{y(1)^2-1}}=\f{\bar{\mv}_1}{|\bar{\mv}|^2}\cdot\f{1}{\sqrt{\f{1}{r^2}\f{|\bar{\mv}|^2}{|\bar{\mv}_2|^2}-1}}=-\f{|\bar{\mv}_1|}{|\bar{\mv}|^2}\cdot\f{1}{\sqrt{\f{1}{r^2}\f{|\bar{\mv}|^2}{|\bar{\mv}_2|^2}-1}},\\
		&\f{\partial y(r)}{\partial \bar{\mv}_1}\cdot \f{1}{y(r)\sqrt{y(r)^2-1}}=\f{\bar{\mv}_1}{|\bar{\mv}|^2}\cdot\f{1}{\sqrt{\f{r^2}{r^2}\f{|\bar{\mv}|^2}{|\bar{\mv}_2|^2}-1}}=-\f{|\bar{\mv}_1|}{|\bar{\mv}|^2}\cdot\f{|\bar{\mv}_2|}{|\bar{\mv}_1|},\\
       &\f{\partial y(\mathcal{X}(s))}{\partial \bar{\mv}_1}\cdot \f{1}{y(\mathcal{X}(s))\sqrt{y(\mathcal{X}(s))^2-1}}\\\
		&\qquad\qquad=\f{\bar{\mv}_1}{|\bar{\mv}|^2}\cdot\f{1}{\sqrt{\f{\mathcal{X}^2}{r^2}\f{|\bar{\mv}|^2}{|\bar{\mv}_2|^2}-1}}+\f{1}{\mathcal{X}}\cdot \f{1}{\sqrt{\f{\mathcal{X}^2}{r^2}\f{|\bar{\mv}|^2}{|\bar{\mv}_2|^2}-1}}\f{\partial \mathcal{X}}{\partial \bar{\mv}_1},
	\end{split}
	\end{align*}
	which, together with \eqref{6-27}, yields that 
	\begin{align}\label{6-28}
		\f{\partial\Phi}{\partial \bar{\mv}_1}=\,&\text{sgn}\bar{\mv}_2\left\{\f{|\bar{\mv}_1|}{|\bar{\mv}|^2}\left[\f{2}{\sqrt{\f{1}{r^2}\f{|\bar{\mv}|^2}{|\bar{\mv}_2|^2}-1}}-\f{|\bar{\mv}_2|}{|\bar{\mv}_1|}-\f{1}{\sqrt{\f{\mathcal{X}^2}{r^2}\f{|\bar{\mv}|^2}{|\bar{\mv}_2|^2}-1}}\right]\right.\nonumber\\
		&\qquad\qquad\left.+\f{1}{\mathcal{X}}\cdot \f{1}{\sqrt{\f{\mathcal{X}^2}{r^2}\f{|\bar{\mv}|^2}{|\bar{\mv}_2|^2}-1}}\f{\partial \mathcal{X}}{\partial \bar{\mv}_1}\right\}
	\end{align}
	{\noindent \it Step 3. Calculation on $\dis \f{\partial\Phi}{\partial \bar{\mv}_2}$.} We have from \eqref{6-26} that 
	\begin{align}\label{6-29}
		\f{\partial\Phi}{\partial \bar{\mv}_2}=-\text{sgn}\bar{\mv}_2&\left\{2\f{\partial y(1)}{\partial \bar{\mv}_2}\cdot \f{1}{y(1)\sqrt{y(1)^2-1}}-\f{\partial y(r)}{\partial \bar{\mv}_2}\cdot \f{1}{y(r)\sqrt{y(r)^2-1}}\right.\nonumber\\
		&\left.\quad-\f{\partial y(\mathcal{X}(s))}{\partial \bar{\mv}_2}\cdot \f{1}{y(\mathcal{X}(s))\sqrt{y(\mathcal{X}(s))^2-1}}\right\}.
	\end{align}
	A direct calculation shows
	\begin{align*}
		\begin{split}
			&\f{\partial y(1)}{\partial \bar{\mv}_2}\cdot \f{1}{y(1)\sqrt{y(1)^2-1}}=-\text{sgn}\bar{\mv}_2\f{\bar{\mv}_1^2}{|\bar{\mv}|^2|\bar{\mv}_2|}\cdot\f{1}{\sqrt{\f{1}{r^2}\f{|\bar{\mv}|^2}{|\bar{\mv}_2|^2}-1}},\\
			&\f{\partial y(r)}{\partial \bar{\mv}_2}\cdot \f{1}{y(r)\sqrt{y(r)^2-1}}=-\text{sgn}\bar{\mv}_2\f{\bar{\mv}_1^2}{|\bar{\mv}|^2|\bar{\mv}_2|}\cdot\f{1}{\sqrt{\f{|\bar{\mv}|^2}{|\bar{\mv}_2|^2}-1}}=-\text{sgn}\bar{\mv}_2\f{\bar{\mv}_1^2}{|\bar{\mv}|^2|\bar{\mv}_2|}\cdot\f{|\bar{\mv}_2|}{|\bar{\mv}_1|},\\
			&\f{\partial y(\mathcal{X}(s))}{\partial \bar{\mv}_2}\cdot \f{1}{y(\mathcal{X}(s))\sqrt{y(\mathcal{X}(s))^2-1}}\\
			&\qquad\qquad=-\text{sgn}\bar{\mv}_2\f{\bar{\mv}_1^2}{|\bar{\mv}|^2|\bar{\mv}_2|}\cdot\f{1}{\sqrt{\f{\mathcal{X}^2}{r^2}\f{|\bar{\mv}|^2}{|\bar{\mv}_2|^2}-1}}+\f{1}{\mathcal{X}}\cdot \f{1}{\sqrt{\f{\mathcal{X}^2}{r^2}\f{|\bar{\mv}|^2}{|\bar{\mv}_2|^2}-1}}\f{\partial \mathcal{X}}{\partial \bar{\mv}_2},
		\end{split}
	\end{align*}
	which, together with \eqref{6-29}, yields that 
	\begin{align}\label{6-30}
		\f{\partial\Phi}{\partial \bar{\mv}_2}=\,&\f{\bar{\mv}_1^2}{|\bar{\mv}|^2|\bar{\mv}_2|}\left[\f{2}{\sqrt{\f{1}{r^2}\f{|\bar{\mv}|^2}{|\bar{\mv}_2|^2}-1}}-\f{|\bar{\mv}_2|}{|\bar{\mv}_1|}-\f{1}{\sqrt{\f{\mathcal{X}^2}{r^2}\f{|\bar{\mv}|^2}{|\bar{\mv}_2|^2}-1}}\right]\nonumber\\
		&+\text{sgn}\bar{\mv}_2\f{1}{\mathcal{X}}\cdot \f{1}{\sqrt{\f{\mathcal{X}^2}{r^2}\f{|\bar{\mv}|^2}{|\bar{\mv}_2|^2}-1}}\f{\partial \mathcal{X}}{\partial \bar{\mv}_2}.
	\end{align}
	
	{\noindent \it Step 4. Conclusion.} For simplicity of presentation, we denote	 
	\begin{align*}
		\mathbf{E}=\f{2}{\sqrt{\f{1}{r^2}\f{|\bar{\mv}|^2}{|\bar{\mv}_2|^2}-1}}-\f{|\bar{\mv}_2|}{|\bar{\mv}_1|}-\f{1}{\sqrt{\f{\mathcal{X}^2}{r^2}\f{|\bar{\mv}|^2}{|\bar{\mv}_2|^2}-1}}.
	\end{align*}
	It follows from \eqref{6-8}, \eqref{6-9}, \eqref{6-28} and \eqref{6-30} that 
	\begin{align}\label{6-31}
		&\mathcal{X}\left[\f{\partial\mathcal{X}}{\partial \bar{\mv}_1} \f{\partial\Phi}{\partial \bar{\mv}_2}- \f{\partial\mathcal{X}}{\partial \bar{\mv}_2} \f{\partial\Phi}{\partial \bar{\mv}_1}\right]
		\nonumber\\
		&=\mathcal{X}\f{\partial\mathcal{X}}{\partial \bar{\mv}_1}\cdot \left\{\f{\bar{\mv}_1^2}{|\bar{\mv}|^2|\bar{\mv}_2|}\mathbf{E}+\text{sgn}\bar{\mv}_2\f{1}{\mathcal{X}}\cdot \f{1}{\sqrt{\f{\mathcal{X}^2}{r^2}\f{|\bar{\mv}|^2}{|\bar{\mv}_2|^2}-1}}\f{\partial \mathcal{X}}{\partial \bar{\mv}_2}\right\}\nonumber\\
		&\quad-\mathcal{X} \f{\partial\mathcal{X}}{\partial \bar{\mv}_2}\cdot\text{sgn}\bar{\mv}_2\left\{\f{|\bar{\mv}_1|}{|\bar{\mv}|^2}\mathbf{E}+\f{1}{\mathcal{X}}\cdot \f{1}{\sqrt{\f{\mathcal{X}^2}{r^2}\f{|\bar{\mv}|^2}{|\bar{\mv}_2|^2}-1}}\f{\partial \mathcal{X}}{\partial \bar{\mv}_1}\right\}\nonumber\\
		&=\left\{-|\bar{\mv}_1|(s-t_{1,0})^2-|\bar{\mv}|^2\f{\partial t_{1,0}}{\partial \bar{\mv}_1}(s-t_{1,0})+\f{|\bar{\mv}_1|\bar{\mv}_2^2}{|\bar{\mv}|^4}r^2\right\}\cdot\f{\bar{\mv}_1^2}{|\bar{\mv}|^2|\bar{\mv}_2|}\mathbf{E}\nonumber\\
		&\quad -\left\{\text{sgn}\bar{\mv}_2|\bar{\mv}_2|(s-t_{1,0})^2-|\bar{\mv}|^2\f{\partial t_{1,0}}{\partial \bar{\mv}_2}(s-t_{1,0})+\text{sgn}\bar{\mv}_2\f{|\bar{\mv}_1|^2|\bar{\mv}_2|}{|\bar{\mv}|^4}r^2\right\}\cdot \text{sgn}\bar{\mv}_2\f{|\bar{\mv}_1|}{|\bar{\mv}|^2}\mathbf{E}\nonumber\\
		&=-\f{|\bar{\mv}_1|}{|\bar{\mv}_2|}(s-t_{1,0})\left\{(s-t_{1,0})-\left[\text{sgn}\bar{\mv}_2|\bar{\mv}_2|\f{\partial t_{1,0}}{\partial \bar{\mv}_2}-|\bar{\mv}_1|\f{\partial t_{1,0}}{\partial \bar{\mv}_1}\right]\right\}\cdot\mathbf{E}.
	\end{align}
	
Using \eqref{6-2}, one has that
	\begin{align}
		t_{1,0}=t-\left(\f{\bar{\mv}_1}{|\bar{\mv}|^2}r+\f{2}{|\bar{\mv}|^2}\sqrt{|\bar{\mv}|^2-\bar{\mv}_2^2r^2}\right).
	\end{align}
	A direct calculation shows
	\begin{align*}
		\begin{split}
		\f{\partial t_{1,0}}{\partial \bar{\mv}_1}&=-\f{1}{|\bar{\mv}|^2}r+2\f{\bar{\mv}_1^2}{|\bar{\mv}|^4}r-\f{4|\bar{\mv}_1|}{|\bar{\mv}|^4}\sqrt{|\bar{\mv}|^2-\bar{\mv}_2^2r^2}+\f{2|\bar{\mv}_1|}{|\bar{\mv}|^2\sqrt{|\bar{\mv}|^2-\bar{\mv}_2^2r^2}},\\
		\f{\partial t_{1,0}}{\partial \bar{\mv}_2}&=\text{sgn}\bar{\mv}_2\cdot \left[-\f{2|\bar{\mv}_1||\bar{\mv}_2|}{|\bar{\mv}|^4}r+\f{4|\bar{\mv}_2|}{|\bar{\mv}|^4}\sqrt{|\bar{\mv}|^2-\bar{\mv}_2^2r^2}-\f{2|\bar{\mv}_2|(1-r^2)}{|\bar{\mv}|^2\sqrt{|\bar{\mv}|^2-\bar{\mv}_2^2r^2}}\right],
		\end{split}
	\end{align*}
	which yields that 
	\begin{align}\label{6-10}
		&\text{sgn}\bar{\mv}_2|\bar{\mv}_2|\f{\partial t_{1,0}}{\partial \bar{\mv}_2}-|\bar{\mv}_1|\f{\partial t_{1,0}}{\partial \bar{\mv}_1}\nonumber\\
		&=\f{|\bar{\mv}_1|}{|\bar{\mv}|^2}r-\f{2|\bar{\mv}_1|}{|\bar{\mv}|^2}r+\f{4}{|\bar{\mv}|^2}\sqrt{|\bar{\mv}|^2-\bar{\mv}_2^2r^2}-\f{2(|\bar{\mv}|^2-\bar{\mv}_2^2r^2)}{|\bar{\mv}|^2\sqrt{|\bar{\mv}|^2-\bar{\mv}_2^2r^2}}\nonumber\\
		&=-\f{|\bar{\mv}_1|}{|\bar{\mv}|^2}r+\f{2}{|\bar{\mv}|^2}\sqrt{|\bar{\mv}|^2-\bar{\mv}_2^2r^2}=t-t_{1,0}.
	\end{align}
	
	Substituting \eqref{6-10} into \eqref{6-31}, one obtains 
	\begin{align*}
		\mathcal{X}\left[\f{\partial\mathcal{X}}{\partial \bar{\mv}_1} \f{\partial\Phi}{\partial \bar{\mv}_2}- \f{\partial\mathcal{X}}{\partial \bar{\mv}_2} \f{\partial\Phi}{\partial \bar{\mv}_1}\right]=-\f{|\bar{\mv}_1|}{|\bar{\mv}_2|}(s-t_{1,0})\left[(s-t_{1,0})-(t-t_{1,0})\right]\cdot\mathbf{E},
	\end{align*}
	which concludes \eqref{6-32} immediately. Therefore the proof of Lemma \ref{lem6.2-0} is completed.
\end{proof}
Assume for $t\in[0,\tau^\d]$,
\begin{align}\label{5.38-1}
	\sup_{0\leq s\leq t}\|\v^4h^\v_R\|_{L^\infty}\leq 1,\quad \sup_{0\leq s\leq t}\|\v f_R^\v\|_{L^2}\leq 1.
\end{align}
We have the following estimate.
\begin{Lemma}\label{lem6.2}
	Let $t\in[0,\tau^\d]$, $\mathbf{k}\geq6,$ it holds that
	\begin{align}\label{6.7}
		\sup_{0\leq s\leq t}\|\v^3h^\v_R(s)\|_{L^\infty}\leq  C(t)\{\|\v^3h^\v_R(0)\|_{L^\infty}+\v^{N-1}+\v^{\mathfrak{b}}\}+C\sup_{0\leq s\leq t}\|f^\v_R(s)\|_{L^2}.
	\end{align}
\end{Lemma}
\begin{proof}
	Since the proof for Lemma \ref{lem6.2} is quite long. We divide it into several steps.\\
		{\it Step 1.} Integrating along the backward characteristic, we have from \eqref{6.3} that
	\begin{align}\label{5.40-0}
		h^\v_R(t,r,\phi,v)=&h^\v_R(0,\mathcal{X}(0),\Phi(0),\mathcal{V}(0))\cdot \exp\{-\f{1}{\v^2}\tilde{\nu} t\}\nonumber\\
		&+\f{1}{\v^2}\int_0^t\tilde{K}_{M,\tilde{w}}h^\v_R(s,\mathcal{X}(s),\Phi(s),\mathcal{V}(s))\cdot e^{-\f{1}{\v^2}\tilde{\nu} (t-s)}\,ds\nonumber\\
		&+\int_0^t\tilde{w}_{\mathbf{k}}\bar{S}_R\cdot \exp\{-\f{1}{\v^2}\tilde{\nu} (t-s)\}\,ds.
	\end{align}
	Noting \eqref{5.2-0}, by a simple iterative argument, one gets
	\begin{align}\label{6.8}
		\|h^\v_R(t)\|_{L^\infty}&\leq  \|h^\v_R(0)\|_{L^\infty}+\sup_{0\leq s\leq t}\|\tilde{K}_{M,\tilde{w}}h^\v_R\|_{L^\infty}+\v^2\sup_{0\leq s\leq t}\|\tilde{\nu}^{-1}\tilde{w}_{\mathbf{k}}\bar{S}_R\|_{L^\infty}\nonumber\\
		&\leq \|h^\v_R(0)\|_{L^\infty}+\sup_{0\leq s\leq t}\|\tilde{K}_{M}\mf^\v_R\|_{L^\infty L^\infty_{\mathbf{k}-1}}+\v^2\sup_{0\leq s\leq t}\|\tilde{\nu}^{-1}\tilde{w}_{\mathbf{k}}\bar{S}_R\|_{L^\infty}\nonumber\\
		&\leq ...\nonumber\\
		&\leq \|h^\v_R(0)\|_{L^\infty}+\sup_{0\leq s\leq t}\|\mf^\v_R\|_{L^\infty L^2}+\v^2\sup_{0\leq s\leq t}\|\tilde{\nu}^{-1}\tilde{w}_{\mathbf{k}}\bar{S}_R\|_{L^\infty},
	\end{align}
	where $\mathbf{f}^\v_R:=\f{1}{\tilde{w}_{\mathbf{k}}}h^\v_R=\f{F_R^\v}{\sqrt{\mu_M}}$. The norms $\|g\|_{L^\infty}$, $\|g\|_{L^\infty L^\infty_\beta}$ and $\|g\|_{L^\infty L^2}$ mean respectively  $\|g\|_{L^\infty_{r,\phi,\bar{\mv}}}$(or equivalently $\|g\|_{L^\infty(B_1\times \R^2)}$), $\|\tilde{w}_{\beta}g\|_{L^\infty_{r,\phi,\bar{\mv}}}$ and $\|g\|_{L^\infty_{r,\phi}L^2_{\bar{\mv}}}$. \\
	{\it Step 2.} To close \eqref{6.8}, we still need to estimate $\dis \int_{\R^2}|\mathbf{f}_R^\v(\bar{\mv})|^2\, d\bar{\mv}.$ Similar to \eqref{5.40-0}, one has
	\begin{align*}
		\mf^\v_R(t,r,\phi,\bar{\mv})=&\mf^\v_R(0,\mathcal{X}(0),\Phi(0),\mathcal{V}(0))\exp\{-\f{1}{\v^2}\tilde{\nu} t\}\nonumber\\
		&+\f{1}{\v^2}\int_0^t\tilde{K}_{M}\mf^\v_R(s,\mathcal{X}(s),\Phi(s),\mathcal{V}(s))\cdot \exp\{-\f{1}{\v^2}\tilde{\nu} (t-s)\}\,ds\nonumber\\
		&+\int_0^t\bar{S}_R\cdot \exp\{-\f{1}{\v^2}\tilde{\nu} (t-s)\}\,ds.
	\end{align*}
	It is obvious that
	\begin{align}\label{6.9}
		\begin{split}
		\int_{\R^2}|\mf^\v_R(0,\mathcal{X}(0),\Phi(0),\mathcal{V}(0))|^2\exp\{-\f{2}{\v^2}\tilde{\nu} t\}d\bar{\mv}\leq C\|h^\v_R(0)\|_{L^\infty_{x,v}}^2,\\
		\int_{\R^2}\left\{\int_0^t\bar{S}_R\cdot \exp\{-\f{1}{\v^2}\tilde{\nu} (t-s)\}\,ds\right\}^2d\bar{\mv}\leq \v^4\sup_{0\leq s\leq t}\|\tilde{\nu}^{-1}\tilde{w}_{\mathbf{k}}\bar{S}_R\|_{L^\infty}^2.
		\end{split}
	\end{align}
	
	Next we consider 
	\begin{align*}
		&\mathcal{P}:=\int_{\R^2}\f{1}{\v^4}\left\{\int_0^t\tilde{K}_M\mf^\v_R\cdot \exp\{-\f{1}{\v^2}\tilde{\nu} (t-s)\}ds\right\}^2d\bar{\mv}\nonumber\\
		&=\int_{\R^2}\f{1}{\v^4}\left\{\int_0^t\int_{\R^2}\tilde{k}_M(\mathcal{V}_{cl}(s),v')\mf^\v_R(s,\mathcal{X}_{cl}(s),\Phi_{cl}(s),v')dv'\cdot \exp\{-\f{1}{\v^2}\tilde{\nu} (t-s)\}ds\right\}^2d\bar{\mv}.
	\end{align*}
	Since we need $\|\mf^\v_R\|_{L^\infty L^2}$, we only consider $\mathcal{P}$ for $r\neq 0$. We divide the discussion of $\mathcal{P}$ into several cases.\\
	{\it Case 1. $|\bar{\mv}|\geq \mathfrak{N}$.} Denote $\Theta_1:=\mathbf{1}_{\{|\bar{\mv}|\geq \mathfrak{N}\}}$. Noting $|\mathcal{V}_{cl}(s)|=|\bar{\mv}|$, we have
	\begin{align}\label{6.11}
		\mathcal{P}_1&=\int_{\R^2}\f{1}{\v^4}\mathbf{1}_{\{|\bar{\mv}|\geq \mathfrak{N}\}}\{\cdot\cdot\cdot\}^2\, d\bar{\mv}\nonumber\\
		&\leq \sup_{0\leq s\leq t}\|h^\v_R(s)\|_{L^\infty}^2\cdot\int_{\R^2}\f{\Theta_1}{\v^4}\left\{\int_0^t(1+|\bar{\mv}|)^{-\mathbf{k}-1}\cdot \exp\{-\f{1}{\v^2}\tilde{\nu} (t-s)\}ds\right\}^2d\bar{\mv}\nonumber\\
		&\leq \f{1}{\mathfrak{N}}\sup_{0\leq s\leq t}\|h^\v_R(s)\|_{L^\infty}^2.
	\end{align}
	{\it Case 2. $|\bar{\mv}|\leq \mathfrak{N},|v'|\geq 2\mathfrak{N}$.}  Denote $\Theta_2:=\mathbf{1}_{\{|\bar{\mv}|\leq \mathfrak{N},|v'|\geq 2\mathfrak{N}\}}.$ Since $|\mathcal{V}_{cl}(s)|=|\bar{\mv}|$, then it holds that $|\mathcal{V}_{cl}(s)-v'|\geq \mathfrak{N}$. In this case, we have from \eqref{6-7} that there exists a positive constant $c_4>0$ such that 
	\begin{align}\label{6.12}
		\mathcal{P}_2&=\int_{\R^2}\f{1}{\v^4}\mathbf{1}_{\{|\bar{\mv}|\leq \mathfrak{N},|v'|\geq 2\mathfrak{N}\}}\{\cdot\cdot\cdot\}^2\, d\bar{\mv}\nonumber\\
		&\leq e^{-c_4\mathfrak{N}^2}\sup_{0\leq s\leq t}\|h^\v_R(s)\|_{L^\infty}^2\nonumber\\
		&\quad\times\int_{\R^2}\f{1}{\v^4}\left\{\int_0^t\int_{\R^2}\Theta_2k_M(\mathcal{V}_{cl}(s),v')e^{c_4|\mathcal{V}_{cl}-v'|^2}(1+v'|^{-\mathbf{k}})\cdot\exp\{-\f{1}{\v^2}\tilde{\nu} (t-s)\}ds\right\}^2d\bar{\mv}\nonumber\\
		&\leq \f{1}{\mathfrak{N}}\sup_{0\leq s\leq t}\|h^\v_R(s)\|_{L^\infty}^2.
	\end{align}
	{\it Case 3. $|\bar{\mv}|\leq \mathfrak{N},|v'|\leq 2\mathfrak{N}$ and further assume: $|\bar{\mv}_{1}/\bar{\mv}_{2}|\leq \bar{m},$ or $ |\bar{\mv}_2|\leq \bar{m}$ where $\bar{m}\ll 1$.} Denote the corresponding indicative function as $\Theta_3$. Thus
	 we have 
	\begin{align*}
		|\bar{\mv}_{1}|\leq \bar{m}\mathfrak{N} \quad \text{or}\quad |\bar{\mv}_2|\leq \bar{m},
	\end{align*} 
	which yields that 
	\begin{align}\label{6.13}
		\mathcal{P}_3&=\int_{\R^2}\f{\Theta_3}{\v^4}\{\cdot\cdot\cdot\}^2\, d\bar{\mv}\nonumber\\
		&\leq \sup_{0\leq s\leq t}\|h_R^\v(s)\|_{L^\infty}^2\cdot \int_{\R^2}(1+|\bar{\mv}|)^{-\mathbf{k}-1}\Theta_3(\bar{\mv}) \,d\bar{\mv}\leq C\bar{m}\mathfrak{N}\sup_{0\leq s\leq t}\|h_R^\v(s)\|_{L^\infty}^2.
	\end{align}
	{\it Case 4. $|\bar{\mv}|\leq \mathfrak{N}, |v'|\leq 2\mathfrak{N}, |\bar{\mv}_{1}/\bar{\mv}_{2}|\geq \bar{m},|\bar{\mv}_2|\geq \bar{m},$ and $ |s-t_{1,0}|\leq m_1$ or $t-s\leq m_1$.}  Denote the corresponding indicative function as $\Theta_4$. The length of time integral domain is less than $2m_1$, hence 
	\begin{align}\label{6.14}
		\mathcal{P}_4&=\int_{\R^2}\f{\Theta_4}{\v^4}\{\cdot\cdot\cdot\}^2\, d\bar{\mv}\nonumber\\
		&\leq \sup_{0\leq s\leq t}\|h_R^\v(s)\|_{L^\infty}^2\cdot \int_{\R^2}\f{1}{\v^4}\left\{\int_0^t \Theta_4\, ds\right\}^2\, d\bar{\mv}\leq C\f{m_1^2}{\v^4}\sup_{0\leq s\leq t}\|h_R^\v(s)\|_{L^\infty}^2.
	\end{align}
	{\it Case 5. $|\bar{\mv}|\leq \mathfrak{N},|v'|\leq 2\mathfrak{N}, |\bar{\mv}_{1}/\bar{\mv}_{2}|\geq \bar{m}, |\bar{\mv}_2|\geq \bar{m}$ and $|s-t_{1,0}|\geq m_1, t-s\geq m_1$. We also assume $$\mathcal{X}(s)\in[(1-m_2)r,(1+m_2)r]\quad \text{with}\quad m_2\ll\bar{m}\ll1.$$} Denote the corresponding indicative function as $\Theta_5$.\\
	
Noting $|\bar{\mv}_{1}/\bar{\mv}_{2}|\geq \bar{m}$, one has 
	\begin{align*}
		1-x_+=1-\f{r}{\sqrt{1+\f{\bar{\mv}_{1}^2}{\bar{\mv}_{2}^2}}}\geq 1-\f{1}{\sqrt{1+\f{\bar{\mv}_{1}^2}{\bar{\mv}_{2}^2}}}\geq \f{\bar{m}^2}{6}.
	\end{align*}
	Since for $l\geq 1$, one has 
	\begin{align*}
		\int_{t_{l+1}}^{t_l}|\mathcal{V}_1|\, ds=2(1-x_+),
	\end{align*}
which yields that
	\begin{align*}
		t_l-t_{l+1}\geq \f{\bar{m}^2}{3\mathfrak{N}}, 
	\end{align*}
	and the number of collisions $k\leq k_{max}=(3\mathfrak{N}t)/\bar{m}^2.$ 
	
	We also have that 
	\begin{align*}
		|\mathcal{V}_1(s)|^2=|\bar{\mv}|^2-\bar{\mv}_2^2\f{r^2}{\mathcal{X}^2(s)}\in\left[|\bar{\mv}|^2-\f{\bar{\mv}_2^2}{(1-m_2)^2},|\bar{\mv}|^2-\f{\bar{\mv}_2^2}{(1+m_2)^2}\right],
	\end{align*}
	which, together with $0<m_2\ll\bar{m}\ll1$, yields that 
	\begin{align}\label{5.47-0}
		|\mathcal{V}_1(s)|^2\geq \bar{\mv}_1^2-\f{3m_2}{(1-m_2)^2}\bar{\mv}_2^2\geq \left[1-\f{3m_2}{\bar{m}^2(1-m_2)^2}\right]\bar{\mv}_1^2\geq \f{\bar{\mv}_1^2}{4}.
	\end{align} 
	Thus we have from \eqref{5.47-0} and \eqref{6.4} that
	\begin{align*}
		\left|\f{d\mathcal{X}}{ds}\right|=|\mathcal{V}_1|\geq \f{|\bar{\mv}_1|}{2}\geq \f{\bar{m}^2}{2}.
	\end{align*}
	Then we use change of variable
	\begin{align}\label{6.15}
		\mathcal{P}_5&=\int_{\R^2}\f{\Theta_5}{\v^4}\{\cdot\cdot\cdot\}^2\, d\bar{\mv}\nonumber\\
		&\leq C\sup_{0\leq s\leq t}\|h_R^\v(s)\|_{L^\infty}^2\cdot \f{1}{\v^4}\left\{\int_0^t e^{-\f{\tilde{\nu}_0}{\v^2}(t-s)}\Theta_5ds\right\}^2\nonumber\\
		&\leq  Ck\sup_{0\leq s\leq t}\|h_R^\v(s)\|_{L^\infty}^2\cdot \f{1}{\v^4}\left\{\int_0^1\Theta_5 \f{1}{|\mathcal{V}_1|}d\mathcal{X}\right\}^2\nonumber\\
		&\leq\f{C\mathfrak{N}t}{\bar{m}^6}\left(\f{m_2}{\v^2}\right)^2\sup_{0\leq s\leq t}\|h_R^\v(s)\|_{L^\infty}^2.
	\end{align}
	{\it Case 6. $|\bar{\mv}|\leq \mathfrak{N},|v'|\leq 2\mathfrak{N}, |\bar{\mv}_{1}/\bar{\mv}_{2}|\geq \bar{m}, |\bar{\mv}_2|\geq \bar{m}$, $|s-t_{1,0}|\geq m_1$ and $t-s\geq m_1$. We also assume 
	\begin{align}\label{6-56}
			\mathcal{X}\leq( 1-m_2)r\quad \text{or} \quad \mathcal{X}\geq (1+m_2)r.
	\end{align}} 
	Denote the corresponding indicative function as $\Theta_6$. Then we have 
	\begin{align*}
		\mathcal{P}_6=\int_{\R^2}\f{\Theta_6}{\v^4}\{\cdot\cdot\cdot\}^2\, d\bar{\mv},
	\end{align*}	
	and divide $\mathcal{P}_6=\mathcal{P}_{61}+\mathcal{P}_{62}$ for $\bar{\mv}_{1}>0$ and $\bar{\mv}_{1}<0$. We rewrite $\mathcal{P}_{61},\mathcal{P}_{62}$ as
	\begin{align*}
		\mathcal{P}_{61}=&\int_{\R^2}\Theta_6\mathbf{1}_{\{\bar{\mv}_1>0\}}\f{1}{\v^4}\\
		&\quad \times\left\{\left(\int_{t_1}^t+\int_0^{t_1}\right)\int_{\R^2}\tilde{k}_M(\mathcal{V}(s),v')\mf_R^\v(\mathcal{X}(s),\Phi(s),v')dv'\cdot e^{-\f{1}{\v^2}\tilde{\nu}(t-s)}ds\right\}^2\, d\bar{\mv},\nonumber\\
		\mathcal{P}_{62}=& \int_{\R^2}\Theta_6\mathbf{1}_{\{\bar{\mv}_1<0\}}\f{1}{\v^4}\nonumber\\
		&\quad\times\left\{\left(\int_{t_{1,0}}^t+\int_0^{t_{1,0}}\right)\int_{\R^2}\tilde{k}_M(\mathcal{V}(s),v')\mf_R^\v(\mathcal{X}(s),\Phi(s),v')dv'\cdot e^{-\f{1}{\v^2}\tilde{\nu}(t-s)}ds\right\}^2\, d\bar{\mv}.
	\end{align*}
	%From the definition of $x_+$ in \eqref{6.5-1}, one has that for $l\geq 1$,
	%\begin{align}\label{6.17-0}
	%	\tilde{\nu}(t_l-t_{l+1})\geq \int_s^t|\mathcal{V}_1|d\tau=2(1-x_+)\geq 2-\f{2}{\sqrt{1+\f{\bar{\mv}_{1}^2}{\bar{\mv}_{2}^2}}}\geq 2-\f{2}{\sqrt{1+\bar{m}^2}}\geq\f{\bar{m}^2}{3}.
	%\end{align}
	
	For $\bar{\mv}_{1}>0$, one gets
	\begin{align}\label{6.17-1}
		\tilde{\nu}(t-t_1)\geq 1-x_+\geq \f{\bar{m}^2}{6}.
	\end{align}
	It follows from \eqref{6.17-1} that
	\begin{align}\label{6.18}
		\int_0^{t_1}e^{-\f{1}{\v^2}\tilde{\nu}(t-s)}ds\leq \int_0^{t_1}e^{-\f{1}{\v^2}\tilde{\nu}(t_1-s)}ds\cdot e^{-\f{1}{\v^2}\cdot \f{\bar{m}^2}{6}}\leq C\v^2e^{-\f{1}{6\v^2}\bar{m}^2},
	\end{align}
which yields that
	\begin{align}\label{6.19}
		\mathcal{P}_{61}&\leq \int_{\R^2}\Theta_6\mathbf{1}_{\{\bar{\mv}_{1}>0\}}\f{1}{\v^4}\left\{\int_{t_1}^{t}\int_{\R^2}\tilde{k}_M(\mathcal{V}(s),v')\mf_R^\v(\mathcal{X}(s),\Phi(s),v')dv'\cdot e^{-\f{1}{\v^2}\tilde{\nu}(t-s)}ds\right\}^2d\bar{\mv}\nonumber\\
		&\qquad\qquad\qquad\qquad \qquad +Ce^{-\f{1}{3\v^2}\bar{m}^2}\|h_R^\v\|_{L^\infty}^2\nonumber\\
		&=:\mathcal{P}_{61R}+Ce^{-\f{1}{3\v^2}\bar{m}^2}\|h_R^\v\|_{L^\infty}^2.
	\end{align}
	
	For $\bar{\mv}_{1}<0$, one has 
	\begin{align*}
		\tilde{\nu}(t-t_{1,0})\geq 1-x_+\geq \f{\bar{m}^2}{6},
	\end{align*}
	then similar to \eqref{6.18}--\eqref{6.19}, one has 
	\begin{align}\label{6.20}
		\mathcal{P}_{62}&\leq \int_{\R^2}\Theta_6\mathbf{1}_{\{\bar{\mv}_{1}<0\}}\f{1}{\v^4}\left\{\int_{t_{1,0}}^{t}\int_{\R^2}\tilde{k}_M(\mathcal{V}(s),v')\mf_R^\v(s,\mathcal{X}(s),\Phi(s),v')dv'\cdot e^{-\f{1}{\v^2}\tilde{\nu}(t-s)}ds\right\}^2d\bar{\mv}\nonumber\\
		&\qquad\qquad\qquad\qquad  +Ce^{-\f{1}{3\v^2}\bar{m}^2}\|h_R^\v\|_{L^\infty}^2\nonumber\\
		&=:\mathcal{P}_{62R}+Ce^{-\f{1}{3\v^2}\bar{m}^2}\|h_R^\v\|_{L^\infty}^2.
	\end{align}
	
	{\it \underline{Estimate for $\mathcal{P}_{61R}$}}. We  rewrite $\mf_R^\v(s,\mathcal{X}(s),\Phi(s),v')$ as $\mf_R^\v(s,\mathbf{X}(s),v')$, with 
	\begin{align}
		\mathbf{X}_1(s)=\mathcal{X}(s)\cos \Phi(s),\quad \mathbf{X}_2(s)=\mathcal{X}(s)\sin \Phi(s).
	\end{align}
Noting for $\bar{\mv}_1>0$ and $s\geq t_1$, there is no collision with the boundary. Then it is better to use the original coordinate. From \eqref{6.05}, we actually have
	\begin{align*}
		\mathbf{X}_1(s)=r\cos\phi-v_1(t-s),\quad \mathbf{X}_2(s)=r\sin\phi-v_2(t-s),
	\end{align*}
	which yields that 
	\begin{align}\label{6-66}
		\left|\f{\partial{\mathbf{X}}}{{\partial v}}\right|=(t-s)^2\geq m_1^2.
	\end{align}
	Hence one obtains
	\begin{align}\label{6.26}
		\mathcal{P}_{61R}&=\int_{\R^2}\Theta_6\mathbf{1}_{\{\bar{\mv}_{1}>0\}}\f{1}{\v^4}\left\{\int_{t_1}^{t}\int_{\R^2}\tilde{k}_M(\mathcal{V}(s),v')\mf_R^\v(s,\mathcal{X}(s),\Phi(s),v')dv'\cdot e^{-\f{1}{\v^2}\tilde{\nu}(t-s)}ds\right\}^2d\bar{\mv}\nonumber\\ 
		&\leq \int_{\R^2}\Theta_6\mathbf{1}_{\{\bar{\mv}_{1}>0\}}\f{1}{\v^2}\int_{t_1}^{t}\left\{\int_{\R^2}\tilde{k}_M(\mathcal{V}(s),v')\mf_R^\v(s,\mathcal{X}(s),\Phi(s),v')dv'\right\}^2\cdot e^{-\f{1}{\v^2}\tilde{\nu}_0(t-s)}dsd\bar{\mv}\nonumber\\
		&\leq C_\mathfrak{N}\int_{\R^2}\Theta_6\mathbf{1}_{\{\bar{\mv}_{1}>0\}}\f{1}{\v^2}\int_{t_1}^{t}\left\{\int_{\R^2}|f_R^\v(s,\mathbf{X}(s),v')|^2dv'\right\}\cdot e^{-\f{1}{\v^2}\tilde{\nu}_0(t-s)}dsdv\nonumber\\
		&\leq \f{C_{\mathfrak{N}}}{m_1^2}\int_{\R^2}\Theta_6\mathbf{1}_{\{\bar{\mv}_{1}>0\}}\f{1}{\v^2}\int_{t_1}^t\left\{\int_{\R^2}|f_R^\v(s,\mathbf{X},v')|^2dv'\right\}\cdot e^{-\f{1}{\v^2}\tilde{\nu}_0(t-s)}\,  d\mathbf{X}\, ds\nonumber\\
		&\leq \f{C_{\mathfrak{N}}}{m_1^2}\sup_{0\leq s\leq t}\|f_R^\v(s)\|_{L^2}^2.
	\end{align}
	%Then for $k\geq1$,
	%\begin{align*}
	%	\f{\partial\phi_k}{\partial\bar{\mv}_{2}}\cdot \text{sgn}\bar{\mv}_{2}&=-\f{\partial y(r)}{\partial\bar{\mv}_{2}}\cdot\f{1}{y(r)\sqrt{y(r)^2-1}}-(2k-1)\f{\partial y(1)}{\partial\bar{\mv}_{2}}\cdot\f{1}{y_2\sqrt{y(1)^2-1}}\nonumber\\
	%	&=\text{sgn}\bar{\mv}_{2}\f{2|\bar{\mv}_{1}|}{|\bar{\mv}|^2}+(2k -1)\f{2\bar{\mv}_{1}^2}{\bar{\mv}_{2}\cdot |\bar{\mv}|^2}\cdot \f{1}{\sqrt{\f{1}{r^2}\left(1+\f{\bar{\mv}_{1}^2}{\bar{\mv}_{2}^2}\right)-1}}.
	%\end{align*}
	
	{\it \underline{Estimate for $\mathcal{P}_{62R}$}}.	
	For $\bar{\mv}_{1}<0$ and $s\in(t_1,t],$ since there is no collision with $\partial B_1$, then by similar arguments in \eqref{6-66}--\eqref{6.26}, we have 
	\begin{align}
		\int_{\R^2}\Theta_6\mathbf{1}_{\{\bar{\mv}_1<0\}}\f{1}{\v^4}\left\{\int_{t_1}^t(\cdot \cdot\cdot)\right\}^2\, d\bar{\mv}\leq \f{C_\mathfrak{N}}{m_1^2}\sup_{0\leq s\leq t}\|f_R^\v(s)\|_{L^2}^2.
	\end{align}
	
	For $\bar{\mv}_1<0,s\in(t_{1,0},t_1]$ and $\bar{\mv}_2\cdot r\neq 0$, we use Lemma \ref{lem6.2-0} to obtain 
	\begin{align}\label{6-67}
		&\mathcal{X}\left |\f{\partial(\mathcal{X},\Phi)}{\partial \bar{\mv}}(s)\right |\nonumber\\
		&=\f{|\bar{\mv}_1|}{|\bar{\mv}_2|}(s-t_{1,0})(t-s) \times\left(\f{|\bar{\mv}_2|}{|\bar{\mv}_1|}+\f{1}{\sqrt{\f{\mathcal{X}^2}{r^2}\f{|\bar{\mv}|^2}{|\bar{\mv}_2|^2}-1}}-\f{2}{\sqrt{\f{1}{r^2}\f{|\bar{\mv}|^2}{|\bar{\mv}_2|^2}-1}}\right)\nonumber\\
		&=	(s-t_{1,0})(t-s)\times\left( 1+\f{1}{\sqrt{\f{\mathcal{X}^2}{r^2}(1+\f{|\bar{\mv}_2|^2}{|\bar{\mv}_1|^2})-\f{|\bar{\mv}_2|^2}{|\bar{\mv}_1|^2}}}-\f{2}{\sqrt{\f{1}{r^2}(1+\f{|\bar{\mv}_2|^2}{|\bar{\mv}_1|^2})-\f{|\bar{\mv}_2|^2}{|\bar{\mv}_1|^2}}}\right).
	\end{align}
	
	For the case of $\mathcal{X}^2/r^2\geq (1+m_2)^2$, it holds 
	\begin{align*}
		\f{\mathcal{X}^2}{r^2}\left(1+\f{\bar{\mv}_{2}^2}{\bar{\mv}_{1}^2}\right)-\f{\bar{\mv}_{2}^2}{\bar{\mv}_{1}^2}\geq (1+m_2)^2,
	\end{align*}
	which yields that
	\begin{align}\label{6.24}
		&1+\f{1}{\sqrt{\f{\mathcal{X}^2}{r^2}(1+\f{|\bar{\mv}_2|^2}{|\bar{\mv}_1|^2})-\f{|\bar{\mv}_2|^2}{|\bar{\mv}_1|^2}}}-\f{2}{\sqrt{\f{1}{r^2}(1+\f{|\bar{\mv}_2|^2}{|\bar{\mv}_1|^2})-\f{|\bar{\mv}_2|^2}{|\bar{\mv}_1|^2}}}\nonumber\\
		&\geq 1-\f{1}{\sqrt{\f{1}{r^2}(1+\f{|\bar{\mv}_2|^2}{|\bar{\mv}_1|^2})-\f{|\bar{\mv}_2|^2}{|\bar{\mv}_1|^2}}}\geq1-\f{1}{1+m_2}\geq \f{m_2}{2}.
	\end{align}
	
	For the case of $\f{\mathcal{X}^2}{r^2}\leq (1-m_2)^2,$ one has 
	\begin{align*}
		\f{\mathcal{X}^2}{r^2}\left(1+\f{\bar{\mv}_{2}^2}{\bar{\mv}_{1}^2}\right)-\f{\bar{\mv}_{2}^2}{\bar{\mv}_{1}^2}\leq (1-m_2)^2,\quad 
		\f{1}{r^2}\left(1+\f{\bar{\mv}_{2}^2}{\bar{\mv}_{1}^2}\right)-\f{\bar{\mv}_{2}^2}{\bar{\mv}_{1}^2}\geq 1,
	\end{align*}
	which yields that
	\begin{align}\label{6-70}
		&1+\f{1}{\sqrt{\f{\mathcal{X}^2}{r^2}(1+\f{|\bar{\mv}_2|^2}{|\bar{\mv}_1|^2})-\f{|\bar{\mv}_2|^2}{|\bar{\mv}_1|^2}}}-\f{2}{\sqrt{\f{1}{r^2}(1+\f{|\bar{\mv}_2|^2}{|\bar{\mv}_1|^2})-\f{|\bar{\mv}_2|^2}{|\bar{\mv}_1|^2}}}\nonumber\\
		&\geq\f{1}{\sqrt{\f{\mathcal{X}^2}{r^2}(1+\f{|\bar{\mv}_2|^2}{|\bar{\mv}_1|^2})-\f{|\bar{\mv}_2|^2}{|\bar{\mv}_1|^2}}}-\f{1}{\sqrt{\f{1}{r^2}(1+\f{|\bar{\mv}_2|^2}{|\bar{\mv}_1|^2})-\f{|\bar{\mv}_2|^2}{|\bar{\mv}_1|^2}}}\nonumber\\
		&\geq \f{1}{1-m_2}-1\geq m_2.
	\end{align}
	Hence it follows from \eqref{6-67}--\eqref{6-70} that  
	\begin{align}\label{6.28-1}
		\mathcal{X}\left |\f{\partial(\mathcal{X},\Phi)}{\partial \bar{\mv}}(s)\right |\geq \f{m_2}{2}(s-t_{1,0})(t-s)\geq \f{m_2}{2}m_1^2.
	\end{align}
Using \eqref{6-66} and \eqref{6.28-1}, similar to \eqref{6.26}, one has 
	\begin{align*}
		\mathcal{P}_{62R}\leq \f{C_\mathfrak{N}}{ m_1^2m_2}\sup_{0\leq s\leq t}\|f_R^\v(s)\|_{L^2}^2,
	\end{align*}
	which, together with \eqref{6.19}, \eqref{6.20} and \eqref{6.26}, yields that
	\begin{align}\label{6.29}
		\mathcal{P}_6\leq \f{C_\mathfrak{N}}{ m_1^2m_2}\sup_{0\leq s\leq t}\|f_R^\v(s)\|_{L^2}^2+Ce^{-\f{1}{3\v^2}\bar{m}^2}\sup_{0\leq s\leq t}\|h_R^\v(s)\|_{L^\infty}^2.
	\end{align}

	We have from \eqref{6.11}--\eqref{6.14}, \eqref{6.15} and \eqref{6.29} that 
	\begin{align*}
		\mathcal{P}\leq & \left\{\f{1}{\mathfrak{N}}+C\left(\bar{m}\mathfrak{N}+Ce^{-\f{1}{3\v^2}\bar{m}^2}\right)\right\}\sup_{0\leq s\leq t}\|\tilde{w}_{\mathbf{k}}\mf^\v_R(s)\|_{L^\infty}^2\nonumber\\
		&+\left(Cm_1^2+\f{C\mathfrak{N}t}{\bar{m}^6}m_2^2\right)\cdot\f{1}{\v^4}\sup_{0\leq s\leq t}\|\tilde{w}_{\mathbf{k}}\mf^\v_R(s)\|_{L^\infty}^2+\f{C_\mathfrak{N}}{ m_1^2m_2}\sup_{0\leq s\leq t}\|f^\v_R(s)\|_{L^2}^2.
	\end{align*}
	Choosing $\bar{m}=\f{1}{\mathfrak{N}^2}$ and $\mathfrak{N}>1$ suitably large, and $m_1=\bar{m}\v^2, m_2=\bar{m}^4\v^2$, one obtains that 
	\begin{align*}
		\mathcal{P}\leq o(1)\sup_{0\leq s\leq t}\|h_R^\v(s)\|_{L^\infty}^2+\f{C_\mathfrak{N}}{\v^6}\sup_{0\leq s\leq t}\|f_R^\v(s)\|_{L^2}^2,
	\end{align*}
	which, together with \eqref{6.9}, yields that
	\begin{align}\label{5.62-1}
		\|\mf^\v_R\|_{L^\infty L^2}^2\leq&  C\|h^\v_R(0)\|_{L^\infty}^2+o(1) \sup_{0\leq s\leq t}\|h_R^\v(s)\|_{L^\infty}^2+\f{C_\mathfrak{N}}{\v^6}\sup_{0\leq s\leq t}\|f^\v_R(s)\|_{L^2}^2\nonumber\\
		&+C\v^4\sup_{0\leq s\leq t}\|\tilde{\nu}^{-1}\tilde{w}_{\mathbf{k}}\bar{S}_R(s)\|_{L^\infty}^2.
	\end{align}
	
	Substituting \eqref{5.62-1} into \eqref{6.8}, one has 
	\begin{align}\label{6.30}
		\|h^\v_R\|_{L^\infty}\leq  C\|h^\v_R(0)\|_{L^\infty}+\f{C}{\v^3}\sup_{0\leq s\leq t}\|f^\v_R(s)\|_{L^2}+\v^2\sup_{0\leq s\leq t}\|\tilde{\nu}^{-1}\tilde{w}_{\mathbf{k}}\bar{S}_R(s)\|_{L^\infty}.
	\end{align}
	Noting the definition of $\bar{S}_R$ in \eqref{6.06}, one has 
	\begin{align*}
		|\tilde{\nu}^{-1}\tilde{w}_{\mathbf{k}}\bar{S}_{R,1}|&\leq C\v^3\|h^\v_R\|_{L^\infty}^2,\\
		|\tilde{\nu}^{-1}\tilde{w}_{\mathbf{k}}\bar{S}_{R,2}|&\leq C\v^{-1}\|h^\v_R\|_{L^\infty},\\
		|\tilde{\nu}^{-1}\tilde{w}_{\mathbf{k}}\bar{S}_{R,3}|&\leq C(\v^{N-6}+\v^{\mathfrak{b}-5}),
	\end{align*}
	which together with \eqref{6.30} and \eqref{5.38-1} yields \eqref{6.7}. Therefore the proof of Lemma \ref{lem6.2} is finished. 
\end{proof}

{\noindent \bf Proof of Theorem \ref{thm}:} Applying \eqref{6.7} to \eqref{6.1}, one has 
\begin{align*}
	&\f{d}{dt}\|f^\v_R\|_{L^2}^2+\f{c_0}{\v^2}\int_{B_1}\|(\FI-\FP)f^\v_R\|_{\tilde{\nu}}^2\, dx\nonumber\\
	&\leq C(\|f^\v_R\|_{L^2}^2+1)\times \left\{1+\v^{\min\{2\mathbf{k}-5,8\}}\left[\|h^\v_R(0)\|_{L^\infty}^2+\f{C}{\v^6}\sup_{0\leq s\leq t}\|f^\v_R(s)\|_{L^2}^2\right]\right\}.
\end{align*}
Since $\mathbf{k}\geq 6$, we use Gronwall inequality and get 
\begin{align}\label{6.31}
	\|f_R^\v\|_{L^2}^2\leq C\exp\left\{\left(1+\v\sup_{0\leq s\leq t}\|f^\v_R(s)\|_{L^2}^2\right)t \right\}\leq C(t).
\end{align}
It follows from \eqref{6.31} and \eqref{6.7} that
\begin{align*}
	\sup_{0\leq s\leq t}(\v^3\|h_R^\v(s)\|_{L^\infty})\leq C(t)\{1+\v^3\|h_R^\v(0)\|_{L^\infty}\}\leq C(t).
\end{align*} 
Taking $\v\leq \v_0$ for some small $\v_0\ll 1$, we close the {\it a priori} assumption \eqref{5.38-1} directly. Therefore the proof of Theorem \ref{thm} is completed. $\hfill\Box$

	\section{Appendix}
\subsection{Proof of Lemma \ref{lem2.3}} We divide the proof into several steps.\\
{\it Step 1.} For later use, we define the corresponding Burnett functions:
\begin{align}\label{7.88}
	\begin{split}
		\mathcal{A}_{11}^0=\frac{1}{2T^0}\left\{(\bar{v}\cdot \vec{n})^2-(\bar{v}\cdot\vec{\tau}-u_\tau^0)^2\right\}\sqrt{\mu_0},\quad\mathcal{A}_{12}^0=\frac{1}{T^0}\left\{(\bar{v}\cdot \vec{n})(\bar{v}\cdot\vec{\tau}-u_\tau^0)\right\}\sqrt{\mu_0},\\
		\mathcal{B}_1^0=\frac{1}{2\sqrt{2T^0}}(\bar{v}\cdot \vec{n})(\frac{|\bar{v}-\bar{\mathfrak{u}}^0|^2}{T^0}-4)\sqrt{\mu_0},\quad\mathcal{B}_2^0=\frac{1}{2\sqrt{2T^0}}(\bar{v}\cdot\vec{\tau}-u_\tau^0)(\frac{|\bar{v}-\bar{\mathfrak{u}}^0|^2}{T^0}-4)\sqrt{\mu_0},
	\end{split}
\end{align}
and $\mathcal{A}_{21}^0:=\mathcal{A}_{12}^0,\mathcal{A}_{22}^0:=-\mathcal{A}_{11}^0.$ A direct calculation shows
\begin{align*}
	\langle \mathcal{A}_{ij}^0,\mathcal{A}_{ij}^0\rangle =\rho^0,\quad \langle \mathcal{B}_i^0,\mathcal{B}_i^0\rangle =\rho^0,
\end{align*}
and $\mathcal{A}_{11}^0,\mathcal{A}_{12}^0,\mathcal{B}_1^0,\mathcal{B}_2^0$ are orthogonal to each other and belong to $\mathcal{N}_0^\perp$. We also define
\begin{align}\label{6.1-1}
	\langle \mathcal{A}_{ij}^0,\mathbf{L}_0^{-1}\mathcal{A}_{ij}^0\rangle =\frac{\kappa_1(T^0)}{T^0},\quad \langle \mathcal{B}_i^0,\mathbf{L}_0^{-1}\mathcal{B}_i^0\rangle =\frac{\kappa_2(T^0)}{T^0}.
\end{align}
Noting \eqref{1.12-0}, one has that 
\begin{align}\label{7-1}
	\begin{split}
	&\langle 1,\bar{\mathfrak{F}}_k\rangle =\bar{\rho}_k,\quad \langle \bar{v}\cdot\vec{\tau},\bar{\mathfrak{F}}_k\rangle =\bar{\rho}_ku_\tau^0+\rho^0\bar{u}_k\cdot\vec{\tau},\quad \langle \bar{v}\cdot \vec{n},\bar{\mathfrak{F}}_k\rangle =\rho^0\bar{u}_k\cdot \vec{n},\\
	&\langle |\bar{v}-\bar{\mathfrak{u}}^0|^2,\bar{\mathfrak{F}}_k\rangle =\rho^0\bar{\theta}_k+2\bar{\rho}_kT^0=2\bar{p}_k,\\
	&\langle (\bar{v}\cdot \vec{n})^2,\bar{\mathfrak{F}}_k\rangle =\bar{p}_k+T^0\langle \mathcal{A}_{11}^0,(\FI-\FP_0)\left\{\f{\bar{\mathfrak{F}}_k}{\sqrt{\mu_0}}\right\}\rangle ,\\
	&\langle (\bar{v}\cdot\vec{\tau})^2,\bar{\mathfrak{F}}_k\rangle =\bar{p}_k +\bar{\rho}_k|u_\tau^0|^2+2\rho^0 u_\tau^0\bar{u}_k\cdot\vec{\tau}-T^0\langle \mathcal{A}_{11}^0,(\FI-\FP_0)\left\{\f{\bar{\mathfrak{F}}_k}{\sqrt{\mu_0}}\right\}\rangle ,\\
	&\langle (\bar{v}\cdot \vec{n})(\bar{v}\cdot\vec{\tau}),\bar{\mathfrak{F}}_k\rangle =\rho^0 u_\tau^0(\bar{u}_k\cdot \vec{n})+T^0\langle\mathcal{A}_{12}^0,(\FI-\FP_0)\left\{\f{\bar{\mathfrak{F}}_k}{\sqrt{\mu_0}}\right\}\rangle.
	\end{split}
	%\langle (\bar{v}\cdot \vec{n})|\bar{v}|^2,\bar{\mathfrak{F}}_k\rangle &=\langle (\bar{v}\cdot \vec{n})(|\bar{v}-\bar{\mathfrak{u}}^0|^2-4T^0),\bar{\mathfrak{F}}_k\rangle +2 u_\tau^0\langle (\bar{v}\cdot \vec{n})(\bar{v}\cdot\vec{\tau}- u_\tau^0),\bar{\mathfrak{F}}_k\rangle \\
	%&\quad +(4\rho^0T^0+\rho^0|\bar{\mathfrak{u}}^0|^2)\bar{u}_k\cdot \vec{n},\\
	%\langle (\bar{v}\cdot\vec{\tau})|\bar{v}|^2,\bar{\mathfrak{F}}_k\rangle &=\langle (\bar{v}
	%\cdot\vec{\tau}- u_\tau^0)(|\bar{v}-\bar{\mathfrak{u}}^0|^2-4T^0),\bar{\mathfrak{F}}_k\rangle\nonumber\\
	%&\quad + u_\tau^0\langle (\bar{v}\cdot\vec{\tau}- u_\tau^0)^2-(v\cdot \vec{n})^2,\bar{\mathfrak{F}}_k\rangle\nonumber\\
	%&\quad +4\rho^0T^0(\bar{u}_k\cdot\vec{\tau})+2 u_\tau^0(\rho^0\bar{\theta}_k+2\bar{\rho}_kT^0)+3\rho^0|\bar{\mathfrak{u}}^0|^2\bar{u}_k\cdot\vec{\tau}+\bar{\rho}_k u_\tau^0| u_\tau^0|^2.
\end{align}

The equation  \eqref{1.16} is formally written as
\begin{align*}
	(\bar{v}\cdot \vec{n})\partial_y \bar{\mathfrak{F}}_{k+1}=Q(\cdot,\cdot)-\partial_t\bar{\mathfrak{F}}_k+\mathcal{H}^\v(y)\partial_\phi\left\{(\bar{v}\cdot \vec{\tau})\bar{\mathfrak{F}}_k\right\}+\mathcal{H}^\v(y)(\bar{v}\cdot \vec{n})\bar{\mathfrak{F}}_k.
\end{align*}
Integrating \eqref{1.16} over $\bar{v}$, one has that
\begin{align}\label{(1)}
	\partial_y\left\{\int_{\R^2}(\bar{v}\cdot \vec{n})\bar{\mathfrak{F}}_{k+1}\, d\bar{v}\right\}=&-\partial_t\left\{\int_{\R^2} \bar{\mathfrak{F}}_k \,d\bar{v}\right\}+\mathcal{H}^\v(y)\partial_\phi\left\{\int_{\R^2}(\bar{v}\cdot\vec{\tau})\bar{\mathfrak{F}}_k\, d\bar{v}\right\}\nonumber\\
	&+\mathcal{H}^\v(y)\int_{\R^2}(\bar{v}\cdot \vec{n})\bar{\mathfrak{F}}_k \,d\bar{v}.
\end{align}
Multiplying \eqref{1.16} by $\bar{v}\cdot\vec{\tau}$, one gets that
\begin{align}\label{(2)}
	\partial_y\left\{\int_{\R^2} (\bar{v}\cdot \vec{n})(\bar{v}\cdot\vec{\tau})\bar{\mathfrak{F}}_{k+1}\,d\bar{v}\right\}&= -\partial_t\left\{\int_{\R^2}(\bar{v}\cdot\vec{\tau})\bar{\mathfrak{F}}_k\,d\bar{v}\right\}+\mathcal{H}^\v(y)\partial_\phi\left\{\int_{\R^2} (\bar{v}\cdot\vec{\tau})^2\bar{\mathfrak{F}}_k\, d\bar{v}\right\} \nonumber\\
	&\quad +2\mathcal{H}^\v(y)\int_{\R^2}(\bar{v}\cdot \vec{n})(\bar{v}\cdot \vec{\tau})\bar{\mathfrak{F}}_k\, d\bar{v}.
\end{align}
Considering $\eqref{(1)}\times(- u_\tau^0)+\eqref{(2)}$, we have
\begin{align*}
	&\partial_y\left\{\int_{\R^2}(\bar{v}\cdot \vec{n})(\bar{v}\cdot\vec{\tau} -u_\tau^0)\bar{\mathfrak{F}}_{k+1}\, d\bar{v}\right\}\nonumber\\
	&=-\partial_t\left\{\int_{\R^2}(\bar{v}\cdot\vec{\tau}-u_\tau^0)\bar{\mathfrak{F}}_k\,d\bar{v}\right\}-(\partial_tu_\tau^0)\int_{\R^2}\bar{\mathfrak{F}}_k\, d\bar{v}\nonumber\\
	&\quad -\mathcal{H}^\v(y)u_\tau^0\partial_\phi\left\{\int_{\R^2} (\bar{v}\cdot\vec{\tau})\bar{\mathfrak{F}}_k\, d\bar{v}\right\}+\mathcal{H}^\v(y)\partial_\phi\left\{\int_{\R^2} (\bar{v}\cdot\vec{\tau})^2\bar{\mathfrak{F}}_k\, d\bar{v}\right\} \nonumber\\
	&\quad -\mathcal{H}^\v(y)u_\tau^0\int_{\R^2}(\bar{v}\cdot \vec{n})\bar{\mathfrak{F}}_k \,d\bar{v}+2\mathcal{H}^\v(y)\int_{\R^2}(\bar{v}\cdot \vec{n})(\bar{v}\cdot \vec{\tau})\bar{\mathfrak{F}}_k\, d\bar{v},
\end{align*}
which, together with \eqref{7-1}, yields that  
\begin{align}\label{(a)}
	&\rho^0\partial_t(\bar{u}_k\cdot\vec{\tau})+(\bar{u}_k\cdot\vec{\tau})\cdot\left\{\partial_t\rho^0-\mathcal{H}^\v(y)u_\tau^0\partial_\phi\rho^0-2\mathcal{H}^\v(y)\rho^0\partial_\phi u_\tau^0 \right\}\nonumber\\
	&-\mathcal{H}^\v(y)\rho^0u_\tau^0\partial_\phi(\bar{u}_k\cdot\vec{\tau})+\bar{\rho}_k\cdot\left\{\partial_t u_\tau^0-\mathcal{H}^\v(y) u_\tau^0\partial_\phi  u_\tau^0\right\}-\mathcal{H}^\v(y)\partial_\phi\bar{p}_k-\mathcal{H}^\v(y)\rho^0u_\tau^0(\bar{u}_k\cdot \vec{n})\nonumber\\
	&+\partial_y\langle (\bar{v}\cdot \vec{n})(\bar{v}\cdot\vec{\tau}-u_\tau^0),\bar{\mathfrak{F}}_{k+1}\rangle =\bar{W}_{k-1},
\end{align}
where
\begin{align}\label{7.4}
	\bar{W}_{k-1}=2\mathcal{H}^\v(y)\langle (\bar{v}
	\cdot \vec{\tau}-u_\tau^0)(\bar{v}\cdot \vec{n}),\bar{\mathfrak{F}}_k\rangle-\f{1}{2}\mathcal{H}^\v(y)\partial_\phi \langle (\bar{v}\cdot\vec{\tau}-u_\tau^0)^2-(\bar{v}\cdot \vec{n})^2,\bar{\mathfrak{F}}_k\rangle.
\end{align}

It follows from \eqref{1.12} that 
\begin{align}\label{7.10}
	(\FI-\FP_0)\bar{f}_k=&-\FL_0^{-1}\{(\FI-\FP_0)(\bar{v}\cdot \vec{n})\partial_y\mathbf{P}_0\bar{f}_k\}\nonumber\\
	&+\mathbf{L_0}^{-1}\left\{\frac{-y}{\sqrt{\mu_0}}\left[Q(\partial_r\mu_0,\sqrt{\mu_0}\FP_0\bar{f}_k)+Q(\sqrt{\mu_0}\FP_0\bar{f}_k,\partial_r\mu_0)\right]\right\}\nonumber\\
	&+\mathbf{L_0}^{-1}\left\{\frac{1}{\sqrt{\mu_0}}\left[Q(\sqrt{\mu_0}f_1^0,\sqrt{\mu_0}\mathbf{P}_0\bar{f}_k)+Q(\sqrt{\mu_0}\FP_0\bar{f}_k,\sqrt{\mu_0}f_1^0)\right]\right\}\nonumber\\
	&+\mathbf{L_0}^{-1}\left\{\frac{1}{\sqrt{\mu_0}}\left[Q(\sqrt{\mu_0}\bar{f}_1,\sqrt{\mu_0}\FP_0\bar{f}_k)+Q(\sqrt{\mu_0}\FP_0\bar{f}_k,\sqrt{\mu_0}\FP_0\bar{f}_1)\right]\right\}\nonumber\\
	&+\bar{J}_{k-1}=:\bar{P}_1+\bar{P}_2+\bar{P}_3+\bar{P}_4+\bar{J}_{k-1},
\end{align}
where
\begin{align}\label{7.55}
	\bar{J}_{k-1}=&-\mathbf{L_0}^{-1}\left\{(\FI-\FP_0)(\bar{v}\cdot \vec{n})\partial_y(\FI-\FP_0)\bar{f}_k\right\}\nonumber\\
	&-\mathbf{L_0}^{-1}\left\{(\FI-\FP_0)\left[\frac{\partial_t\bar{\mathfrak{F}}_{k-1}}{\sqrt{\mu_0}}-\mathcal{H}^\v(y)\frac{(\bar{v}\cdot\vec{\tau})\partial_\phi\bar{\mathfrak{F}}_{k-1}}{\sqrt{\mu_0}}\right]\right\}\nonumber\\
	&+\mathbf{L_0}^{-1}\left\{\frac{-y}{\sqrt{\mu_0}}[Q(\partial_r\mu_0,\sqrt{\mu_0}
	\{\FI-\FP_0\}\bar{f}_k)+Q(\sqrt{\mu_0}\{\FI-\FP_0\}\bar{f}_k,\partial_r\mu_0)]\right\}\nonumber\\
	&+\mathbf{L_0}^{-1}\left\{\frac{1}{\sqrt{\mu_0}}[Q(\sqrt{\mu_0}f_1^0,\sqrt{\mu_0}\{\FI-\FP_0\}\bar{f}_k)+Q(\sqrt{\mu_0}\{\FI-\FP_0\}\bar{f}_k,\sqrt{\mu_0}f_1^0)]\right\}\nonumber\\
	&+\mathbf{L_0}^{-1}\left\{\frac{1}{\sqrt{\mu_0}}[Q(\sqrt{\mu_0}\bar{f}_1,\sqrt{\mu_0}\{\FI-\FP_0\}\bar{f}_k)+Q(\sqrt{\mu_0}\FP_0\bar{f}_k,\sqrt{\mu_0}\{\FI-\FP_0\}\bar{f}_1)]\right\}\nonumber\\
	&+\mathbf{L_0}^{-1}\left\{\frac{(-y)^l}{\sqrt{\mu_0}}\sum\limits_{l+j=k+1,\atop l\geq2}[Q(\partial_r^l\mu_0,\sqrt{\mu_0}\bar{f}_j)+Q(\sqrt{\mu_0}\bar{f}_j,\partial_r^l\mu_0)]\right\}\nonumber\\
	&+\mathbf{L_0}^{-1}\left\{\frac{(-y)^l}{\sqrt{\mu_0}}\sum\limits_{l+i+j=k+1,\atop l,i\geq 1}[Q(\partial_r^lF_i^0,\sqrt{\mu_0}\bar{f}_j)+Q(\sqrt{\mu_0}\bar{f}_j,\partial_r^lF_i^0)]\right\}\nonumber\\
	&+\mathbf{L_0}^{-1}\left\{\sum\limits_{i+j=k+1,\atop i\geq2}\frac{1}{\sqrt{\mu_0}}[Q(F_i^0,\sqrt{\mu_0}\bar{f}_j)+Q(\sqrt{\mu_0}\bar{f}_j,F_i^0)]\right\}\nonumber\\
	&+\mathbf{L_0}^{-1}\left\{\sum\limits_{i+j=k+1,\atop i,j\geq2}\frac{1}{\sqrt{\mu_0}}[Q(\sqrt{\mu_0}\bar{f}_i,\sqrt{\mu_0}\bar{f}_j)+Q(\sqrt{\mu_0}\bar{f}_j,\sqrt{\mu_0}\bar{f}_i)]\right\}.
\end{align}

It is clear that
\begin{align}\label{7.6}
	\langle (\bar{v}\cdot\vec{\tau}-u_\tau^0)(\bar{v}\cdot \vec{n}),\bar{\mathfrak{F}}_{k+1}\rangle =T^0\left\langle \mathcal{A}_{12}^0,(\FI-\FP_0)\left\{\frac{\bar{\mathfrak{F}}_{k+1}}{\sqrt{\mu_0}}\right\}\right\rangle,
\end{align}
and
\begin{align}\label{7.7}
	&T^0	\langle\mathcal{A}_{12}^0 ,\bar{P}_1\rangle\nonumber\\
	&=-T^0\left\langle (\bar{v}\cdot \vec{n})\left[\frac{\partial_y\bar{\rho}_k}{\rho^0}+\partial_y\bar{u}_k\cdot\frac{\bar{v}-\bar{\mathfrak{u}}^0}{T^0}+\partial_y\bar{\theta}_k\cdot\frac{1}{4T^0}\left(\frac{|\bar{v}-\bar{\mathfrak{u}}^0|^2}{T^0}-2\right)\right]\sqrt{\mu_0},\mathbf{L}_0^{-1}\mathcal{A}_{12}^0\right\rangle \nonumber\\
	&=-T^0\partial_y(\bar{u}_k\cdot\vec{\tau})\langle \mathcal{A}_{12}^0,\mathbf{L}_0^{-1}\mathcal{A}_{12}^0\rangle=-\kappa_1(T^0)\partial_y(\bar{u}_k\cdot\vec{\tau}).
\end{align}

Noting
\begin{align*}
	\FP_0g=\left\{\frac{a}{\rho^0}+b\cdot\frac{\bar{v}-\bar{\mathfrak{u}}^0}{T^0}+\frac{c}{4T^0}\left(\frac{|\bar{v}-\mathbf{\bar{u}}^0|^2}{T^0}-2\right)\right\}\sqrt{\mu_0},
\end{align*}
we know from \cite{Guo2006} that
\begin{align*}
	\mathbf{L}_0^{-1}\left\{\frac{1}{\sqrt{\mu_0}}[Q(\sqrt{\mu_0}\mathbf{P}_0g,\sqrt{\mu_0}\mathbf{P}_0\bar{f}_k)+Q(\sqrt{\mu_0}\mathbf{P}_0\bar{f}_k,\sqrt{\mu_0}\mathbf{P}_0g)]\right\}=(\mathbf{I-P}_0)\left\{\frac{\mathbf{P}_0g\cdot \mathbf{P}_0\bar{f}_k}{\sqrt{\mu_0}}\right\},
\end{align*}
and 
\begin{align*}
	(\FI-\FP_0)\left\{\frac{\mathbf{P}_0g\cdot\mathbf{P}_0\bar{f}_k}{\sqrt{\mu_0}}\right\}=&\frac{(b\cdot \vec{n})(\bar{u}_k\cdot \vec{n})-(b\cdot\vec{\tau})(\bar{u}_k\cdot\vec{\tau})}{T^0}\mathcal{A}_{11}^0+\frac{(b\cdot \vec{n})(\bar{u}_k\cdot\vec{\tau})+(b\cdot\vec{\tau})(\bar{u}_k\cdot \vec{n})}{T^0}\mathcal{A}_{12}^0\nonumber\\
	&+\frac{1}{\sqrt{2T^0}}\left\{\frac{\bar{\theta}_k}{T^0}(b\cdot \vec{n})\mathcal{B}_1^0+\frac{\bar{\theta}_k}{T^0}(b\cdot\vec{\tau})\mathcal{B}_2^0+\frac{c}{T^0}(\bar{u}_k\cdot \vec{n})\mathcal{B}_1^0+\frac{c}{T^0}(\bar{u}_k\cdot\vec{\tau})\mathcal{B}_2^0\right\}\nonumber\\
	&+\frac{c\cdot\bar{\theta}_k}{16(T^0)^2}(\FI-\FP_0)\left\{\left(\frac{|\bar{v}-\bar{\mathfrak{u}}^0|^2}{T^0}-4\right)^2\sqrt{\mu_0}\right\}.
\end{align*}
Using \eqref{2.4}, one has
\begin{align*}
	&T^0\left\langle \mathcal{A}_{12}^0,\bar{P}_2+\bar{P}_3+\bar{P}_4\right\rangle \nonumber\\
	&=\rho^0[(\bar{u}_k\cdot\vec{\tau})(-y\partial_r\bar{\mathfrak{u}}^0\cdot \vec{n}-u_1\cdot \vec{n})+(\bar{u}_k\cdot \vec{n})(-y\partial_r \bar{\mathfrak{u}}^0\cdot\vec{\tau}-u^0_1\cdot\vec{\tau}+\bar{u}_1\cdot\vec{\tau})],
\end{align*}
which, together with \eqref{7.6}--\eqref{7.7}, yields that
\begin{align}\label{7.8}
	&\partial_y\langle (\bar{v}\cdot\vec{\tau}-u_\tau^0)(\bar{v}\cdot \vec{n}),\bar{\mathfrak{F}}_{k+1}\rangle\nonumber\\ &=-\kappa_1(T^0)\partial_{yy}(\bar{u}_k\cdot\vec{\tau})+\rho^0\partial_y[(\bar{u}_k\cdot\vec{\tau})(-y\partial_r\bar{\mathfrak{u}}^0\cdot \vec{n}-u^0_1\cdot \vec{n})\nonumber\\
	&\quad+(\bar{u}_k\cdot \vec{n})(-y\partial_r \bar{\mathfrak{u}}^0\cdot \vec{\tau}-u^0_1\cdot\vec{\tau}+\bar{u}_1\cdot\vec{\tau})]+T^0\partial_y\langle \mathcal{A}_{12}^0,\bar{J}_{k-1}\rangle\\
	&=-\kappa_1(T^0)\partial_{yy}(\bar{u}_k\cdot\vec{\tau})+\rho^0\{\partial_y(\bar{u}_k\cdot\vec{\tau})\}\cdot (-y\partial_r\bar{\mathfrak{u}}^0\cdot \vec{n}-u^0_1\cdot \vec{n})-\rho^0(\partial_r\bar{\mathfrak{u}}^0\cdot \vec{n})(\bar{u}_k\cdot\vec{\tau})\nonumber\\
	&\quad+\rho^0\partial_y(\bar{u}_k\cdot \vec{n})(-y\partial_r \bar{\mathfrak{u}}^0\cdot \vec{\tau}-u^0_1\cdot\vec{\tau}+\bar{u}_1\cdot\vec{\tau})-\rho^0 (\partial_r \bar{\mathfrak{u}}^0\cdot \vec{\tau})(\bar{u}_k\cdot \vec{n})+T^0\partial_y\langle \mathcal{A}_{12}^0,\bar{J}_{k-1}\rangle,\nonumber
\end{align}
where we have used
\begin{align*}
	\partial_r\mu=\left\{\frac{\partial_r\rho}{\rho}+\partial_r\bar{\mathfrak{u}}\cdot\frac{\bar{v}-\bar{\mathfrak{u}}}{T}+\frac{\partial_r T}{2T}\left(\frac{|\bar{v}-\bar{\mathfrak{u}}|^2}{T}-2\right)\right\}\mu,
\end{align*}
and $\partial_r\bar{\mathfrak{u}}^0=\partial_r\bar{\mathfrak{u}}|_{r=1}$. Combining \eqref{(a)} and \eqref{7.8}, one has that
\begin{align}\label{7.9}
	&\rho^0\partial_t(\bar{u}_k\cdot\vec{\tau})-\kappa_1(T^0)\partial_{yy}(\bar{u}_k\cdot\vec{\tau})\nonumber\\
	&-\mathcal{H}^\v(y)\rho^0u_\tau^0\partial_\phi(\bar{u}_k\cdot\vec{\tau})+\rho^0(-y\partial_r\bar{\mathfrak{u}}^0\cdot \vec{n}-u^0_1\cdot \vec{n})\partial_y(\bar{u}_k\cdot\vec{\tau})\nonumber\\
	&+(\bar{u}_k\cdot\vec{\tau})\cdot\left\{\partial_t\rho^0-\mathcal{H}^\v(y)u_\tau^0\partial_\phi\rho^0-2\mathcal{H}^\v(y)\rho^0\partial_\phi u_\tau^0-\rho^0\partial_r\bar{\mathfrak{u}}^0\cdot \vec{n} \right\}\nonumber\\
	&+\bar{\rho}_k\cdot\left\{\partial_t u_\tau^0-\mathcal{H}^\v(y) u_\tau^0\partial_\phi  u_\tau^0\right\}-\mathcal{H}^\v(y)\partial_\phi\bar{p}_k-\left(\mathcal{H}^\v(y)u_\tau^0+\partial_r\bar{\mathfrak{u}}^0\cdot\vec{\tau}\right)\rho^0(\bar{u}_k\cdot \vec{n})\nonumber\\
	&+\rho^0(-y\partial_r \bar{\mathfrak{u}}^0\cdot \vec{\tau}-u^0_1\cdot\vec{\tau}+\bar{u}_1\cdot\vec{\tau})\partial_y(\bar{u}_k\cdot \vec{n})=\bar{W}_{k-1}-T^0\partial_y\langle \mathcal{A}_{12}^0,\bar{J}_{k-1}\rangle.
\end{align}

We rewrite the Euler equations \eqref{1.9} in polar coordinate
\begin{align*}
	&\partial_t\rho+\partial_r(\rho\mathfrak{u}\cdot \vec{n})+\f{1}{r}\partial_\phi(\rho\mathfrak{u}\cdot\vec{\tau})+\f{1}{r}\rho(\mathfrak{u}\cdot \vec{n})=0,\\
	&\partial_t\mathfrak{u}\cdot\vec{\tau}+(\mathfrak{u}\cdot \vec{n})\partial_r\mathfrak{u}\cdot\vec{\tau}+\f{1}{r}(\mathfrak{u}\cdot\vec{\tau})\partial_\phi(\mathfrak{u}\cdot\vec{\tau})+\f{1}{\rho r}\partial_\phi(\rho T)=0,\\
	&\partial_tT+\partial_r(T\mathfrak{u}\cdot \vec{n})+\f{1}{r}\partial_\phi(T\mathfrak{u}\cdot\vec{\tau})+\f{1}{r}T(\mathfrak{u}\cdot \vec{n})=0,
\end{align*}
which yields that
\begin{align}\label{7.11}
	&\partial_t\rho^0=\rho^0\partial_r\bar{\mathfrak{u}}^0\cdot \vec{n}+\rho^0\partial_\phi u_\tau^0+u_\tau^0\partial_\phi\rho^0,\nonumber\\
	&\partial_tu_\tau^0=u_\tau^0\partial_\phi u_\tau^0+\partial_\phi T^0+\f{T^0}{\rho^0}\partial_\phi\rho^0,\nonumber\\
	&\partial_tT^0=T^0\partial_r\bar{\mathfrak{u}}^0\cdot \vec{n}+T^0\partial_\phi u_\tau^0+u_\tau^0\partial_\phi T^0.
\end{align}
Combining \eqref{7.9}--\eqref{7.11}, noting the definition of $\bar{p}_k$, we get
\begin{align*}
	&\rho^0\partial_t(\bar{u}_k\cdot\vec{\tau})-\kappa_1(T^0)\partial_{yy}(\bar{u}_k\cdot\vec{\tau})\nonumber\\
	&-\mathcal{H}^\v(y)\rho^0u_\tau^0\partial_\phi(\bar{u}_k\cdot\vec{\tau})+\rho^0(-y\partial_r\bar{\mathfrak{u}}^0\cdot \vec{n}-u^0_1\cdot \vec{n})\partial_y(\bar{u}_k\cdot\vec{\tau})\nonumber\\
	&+(\bar{u}_k\cdot\vec{\tau})\cdot\left\{\left(1-\mathcal{H}^\v(y)\right)u_\tau^0\partial_\phi\rho^0+\left(1-2\mathcal{H}^\v(y)\right)\rho^0\partial_\phi u_\tau^0 \right\}\nonumber\\
	&-\f{\rho^0}{2T^0}\bar{\theta}_k\cdot\left\{\left(1-\mathcal{H}^\v(y)\right) u_\tau^0\partial_\phi  u_\tau^0+\partial_\phi T^0+\f{T^0}{\rho^0}\partial_\phi\rho^0\right\}\nonumber\\
	&+\f{\bar{p}_k}{T^0}\cdot\left\{\left(1-\mathcal{H}^\v(y)\right) u_\tau^0\partial_\phi  u_\tau^0+\partial_\phi T^0+\f{T^0}{\rho^0}\partial_\phi\rho^0\right\}-\mathcal{H}^\v(y)\partial_\phi\bar{p}_k\nonumber\\
	&-\left(\mathcal{H}^\v(y)u_\tau^0+\partial_r\bar{\mathfrak{u}}^0\cdot\vec{\tau}\right)\rho^0(\bar{u}_k\cdot \vec{n})+\rho^0(-y\partial_r \bar{\mathfrak{u}}^0\cdot \vec{\tau}-u^0_1\cdot\vec{\tau}+\bar{u}_1\cdot\vec{\tau})\partial_y(\bar{u}_k\cdot \vec{n})\nonumber\\
	&=\bar{W}_{k-1}-T^0\partial_y\langle \mathcal{A}_{12}^0,\bar{J}_{k-1}\rangle,
\end{align*} 
which concludes $\eqref{bu-0}_1$.\\
{\it Step 2.}
Multiplying \eqref{1.16} by $|\bar{v}-\bar{\mathfrak{u}}^0|^2$, one has 
\begin{align*}
	&\partial_y\left\{\int_{\R^2}(\bar{v}\cdot \vec{n})|\bar{v}-\bar{\mathfrak{u}}^0|^2\bar{\mathfrak{F}}_{k+1}\,d\bar{v}\right\}\nonumber\\
	&=-\partial_t\left\{\int_{\R^2}|\bar{v}-\bar{\mathfrak{u}}^0|^2\bar{\mathfrak{F}}_k\, d\bar{v}\right\}+\int_{\R^2}\{\partial_t(|\bar{v}-\bar{\mathfrak{u}}^0|^2)\}\bar{\mathfrak{F}}_k\, d\bar{v}\nonumber\\
	&\quad -\mathcal{H}^\v(y)\int_{\R^2}\left\{\partial_\phi\left(|\bar{v}-\bar{\mathfrak{u}}^0|^2\right)\right\}(\bar{v}\cdot\vec{\tau})\bar{\mathfrak{F}}_k\, d\bar{v}\nonumber\\
	&\quad +\mathcal{H}^\v(y)\int_{\R^2}(\bar{v}\cdot \vec{n})|\bar{v}-\bar{\mathfrak{u}}^0|^2\bar{\mathfrak{F}}_k\, d\bar{v}+\mathcal{H}^\v(y)\partial_\phi\left\{\int_{\R^2}(\bar{v}\cdot\vec{\tau})|\bar{v}-\bar{\mathfrak{u}}^0|^2\bar{\mathfrak{F}}_k\, d\bar{v}\right\}.
\end{align*}
Noting
\begin{align*}
	\partial_\phi \left(|\bar{v}-\bar{\mathfrak{u}}^0|^2\right)=2u_\tau^0\bar{v}\cdot \vec{n}-2\partial_\phi u_\tau^0(\bar{v}\cdot\vec{\tau}-u_\tau^0),\quad \partial_t\left(|\bar{v}-\bar{\mathfrak{u}}^0|^2\right)=-2\partial_tu_\tau^0(\bar{v}\cdot\vec{\tau}-u_\tau^0),
\end{align*}
we further derive
\begin{align}\label{(3)}
	&\partial_y\left\{\int_{\R^2}(\bar{v}\cdot \vec{n})|\bar{v}-\bar{\mathfrak{u}}^0|^2\bar{\mathfrak{F}}_{k+1}\,d\bar{v}\right\}\nonumber\\
	&=-\partial_t\left\{\int_{\R^2}|\bar{v}-\bar{\mathfrak{u}}^0|^2\bar{\mathfrak{F}}_k\, d\bar{v}\right\}-2(\partial_tu_\tau^0)\int_{\R^2}(\bar{v}\cdot\vec{\tau}-u_\tau^0)\bar{\mathfrak{F}}_k\,d\bar{v}\nonumber\\
	&\quad -2\mathcal{H}^\v(y)\int_{\R^2}u_\tau^0(\bar{v}\cdot \vec{n})(\bar{v}\cdot\vec{\tau})\bar{\mathfrak{F}}_k\, d\bar{v} +2\mathcal{H}^\v(y)\int_{\R^2}(\partial_\phi u_\tau^0)(\bar{v}\cdot\vec{\tau})(\bar{v}\cdot\vec{\tau}-u_\tau^0)\bar{\mathfrak{F}}_k\, d\bar{v}\nonumber\\
	&+\mathcal{H}^\v(y)\int_{\R^2}(\bar{v}\cdot \vec{n})|\bar{v}-\bar{\mathfrak{u}}^0|^2\bar{\mathfrak{F}}_k\, d\bar{v}+\mathcal{H}^\v(y)\partial_\phi\left\{\int_{\R^2}(\bar{v}\cdot\vec{\tau})|\bar{v}-\bar{\mathfrak{u}}^0|^2\bar{\mathfrak{F}}_k\, d\bar{v}\right\}.
\end{align}
Considering $(-4T^0)\times\eqref{(1)}+\eqref{(3)}$, one obtains 
\begin{align*}
	&\partial_y\left\{\int_{\R^2}(\bar{v}\cdot \vec{n})(|\bar{v}-\bar{\mathfrak{u}}^0|^2-4T^0)\bar{\mathfrak{F}}_{k+1}\,d\bar{v}\right\}\nonumber\\
	&=-\partial_t\left\{\int_{\R^2}(|\bar{v}-\bar{\mathfrak{u}}^0|^2-4T^0)\bar{\mathfrak{F}}_k\, d\bar{v}\right\}-2(\partial_tu_\tau^0)\int_{\R^2}(\bar{v}\cdot\vec{\tau}-u_\tau^0)\bar{\mathfrak{F}}_k\,d\bar{v}-4(\partial_tT^0)\int_{\R^2}\bar{\mathfrak{F}}_k\, d\bar{v}\nonumber\\
	&\quad -2\mathcal{H}^\v(y)\int_{\R^2}u_\tau^0(\bar{v}\cdot \vec{n})(\bar{v}\cdot\vec{\tau})\bar{\mathfrak{F}}_k\, d\bar{v}+2\mathcal{H}^\v(y)\int_{\R^2}\partial_\phi u_\tau^0(\bar{v}\cdot\vec{\tau})(\bar{v}\cdot\vec{\tau}-u_\tau^0)\bar{\mathfrak{F}}_k\, d\bar{v}\nonumber\\
	&\quad +\mathcal{H}^\v(y)\partial_\phi\left\{\int_{\R^2}(\bar{v}\cdot\vec{\tau}-u_\tau^0)(|\bar{v}-\bar{\mathfrak{u}}^0|^2-4T^0)\bar{\mathfrak{F}}_k\, d\bar{v}\right\}+4\mathcal{H}^\v(y)\partial_\phi T^0\int_{\R^2}(\bar{v}\cdot\vec{\tau})\bar{\mathfrak{F}}_k\, d\bar{v}\nonumber\\
	&\quad+\mathcal{H}^\v(y)\partial_\phi\left\{u_\tau^0\int_{\R^2}(|\bar{v}-\bar{\mathfrak{u}}^0|^2-4T^0)\bar{\mathfrak{F}}_k\, d\bar{v}\right\}+\mathcal{H}^\v(y)\int_{\R^2}(\bar{v}\cdot \vec{n})(|\bar{v}-\bar{\mathfrak{u}}^0|^2-4T^0)\bar{\mathfrak{F}}_k\, d\bar{v},
\end{align*}
which yields that
\begin{align}\label{(5)}
	&\partial_t\left\{-2\bar{p}_k+2\rho^0\bar{\theta}_k\right\}+4(\partial_tT^0)\bar{\rho}_k+2(\partial_tu_\tau^0)\rho^0\bar{u}_k\cdot\vec{\tau}\nonumber\\
	&+2\mathcal{H}^\v(y)\rho^0|u_\tau^0|^2(\bar{u}_k\cdot \vec{n})-2\mathcal{H}^\v(y)(\partial_\phi u_\tau^0)\left\{\bar{p}_k+\rho^0u_\tau^0(\bar{u}_k\cdot\vec{\tau})\right\}\nonumber\\
	&-4\mathcal{H}^\v(y)(\partial_\phi T^0)(\bar{\rho}_ku_\tau^0+\rho^0\bar{u}_k\cdot\vec{\tau})-\mathcal{H}^\v(y)\partial_\phi\left\{-2u_\tau^0\bar{p}_k+2\rho^0u_\tau^0\bar{\theta}_k\right\}\nonumber\\
	&+\partial_y\left\{\int_{\R^2}(\bar{v}\cdot \vec{n})(|\bar{v}-\bar{\mathfrak{u}}^0|^2-4T^0)\bar{\mathfrak{F}}_{k+1}\, d\bar{v}\right\}=\bar{H}_{k-1},
\end{align}
where
\begin{align}\label{7.12}
	\bar{H}_{k-1}=&-2\mathcal{H}^\v(y) u_\tau^0\langle(\bar{v}\cdot \vec{n})(\bar{v}\cdot\vec{\tau}-u_\tau^0),\bar{\mathfrak{F}}_k\rangle +\mathcal{H}^\v(y)\partial_\phi\langle(\bar{v}\cdot\vec{\tau}-u_\tau^0)(|\bar{v}-\bar{\mathfrak{u}}^0|^2-4T^0),\bar{\mathfrak{F}}_k\rangle \nonumber\\
	&+\mathcal{H}^\v(y)\langle(\bar{v}\cdot \vec{n})(|\bar{v}-\bar{\mathfrak{u}}^0|^2-4T^0),\bar{\mathfrak{F}}_k\rangle-\mathcal{H}^\v(y)\partial_\phi u_\tau^0\langle (\bar{v}\cdot\vec{\tau}-u_\tau^0)^2-(\bar{v}\cdot \vec{n})^2,\bar{\mathfrak{F}}_k\rangle.
\end{align}
To calculate $\partial_y\langle (\bar{v}\cdot \vec{n})(|\bar{v}-\bar{\mathfrak{u}}^0|^2-4T^0),\bar{\mathfrak{F}}_{k+1}\rangle $, we have from \eqref{7.10} that
\begin{align}\label{7.13}
	&\langle (\bar{v}\cdot \vec{n})(|\bar{v}-\bar{\mathfrak{u}}^0|^2-4T^0),\bar{\mathfrak{F}}_{k+1}\rangle\nonumber\\ &=(2T^0)^{\frac32}\left\langle \mathcal{B}^0_1,\{\FI-\FP_0\}\left\{\frac{\bar{\mathfrak{F}}_{k+1}}{\sqrt{\mu_0}}\right\}\right\rangle\nonumber \\
	&=(2T^0)^{\frac32}\langle \mathcal{B}^0_1,\bar{P}_1\rangle+(2T^0)^{\frac32}\langle \mathcal{B}^0_1,\bar{P}_2+\bar{P}_3+\bar{P}_4\rangle+(2T^0)^{\frac32}\langle \bar{J}_{k-1},\mathcal{B}_1^0\rangle\nonumber\\
	&=-2\kappa_2(T^0)\partial_y\bar{\theta}_k+2\rho^0\bar{\theta}_k(-y\partial_r\bar{\mathfrak{u}}^0\cdot \vec{n}-u^0_1\cdot \vec{n})\nonumber\\
	&\quad +2\rho^0(\bar{u}_k\cdot \vec{n})\cdot (-2y\partial_rT^0+\theta_1^0+\bar{\theta}_1)+(2T^0)^{\frac32}\langle \bar{J}_{k-1},\mathcal{B}_1^0\rangle.
\end{align}
Combining \eqref{(5)} and \eqref{7.13}, one has that
\begin{align*}
	&\rho^0\partial_t\bar{\theta}_k-\kappa_2(T^0)\partial_{yy}\bar{\theta}_k+\rho^0(-y\partial_r\bar{\mathfrak{u}}^0\cdot \vec{n}-u_1^0\cdot \vec{n})\partial_y\bar{\theta}_k\nonumber\\
	&-\mathcal{H}^\v(y)\rho^0u_\tau^0\partial_\phi\bar{\theta}_k+\bar{\theta}_k\left\{\partial_t\rho^0-\mathcal{H}^\v(y)\partial_\phi(\rho^0u_\tau^0)-\rho^0\partial_r\bar{\mathfrak{u}}^0\cdot \vec{n}\right\}\nonumber\\
	&+\rho^0(\bar{u}_k\cdot \vec{\tau})\left\{\partial_tu_\tau^0-2\mathcal{H}^\v(y)\partial_\phi T^0-\mathcal{H}^\v(y)u_\tau^0\partial_\phi u_\tau^0\right\}\nonumber\\
	&-\partial_t\bar{p}_k+\mathcal{H}^\v(y)u_\tau^0\partial_\phi\bar{p}_k+\bar{\rho}_k\left\{2\partial_tT^0-2\mathcal{H}^\v(y)u_\tau^0\partial_\phi T^0\right\}\nonumber\\
	&+\rho^0\left\{\mathcal{H}^\v(y)|u_\tau^0|^2-2\partial_rT^0\right\}(\bar{u}_k\cdot \vec{n})+\rho^0(-2y\partial_rT^0+\theta_1^0+\bar{\theta}_1)\partial_y(\bar{u}_k\cdot \vec{n})\nonumber\\
	&=\f{1}{2}\bar{H}_{k-1}-\f{1}{2}(2T^0)^{\f32}\partial_y\langle \mathcal{B}_1^0,\bar{J}_{k-1}\rangle,
\end{align*}
which together with \eqref{7.11} and the definition of $\bar{p}_k$ yields that
\begin{align*}
	&\rho^0\partial_t\bar{\theta}_k-\kappa_2(T^0)\partial_{yy}\bar{\theta}_k+\rho^0(-y\partial_r\bar{\mathfrak{u}}^0\cdot \vec{n}-u_1^0\cdot \vec{n})\partial_y\bar{\theta}_k-\mathcal{H}^\v(y)\rho^0u_\tau^0\partial_\phi\bar{\theta}_k\nonumber\\
	&+\bar{\theta}_k\left\{\left(1-\mathcal{H}^\v(y)\right)\partial_\phi(\rho^0u_\tau^0)-\f{\rho^0}{T^0}\left(1-\mathcal{H}^\v(y)\right)u_\tau^0\partial_\phi T^0-\rho^0\partial_\phi u_\tau^0-\rho^0\partial_r\bar{\mathfrak{u}}^0\cdot \vec{n}\right\}\nonumber\\
	&+\rho^0(\bar{u}_k\cdot \vec{\tau})\left\{\left(1-2\mathcal{H}^\v(y)\right)\partial_\phi T^0+\f{T^0}{\rho^0}\partial_\phi\rho^0+\left(1-\mathcal{H}^\v(y)\right)u_\tau^0\partial_\phi u_\tau^0\right\}\nonumber\\
	&-\partial_t\bar{p}_k+\mathcal{H}^\v(y)u_\tau^0\partial_\phi\bar{p}_k+\f{\bar{p}_k}{T^0}\left\{\left(2-2\mathcal{H}^\v(y)\right)u_\tau^0\partial_\phi T^0+2T^0\partial_\phi u_\tau^0+2T^0\partial_r\bar{\mathfrak{u}}^0\cdot \vec{n}\right\}\nonumber\\
	&+\rho^0\left\{\mathcal{H}^\v(y)|u_\tau^0|^2-2\partial_rT^0\right\}(\bar{u}_k\cdot \vec{n})+\rho^0(-2y\partial_rT^0+\theta_1^0+\bar{\theta}_1)\partial_y(\bar{u}_k\cdot \vec{n})\nonumber\\
	&=\f{1}{2}\bar{H}_{k-1}-\f{1}{2}(2T^0)^{\f32}\partial_y\langle \mathcal{B}_1^0,\bar{J}_{k-1}\rangle,
\end{align*}
which is exactly $\eqref{bu-0}_2$. $\hfill\Box$
\subsection{Proof of Lemma \ref{2.2}}
	Recall $\hat{S}_{k,1}$ in \eqref{eq2.22}, the condition $\eqref{eq2.24}_1$ can be expressed as
	\begin{align*}
		&\partial_\eta\left\{\int_{\R^2}\mv_1\sqrt{\mu_0}\hat{f}_{k,1}\, d\mv\right\}+G(\eta)\int_{\R^2}\mv_1\sqrt{\mu_0}\hat{f}_{k,1}\, d\mv=\rho^0\hat{a}_k,\\
		&\partial_\eta\left\{\int_{\R^2}\mv_1^2\sqrt{\mu_0}\hat{f}_{k,1}\, d\mv\right\}+G(\eta)\int_{\R^2}(\mv_1^2-\mv_2^2)\sqrt{\mu_0}\hat{f}_{k,1}\, d\mv=\rho^0\hat{b}_{k,1},\\
		&\partial_\eta\left\{\int_{\R^2}\mv_1 \mv_2\sqrt{\mu_0}\hat{f}_{k,1}\, d\mv\right\}+2G(\eta)\int_{\R^2}\mv_1 \mv_2\sqrt{\mu_0}\hat{f}_{k,1}\, d\mv=\rho^0(\hat{b}_{k,2}+u_\tau^0\hat{a}_k),\\
		&\partial_\eta\left\{\int_{\R^2}\mv_1 |v|^2\sqrt{\mu_0}\hat{f}_{k,1}\, d\mv\right\}+G(\eta)\int_{\R^2}\mv_1 |v|^2\sqrt{\mu_0}\hat{f}_{k,1}\, d\mv\nonumber\\
		&=\rho^0\left\{(2T^0+|u_\tau^0|^2)\hat{a}_k+2u_\tau^0\hat{b}_{k,2}+4T^0\hat{c}_k\right\}.
	\end{align*}
	Noting the construction of $\hat{f}_{k,1}$ in \eqref{eq2.23}, by serious calculation, one has that
	\begin{align}\label{7.14}
		&\partial_\eta\left\{\hat{A}_k\cdot e^{-W(\eta)}\right\}=\hat{a}_k\cdot e^{-W(\eta)},\nonumber\\
		&\partial_\eta\left\{\hat{D}_k\cdot e^{\f{|u_\tau^0|^2}{T^0}W(\eta)}\right\}=\f{\hat{b}_{k,1}}{T^0}\cdot e^{\f{|u_\tau^0|^2}{T^0}W(\eta)},\nonumber\\ 
		&\partial_\eta\left\{\left(\hat{B}_k+\f{u_\tau^0}{T^0}\hat{A}_k\right)\cdot e^{-2W(\eta)}\right\}=\left(\f{\hat{b}_{k,2}+u_\tau^0\hat{a}_k}{T^0}\right)\cdot e^{-2W(\eta)},\nonumber\\
		&\partial_\eta\left\{\left[8(T^0)^2\hat{C}_k+2u_\tau^0T^0\hat{B}_k+(4T^0+|u_\tau^0|^2)\hat{A}_k\right]\cdot e^{-W(\eta)}\right\}\nonumber\\
		&=\left\{(2T^0+|u_\tau^0|^2)\hat{a}_k+2u_\tau^0\hat{b}_{k,2}+4T^0\hat{c}_k\right\}\cdot e^{-W(\eta)}.
	\end{align}
	Assuming  $\hat{A}_{k}(d)=\hat{B}_k(d)=\hat{C}_k(d)=\hat{D}_k(d)\equiv 0$, integrating \eqref{7.14} on $[\eta,d]$, one can prove Lemma \ref{2.2}. $\hfill\Box$
	\subsection{Proof of Lemma \ref{lem4.1}}
			We first prove that $$k_M(\mv,u)e^{\zeta(|u|^2-|\mv|^2)}\cdot \frac{(1+|u|^2)^{\f{\beta}{2}}}{(1+|\mv|^2)^{\f{\beta}{2}}}\in L_u^1.$$  One has that
			\begin{align*}
				k_{M,i}(\mv,u)=  k_i(\mv,u)\cdot e^{\f{1}{4T^0}(|u-\bar{\fu}^0|^2-|\mv-\bar{\fu}^0|^2)}\cdot e^{\f{1}{4T_M}(|\mv|^2-|u|^2)},
			\end{align*}
			which yields that
			\begin{align*}
				k_{M,i}(\mv,u)\cdot e^{\zeta(|u|^2-|\mv|^2)}=k_i(\mv,u)\cdot e^{(\f{1}{4T_M}-\f{1}{4T^0}-\zeta)(|\mv|^2-|u|^2)}\cdot e^{\f{\bar{\fu}^0}{2T^0}\cdot (\mv-u)}.
			\end{align*}
			Noting $\frac{(1+|u|^2)^{\f{\beta}{2}}}{(1+|\mv|^2)^{\f{\beta}{2}}}\leq (1+|\mv-u|^2)^{\f{\beta}{2}}$, recall \eqref{5.0-1}, we get that 
			\begin{align*}
				&k_{M,1}(\mv,u)\cdot e^{\zeta(|u|^2-|\mv|^2)}\cdot \frac{(1+|u|^2)^{\f{\beta}{2}}}{(1+|\mv|^2)^{\f{\beta}{2}}}\\
				\leq &\, k_1(\mv, u)(1+|\mv-u|^2)^{\f{\beta}{2}}\cdot e^{(\f{1}{4T_M}-\f{1}{4T^0}-\zeta)(|\mv|^2-|u|^2)}\cdot e^{\f{\bar{\fu}^0}{2T^0}\cdot (\mv-u)}\\
				\lesssim&\,  (1+|\mv-u|)^{\beta+1} e^{-(\f{1}{2T^0}-\f{1}{4T_M}+\zeta)|\mv|^2} e^{-(\f{1}{4T_M}-\zeta)|u|^2}\cdot e^{\f{\bar{\fu}^0}{T^0}\mv}\quad \in L^1_u\cap L^2_u,
			\end{align*}
			which yields that 
			\begin{align*}
				\int_{\R^2}k_{M,1}(\mv,u)e^{\zeta(|u|^2-|\mv|^2)}\cdot \frac{(1+|u|^2)^{\f{\beta}{2}}}{(1+|\mv|^2)^{\f{\beta}{2}}}\, du \leq C(1+|\mv)^{-1}.
			\end{align*}
			
			A direct calculation gives
			\begin{align}\label{5.0-2}
				&k_{M,2}(\mv,u)e^{\zeta(|u|^2-|\mv|^2)}\cdot (1+|\mv-u|^2)^{\f{\beta}{2}}\nonumber\\
				=&k_2(\mv,u)(1+|\mv-u|^2)^{\f{\beta}{2}}\cdot e^{(\f{1}{4T_M}-\f{1}{4T^0}-\zeta)(|\mv-\bar{\fu}^0|^2-|u-\bar{\fu}^0|^2)}\cdot e^{(\f{1}{2T_M}-2\zeta)\bar{\fu}^0\cdot (\mv-u)}.
			\end{align}
			Noting \eqref{ta}, we have $$-\f{1}{4T^0}<\f{1}{4T_M}-\f{1}{4T^0}-\zeta<\f{1}{4T_M}-\f{1}{4T^0}<\f{1}{4T^0},$$
			then it follows from \eqref{5.0-1} and \cite{Gl} that 
			\begin{align*}
				&k_{M,2}(\mv,u)e^{\zeta(|u|^2-|\mv|^2)}\cdot \frac{(1+|u|^2)^{\f{\beta}{2}}}{(1+|\mv|^2)^{\f{\beta}{2}}}\\
				\lesssim &\, \exp\left\{-\f{1}{8T^0}\left(|\mv-u|^2+\f{\left||\mv-\bar{\mathfrak{u}}^0|^2-|u-\bar{\mathfrak{u}}^0|^2\right|^2}{|\mv-u|^2}\right)\right\}\\
				&\times \exp\left\{\left|\f{1}{4T_M}-\f{1}{4T^0}-\zeta\right|\cdot \left||\mv-\bar{\mathfrak{u}}^0|^2-|u-\bar{\mathfrak{u}}^0|^2\right|\right\}\exp\left\{(\f{1}{2T_M}-2\zeta)\bar{\fu}^0\cdot (\mv-u)\right\}\\
				\in  &L_u^1\cap L^2_u.
			\end{align*}  Then we only need to prove that 
			\begin{align*}
				\int_{\R^2}k_{M,2}(\mv,u)(1+|\mv-u|^2)^{\f{\beta}{2}}du\leq C(1+|\mv|)^{-1}.
			\end{align*}
			When $|\mv|\leq C_0$ with $C_0$ a constant satisfying $C_0\gg |u_\tau^0|$, the above equation obviously holds. If $|\mv|\geq C_0,$ we have from \eqref{5.0-1} and \eqref{5.0-2} that there exists a $c_1>0$ such that 
			\begin{align*}
				&\int_{\R^2}k_{M,2}(\mv,u)(1+|\mv-u|^2)^{\f{\beta}{2}}du\nonumber\\
				&\leq C\int_{\R^2}\exp\left\{-\f{c_1}{2}|\mv-u|^2-c_1\f{\left||\mv-\bar{\mathfrak{u}}^0|^2-|u-\bar{\mathfrak{u}}^0|^2\right|^2}{|\mv-u|^2}\right\}\, du\nonumber\\
				&\leq C\int_0^\infty\int_0^\pi r \exp\left\{-\f{c_1}{2}r^2-c_1(r+2|\mv-\bar{\mathfrak{u}}^0|\cos \theta)^2\right\}\,d\theta dr\nonumber\\
				&=C\int_0^\infty r\exp\left\{-\f{3c_1}{2}(r+\f43|\mv-\bar{\mathfrak{u}}^0|\cos\theta)^2\right\}dr\cdot\int_0^\pi \exp\left\{-\f{4}{3}c_1|\mv-\bar{\mathfrak{u}}^0|^2\cos^2\theta\right\}d\theta\nonumber\\
				&\leq C\int_0^{\f{\pi}{4}}\exp\left\{-\f{4}{3}c_1|\mv-\bar{\mathfrak{u}}^0|^2\cos^2\theta\right\}d\theta+C\int_{\f{\pi}{4}}^{\f{\pi}{2}}\exp\left\{-\f{4}{3}c_1|\mv-\bar{\mathfrak{u}}^0|^2\cos^2\theta\right\}\sin\theta \, d\theta\nonumber\\
				&\leq \f{C}{|\mv-\bar{\mathfrak{u}}^0|}.
			\end{align*} 
			Hence we can easily get \eqref{5.2-0}. Therefore the proof of Lemma \ref{lem4.1} is completed. $\hfill\Box$
			
\

		\noindent{\bf Acknowledgments.}  Feimin Huang's research is partially supported by National Key R\&D
		Program of China, grant No. 2021YFA1000800, and National Natural Sciences Foundation of
		China, grant No. 12288201. Yong Wang's research is partially supported by the National Natural Science Foundation of China, grants No. 12022114, No. 12288201, No. 12421001 and CAS Project for Young Scientists in Basic Research, grant No. YSBR-031. The authors would thank the anonymous referees for the valuable and helpful comments on the paper.
		
		\
		
		\noindent{\bf Conflict of interest.} The authors declare that they have no conflict of interest.
		
	\end{document}